\documentclass{article}

\usepackage{amsmath, latexsym, amssymb, amsthm, mathrsfs, xcolor, bbm}

\allowdisplaybreaks

\def\word{\omega}

\newtheorem{theo}{Theorem}[section]
\newtheorem{pro}[theo]{Proposition}
\newtheorem{lem}[theo]{Lemma}
\newtheorem{cor}[theo]{Corollary}

\newcommand{\ra}{\rightarrow}

\theoremstyle{definition}
\newtheorem{defin}[theo]{Definition}
\newtheorem{exa}[theo]{Example}

\newtheorem{exer}[theo]{Exercise}

\theoremstyle{remark}
\newtheorem{rem}[theo]{Remark}

 \def\bee{\begin{equation}}
\def\eee{\end{equation}}
\def\lam{{\lambda}}

\def\1{{\mathbbm 1}}

\def\sE {{\cal E}}
\def\sF {{\cal F}}
\def\sL {{\cal L}}

\def\sB {{\cal B}}
\def\sX {{\cal X}}

\def\R{{\mathbb R}}
\def\D{{\mathbb D}}
\def\E{{{\mathbb E}\,}}
\def\P{{\mathbb P}}

\def\Z{{\mathbb Z}}
\def\F{{\cal F}}
\def\G{{\cal G}}
\def\sC {{\cal C}}
\def\N{{\cal N}}

\def\bD{\mathbb{D}}

\def\bR{\mathbb{R}}
\def\bE{\mathbb{E}}
\def\bN{\mathbb{N}}

\def\bP{\mathbb{P}}

\def\FF{{\cal F}}

\def\angel#1{{\langle#1\rangle}}
\def\norm#1{\Vert #1 \Vert}
\def\ol{\overline}
 \def\qq {\qquad}

\def\eps{\varepsilon}

\def\wh{\widehat}
\def\Lip{{\rm Lip}}

\newcommand{\HH}{\mathbb H}
\def\wt{\widetilde}

\def\group{N}

\numberwithin{equation}{section}

\title{Limit theorems for some long range random walks on torsion free nilpotent groups}

\author{Zhen-Qing Chen, Takashi Kumagai, Laurent Saloff-Coste,\\ Jian Wang and Tianyi Zheng}

\begin{document}

\maketitle

\begin{abstract} We consider a natural class of long range random walks on torsion free nilpotent groups and develop limit theorems
for these walks. Given the original discrete group $\Gamma$  and a random walk $(S_n)_ {n\ge1}$
driven by a certain type of symmetric probability measure $\mu$,
we construct a homogeneous nilpotent Lie group $G_\bullet(\Gamma,\mu)$ which carries an adapted dilation structure and a stable-like process
$(X_t)_{ t\ge0}$ which appears in a Donsker-type functional limit theorem as the limit of a rescaled version of the random walk.  Both the limit group and the limit process on that group depend on the measure $\mu$. In addition, the functional limit theorem is complemented by a local limit theorem. 

\bigskip


\bigskip


\end{abstract}

\tableofcontents

\section{Setting the stage}\label{S:1}

 The aim of this work is to present limit theorems (of both functional and local types) for certain long jump random walks on nilpotent groups.  
Recall  that    a nilpotent group is a group $G$ with identity element $e$ that  has a central series of finite length, that is, there is a finite sequence of normal subgroups
 so that  
 $$
 \{e\} =K_0 \lhd K_1 \lhd  \cdots \lhd K_n=G
 $$
 with $K_{i+1}/K_i$ contained in the center of $G/K_i$ for $0\leq i \leq n -1$.
  See the Appendix for a very brief introduction to nilpotent groups. 
  
  \medskip
  
 Before we explain our particular setup and the tools  and techniques that we will use, we attempt to put this research in perspective by discussing a small selection of related results concerning random walks and limit theorems in finite dimensional vector spaces (i.e., the  torsion free  abelian case), and applications of these classical results to the simplest
example of non-abelian nilpotent groups, the (discrete and continuous) Heisenberg groups $\mathbb H_3(\mathbb Z)\subseteq \mathbb H_3(\mathbb R)$.

The ``limit theorems'' that concern us always have three key ingredients: The first ingredient is
 a discrete random walk $S=\{S_n; n=1, 2, \ldots\}$ on a group $G$ with   independent identically distributed (i.i.d.  in short)  increments of probability distribution $\mu$.
  The second ingredient is a method of renormalization via some sort of ``dilations'' acting on the underlying space $G$. We remain vague here on purpose. The third ingredient is a continuous time process that
  appears  in the limit, call it $Z$. Hopefully, $Z$ has properties that
     make  it relatively easy to study although this entire story can also be viewed as a way to understand $Z$ in terms of the more elementary process $S$. The following fundamental questions arise:
\begin{itemize}
\item  What is the nature of those limiting processes $Z$ that may appear through such a scheme?

\item Given a  limit process $Z$, what are all
the one-step increment probability distributions $\mu$  whose
 associated random walk converges to $Z$ after renormalization?

\item Given
 a one-step increment probability distribution,
how to find a renormalization procedure that leads to an interesting
 non-trivial limit process $Z$,
if any exists?
\end{itemize}

\subsection{Review of some abelian results} \label{S:1.1}

In this section, we discuss some aspects of these vaguely stated questions in the context of finite dimensional vector spaces where detailed answers to the first two questions are well-known and understood. The answer to the first question involves the notions of infinitely divisible probability distribution and L\'evy   process, and the additional notion of operator stability which relates directly to the ``normalization procedure'' that allows us to pass from $S$ to $Z$. See, e.g., \cite[Chapter 6]{Feller}, \cite[Section 1.6]{Hazod2001}  and \cite[Chapter 8]{MS1}.
The second question concerns the ``domain of operator-attraction'' of the limit $Z$ and falls outside the scope of our interest. The third question is not easily answered in general (see \cite{Griffin1986}) but it plays an important role in the results we develop in this work for nilpotent groups. Indeed, for the particular class of examples we treat on nilpotent groups,  a key step consists in identifying appropriate renormalization procedures.

Recall that an $\R^d$-valued random variable $Y$ (or its probability distribution) is said to be
{\em infinitely divisible} if, for each integer $n\geq 1$,
there are i.i.d $\R^d$-valued random variables $\{X_1, \ldots, X_n\}$ such that $\sum_{k=1}^n X_k$ has
the same distribution as $Y$.  It is well known (see, e.g., \cite{Bertoin96, JuMa, MS1, Sato})
that the distribution of $Y$ is infinitely divisible if and only if it is the distribution at time $1$ of a L\'evy process $Z=\{Z_t; t\geq 0\}$ with $Z_0=0$. An infinitely divisible probability
is uniquely characterized by the L\'evy exponent $\phi$
of its characteristic function
$$\phi (\lambda) :=- \log \E \left[ e^{i \lambda \cdot Y} \right],
$$
which takes the following form.
There are a symmetric non-negative definite constant matrix $A=(a_{ij})_{1\leq i, j\leq d}$,
a constant  vector $b=(b_1, \ldots, b_d)$, and a  non-negative  Borel measure $\nu$ on $\R^d \setminus \{0\}$
satisfying $\int_{\R^d} (1\wedge | z|^2) \nu (dz)<\infty$  so that
\begin{equation} \label{e:1.1}
\phi (\lambda) = \frac12 \sum_{i,j=1}^d a_{ij} \lambda_i \lambda_j + \sum_{i=1}^d b_i \lambda_i
+ \int_{\R^d} \left( 1 -e^{i\lambda \cdot z} + i \lambda \cdot z \1_{\{|z|\leq 1\}}\right)
\nu (dz)
\end{equation}
for any $\lambda=(\lambda_1, \ldots, \lambda_d) \in \R^d$.
The triplet $(A,b,\nu)$ and the measure $\nu$ are called the L\'evy triplet and the  L\'evy measure of the infinitely divisible distribution of $Y$, respectively.  They are
uniquely determined by $Y$, and vice-versa. See, e.g., \cite[1.3.2]{Hazod2001}.
 The expression \eqref{e:1.1} is called the L\'evy-Khintchine formula for the infinitely divisible
distribution of $Y$.
We say the random variable $Y$ has no Gaussian part if $A=0$.
Clearly, if the distribution of $Y$ is symmetric, that is, $-Y$ has the same distribution as $Y$,
then $b=0$ and the L\'evy measure $\nu$ is symmetric.
We say that an $\R^d$-valued random variable $X$ is full if there
   is  no non-zero $\lambda \in \R^d$
so that $\lambda \cdot X$ is a constant,   that is,  if  the distribution of $X$ is not supported
on any $(d-1)$-dimensional affine subspace of $\R^d$.
An infinite divisible probability distribution having no Gaussian part
is full if and only if its L\'evy measure $\nu$ is not supported
on any $d-1$ dimensional linear subspace of $\R^d$ ; see \cite[Proposition 3.1.20]{MS1}.  In this work, we are interested
in results involving limits that are symmetric with no Gaussian part  (symmetric random walks with  jumps having heavy tails).

We start with the following elegant result. Let  $S=\{S_n; n\geq 0\}$ be a  random walk in $ \Z^d$ with i.i.d.\ steps  $\xi_k$
having  distribution $\mu$.  That is,
\begin{equation}\label{e:2.1}
\P(\xi_k = (j_1, \dots, j_d))= \mu ((j_1, \dots, j_d))  \quad \hbox{for } (j_1, \dots, j_d)\in \Z^d,
\end{equation}
and $$S_n= \xi_1+\dots+\xi_n.$$

\begin{pro}\label{P:1.1} {\rm (\cite[Corollary 8.2.12]{MS1})}    
Let $\eta$ be a full infinitely divisible probability distribution on $\R^d$ with no Gaussian part and  L\'evy measure $\nu$.
Let $\{S_n; n\geq 0\}$ be a  random walk in $ \R^d$ driven by a probability measure $\mu$ as above.

There are linear operators $A_n: \R^d \to \R^d$ and vectors $b_n\in \R^d$ such that
$A_n S_n+b_n  $ converges in distribution to  $\eta $
 if and only if
\begin{equation}\label{e:2.2a}
n  \, \mu \circ A_n^{-1}  \hbox{ converges   vaguely to } \nu \hbox{ on } \R^d \setminus \{0\}.
\end{equation}
In this case,  $\lim_{n\to \infty} \| A_n\| = 0$.
\end{pro}

Here,  $\mu \circ A_n^{-1} $ is the probability measure on $\R^d$ defined by
$$
 \mu \circ A_n^{-1} (B) = \mu (\{ x\in \R^d: A_n x \in B\})  \quad
\hbox{for every } B\in {\mathcal B} (\R^d ).
$$
 Denote by $C_c(\R^d \setminus \{0\})$ the space of continuous functions  on $\R^d\setminus \{0\}$  with compact support. 
Then  \eqref{e:2.2a} means that 
\begin{equation}\label{e:2.2b}  
\lim_{n\to \infty} n  \int_{\R^d} f(A_n x)  \mu (dx) =  \int_{\R^d} f(x)   \nu (dx)
\quad \hbox{for any }  f \in C_c(\R^d \setminus \{0\}),
\end{equation}
  Note that from the L\'evy-Khintchine formula \eqref{e:1.1}, two infinitely divisible random variables without Gaussian components
and having the same L\'evy measure $\nu$ can only differ by a constant vector.

If \eqref{e:2.2a} holds,  we say the L\'evy measure $\nu$ is  operator-stable (see below) and
the measure $\mu$ (or equivalently $\xi_1$)   belongs to the generalized domain of
attraction of $\eta$ (or  $\nu $, by abuse of language). The matrix
$A_n$ is automatically invertible for all large  $n$. See \cite[Lemma 3.3.25]{MS1}.

\medskip

\begin{rem} \label{R:1.2} \rm Suppose \eqref{e:2.2a} holds with the L\'evy measure $\nu$ not supported in a   $(d-1)$-dimensional  vector subspace and  $\mu$ being  symmetric (that is, $\mu (A)=\mu (-A)$).
\begin{enumerate}
\item[(i)] The vector   $b_n$ in Proposition \ref{P:1.1} can be taken
to be the zero vector in $\R^d$ and the limiting distribution $\eta $ is symmetric.
This is because in this case, $\{S_n; n\in \bN\}$ has the same distribution as $\{-S_n; n\in \bN\}$,
and so, $\{A_nS_n - b_n; n\in \bN\}$ has the same distribution as $\{-(A_n S_n +b_n); n\in \bN\}$.
Consequently, $\{A_n S_n -b_n; n\in \bN\}$ also converges weakly.
It then follows from the characterization of weak convergence
that  $\{A_nS_n  ; n\in \bN\}$ converges weakly to a symmetric random variable $\eta$.

\item[(ii)]   By \cite[Theorem  8.1.5]{MS1} and its proof,
there  are a sequence of invertible $d\times d$ matrices $ (M_n)_{n\geq 1}$ that keeps the distribution of
$\eta $ invariant
(that is, $   M_n \eta$ has the same distribution as $\eta$   for each $n\geq 1$)
and a $d\times d$-matrix $E$ with real entries such that $\wt A_n:= M_n A_n$   satisfies 
\begin{equation}\label{e:1.4}
\lim_{n\to \infty} \wt A_{[\lambda n]} {\wt A_n}^{-1} =\lambda^E
\quad \hbox{for all } \lambda >0,
\end{equation}
and $  {\wt A_n}S_n  $   converges in distribution to $\eta$    as $n\to \infty$.
 Here, $[a]$ stands for the largest integer not exceeding the real $a$.

 Using the independent
stationary increments property of random walks, we can easily deduce from Proposition \ref{P:1.1}that
both $\{  A_n S_{[nt]}; t\geq 0\}$ and
$\{ \wt A_n S_{[nt]}; t\geq 0\}$ converge  in finite dimensional distributions to the symmetric
L\'evy process $Z=\{Z_t; t\geq 0\}$ with $Z_1$ having the same distribution as $\eta $;
 see the proof of Proposition \ref{P:1.3}.  
  Further, $Z$ has the following scaling property by \eqref{e:1.4}:
 for any $\lambda >0$,
$$
\{Z_{\lambda t}; t\geq 0\} \hbox{ has the same distribution as }
\lambda^E Z=\{ \lambda^E Z_t; t\geq 0\} .
$$
See \cite[Example 11.2.18]{MS1} and \cite[p.625]{MS2}.
For this reason, the L\'evy process $Z$ is called an operator-stable process (or operator-L\'evy motion)
  and its L\'evy measure, $\nu$, is also said to be operator-stable 
in the literature. If $E= \alpha^{-1} I_{d\times d}$, where $I_{d\times d}$ denotes the $d\times d$ identity matrix,
$\lambda^{E}= \lambda^{1/\alpha} I_{d\times d}$. In this case, $Z$ is an $\alpha$-stable L\'evy process on $\R^d$.

\item[(iii)] The matrices $\{A_n; n\in \bN \}$ and the limiting L\'evy measure $\nu$ in \eqref{e:2.2a}
are not unique. Suppose \eqref{e:2.2a} holds. Then for any non-degenerate matrix $M$,
we clearly have that $n  \, \mu \circ (MA_n)^{-1} $ converges   vaguely to $\nu \circ M^{-1}$ on
$ \R^d \setminus \{0\}$. Thus $\nu$ depends not only on $\mu$ but also on the ``dilation structure" $A_n$.
  \qed
\end{enumerate}
\end{rem}

Denote by $\D([0, \infty); \R^d)$ the space of right continuous $\R^d$-valued functions on $[0, \infty)$
having left limits. We refer the reader to \cite{GS80}
for the definition of ${\cal J}_1$-topology on the Skorohod space $\D([0, \infty); \R^d)$.

\begin{pro}\label{P:1.3}
 Suppose that the one step distribution $\mu$ of the random walk $\{S_n; n=0, 1, 2, \ldots\}$
is symmetric and satisfies condition \eqref{e:2.2a}. Let $\eta$ be an infinitely divisible symmetric probability distribution with no Gaussian component and L\'evy measure $\nu$.
Let $Z=\{Z_t; t\geq 0\}$ be the symmetric
  L\'evy process on $\R^d$ so that $Z_1$ has distribution $\eta$.
 Then $\{ A_n S_{[nt]}; t\geq 0\}$ converges   weakly in the Skorohod space $\D([0, \infty); \R^d)$
equipped with ${\cal J}_1$-topology to the L\'evy process $Z$ as $n\to \infty$.
\end{pro}

\proof  Let $\wt A_n=M_n A_n$ be defined as in Remark \ref{R:1.2}(ii),  where $(M_n)_{n\geq 1}$ is a sequence of invertible matrices that keeps
the distribution of  $\eta$ invariant. We know from \cite[Theorem 4.1]{MS2}
that  $\{ \wt A_n S_{[nt]}; t\geq 0\}$ converges   weakly in the Skorohod space $\D([0, \infty); \R^d)$
equipped with ${\cal J}_1$-topology to $Z$ as $n\to \infty$.
By \cite[Theorem 3.2.10]{MS1}, $(M_n)_{n\geq 1}$ is relatively compact in the spaces of invertible $d\times d$-matrices.
Thus for any subsequence of $(n)_{n\geq 1}$, there is a sub-subsequence $(n')_{n'\ge1}$  so that $M_{n'}$ converges to
a non-degenerate $d\times d$-matrix $M$ that also keeps the distribution of $\eta$ and hence its
 L\'evy measure $\nu$ invariant. Note that  $A_n = M_n^{-1} \wt A_n$ and the L\'evy process
 $M^{-1}Z$ is of the same distribution as that of $Z$.
It follows that   $\{   A_{n'} S_{[n't]}; t\geq 0\} $ converges
weakly in the Skorohod space $\D([0, \infty); \R^d)$
equipped with ${\cal J}_1$-topology to $ Z$ as $n'\to \infty$.
Since this holds for any subsequence of $n$,  we conclude that $\{   A_{n} S_{[nt]}; t\geq 0\} $
converges weakly to $Z$ as  $n\to \infty$.   \qed

\medskip

The two propositions above and the accompanying remarks  tell us that, if we expect that a given symmetric measure $\mu$ on $\mathbb Z^d$ drives a random walk whose functional limit process  $Z=\{Z_t; t\geq 0\}$ has no Gaussian part and L\'evy measure $\nu$, we should concentrate on finding the  sequence of invertible matrices $A_n$ such that
(\ref{e:2.2a}) holds. Indeed, that property is necessary and sufficient for the desired limit theorems to hold.

\subsection{Illustrative examples on nilpotent matrix groups}\label{S:1.2} 

In this section,  we describe some
illustrative examples, let us emphasize that, although one can easily formulate  versions of Proposition \ref{P:1.1} in the context of certain nilpotent groups, it is not known if such generalizations hold true.  In a similar vein, in $\mathbb R^d$, a full operator-stable L\'evy process  always admits a smooth density whereas in the context of nilpotent group, it is not known if a full  operator-stable L\'evy process always has a density. For details on how to formulate these questions more precisely on nilpotent groups, see \cite[Chapter 2]{Hazod2001}.

 \begin{exa} \label{E:1.4}  \rm
In this example,  we consider a random walk on $\Z^3$ with  i.i.d.\ steps $\xi_k$ distributed according to the probability measure $\mu$ concentrated along the coordinate axes of $\Z^3$  given by
\begin{align*}
\mu  ((i_1, i_2, i_3))   = &   \frac{\kappa_1}{(1+|i_1|)^{1+\alpha_1}}  \1_{\{i_2=i_3=0\}} +
 \frac{\kappa_2}{(1+|i_2|)^{1+\alpha_2}}  \1_{\{i_1=i_3=0\}}  \\
 &  +   \frac{\kappa_3}{(1+|i_3|)^{1+\alpha_3}}  \1_{\{i_1=i_2=0\}}
 \quad \hbox{for } (i_1, i_2, i_3 ) \in \Z^3 \setminus \{ {0} \} .
\end{align*}
We assume $\alpha_i\in (0,2)$, $i=1,2,3$. Let
$  A_n=
\begin{pmatrix}
    n^{-1/\alpha_1}      & 0  & 0    \\
    0      & n^{-1/ \alpha_2} & 0 \\
    0     & 0 &    n^{-1/ \alpha_3}
 \end{pmatrix}
$.
It is easy to check that
$$
n \mu \circ A_n^{-1} \ \hbox{ converges vaguely to }  \nu     \hbox{ on } \R^3 \setminus \{0\},
$$
where
$$
\nu (dx) = \sum_{i=1}^3  \frac{\kappa_i}{|x_i |^{1+\alpha_i}} dx_i
 \otimes_{j\in\{1,2,3\}\setminus\{ i\} }
\delta_{\{0\}} (dx_j) .
$$
Here $\delta_{\{0\}} $ is the Dirac measure concentrated at $0$.
Since $\mu$ is symmetric, by Proposition \ref{P:1.1}and Remark \ref{R:1.2}(i),
$ A_n S_n $ converges weakly  to a random vector whose distribution is  symmetric infinite divisible with no Gaussian part
 and L\'evy measure $\nu$.
 By Proposition \ref{P:1.3},
$\{ A_n S_{[nt]}; t\geq 0\} $ converges weakly in the Skorohod space $\D( [0, \infty); \R^3)$
 equipped with ${\cal J}_1$-topology  to the purely discontinuous symmetric L\'evy process $Z=\{Z_t; t\geq 0\}$ having  $\nu$ as its L\'evy measure. Note that the coordinate processes of  $Z=(Z^{(i)})_1^3$ are independent with $Z^{(i)}$ being a one-dimensional  symmetric $\alpha_i$-stable process, $1
\le i\le 3$.

 As pointed out earlier in Remark \ref{R:1.2}(iii),
it is worth noting that the choice of $A_n$ above is not unique even though it seems most natural in this example. To simplify the discussion, assume that $\alpha_1 <\alpha_2<\alpha_3$. Let $(e_1,e_2,e_3)$ be the canonical basis of $\mathbb R^3$ used implicitly above. Construct a linear operator $B_n$ as follows. First,
set $B_ne_1=n^{-1/\alpha_1}e_1$. Second, pick an arbitrary vector  $e'_2$ which is linearly independent from $e_1$ and belongs to the plane spanned by $e_1$ and $e_2$, and set
$B_ne'_2=n^{-1/\alpha_2}e'_2$.  Finally, pick an arbitrary non-zero vector $e'_3$ that does not belong to the plane
  spanned by $e_1$ and $e_2$, and set $B_n e'_3=n^{-1/\alpha_3}e'_3$.  Then $n\mu\circ B_n^{-1}$ converges vaguely to a L\'evy measure
$\nu'$ having essentially the same form as $\nu$ but carried by the axes associated with $e_1,e'_2,e'_3$.
 More precisely,
  $$
\nu '  (dx) =  \sum_{i=1}^3  \frac{c_i}{|x'_i |^{1+\alpha_i}} dx'_i
 \otimes_{j\in\{1,2,3\}\setminus\{ i\} }
\delta_{\{0\}} (dx'_j) ,
$$
where $(x_1', x_2', x_3')$ is the coordinate  of $x\in \R^3$  under the coordinate system $(e_1, e_2', e_3')$. 
  Note that the L\'evy measure $\nu'$ is thus a linear transformation of $\nu$.

\end{exa}

\medskip

 \begin{exa} [Random walk on the Heisenberg group $\HH_3 (\Z)$]  \label{E:1.5} \rm
Recall that the discrete  Heisenberg group $\HH_3 (\Z)$ is the family of upper triangle matrices of the form
  $
\begin{pmatrix}
    1      & x  & z    \\
    0      & 1 & y   \\
     0     & 0   & 1
\end{pmatrix} $, with $x, y, z\in \Z$,
equipped with matrix multiplication; that is
$$
\begin{pmatrix}
    1      & x_1 & z_1   \\
    0      & 1 & y _1  \\
     0     & 0   & 1
\end{pmatrix}
\cdot
\begin{pmatrix}
    1      & x_2 & z_2   \\
    0      & 1 & y _2  \\
     0     & 0   & 1
\end{pmatrix}
=
\begin{pmatrix}
    1      & x_1+x_2 & z_1 +z_2+ x_1 y_2  \\
    0      & 1 & y _1 +y_2 \\
     0     & 0   & 1
\end{pmatrix} .
$$
For $a=
\begin{pmatrix}
    1      & x  & z    \\
    0      & 1 & y   \\
     0     & 0   & 1
\end{pmatrix}$,  its inverse $a^{-1}$ is  
$ \begin{pmatrix}
    1      & - x & xy-z   \\
    0      & 1 &  -y   \\
     0     & 0   & 1
\end{pmatrix}$.
If we identify matrix $a$ with $(x, y, z)$, then the discrete  Heisenberg group $\HH_3 (\Z)$ can be identified with $\Z^3$
equipped with the group multiplication
\begin{equation} \label{e:3.1}
(x_1, y_1, z_1) \cdot (x_2, y_2, z_2) := (x_1+x_2, y_1+y_2, z_1+z_2 + x_1 y_2).
\end{equation}
We will use this realization of $\HH_3 (\Z)$. This is one of the simplest example of a non-abelian nilpotent group.

Let $e_1=(1, 0, 0)$, $e_2=(0, 1, 0)$ and $e_3=(0, 0, 1)$, which are generators of $\HH_3 (\Z)$.
Note that for $k\in \Z\setminus \{0\}$, $e_1^k=(k, 0, 0)$, $e_2=(0, k, 0)$ and $e_3^k=(0, 0, k)$.
Let  $\alpha_k \in (0, 2)$ be a constant, $1\leq k  \leq 3$,   and write $ \boldsymbol   \alpha=(\alpha_1, \alpha_2, \alpha_3)$.
Consider the following probability measure on $\HH_3 (\Z)=\Z^3$:
$$
\mu_{ \boldsymbol \alpha}(g)= \sum_{i=1}^3 \sum_{n\in \Z}\frac{\kappa_i}{(1+|n|)^{1+\alpha_i}}\1_{\{e_i^n\}}(g),\quad g\in \HH_3 (\Z)  ,
$$
where $\kappa_j$, $1\leq j\leq 3$, are appropriate positive constants.
Let
$$
\Big( \xi_k =(\xi_k^{(1)}, \xi_k^{(2)}, \xi_k^{(3)}) \Big)_{ k\geq 1 }
$$
 be an i.i.d sequence of random variables taking values in $\HH_3 (\Z)$ of distribution $\mu_{ \boldsymbol \alpha}$. Then
$$
S_n =S_0 \cdot \xi_1 \cdot \ldots \cdot \xi_n, \quad n\geq 1,
$$
 defines a random walk on the Heisenberg group $\HH_3 (\Z)$.
Write  $S_n$ as $(X_n, Y_n, Z_n)$. By \eqref{e:3.1},
\begin{equation}\label{e:3.2}
X_{n+1}=X_n + \xi_{n+1}^{(1)}, \quad Y_{n+1}=Y_n + \xi_{n+1}^{(2)},
\quad
Z_{n+1}=Z_n + \xi_{n+1}^{(3)} + X_n \xi_{n+1}^{(2)}.
\end{equation}
If we define $\wh Z_n=Z_0+\sum_{k=1}^n \xi_k^{(3)}$, then
\begin{equation}\label{e:3.3}
Z_n=\wh Z_n + \sum_{k=1}^n X_{k-1} \xi_k^{(2)}
=\wh Z_n + \sum_{k=1}^n X_{k-1} (Y_k-Y_{k-1}), \quad n\geq 1.
\end{equation}

We know from Example \ref{E:1.4}
that
\begin{equation}\label{e:3.4}
\Big\{ \left(  n^{-1/\alpha_1} X_{[nt]},  \,  n^{-1/\alpha_2} Y_{[nt]},  \, n^{-1/\alpha_3} \wh Z_{[nt]}, \right) ; t\geq 0 \Big\}
\Longrightarrow \{(\bar X_t,  \bar Y_t, \bar Z_t), \, t\geq 0\}
\end{equation}
weakly in the Skorohod space $\D ([0, \infty), \R^3)$ equipped with ${\cal J}_1$-topology as $n\to \infty$,
where   $\bar X$, $\bar Y$, $\bar Z$  are symmetric $\alpha_1$-, $\alpha_2$- and $\alpha_3$-stable  processes on $\R$, respectively,
 and they are independent.
 For simplicity, let
$$
\wt X^n_t:=  n^{-1/\alpha_1} X_{[nt]}, \quad   \wt Y^n_t :=n^{-1/\alpha_2} Y_{[nt]}
 \quad \hbox{and} \quad \wt Z^n_t :=n^{-1/\alpha_2} \wh Z_{[nt]}.
$$

Now, we can use the following key facts.  L\'evy processes are semimartingales so stochastic integrals such as L\'evy area
$$\int_0^t  \bar X_{s-} d\bar Y_s$$   are well-defined. Furthermore, \cite[Theorem 7.10]{KP}  shows that
 $$
\left\{ \Big(   \wt X^n_t, \wt Y^n_t,   \wt Z^n_t,   \int_0^t  \wt X^n_{s-} d \wt Y^n_s  \Big); \,  t\geq 0 \right\}
$$
converges weakly in the Skorohod space $\D([0, \infty);    \R^4)$ equipped with ${\cal J}_1$-topology  as $n\to \infty$ to
\begin{equation}\label{e:3.5}
\left\{  \left( \bar X_t, \bar Y_t,  \bar Z_t, \int_0^t \bar X_{s-} d \bar Y_s \right); \,  t\geq 0 \right\} .
\end{equation}

Indeed, to prove (\ref{e:3.5}), for any $\delta >0$, let  $h_\delta (r)=(1-\delta /r)^+$.
Define
$$
\wt X^{n, \delta}_t = \wt X^n_t -\sum_{0<s\leq t} h_\delta (|\Delta \wt X^n_s|) \Delta \wt X^n_s
\quad \hbox{and} \quad
\wt Y^{n, \delta}_t = \wt Y^n_t -\sum_{0<s\leq t} h_\delta (|\Delta \wt Y^n_s|) \Delta \wt Y^n_s.
$$
One can define $\wt Z^{n, \delta}$ in a similar way.
Observe that $\wt X^{n, \delta}$ ,  $\wt Y^{n, \delta }$ and $\wt Z^{n, \delta }$ are again   symmetric random walks but with i.i.d step  sizes
$$
\left\{  \left(1-h_\delta ( n^{- 1/\alpha_j} n^{- 1/\alpha_j}\xi^{(j)}_k) \right)
   n^{- 1/\alpha_j}  
\xi^{(j)}_k ; k\geq 1 \right\} \quad \hbox{for }
j=1, 2, 3,
$$
respectively.
 Let $[\wt X^{n, \delta}]$,  $[ \wt Y^{n, \delta }]$ and
$[ \wt Z^{n, \delta }]$ denote the quadratic variation processes
of  the square integrable martingales $\wt X^{n, \delta}$,  $\wt Y^{n, \delta }$ and $\wt Z^{n, \delta }$, respectively.
Note that \begin{eqnarray*}
&& \E \left(  \big[ \wt X^{n, \delta}\big]_t  \right) \\
&&=    [nt] \, \E \left[ \left(1-h_\delta ( n^{- 1/\alpha_1} \xi^{(1)}_1 )\right)^2
 \left(n^{-1/\alpha_1} \xi^{(1)}_1 \right)^2   \right] \\
&&\leq    c_1 \kappa_1 n^{-2/\alpha_1}  [nt]  \left(  \sum_{k=1}^{[n^{1/\alpha_1} \delta]}  \frac{k^2}{(1+k)^{1+\alpha_1}}  +
  \sum_{[n^{1/\alpha_1} \delta] +1}^\infty \frac{n^{2/\alpha_1} \delta^2}{k^2} \frac{k^2}{(1+k)^{1+\alpha_1}}  \right)\\
  &&\leq  c_1 \kappa_1 n^{-2/\alpha_1}  [nt]  \left(  \frac{[n^{1/\alpha_1} \delta]^{2-\alpha_1}}{2-\alpha_1} +
      \frac{n^{2/\alpha_1} \delta^2}{\alpha_1 (1+[n^{1/\alpha_1} \delta])^{\alpha_1}}      \right)\\
     &&\leq   c_1 \kappa_1 n^{1 -2/\alpha_1}   t  \left( \frac{(n^{1/\alpha_1} \delta)^{2-\alpha_1}}{2-\alpha_1} +
     \frac{n^{2/\alpha_1} \delta^2}{\alpha_1 ( n^{1/\alpha_1}  \delta )^{\alpha_1} }        \right)  \\
     &&= \frac{2 c_1 \kappa_1 \delta^{2-\alpha_1} }{\alpha_1 (2-\alpha_1)} \, t,
     \end{eqnarray*}
where  $c_1>0$ is a constant   independent of $n$ and $\delta$.
In the same way, there is a constant $c_k>0$, $k=2, 3$,  independent of $n$ and $\delta$ so that
$$ \E \left(  \big[ \wt Y^{n, \delta}\big]_t  \right)   \leq \frac{2 c_2 \kappa_2 \delta^{2-\alpha_2} }{\alpha_2 (2-\alpha_2)} \, t
\quad \hbox{and} \quad
 \E \left(  \big[ \wt Z^{n, \delta}\big]_t  \right)   \leq \frac{2 c_3 \kappa_3 \delta^{2-\alpha_3} }{\alpha_3 (2-\alpha_3)} \, t
$$
for all $n\geq 1$ and $t>0$.
So  these three sequences of square integrable martingales $\{\wt X^n; n\geq 1\}$,  $\{\wt Y^n; n\geq 1\}$
and $\{\wt Z^n; n\geq 1\}$  have uniformly controlled
variations  in the sense of \cite[Definition 7.5]{KP}.  Thus by taking  $ (\wt X^n_t, 0)$ and $(\wt Y^n, \wt Z_n)$ for  the
vector-valued process $H^n  $ and $X^n$ in  \cite[Theorem 7.10]{KP}, we conclude
$\{ (\wt X^n_t, \wt Y^n_t,   \wt Z^n_t,   \int_0^t  \wt X^n_{s-} d \wt Y^n_s); t\geq 0\}$ converges weakly in the Skorohod space
$\D([0, \infty); \R^4)$ equipped with ${\cal J}_1$-topology to $\{ (\bar X_t,  \bar Y_t, \bar Z, \int_0^t \bar X_{s-} d \bar Y_s); t\geq 0\}$.
This proves the claim \eqref{e:3.5}.

Using the almost sure Skorohod representation theorem, we can assume without loss of generality that
$$
\left\{  \Big(\wt X^n_t, \wt Y^n_t,   \wt Z^n_t,   \int_0^t  \wt X^n_{s-} d \wt Y^n_s \Big);  \ t\geq 0 \right\}
$$
 converges a.s. in the Skorohod space
$\D([0, \infty); \R^4)$ as $n\to \infty$  to
$$
\left\{  \Big(\bar X_t,  \bar Y_t, \bar Z_t, \int_0^t \bar X_{s-} d \bar Y_s \Big); \  t\geq 0 \right\}.
$$
Consequently, we have  the following conclusions. The weak convergence below
 (denoted by $\Longrightarrow$) is in the Skorohod space $\D([0, \infty); \R^3)$ equipped with
${\cal J}_1$-topology.

\medskip

(i) If $1/\alpha_3< 1/\alpha_1 +1/\alpha_2$,
\begin{eqnarray*}
\lefteqn{ \Big\{ \left(  n^{-1/\alpha_1} X_{[nt]},  \,  n^{-1/\alpha_2} Y_{[nt]},  \, n^{-1/\alpha_1 -1/\alpha_2}   Z_{[nt]} \right) ; t\geq 0 \Big\}} \hspace{1in}&& \\
& \Longrightarrow&  \Big\{  \Big( \bar X_t,  \bar Y_t, \int_0^t \bar X_{s-} d\bar Y_s \Big) ; \, t\geq 0 \Big\}
\quad \hbox{as } n\to \infty ;
\end{eqnarray*}

(ii) If $1/\alpha_3= 1/\alpha_1 +1/\alpha_2$,

\begin{eqnarray*}
\lefteqn{ \Big\{ \left(  n^{-1/\alpha_1} X_{[nt]},  \,  n^{-1/\alpha_2} Y_{[nt]},  \, n^{-1/\alpha_3}   Z_{[nt]} \right) ; t\geq 0 \Big\} }
\hspace{.5in} &&\\
 &\Longrightarrow &
 \Big\{  \Big( \bar X_t,  \bar Y_t, \bar Z_t + \int_0^t \bar X_{s-} d\bar Y_s \Big); \, t\geq 0 \Big\}
 \quad \hbox{as } n\to \infty ;
 \end{eqnarray*}

(iii) If $1/\alpha_3> 1/\alpha_1 +1/\alpha_2$,
\begin{eqnarray*}
\lefteqn{  \Big\{ \left(  n^{-1/\alpha_1} X_{[nt]},  \,  n^{-1/\alpha_2} Y_{[nt]},  \, n^{-1/\alpha_3}   Z_{[nt]} \right) ; t\geq 0 \Big\} }\hspace{1in}&&\\
&  \Longrightarrow & \Big\{  \Big( \bar X_t,  \bar Y_t, \bar Z_t   \Big) ; \, t \geq 0 \Big\} \quad \hbox{as } n\to \infty.   \end{eqnarray*}
  \end{exa}

Let us interpret the results above in group theoretical terms. In the treatment above, we have taken
the coordinate components of the measure $\mu$ and considered the one-dimensional random walks, $X$, $Y$, $Z$, independently of each other.  We have then reconstructed the group law effect of the random walk on $\mathbb H_3(\mathbb Z)$ by considering the L\'evy area generated by the $X$ and $Y$ components. This is easy to do in this case because the $Z$ component commutes with anything else (it is in the center of the group).  Now, the renormalization process involved   
 is to make particular somewhat ad-hoc choices of scalings.

 In the first two cases, (i)-(ii), we used  the anisotropic dilations
$$ \delta_t((x,y,z))=(t^{-1/\alpha_1} x,t^{-1/\alpha_2}y,t^{-(1/\alpha_1+1/\alpha_2)}z), \quad t>0.$$
This one parameter group of diffeomorphisms  has the very special property of being a one parameter group of
 automophisms of $\mathbb H_3(\mathbb R)$. That is,
$$
\delta_t((x,y,z)\cdot (x',y',z'))=\delta_t((x,y,z))\cdot\delta_t ((x',y',z')).
$$
 The consequence of this property is that the limit processes obtained above,
$ \Big\{  \Big( \bar X_t,  \bar Y_t,  \int_0^t \bar X_{s-} d\bar Y_s \Big); t\geq 0 \Big\} $ in case (i),  $\Big\{  \Big( \bar X_t,  \bar Y_t, \bar Z_t + \int_0^t \bar X_{s-} d\bar Y_s \Big);  t\geq 0 \Big\} $ in case (ii), are  symmetric L\'evy processes on the real nilpotent group $\mathbb H_3(\mathbb R)$ which are operator-stable with respect to the one parameter group of automorphisms $\{\delta_t
: t>0\}$. See \cite[Chapter 2, Definition 2.3.13]{Hazod2001}.

In the third case when $1/\alpha+1/\alpha_2<1/\alpha_3$, we used
$$
\delta_t((x,y,z))=(t^{-1/\alpha_1} x,t^{-1/\alpha_2}y,t^{-1/\alpha_3}z), \quad  t>0.
$$
These diffeomorphisms  are not automorphisms of $\mathbb H_3(\mathbb R)$ and it follows that using them
in rescaling the random walk driven by $\mu$ on $\mathbb H_3(\mathbb Z)\subset \mathbb H_3(\mathbb R)$ produces a non-trivial change in the underlying group structure. This is visible in the nature of the limiting process,
$ \Big\{  \Big( \bar X_t,  \bar Y_t, \bar Z_t   \Big) ; \, t \geq 0 \Big\} $, which is not a L\'evy process on $\mathbb H_3(\mathbb R)$ but a L\'evy process on the abelian group $\mathbb R^3$.

Although it is certainly possible to push this approach further in specific examples, there are serious difficulties in treating large classes of examples in this way. For this reason, the approach presented in this monograph is quite different. It does not involve explicitly the stochastic calculus involved in studying the L\'evy area and higher degree functionals of the same type that are known to appear when expressing random walks on nilpotent groups in coordinates.  The interested reader might try the following two informal exercises before reading further.

 \medskip

\begin{exer} \label{Exer:1} \rm 
Pick a tuple of 10 elements $(s_1,\dots,s_{10})$ in either $\mathbb Z^3$ or in
$\mathbb H_3(\mathbb Z)$, $s_i=(x_i,y_i,z_i)$,  and a tuple of ten reals $\alpha_i\in (0,2)$, $1\le i\le 10$.
 Consider the probability measure
 $$\mu(g)= \sum_{i=1}^{10} \sum_{n\in \Z}\frac{\kappa_i}{(1+|n|)^{1+\alpha_i}}\1_{\{s_i^n\}}(g).$$
 What to do to formulate a limit theorem? in $\mathbb Z^3$? in $\mathbb H_3(\mathbb Z)$?
\end{exer} 

 \medskip

\begin{exer} \label{Exer:2} \rm 
Repeat Exercise \ref{Exer:1}  with $\mathbb H_3(\mathbb Z)$ replaced by the group of four by four upper-triangular matrices with diagonal entries equal to $1$ (this group is nilpotent).
\end{exer}

\medskip

As the reader will see,  the approach developed in this work is amenable to detailed computation in concrete cases. Using the theory developed in this monograph, we will revisit Example \ref{E:1.5} in 
 Subsection  \ref{S:7.3}.

\medskip

We close this preliminary section by describing the organization of this monograph. The next section  provides an introduction to our main results while avoiding most technical details. In particular, Subsection \ref{sec-IntroExa} describes special cases which we hope the reader will find both interesting and informative, and Subsection \ref{intro-prior} discusses prior results. 
  Subsection  \ref{S:3.1} introduces polynomial coordinate systems and the key notions of group dilation and approximate group dilation relative to such a coordinate system.  Approximate group dilations
   lead 
 to the definition of ``limit group structures'' and we present some basic properties of these limit group structures that are important for our purpose.  Section \ref{3-4weak}
  introduces   the vague convergence of a probability measures under rescaling by an approximate group  dilation   and how the vague limit and the limit group structure interact  (see Proposition
\ref{weaklimmeas}).  Section \ref{sect-weakc}  describes our main technical results concerning functional limit theorem. It identifies a list of strong hypotheses that allows us to state such a theorem. See Theorem \ref{WT1}.  Section \ref{sec-LCLT}
  presents the corresponding local limit theorem, Theorem \ref{localCLT}. Section 7 describes how to identify in concrete terms (in coordinates), the limiting L\'evy process (on the associated  limit group).
They are then used   together with the main results of this monograph to give several examples on the weak convergence of long range random walks
on various  nilpotent groups.  
Section  \ref{sec-stab} describes the main class of probability measures, $\mathcal{SM}$, to which we want to apply the results obtained in previous   sections.
Section \ref{S:9}   shows
 how to choose appropriate coordinate systems and dilations for measures in $\mathcal{SM}$ whereas Section \ref{S:10}
demonstrates that the hypotheses needed in Sections \ref{sect-weakc}-\ref{sec-LCLT} are  essentially satisfied  by measures in $\mathcal{SM}$.

\subsubsection*{Notation}We use $:=$ as a way of definition.
For $a, b\in \R$, $a\wedge b:= \min\{a, b\}$.
We use  $\delta_{\{x_0\}}$ to denote the Dirac measure concentrated at $x_0\in \R^d$,
and $\1_A$ for the indicator function of a Borel measurable set $A\subset \R^d$.
For an open subset $D\subset \R^d$,  the space of bounded continuous functions on $D$ and the space
of    continuous functions on $D$ with compact support  will be denoted by $C_b(D)$ and $C_c(D)$, respectively.

\section{Introduction}
\subsection{Basic question}

The aim of this work is to prove limit theorems for a class of random walks on nilpotent groups driven by
{ probability  measures allowing for long jumps in certain directions.  The class of   probability  measures we study can be described roughly as follows.  Let $\Gamma$ be a finitely generated nilpotent group with neutral element $e$.  Assume that we are given a finite family of subgroups of $\Gamma$, $H_1,\dots,H_k$, each equipped with a finite symmetric generating set $S_i$ and the associated word-length $|\cdot|_{H_i,S_i}=|\cdot|_i$.  For each $i\in \{1,\dots,k\}$, fix $\alpha_i\in (0,2)$.
On each $H_i$, set $V_i(r)=\#\{g\in  H_i:  |g|_i\le r\}$
and consider the probability measure  on $H_i$:
$$\mu_i(g)= \frac{c_i}{(1+|g|_i)^{\alpha_i} V_i(|g|_i)},
\quad   g\in H_i.  
$$
 Now, on $\Gamma$, consider the symmetric probability measure
 $$\mu=
    \sum_{i=1}^k  \lambda_i \mu_i,  
 $$
  where $\lambda_i$'s  are positive constants with $\sum_{i=1}^k \lambda_i=1$. 
 The class of measures we will treat is slightly larger than what we just described. Two special cases of this construction  are particularly compelling. The first is the case when $k=1$, $H_1=\Gamma$, $S_1=S$ is a finite symmetric generating set for $\Gamma$, and $\mu(g)=\frac{c_\Gamma}{(1+|g|_S)^\alpha
    V_{\Gamma} (|g|_S)} $.
 This is reminiscent of radially symmetric $\alpha$-stable process. The second is the case when each $H_i$ is an infinite cyclic subgroup in $\Gamma$, a case reminiscent of more singular symmetric
     operator-stable  process whose coordinate processes are independent to each other.  See \cite{ChKa,JuMa,MS1,SCZ-nil}.

In an earlier work  \cite{CKSWZ1},
we proved that there  are 
  a   positive constant $\gamma_0=\gamma_0(\mu)$ 
(which can be computed relatively easily from the data) and positive constants
$c=c(\mu),C=C(\mu)$ such that
 $$
  cn^{-\gamma_0}\le \mu^{(n)}(e) \le Cn^{-\gamma_0}.  
$$
Here $\mu^{(n)}$ is the $n$-fold convolution power of the measure $\mu$.
One motivation for the present work is to provide the more precise asymptotic
$$ \lim_{n\ra \infty}  n^{\gamma_0} \mu^{(n)}(e) =a(\mu)$$
with, hopefully, a description of the constant $a(\mu)$. One classical approach to such problems is to find a way to rescale the random walk on $\Gamma$ so as to obtain some sort of limit theorem proving convergence of the law of the rescaled random walk towards the law of a limit process on an appropriate limit space. Typically,
the limit space and the limit process will have some self-similarity properties with respect to some scaling structure.  In the most classical cases, e.g., when $\mu$ is a  symmetric probability measure on $\mathbb Z^d$ which drives a  symmetric random walk  converging towards  some symmetric stable process on $\mathbb R^d$, the limit space supporting the limit process, and its group law, are always the same, $(\mathbb R^d,+)$, independently of $\mu$.  In the present context, one interesting new phenomenon is that the group structure of the limit space supporting the limit process depends not only of the discrete group $\Gamma $ but also on the measure $\mu$.

  \subsection{Description of the basic ingredients and results}

 For simplicity, in this work, we restrict ourselves to random walks on torsion-free finitely generated nilpotent groups,
   that is,  finitely generated nilpotent groups whose only element of finite order is the identity element.
 These countable groups are both similar  to and more complicated than the square lattice $\mathbb Z^d$ in $\mathbb R^d$.  Let $\Gamma$ be such a group. By a celebrated theorem of Malcev \cite{Mal}, the countable group $\Gamma$ can be realized as a co-compact discrete subgroup of a simply connected nilpotent Lie group $G$.  Moreover, any simply connected nilpotent Lie group $G$ can be identified with the $d$-dimensional coordinate space $\mathbb R^d$ equipped with an appropriate group structure whose (multiplication) law is given, in coordinates, by polynomial functions. This  accounts 
  for the similarity with the square lattice in dimension $d$. Note however that the description of $G$ as $\mathbb R^d$ equipped with a polynomial product is very far from being unique (and it may sometimes be difficult to recognize that two such descriptions  give  the same group $G$ up to isomorphism).    One way to understand the complexity of such structures is to attempt to  give a list of all non-isomorphic simply connected nilpotent groups
    in a fix dimension $d$. No such lists exist for relatively large $d$ (we are not aware of such lists when $d$ is greater than $8$). See \cite{BFNT} and the references therein.

 Once $\Gamma$ is represented as a subset of $\mathbb R^d$, a probability measure $\mu$ on $\Gamma$ can be viewed as a weighted series of Dirac masses  on $\mathbb R^d$.  For such a measure $\mu$ in a certain relatively large class of ``stable-like'' probability measures on $\Gamma$, we are going to find an adapted dilation structure $(\delta_t^\mu)_{t>0}$, expressed in coordinates over $G=(\mathbb R^d,\cdot)$ by
 $\delta_t^\mu(u)=(   t^{1/\alpha^\mu_i}  u_i)_1^d$, with  carefully chosen exponents
   $   \alpha_i^\mu\in (0,2),$
  so that the measure
 $$
 \mu_t= t\delta^\mu_{1/t}(\mu):  \ \phi \mapsto  t\int_{\mathbb R^d} \phi(\delta^\mu_{1/t}(u)) \mu(du)
 $$
 has a vague limit $\mu_\bullet$ (a non-negative Radon measure) on $\mathbb R^d\setminus\{0\}$ as $t$ tends to $\infty$.  By construction, the limit $\mu_\bullet$ will satisfy the self-similar property
 $$ 
 (\mu_\bullet)_t=\mu_\bullet    \quad  \hbox{for any } t> 0 .  
  $$
    At the same time, the rescaled group laws
 $$x\cdot_t y= \delta^\mu_{1/t}\left( \delta^\mu_t(x)\cdot \delta^\mu_t(y)\right), \quad x,y\in \mathbb R^d, \quad t>0,$$
  will have a limit as $t$ tends to infinity
  $$\lim_{t\to \infty} x\cdot_ty= x\bullet^\mu y, $$
  which  defines a  group law $\bullet^\mu$ on $\mathbb R^d$.  Most of the time, we will drop the reference to $\mu$ and write $\bullet^\mu=\bullet$ but it is an essential feature of this work that  this limit law actually 
    depends  on $\mu$ via the choice of a proper dilation structure. It will automatically have the self-similar property
  $$
  x\bullet^\mu y= \delta^\mu_{1/t}\left( \delta^\mu_t(x)\bullet^\mu \delta^\mu_t(y)\right)
  \quad \hbox{for  every } x,y\in \mathbb R^d \hbox{ and } t>0.
  $$
  Of course, we are most interested in cases when this can be done in such a way that the symmetric  measure $\mu_\bullet$ is not supported on a proper closed connected subgroup of $G_\bullet=(\mathbb R^d,\bullet^\mu)$.  In general, the limit measure $\mu_\bullet$  defines a left-invariant jump process on the group $G_\bullet$ and the key results of this monograph are:
  \begin{itemize}
  \item  A ``stable-like'' limit theorem expressing the convergence of  the rescaled long jump random walk on $\Gamma$ associated with $\mu$ to the
   left-invariant L\'evy  process on the group $(\mathbb R^d,\bullet^\mu)$   associated with $\mu_\bullet$. 

   \item A  characterization of the   left-invariant L\'evy process
on the  nilpotent  group $(\mathbb R^d,\bullet^\mu)$ associated with $\mu_\bullet$. 

  \item A companion local limit theorem providing a proper statement of convergence relating the densities of the distributions of these processes.
  \end{itemize}
  The reader should be warned that, given $\mu$, the choice of the appropriate dilation structure $(\delta^\mu_t)_{t>0}$ is not unique and that, consequently, we have made various abuse of notation in the explanations given above.

 The simplest instances of these results  are the well-known convergence theorems relating the ``stable-like'' random walk on $\mathbb Z$ associated with
 the probability measure $\mu(x)=c_\alpha (1+|x|)^{-1-\alpha}, x\in \mathbb Z$, $\alpha\in (0,2)$, to the symmetric $\alpha$-stable process on $\mathbb R$, and its rather rich and complex extension to higher dimensions which includes both rotationally symmetric stable processes and some more singular operator stable processes as illustrated in Section \ref{S:1}. See also \cite{JuMa,MS1,MS2}. We note that, in so far as this monograph focusses on a particular class of probability measures,  it only offers a limited extension of these classical abelian theories to nilpotent groups.

    \subsection{Detailed description of some special cases}\label{sec-IntroExa}

  In this section, we spell out in an informal way how our results  of this monograph apply to a series of specific examples that are of particular interest.  These cases all illustrate our main result, Theorem \ref{WTStab},
  which follows from Theorems \ref{WT1} and \ref{localCLT},
   and the discussions in  Subsections  \ref{S:10.2}-\ref{S:10.4}.

  \subsubsection*{Word length radial stable walks}  
  
  On a finitely generated group
   $\Gamma$ equipped with a symmetric finite generating set $S$, 
  the word length $|g|_S$ is the minimal length $k$ of a string $(g_1,\dots,g_k)$ of elements of $S$ such that $g$ is equal to the product of that string, $g=g_1\dots g_k$.  
    By Gromov's polynomial volume growth theorem \cite{Gromov},
  to say that $\Gamma$ has polynomial volume growth is equivalent to the fact that there are an integer $D$ (independent of $S$) and constants $0<c_S\le C_S<\infty$ such that
  $$
  c_S r^D\le\#\{g\in \Gamma: |g|_S\le r\}\le C_Sr^D
  \quad    \hbox{for all  } r\geq 1  .
  $$
 This is known to hold for any  finitely generated nilpotent group; see Subsection \ref{S:11.4}.
       In this context, we call {\em word length radial stable probability measure of index $\alpha\in (0,2)$} the probability measure
  $$
   \mu_{S,\alpha}(g)= \frac{c(\Gamma,S,\alpha)}{(1+|g|_S)^{\alpha+D}}, \quad  g\in \Gamma.
  $$
  It is known (see \cite[Section 5.1]{SCZ-nil} and \cite[Theorem 1,1]{MuSC} as well as the references given therein)
     that there are constants $0<a=a(\Gamma,S)\le A=A(\Gamma,S)<\infty $ such that the iterated convolutions of this measure satisfy
  $$  \frac{a n}{(n+|g|^\alpha_S)^{1+D/\alpha}}\le  \mu^{(n)}_{S,\alpha}(g)\le \frac{A n}{(n+|g|^\alpha _S)^{1+D/\alpha}},
  \quad g\in \Gamma,\,  n\in \mathbb N.$$
  So, one has a remarkably good control of the behavior of the associated random walk. However, there are no existing limit theorems in the literature for such walks, even if we assume that $\Gamma$ is a torsion free nilpotent group (such groups are basic examples of groups with polynomial volume growth).  Our results provide limit theorems (functional, and also local) for any random walk driven by a word length radial stable probability measure  $\mu_{S,\alpha}$, $\alpha \in (0,2)$, on a torsion free finitely generated nilpotent group.  We now briefly describe these results.

  First, because we assume that $\Gamma$ is a finitely generated torsion free nilpotent group, there is a simply connected nilpotent Lie group $G=(\mathbb R^d,\cdot)$ which contains $\Gamma$ as a co-compact discrete subgroup. The Lie algebra, $\mathfrak g$, of this Lie group is equipped with its central descending series
  $$\mathfrak g_1=\mathfrak g\supseteq \mathfrak g_2=[\mathfrak g,\mathfrak g]\supseteq \cdots \supseteq \mathfrak g_j=[\mathfrak g_{j-1},\mathfrak g]\supseteq\cdots  \supseteq \{0\},$$
  and this series become trivial (i.e., constant equal to $\{0\}$) after finitely many steps. Let $j^*$ be the smallest $j$ such that $\mathfrak g_{j+1}=\{0\}$. One can choose a direct sum  decomposition by vector subspaces, $\mathfrak n_i$, $1\le i\le j^*$, compatible with the central descending series above, so that
  $$\mathbb R^d=\mathfrak g=\oplus_{i=1}^{j^*} \mathfrak n_i \quad \mbox{and} \quad \mathfrak g_j=\sum_{ i\ge j} \mathfrak n_i,\; j\in \{1,\dots,j^*\}.$$
  The linear invertible maps
  $$
  \delta_t(x)= t^i x  \quad \hbox{ if } x\in \mathfrak n_i,\;1\le i\le j^*,\quad t>0
  $$
  form an approximate Lie dilation structure in the sense that
  $$[x,y]_\bullet=\lim _{t\ra \infty} \delta_t^{-1}([\delta_t(x),\delta_t(y)])$$
  is a Lie bracket on $\mathbb R^d$ with the property that $\delta_t([x,y]_\bullet)=[\delta_t(x),\delta_t(y)]_\bullet$.
  Using exponential coordinate (of the first type) to represent $G$ as $(\mathbb R^d,\cdot)$, the approximate Lie dilations $\delta_t, t>0,$ define  approximate group dilations on $G$ for which we use the same notation. 
  The limit group $G_\bullet$ is the simply connected Lie group associated with the Lie algebra $(\mathbb R^d,[\cdot,\cdot]_\bullet)$ defined above. 
 It follows from \eqref{e:11.1} that  the volume growth exponent $D$ of the original group $\Gamma$ is given by $D=\sum_{i=1}^{j^*} i\dim(\mathfrak n_i)$. Thus      we have  $\det (\delta _t)=t^D$ for every $t>0$.
In \cite{Pansu1983}, Pansu proves the fundamental results that  there is a norm $\|\cdot\|_\bullet$ on $(\mathbb R^d,\bullet)$, homogeneous with respect to $(\delta_t)_{t>0}$, such that the geometry of $(\Gamma,|\cdot |_S)$ is well approximated at large scale by that of $(\mathbb R^d, \|\cdot \|_\bullet)$ in the sense that
  $$\lim_{g\in \Gamma,g\to \infty}\frac{|g|_S}{\|g\|_\bullet}=1.$$
Further, one has
 $$\lim_{r\to \infty} \frac{\#\{g\in \Gamma: |g|_S\le r\}}{|\{x\in \mathbb R^d:\|x\|_\bullet\le r\}|}=1,$$
 where $|\Omega|$ is the
  Haar    volume of $\Omega\subset G_\bullet$.
  See also, \cite{Breuillard}.
 Haar measures on $G$ and $G_\bullet$ are both Lebesgue measure $dx$ on $\R^d$.
 When considering densities on these groups, we mean densities with respect to $dx$.

 The importance of these results for us is that
  they enable us to  establish the convergence of the measure $t\delta_{1/t^{1/\alpha}}(\mu_{S,\alpha})$, vaguely on $\mathbb R^d\setminus \{0\}$, to the radial stable jump measure $\mu_{\bullet,\alpha}$  with density   $$\phi_{\bullet,\alpha} (x)= \frac{c(\Gamma,S,\alpha)}{\|x\|_\bullet^{\alpha+D}}
.$$
This measure is the jumping  measure of a left-invariant (strong) Markov process $(X^{\bullet}_t)_{t>0}$ on $(\mathbb R^d,\bullet)$ which is self-similar in the sense that
  $( X^\bullet_s)_{s>0}$ equals 
     $(\delta_{1/t^{1/\alpha}}(X^{\bullet}_{ts}))_{s>0}$ in distribution.
   In a proper global coordinate system, the coordinates of this process can be expressed in terms of suitable  stable processes
  	and their (possibly iterated) L\'evy areas.
  The L\'evy  process $X^\bullet$
   admits a continuous  convolution density
 with respect to the Lebesgue measure    on $\R^d$:
$$p_{\bullet,\alpha}(t,x),\quad  (t,x)\in (0,\infty)\times \mathbb R^d.$$
This density satisfies
$$
 p_{\bullet,\alpha} (t,x) =    t^{-D/\alpha}
p_{\bullet,\alpha} (1,\delta_{1/t^{1/\alpha}}(x)) ,\quad  (t,x)\in (0,\infty)\times \mathbb R^d, $$
and
$$\frac{at}{(t+\|x\|_{\bullet}^\alpha)^{1+D/\alpha}} \le p_{\bullet,\alpha}(t,x)\le \frac{ At}{(t+\|x\|_{\bullet}^\alpha)^{1+D/\alpha}},\quad  (t,x)\in (0,\infty)\times \mathbb R^d.$$

In this context, the results developed in this work establish two limit theorems for the random walk $(X_n)_{n\ge 0}$ on $\Gamma$ driven by $\mu_{S,\alpha}$. These limit theorems capture the fact that, after proper rescaling in time and space, the limit of the random walk $(X_n)_{n\ge 0}$ is the Markov process $( X^\bullet_s)_{s>0}$.  Namely, the functional limit theorem establishes the convergence of $(\delta_{1/t^{1/\alpha}}(X _{[st]}))_{s>0}$ to
$( X^\bullet_s)_{s>0}$ as $t$ tends to infinity. In particular, for any continuous function $\phi$ with compact support on $\mathbb R^d$,
$$\sum_{g\in \Gamma} \phi(\delta_{t^{1/\alpha}}(g))\mu_{S,\alpha}^{[(ts])}(g)\to \int_{\mathbb R^d}\phi(x)
p_{\bullet,\alpha}(s, x)dx
$$
as $t$ tends to infinity.
For any compact set $K\subset \mathbb R^d$ and  any functions $g_n:K\to \Gamma$, $n=1,2,\dots$,  such that the sequence of functions
$\delta_{1/n^{1/\alpha}}\circ g_n:K\to \mathbb R^d$, $n=1,2,\dots$, converges uniformly over $K$ to the identity function, the local limit theorem  of this monograph establishes the uniform convergence to zero over $K$ of
   $$n^{D/\alpha}\mu_{S,\alpha}^{(n)}(g_n(x))-p_{\bullet,\alpha}(1,x)$$
when $n$ tends to infinity.
In particular, this shows that, for any fixed $g\in \Gamma$ (e.g., $g=e$),  $$\lim_{n\to \infty} n^{D/\alpha} \mu^{(n)}(g) = p_{\bullet,\alpha}(1,e).$$

 \subsubsection*{Walks taking stable-like steps along one parameter subgroups}

Let $\Gamma$ be a torsion free finitely generated nilpotent subgroup of a simply connected nilpotent group $G$.
One of the cases that motivates our study can be described as follows: We are given a tuple $S=(s_1,\dots,s_k)$ of elements of $\Gamma$, which, together with their inverses,  
  generates  $\Gamma$. We are also given a tuple of reals  $\boldsymbol{\alpha}=(\alpha_1,\dots,\alpha_k)\in (0,2)^k$.  Note that the letter $S$ is used here in a slightly different way than in the previous case.  Now, set
$$\mu_{S,\boldsymbol \alpha}(g)=\frac{1}{k}\sum_{i=1}^k \sum_{m\in \mathbb Z}\frac{c_{\alpha_i}}{(1+|m|)^{1+\alpha_i}}\1_{\{s_i^m\}}(g).$$
It was proved in \cite{SCZ-nil} that, for any such probability measure, there exist $0<a=a(\Gamma,S,\boldsymbol{\alpha})\le A=A(\Gamma,S,\boldsymbol {\alpha})<\infty$ and
   $\gamma_0 =\gamma_0 (\Gamma,S,\boldsymbol{\alpha})$  such that
\begin{equation}\label{D-SCZ}
a n^{ -\gamma_0}\le \mu_{S,\boldsymbol \alpha}^{(n)}(e)\le An^{ -\gamma_0}.
\end{equation}
In Section \ref{sec-stab}, we introduce the space of probability measures $\mathcal{SM}_1(\Gamma)$, see Definition \ref{def-Stab1}, which contains all such measures. We then  explain how to choose a coordinate system  of polynomial type,  $G=(\mathbb R^d,\cdot)$, and an approximate dilation structure $(\delta_t)_{t>0}$ with limit group $G_\bullet=(\mathbb R^d,\bullet)$, which are adapted to the pair $(S,\boldsymbol{\alpha})$, and
  such that, with $\mu_t:=t\delta_{1/t}(\mu_{S,\boldsymbol \alpha})$,   the family of measures
$(\|z\|_2^2 \wedge 1) \mu_t (dz)$  converges weakly on $\mathbb R^d\setminus \{0\}$ to a measure $ (\|z\|_2^2 \wedge 1)
\mu_\bullet (dz)$   as $t\to \infty$, that is, 
$$
\lim_{t\to \infty} \int_{\R^d \setminus \{0\} } f(z) (\|z\|_2^2 \wedge 1) \mu_t (dz)
=  \int_{\R^d \setminus \{0\} }  f(z) (\|z\|_2^2 \wedge 1) \mu_\bullet (dz) 
$$
for any  $  f\in C_b ( \R^d \setminus \{0\})$.
Here $C_b ( \R^d \setminus \{0\})$ denotes  the space of bounded continuous functions on $\R^d \setminus \{0\}$. 
   The   measure $\mu_\bullet$  is supported on the union of a finite number of one parameters subgroups of $G_\bullet$ and its support generates $G_\bullet$.
It can be interpreted as the L\'evy measure of  a  convolution semigroup of probability measures, associated with a left-invariant L\'evy process on $G_\bullet$.
The convolution transition kernel of this semigroup admits a continuous density, $p_\bullet(t,x)$,
 with respect to the Lebesgue measure
on $(\mathbb
 R^d,\bullet)$ and satisfies
 $$
   p_\bullet(t,x)=
    t^{-\gamma_0}   p_\bullet(1,\delta_{1/t}(x))
 \quad \hbox{for } (t,x)\in (0,\infty)\times \mathbb R^d.
 $$
 Note that the  limit objects introduced here, e.g., $G_\bullet$ and  $p_\bullet$, all depend on $S$ and $\boldsymbol \alpha$, even so we did not capture that dependence in the notation used above. A  notable difference with the earlier description of the radial stable-like case is that, in general, there are no particular canonical choices of the approximate dilation structure $(\delta_t)_{t>0}$ and we have not made a canonical choice of coordinates either. To a certain extent, the entire results and the associated limit objects depend on the choices of coordinates and adapted dilation structure while, of course, there are great commonalities shared by all the limit objects obtained based on these different choices.
 This, however, will not be deeply investigated here.

As in the case of radial stable walks, the results of this monograph establish the  convergence of the discrete time random walk driven by $\mu_{S,\boldsymbol \alpha}$, properly rescaled in time and space, to the left-invariant L\'evy process $( X^\bullet_s)_{s>0}$ with convolution density $p_\bullet(t,x)$ mentioned  above. More precisely, the functional limit theorem establishes the   weak  convergence of $(\delta_{1/t}(X _{[st]}))_{s>0}$ to
$( X^\bullet_s)_{s>0}$ as $t$ tends to infinity. In particular, for any continuous function $\phi$ with compact support
  on $\R^d$,
$$\sum_{g\in \Gamma} \phi(\delta_{t}(g))\mu_{S,\alpha}^{[(ts])}(g)\to \int_{\mathbb R^d}\phi(x)p_\bullet(s, x)dx$$
as $t$ tends to infinity. The local limit theorem asserts that
 \begin{equation} \label{D-here}
 \lim_{n\to \infty}\sup_{x\in K}\left|
   n^{\gamma_0}
 \mu_{S,\boldsymbol \alpha}^{(n)}(g_n(x))- p_{\bullet}(1,x)\right|=0,
 \end{equation}
 where $K$ is a compact in $\mathbb R^d$ and $g_n: K\to \Gamma$ is a sequence of functions such that  $\delta_{1/n}\circ g_n:K\to\mathbb R^d$ converges uniformly over $K$ to the identity function.
 Of course, the non-negative real   $\gamma_0$ 
  appearing in (\ref{D-SCZ}) and in (\ref{D-here}) is the same in both equations. It is also given by
 $\det(\delta_t)=  t^{\gamma_0} $.

 \subsubsection*{Walks associated with measure in $\mathcal{SM}(\Gamma)$}
 In Section \ref{sec-stab}, we introduce a particular set of ``stable-like'' measures on $\Gamma$, $\mathcal{SM}(\Gamma)$, which interpolates between the radially symmetric measures considered above and the convex combinations of one dimensional measures described in the last   subsection. These measures 
   were  studied in   our earlier work \cite{CKSWZ1}.  With any measure in $\mathcal{SM}(\Gamma)$ we can associate in a natural way a (non-unique) polynomial coordinate system $G=\mathbb R^d$ and a family of dilations $(\delta_t)_{t>0}$ which define a limit group structure $G_\bullet=(\mathbb R^d,\bullet)$.
 The approximate dilation structure $(\delta_t)_{t>0}$ is
 built   so that the family $\mu_t=t\delta_{1/t}(\mu)$, $t>0$, has well defined limit points which are all L\'evy measures of $(\delta_t)_{t>0}$-stable symmetric convolution semigroups of probability measures on $G_\bullet$ with continuous positive densities on $G_\bullet$. One of the key contributions of this work is to describe explicitly how one can construct such an approximate dilation structure based on a proper description of $\mu$ on $\Gamma$. If it is the case
   that, with $\mu_t:=t\delta_{1/t}(\mu)$, the measure $(\| z\|_2^2 \wedge 1) \mu_t (dz)$
   converges weakly to a finite measure $(\| z\|_2^2 \wedge 1)  \mu_\bullet (dz)$  on $\mathbb R^d\setminus\{0\}$ as $t\to \infty$,
      then we obtain both a functional theorem and local limit theorem.
 The results described in the previous two paragraphs are, in fact,  special cases of these more general theorems.
 The structure of the L\'evy measures of the limit L\'evy processes on $G_\bullet$ appearing in these limit theorem is described at the end of the next subsection.

 \subsection{Symmetric continuous convolution semigroup of probability measure and L\'evy processes}

  For this very minimal vocabulary review, we follow \cite{Hazod2001}.  Let $G$ be a connected Lie group. Recall that there is a one-to-one correspondence
  between  symmetric continuous convolution semigroups of probability measures  on $G$ and symmetric L\'evy processes  on $G$.
Here   $(\mu_t)_{t>0}$ is a symmetric continuous convolution semigroup of probability measures    on $G$  if  the map $t\mapsto \mu_t$ is continuous, $\mu_t*\mu_s=\mu_{t+s}, s,t>0$, $\mu_0=\delta_e$, and  $\mu_t(\phi)=\mu_t(\check{\phi})$ for any continuous   function $\phi$ on $G$ 
with compact support,    where  $\check{\phi}(y):=\phi(y^{-1})$ for $y\in G$.  A symmetric L\'evy process
  $X$ on $G$ is a $G$-valued  time-homogeneous  c\`adl\`ag Markov process $(X_t)_{t\ge 0}$ with
  stationary independent increments, started at $e$ and such that $X_t^{-1}=X_t$ in distribution
  for every $t>0$.  In this setting,
the notion of  infinitesimal generator of $X$
can be captured in a more elementary way via the so-called generating functional (defined on smooth compactly supported functions): if the infinitesimal generator   of the symmetric L\'evy process $X$ is  $\sL$,
 the associated generating functional is simply
$\phi\mapsto \sL\phi(e)$. The L\'evy-Khinchin-Hunt formula provides a description of the generating functional of a L\'evy process. Under the symmetry condition, the generating functional has two parts, a diffusion part, and a jump part described
by a  symmetric measure $\nu$ on $G\setminus \{e\}$
in the form $\phi\mapsto  { \rm p.v.} \int _{G\setminus \{e\}}(\phi(y)-\phi(e)) \nu(dy)$
 with
\begin{equation}\label{LM}
	\int_{G\setminus\{e\}} \min\{1,\|y\|_2^2\} \nu(dy) <\infty.
\end{equation}
 Here, $\|y\|_2$ is the Riemannian distance between $e$ and $y$ in some fixed left-invariant metric on $G$,
  and
$$ {\rm p.v.} \int _{G\setminus \{e\}}(\phi(y)-\phi(e)) \nu (dy):=\frac 12
 \int _{G\setminus \{e\}}(\phi(y) + \phi (y^{-1})-2\phi(e)) \nu (dy).
 $$

 In this work, we are only interested in pure-jump symmetric L\'evy processes, that is, generating functional of the form
 $$\phi\mapsto \sL \phi(e)=  { \rm p.v.}
 \int _{G\setminus \{e\}}(\phi(y)-\phi(e)) \nu(dy)
 $$
 where  $\nu$ is a symmetric measure on $G\setminus \{e\}$ satisfying \eqref{LM}.
Equivalently, the infinitesimal generator is given on smooth compactly supported functions by
 $$
 \langle     - \sL   u, v\rangle =\frac{1}{2}\int_G\int _{G\setminus \{e\} }(u(xy)-u(x))(v(xy)-v(x))\nu(dy)dx.
 $$
  This, of course, is also a description of the associated Dirichlet form (on a dense subspace of its domain).

  In this work, these objects come about through a limit procedure which implies that they have additional properties.
  First, the underlying Lie group is a simply connected nilpotent Lie group which we call $G_\bullet$. Second, by construction, $G_\bullet$ carries a group of dilations, $(\delta_t)_{t>0}$, $\delta_t: G\to G$, $\delta_1=\mbox{Id}$, $\delta_{ts}=\delta_t\circ \delta_s=\delta_s\circ \delta_t$, $s,t>0$, where
     $\delta_t$ is also  a group isomorphism for every $t>0$, and $\lim_{t\to 0}\delta_t(x)=e$ for all $x\in G$.  In addition, the convolution semigroups and associated L\'evy processes of interest to us are self-similar with respect to such a dilation structure, that is, $(X_s)_{s>0}$ equals $(\delta_{1/t}(X_{ts}))_{s>0}$,  in distribution, for any $t>0$.
	 Moreover,
	there is a linear basis $\eps=(\eps_1,\dots,\eps_d)$ of the Lie algebra of $G_\bullet$ in which the dilation $\delta_t$ has the form $\delta_t(\eps_i)= t^{1/\beta_i}\eps_i$, $\beta_i\in (0,2)$, $1\le i\le d$. This last condition, $\beta_i\in (0,2)$, $1\le i\le d$,  is related to the fact the processes in question are pure-jump operator-stable L\'evy processes.   See, e.g., \cite[Theorem 2.3.17]{Hazod2001}.  Finally, when the original random walk is driven by a probability measure $\mu$ in $\mathcal{SM}(\Gamma)$ (a class of stable-like measures on $\Gamma$ described in Section \ref{sec-stab}), the L\'evy measure
   $$
  \mu_\bullet = \lim_{t\to \infty} t \delta_{1/t} (\mu) = \lim_{t\to \infty} t  \mu \circ \delta_t
  $$
 of our limit process has a particular structure that it inherits from the facts that $\mu\in \mathcal{SM}(\Gamma)$.
  Namely, there is a finite family of closed Lie subgroups of $G_\bullet$, call them $H_{\bullet,i}$, $1\le i\le  k$, which are each invariant under $(\delta_t)_{t>0}$, and functions
  $\psi_i: H_{\bullet, i}\to (0,\infty)$ satisfying $t\psi_i (\delta_t(x))=\psi_i(x)$, and $\psi_i(x^{-1})=\psi_i(x)$, $x\in H_{\bullet,i}$, such that
  $$
    \mu_\bullet  (dx)= \sum_1^m \nu_i (dx),   \quad  \nu_i(\phi)=\int_{H_{\bullet,i}}\phi (x) \psi_i(x) d_{H_{\bullet,i}}x      \  \hbox{ for } 1\le i\le m,
  $$
  where $d_{H_{\bullet,i}}x$ is the Haar measure on $H_{\bullet,i}$; see Proposition \ref{pro-deslim}.
  Each $\nu_i$ satisfies (\ref{LM}).
  The group generated by the union $\cup_1^mH_{\bullet,i}$ of the subgroups $\{H_{\bullet,i}, 1\leq i\leq m\}$
   is $G_\bullet$.

  \subsection{Prior results} \label{intro-prior}
  To put our results in perspective, we briefly review limit theorems (functional and/or local) relating random walks on discrete groups to L\'evy processes on a related Lie group.  Very few results of this type exist outside the setting of nilpotent groups (and closely related groups such as groups of polynomial volume growth).
  The classical (functional) limit theorems can be interpreted in two distinct ways:
  \begin{enumerate}
  \item [(i)]
	As providing  approximation of a (continuous time) L\'evy process by a discrete time process. This can be motivated by the desire to actually construct the limiting process, or to simulate it, or to understand it in more concrete terms. In this case, one should read the limit theorem as follows: at each stage, we take a greater number of smaller steps to approximate the behavior of a continuous time process on a fix bounded time interval.
	  A 	natural setup for this interpretation is the triangular array setup.

\item[(ii)] As a result illuminating the long term behavior of a discrete time process by providing a continuous time
  scaling  limit. In this case, at each stage we take a greater number of identically distributed steps and approximate the probability of larger and larger scale events for the discrete time process by the probability of the same large scale events for the limiting continuous time process, at a large time. Whenever that limiting process is self-similar, the limit computation can be rephrased as a computation within a fixed bounded time interval. There is more rigidity in this viewpoint than in the first as we cannot choose the different individual steps taken as one possibly can in a triangular array formulation of the first viewpoint.
\end{enumerate}

  For random walks in $\mathbb R^n$,
 it is somewhat difficult to see the differences between these two interpretations. The reason is that we have a relatively obvious way to turn the identically distributed steps appearing in the second interpretations into smaller and smaller steps appearing in the first interpretation. Indeed, we typically assume that the limiting process is self-similar with respect to a dilation structure that commute with addition and this dilation structure can be used to turn the fixed-size steps of (ii) into the small-size steps of (i).

	\medskip
	
  Both viewpoints are present in this work. Our main focus is on using (ii) to study long term behavior of  a class of discrete long range random walks on a finitely generated torsion free nilpotent group $\Gamma$. 
     One can then 
    use (i) to better understand the limiting self-similar L\'evy processes on the limit nilpotent group $G_\bullet$. See Section \ref{S:7}.

  \subsubsection{Functional type limit theorems} On a general Lie group, there are results stated in terms of triangular arrays that go back to Wehn \cite{WehnThesis,WehnNAS}. Later, Stroock and Varadhan \cite{StrV} rediscovered Wehn's results. These works concern the case when the limit L\'evy process is a diffusion.  These triangular array results have been extended to cover the case when the limit L\'evy process may have jumps. An exposition of such results  is found in \cite{Hazod2001} which contains a very long list of references.   They are also found in work by Kunita \cite{KunPJAconv, KunBR, KunCornell,  KunPJAstab}.    These results must be understood as an extension of the first
  interpretation of the classical limit theorem discussed above. From this viewpoint, the title of the Stroock-Varadhan paper, Limit Theorems for Random Walks on {L}ie Groups, is somewhat misleading. What the results of Wehn and Stroock-Varadhan do is
   to    provide discrete time steps approximations of diffusions on Lie groups. They do not, in general, help us understand the behavior of random walks on Lie groups. That is because, on a general Lie group, there is no clear way to
  turn identically distributed steps into small-size steps. There are, however, many ways to create arrays of smaller and smaller size  steps, not related to any identically distributed model. The theorems described by Wehn, Stroock-Varadhan, Hazod and Siebert, Kunita, and others thus provide functional limit theorems along the line of the first interpretation. See the excellent discussion in \cite{BreuillardRW}.

  There is one setting in which these triangular array limit theorems provide an understanding of random walk (in the sense of a process taking repeated identically distributed steps). This is, informally, when the limiting continuous time process is self-similar with respect to a  dilation structure that preserve the multiplication law of the underlying group.  Unfortunately, this is a rather rare occurrence as the only Lie groups admitting such dilation structures are simply connected nilpotent Lie groups of a very special kind. Moreover, outside the case of diffusion limit, whether or not  a dilation structure exists that is suitable for a given random walk on a given group depends, to a large extend, on the particular random walk in question. In fact, given a driving measure $\mu$, constructing a proper dilation structure for $\mu$ (deciding if such exists!) is a major problem, one that is  completely ignored by the triangular array formulation of limit theorems. This is illustrated by the results of the present work.

 Somewhat independently of the above circle of ideas, Crepel, Raugi, and others, obtained rather satisfying random walk limit theorems for general nilpotent groups in the case the limit is a diffusion \cite{Crepel1978,RaugiZWVG, RaugiHeyer}.
The proofs in these works can be viewed as using two steps: the first step proves the result  in the presence of a canonical  adapted dilation structure (that is, in the case of stratified nilpotent groups). The second step is closely related to one of the key ingredients we will use here and involves the idea behind our definition of an approximate group dilation structure, a dilation structure that does not preserve the group structure of the original underlying Lie group.  In general, because of the second step, the original group carrying the random walk has a group structure that is different from that of the Lie group carrying the limit diffusion. This, clearly, takes us outside the realm  of Wehn-type results.

 One key point  in the results by Cr\'epel and Raugi is that the structure of the group carrying the limit diffusion depend only on the original group, not of the particular (diffusive) random walk one wants to study. In general, this cannot be the case when the random walk to be studied calls for a limit process that has jumps as we do here. As we shall see, in this case, the limit structure depends on both the original group and the particular probability measure that drive the given random walk. One thus has to discover what this proper limit structure is for each studied random walk.

\subsubsection{Local limit theorems}  The first local limit theorem  in the context of general nilpotent groups and groups of polynomial volume growth is due to G. Alexopoulos \cite{AlexoPTRF,  Alexdis, AlexoMem}. See also the discussion in \cite{BreuillardRW}. It concerns centered random walks driven by a finitely supported measure.  For nilpotent groups, following a very different approach, and covering random walks driven by measures that have a high enough finite moment (much higher than 2, in general), the best known results are due to R. Hough \cite{Hough} which provides an informative review of earlier results. We do not know of references treating cases when the limit is not a diffusion process.

\section{Polynomial coordinates and approximate dilations}\label{S:3}

 \subsection{Polynomial coordinate systems}\label{S:3.1}

  Even though some related results can be stated in an intrinsic manner, in practice, limit theorems are coordinate dependent. This applies to the results of this monograph and, consequently, we discuss in some details the notion of global coordinate system for simply connected nilpotent Lie groups.  A number of different choices are possible for this purpose. In this section, we outline basic characteristics of the coordinate systems  we will use. A given group $G$ can be described via many different such global polynomial coordinate charts and it is often desirable to allow for such a choice to be made by circumstances.   This is discussed further in Section \ref{S:9}.
  
  \medskip

  A simply connected nilpotent Lie group $G$ can always be described by a global coordinate chart $\mathbb R^d \to G$, $0\to e$, in which the group multiplication and inverse map
 are given by polynomials
 $$x\cdot y= P(x,y)=\left(p_1(x,y),\dots,p_d(x,y)\right),\quad  x^{-1}=Q(x)=\left(q_1(x),\dots,q_d(x)\right).
$$
 As  $P(x, 0)=x$ and $P(0, y)=y$  for any $x, y\in\R^d$, we have
\begin{equation}\label{e:3.1a}
p_i (x, y) = x_i+y_i+ \bar p_i (x, y), \quad 1\leq i \leq d,
\end{equation}
where $\bar p_i (x, y)$'s are polynomials having no constant nor first order terms.
Moreover,
for any compact $K\subset \mathbb R^d$, there is a constant $C_K$ such that
  \begin{equation}
	  \|x^{-1}\cdot y\|_2 =|P(Q(x),y)|\le C_K \|x-y\|_2  \quad \hbox{for every } x,y\in K,
	  \label{Lip1}
	\end{equation}
because $P(Q(x),x)=0$.  Here $\|\cdot\|_2$ is the canonical Euclidean norm in $\mathbb R^d$.
Similarly, for any compact $K\subset \mathbb R^d$, there is a constant $C'_K$ such that
  \begin{equation}  \label{Lip2}
	   \|x-y\|_2 \le C'_K \|x^{-1}\cdot y\|_2  \quad \hbox{for every } x,y\in K.
	    	\end{equation}
Indeed, this is the same as $\|x-x\cdot z\|_2\le C'_K\|z\|_2$ and the polynomial
$x-P(x,z)$ vanishes at $z=0$.

We assume throughout that
the Jacobian of the maps $y\mapsto x\cdot y$, $x\in G$,  is $1$ so that the Lebesgue measure on $\mathbb R^d$ is a Haar measure for our group $G$.   This assumption follows from the much more demanding assumption that
 \eqref{e:3.1a}    has  the additional property that
 \begin{equation}\label{triangular}
 \bar p_1(x, y)=0 \quad \hbox{and} \quad
 \bar  p_i(x,y)= \bar p_i((x_j)_{1}^{i-1},(y_j)_{1}^{i-1}) \ \hbox{ for }  2\le i\le d.
\end{equation}

In other word, the polynomial
$$
\bar p_i (x,y)=p_i(x,y)-x_i-y_i$$ depends only  
 on the first $i-1$ coordinates of $x$ and $y$ and has no constant nor first order terms.
 Clearly, this triangular structure implies that the Jacobian of the map $y\mapsto x\cdot y$ is $1$.
 Morever, for $x^{-1}= Q(x)=(q_1(x), \ldots , q_d (x))$, we deduce from $P(Q(x), x)=0$ that
\begin{equation}\label{e:3.4a}
q_1(x)=-x_1 \quad \hbox{and} \quad
q_i(x)= -x_i + \bar q_i(x_1, \ldots, x_{i-1}) \ \hbox{ for } 2\leq i \leq d,
\end{equation}
where $\bar q_i (x_1, \ldots, x_{i-1})$, $2\leq i \leq d$, are polynomials having no constant nor first order terms.

\medskip

\begin{exa}[Matrix coordinates] The most commonly used  coordinate system (as Moli${\rm\grave{e}}$re's Mr. Jourdain with {\em prose}, we may use it without realizing we do!) comes from matrix groups. Indeed,  the group $G$ is often given as a subgroup of a group of invertible matrices of a certain dimension, say $N$.  In particular, a nilpotent
group is often given as a subgroup of the group of unipotent upper-triangular matrices. The most obvious example is when $G$ is the group of unipotent upper-triangular matrices itself
$$\mathbf U_N=\left\{\begin{pmatrix} 1&x_{12}&x_{13}&\dots&x_{1N}\\
0&1&x_{23}&\dots &x_{2N}\\
0&0&1&\dots &x_{3N}\\
\vdots&\vdots&\vdots&\vdots &\vdots\\
0&0&0&\dots&1\end{pmatrix} : x_{ij}\in \mathbb R, 1\le i<j\le d\right\}.$$
This group has dimension $d={N\choose 2}$.
In the case $N=3$, this is the Heisenberg group $\mathbb H_3(\mathbb R)$ in its matrix form with 
$$P(x,y)=(x_1+y_1,x_2+y_2, x_3+y_3+x_1y_2), \quad Q(x_1,x_2,x_3)=(-x_1,-x_2,-x_3+x_1x_2) $$ and
$$P(x^{-1},y) = (y_1-x_1,y_2-x_2, y_3-x_3  -x_1(y_2-x_2)).$$
\end{exa}

 \begin{exa}[Exponential coordinates of the first type] The second most commonly encountered coordinate system is given by the canonical exponential map  $$\exp: \mathfrak g\to G$$
 between the Lie algebra $\mathfrak g$ of the group $G$ and $G$ itself. We can think of $\mathfrak g=\mathbb R^d$ as the tangent space at $e$.
 Given a tangent vector $x\in \mathbb R^d$, we first consider the (unique) left invariant vector field $X$ on $G$ such that $X(e)=x$ and  the solution $\gamma_x:[0,1]\to G$ of $\frac{d}{dt}\gamma_x(t)= X(\gamma_x(t))$ with initial condition $\gamma_x(0)=e$, and set
   $$\exp(x)=\gamma_x(1).$$
   Using the fact that, for any two left-invariant vector fields $X,Y$, the well defined differential operator $XY-YX$ is a left-invariant vector field, we
   obtain the Lie bracket  $(x,y)\mapsto [x,y]=(XY-YX)(e).$ Moreover,
   $$
   \left.[x,y]= \partial_s\partial_t \left(\exp( tx) \cdot \exp(sy)\cdot \exp {(-tx)}\right)\right|_{s=t=0}.
   $$ 
     In the case of simply connected Lie group, the exponential  map is a global invertible diffeomorphism
   and the multiplication is given in universal form by the famous Campbell-Hausdorff formula
   $$\exp(x)\cdot \exp(y)= \exp \left(P_{\mbox{\tiny CH}}(x,y)\right), $$
   where
     \begin{equation} \label{e:2.3}
   P_{\mbox{\tiny CH}}(x,y)= x+y+\frac{1}{2}[x,y]+\frac{1}{12}(    [x,[x,   y  ]]     +[y,[y,x]])+\cdots .
    \end{equation}
   In other words, in the exponential coordinate system, the group law is
   $$x\cdot y= P_{\mbox{\tiny CH}}(x,y)=    x+y+\frac{1}{2}[x,y]+\frac{1}{12}([x,[x,   y ]]+[y,[y,x]])+\cdots .$$
     This has the desirable polynomial form because iterated Lie brackets with more than $r$ entries are equal to $0$ if $r$ is the nilpotency class of $G$.
  In these coordinates, it is always the case that
  $$x^{-1}=-x.$$   Applying this to the Heisenberg group we obtain the often-used description of $\mathbb H_3(\mathbb R)$ as
  $\mathbb R^3$ equipped with the product
  $$x\cdot y= P_{\mbox{\tiny CH}}(x,y)=\left(x_1+y_1,x_2+y_2, y_3+x_3+{\textstyle \frac{1}{2}}(x_1y_2-x_2y_1)\right)$$
  and $$P_{\mbox{\tiny CH}}(x^{-1},y)=   \left(y_1-x_1,y_2-x_2,y_3-x_3+{\textstyle \frac{1}{2}}(x_2(y_1-x_1) -x_1(y_2-x_2))\right).$$
 \end{exa}

  \begin{exa} [Exponential coordinates of the second kind]  \label{exa:PhiPsi}
  For a simply connected nilpotent Lie group $G$, exponential coordinate systems of the second kind are typically associated with a filtration of the Lie algebra $\mathfrak g$ by subalgebras (resp. ideals)
  $$\mathfrak g=\mathfrak g_1 \supset \mathfrak g_2\supset\cdots \supset \mathfrak g_\ell \supset\{0\} $$
  with $\mathfrak g_j$ of dimension  $m_j$,
  and a linear basis $(\varepsilon _i)_1^d$  such that the linear span of $(\varepsilon _i)_{i\ge j}$ is a subalgebra  (resp. ideal) for all
     $1\le j\le d :=m_1$, and $(\varepsilon _i)_{d-m_j +1}^d$ is a basis of $\mathfrak g_j$.  In such a situation, the maps from $\mathbb R^d$ to $G$ defined by
 $$\Phi(x_1,\dots,x_d)=\exp(x_1 \varepsilon_1)\cdot \cdots \cdot \exp(x_d\varepsilon_d)$$
 and
 $$\Psi(x_1,\cdots,x_d)=\exp(x_d \varepsilon_d)\cdot \cdots \cdot \exp(x_1\varepsilon_1)$$
 give two distinct global polynomial coordinate systems for $G$.

 For example, the matrix coordinate system of the group of $n\times n$ upper-triangular matrices with entries equal to $1$ on the diagonal, is an exponential coordinate system of the second kind associated with the  lower central series
 $$\mathfrak g=\mathfrak g_1, \quad \mathfrak g_{i+1}=[\mathfrak g_i,\mathfrak g],\; 1\le j\le n,$$
 which, in this case, has last non-trivial member $\mathfrak g_{n-1}$ corresponding to the upper-right corner entry.
 Here, we can realize $\mathfrak g$ as the algebra of the strictly upper-triangular matrix. We then enumerate the entries $(x_i)_1^{d}$, $d=n(n-1)/2$, going down along each upper-diagonal in order so that $x_d$ is the entry in the upper-right corner, and consider the corresponding map $\Psi$. For instance, in the $4\times 4$ case,
$$ \begin{pmatrix} 1&x_{1}&x_{4}&x_{6} \\
0&1&x_{2}&x_{5}  \\
0&0&1& x_{3}\\
0&0&0&1\end{pmatrix} =\begin{pmatrix} 1&0&0&x_{6} \\
0&1&0&0  \\
0&0&1& 0\\
0&0&0&1\end{pmatrix}    \begin{pmatrix} 1&0&0&0 \\
0&1&0&x_{5}  \\
0&0&1& 0\\
0&0&0&1\end{pmatrix}\cdots  \begin{pmatrix} 1&x_1&0&0 \\
0&1&0&0  \\
0&0&1& 0\\
0&0&0&1\end{pmatrix}. $$
Each of the matrices on the right is the matrix exponential of the corresponding strictly triangular matrix.
Note that $\Phi$ defined above leads to a different coordinate system.  \end{exa}
These classical constructions concerning exponential coordinates of the first and second kinds are explained in more details in \cite[Section 1.2]{CorGr}.  See also \cite{CMZ,GrTa,Mal}.

   \subsection{Dilations, approximate dilations, and $G_\bullet$}

 \subsubsection*{Straight dilations}
Let $G$ be a nilpotent simply connected Lie group given in a global  polynomial coordinate system  $G=(\mathbb R^d,\cdot)$. Call {\em straight dilations with exponents} $\mathbf a=(a_1,\dots,a_d)\in \mathbb R_+^d$, the group of diffeomorphisms
 $$
 \phi_t(x)= (t^{a_1}x_1,\dots,t^{a_d}x_d),\quad t>0.
 $$
Note that $\phi_s\circ \phi_t=\phi_{st}$, $s,t>0$, and $\phi_1=\mbox{Id}$.

 \begin{defin}
 We say that $(\phi_t)_{t>0}$  as above is a {\em straight    group 
 dilation structure}   if
 \begin{equation}\label{Ghom} \phi_t(x\cdot y)=\phi_t(x)\cdot \phi_t(y),\quad \;t>0,x,y\in G. \end{equation}
 \end{defin}
 This, of course, is a very restrictive property and not every simply connected nilpotent Lie group $G$ admits such a structure. In the case of the Heisenberg group in matrix form, for given $a,b,c\ge 0$, set
 $$
 \phi_t\left(\begin{pmatrix} 1&x&z\\0&1&y\\0&0&1\end{pmatrix}\right)=\begin{pmatrix} 1&t^ax&t^cz\\0&1&t^by\\0&0&1\end{pmatrix}, \quad t>0,\; x,y,z\in \mathbb R.
 $$
 These straight dilations structures are group dilation structures if and only if $a+b=c$.
 \begin{rem} More generally, without reference to any coordinate system, a group of diffeomorphisms $(\phi_t)_{t>0}$, $\phi_t: G\to G$, $\phi_1=\mbox{Id}$, satisfying (\ref{Ghom}) and such that
 $\lim_{t\to 0} \phi_{t}(g)=e$ is called an expanding group dilation structure. See \cite{Hazod2001,LeDonne}. By a theorem of Siebert \cite{Siebert},  a   connected locally compact group carrying such a structure must be a simply connected nilpotent Lie group (and not every simply connected nilpotent groups admit such a structure).
 \end{rem}

 \begin{defin} Let $\mathbb R^d$ be equipped with a straight dilation structure
 $$(\phi_t)_{t>0}, \quad \phi_t(x)=(t^{a_i} x_i)_1^d, \quad a_i>0, \quad 1\le i\le d.$$
 A positive  function $N$ on $\mathbb R^d$ is called homogeneous with respect to $(\phi_t)_{t>0}$ if $N(\phi_t(x))=tN(x)$.
 \end{defin}
 
 \begin{exa}  The function $x\mapsto N(x)=\max_{1\le i\le d}\{|x_i|^{1/a_i}\}$ is homogenous with respect to  $(\phi_t)_{t>0}$, 
 $\phi_t(x)=(t^{a_i} x_i)_1^d$, $a_i>0, \;1\le i\le d$. It is a norm on $(\mathbb R^d,+)$ (i.e., satisfies the triangle inequality)  if $a_i\ge 1$ for all $1\le i\le d$. If $M$ is another homogeneous function with respect to $(\phi_t)_{t>0}$, such that the set ${x:M(x)\le 1}$ is compact, then
  there are constants $0<c\le C<\infty$ such that
 $cN\le M\le CN$.
 \end{exa}

 \subsubsection*{Approximate group dilations and $G_\bullet$}
  Let $G$ be a simply connected nilpotent Lie group given in a  global polynomial chart $G=(\mathbb R^d,\cdot)$ and equipped with a straight dilation (not necessarily a group dilation structure)
  $(\phi_t)_{t>0}$.  For each $t>0$, we obtain a new group structure $\cdot_t$ on $\mathbb R^d$ by setting
  $$
  x\cdot_t y= \phi_{1/t}(\phi_t(x)\cdot \phi_t(y)),
  \quad x,y\in \mathbb R^d.
  $$
  Moreover, $$\phi_{1/t}: (\mathbb R^d, \cdot)\to (\mathbb R^d,\cdot_t)$$
  is a group isomorphism between $G=(\mathbb R^d,\cdot)$ and $G_t=(\mathbb R^d,\cdot_t)$. Additionally, $(\phi_t)_{t>0}$ is a group dilation structure
if and only if $\cdot_t=\cdot$ for all $t>0$.

 \begin{defin}[Approximate group dilation structure] \label{bullet}
 Let $G$ be a simply connected nilpotent Lie group  described by a global polynomial chart $(\mathbb R^d,\cdot)$. Let $(\phi_t)_{t>0}$ be a straight dilation structure. We say that this dilation structure is an {\em approximate group dilation structure}  if, for any
 $x,y\in \mathbb R^d$, the limits
 $$
  \lim_{t\to \infty} \phi_{1/t}(\phi_t(x)^{-1}) =x_\bullet^{-1}
 \quad \hbox{and}  \quad
 \lim _{t\to \infty} \phi_{1/t}(\phi_t(x)\cdot \phi_t(y))=x\bullet y
 $$
 exist.
  \end{defin}

  \begin{lem}\label{L:2.5}
    The pairing $(x,y)\mapsto x\bullet y$ yields a nilpotent Lie group $G_\bullet=(\mathbb R^d,\bullet)$ and $x_{\bullet }^{-1}$ is the inverse of $x$ for the
  group law $\bullet$, that is $x_\bullet^{-1} \bullet x=x\bullet x_\bullet^{-1}=e_\bullet$.  For the group $(\mathbb R^d,\bullet)$, the straight dilations $\{\phi_t; \, t>0\}$ form a {\em group dilation structure}, i.e., satisfy {\rm (\ref{Ghom})}.
    \end{lem}
   \begin{proof} By construction, the maps $P_t(x,y)=\phi_{1/t}(\phi_t(x)\cdot\phi_t(y))$ and $I_t(x)=\phi_{1/t}(\phi_t(x)^{-1})$  are polynomial maps in $x,y$ with coefficients equal to linear combinations of power functions of $t$ with exponents in $\mathbb R$. If the limits $\lim_{t\to \infty} P_t(x,y)$ and $\lim _{t\to \infty}I_t(x)$
   exist for all $x,y$, it means that only non-positive powers of $t$ occur and this implies that the families $P_t,I_t$ are uniformly equicontinuous on compact sets. A sequence of simple considerations then yields that   $$x\bullet (y\bullet z)=\lim_{t\to \infty} \phi_{1/t}(\phi_t(x)\cdot\phi_t(y)\cdot\phi_t(z))=(x\bullet y)\bullet z$$ and
   $$x_{\bullet }^{-1}\bullet x=x\bullet x_{\bullet }^{-1}=e_\bullet=0.$$  Note that this also implies
   \begin{equation} \label{keylim}
   \lim_{t\to \infty} \phi_{1/t}\left(\phi_t(x)^{-1}\cdot\phi_t ( y)\right)= x_{\bullet }^{-1}\bullet y.\end{equation}
   \end{proof}

  \begin{lem}\label{lem2-6}
  Let $(\phi_t)_{t>0}$ be a straight approximate group dilation structure   on $(\mathbb R^d,\cdot)$. For any compact $K\subset \mathbb R^d$
  there is a constant $C_K$ such that, for any $x,y\in K$ and $t\ge 1$,
  $$\| \phi_{1/t}\left(\phi_t(x)^{-1}\cdot\phi_t(y)\right)\|_2\le C_K \|y-x\|_2$$
  and
  $$\| \phi_{1/t}\left(\phi_t(x)^{-1}\cdot \phi_t(y)\right)\|_2\le C_K \|x_{\bullet }^{-1}\bullet y\|_2.$$  
  \end{lem}
  
  \begin{proof} The function $(t,x,y)\mapsto \phi_{1/t}\left(\phi_t(x)^{-1}\phi_t(y)\right)$ is a polynomial in
  $$(x,y)=(x_1,\dots,x_d,y_1,\dots, y_d)$$ with coefficients equal  to linear combinations of powers of $t$ with exponents in $\mathbb R$.   By (\ref{keylim}), only  non-positive powers of $t$ appear.  The desired inequality follows because this polynomial function equals $0$ when $x=y$.
  The second inequality follows from the first and (\ref{Lip2}) applied to $(\mathbb R^d,\bullet)$.
   \end{proof}

\begin{rem} When working in exponential coordinates, we have the extra structure of the Lie bracket  $[\cdot,\cdot]$ at our disposal and we can replace the
conditions in Definition \ref{bullet} by the condition that
$$\lim_{t\to \infty} \phi_{1/t}([\phi_t(x),\phi_t(y)])=[x,y]_\bullet$$
exists. Call this an approximate Lie dilation structure. Note that, in this case, $x^{-1}=x_\bullet^{-1}=-x$ and $\phi_t(x)^{-1}=\phi_t(x^{-1})$ so that the inverse map condition is automatically satisfied.
 \end{rem}

\begin{rem} \label{scale}
 If $(\phi_t)_{t>0}$ is a  group dilation structure (resp. an approximate group dilation structure) then so is $ (\phi_{t^a})_{t>0}$, for any $a>0$.  Moreover, in the case of an approximate group dilation structure, this change does not affect the limit structure.
\end{rem}

\begin{rem}The basic idea behind Definition \ref{bullet} is well-known in two different related contexts. It appears in the study
of the large scale geometry of groups of polynomial volume growth, see, e.g., \cite[Section 2.2]{Breuillard}, and in the work of Alexopoulos on local limit theorems in the context of groups of polynomial volume growth, see \cite[Section 5.2]{AlexoMem}. In these works, there is a unique relevant structure at infinity and it follows that the ``dilation structures'' considered there are very special examples of those defined here. Various forms of the same idea play an important role in the local study of sub-elliptic second order operators but in that context the limit is taken when the parameter $t$ goes to $0$. See for
 instance \cite[Chapter V]{VSCC}.
\end{rem}

  The following lemma is not used explicitly but serves as an exercise in manipulating the notion introduced above. See also Section \ref{S:10.3}.

\begin{lem}\label{lem-subgroup} Let $H$ be a  subgroup of $G=(\mathbb R^d,\cdot)$ and $(\phi_t)_{t>0}$ be an approximate group dilation structure with limit law $\bullet$. Set
$$
H_\bullet= \left\{x\in \mathbb R^d:    \hbox{ there exists }   (x_k)_1^\infty \subset  H \hbox{ so that }  \lim_{k\to \infty}\phi_{1/k}(x_k)=x \right\}.
$$
Then $H_\bullet$ is a subgroup of $G_\bullet=(\mathbb R^d,\bullet)$.

\end{lem}
\begin{proof} Let $x,y\in H_\bullet$ with witness sequences $(x_k)_1^\infty$, $(y_k)_1^\infty$ in $H$. Fix $\eps>0$.   By
  the  continuity of $\bullet$, there exists $\delta>0$ such that $\|x-x'\|_2<\delta$ and $\|y-y'\|_2\le \delta$ imply $\| x\bullet y -x'\bullet y'\|_2<\eps/2$.
By  the efinition of $H_\bullet$, there  exists $N>0$ such that  $\|x-\phi_{1/k}(x_k)\|_2<\delta$ and $\|y-\phi_{1/k}(y_k)\|_2< \delta$  for all $k\ge N$.
By the uniform convergence of $\phi_{1/t}(\phi_t(u)\cdot \phi_t(v))$ to $u\bullet v$ on compact sets, there exists $N'$ such that, for all $k\ge N$ and $k'\ge N'$,
$$\|\phi_{1/k}(x_k)\bullet \phi_{1/k}(y_k)-\phi_{1/k'}(\phi_{k'}(\phi_{1/k}(x_k))\cdot\phi_{k'}(\phi_{1/k}(y_k)))\|_2<\eps/2.$$
Hence, for $k\ge \max\{N,N'\}$,
$$\|x\bullet y -\phi_{1/k}(\phi_k(\phi_{1/k}(x_k))\cdot \phi_k(\phi_{1/k}(y)))\|_2<\eps,$$
and thus, $\|x\bullet y-\phi_{1/k}(x_k\cdot y_k)\|_2<\eps$.  Because $x_k\cdot y_k \in H $, this proves that $x\bullet y =\lim_{k\to \infty} \phi_{1/k}(x_k\cdot y_k)\in H_\bullet.$  A similar proof applies to show that $x_\bullet^{- 1}\in H_\bullet$  for  $x\in H_\bullet$.
\end{proof}

\begin{exa}\label{E:3.4}
 Consider the Heisenberg group viewed as the group of matrices
$$\mathbb H_3(\mathbb R)=\left\{\begin{pmatrix} 1&x&z\\0&1&y\\0&0&1\end{pmatrix}: (x,y,z)\in \mathbb R^3\right\}.$$
Here, the product of the matrices associated with $(x,y,z)$ and $(x',y',z')$ is associated with the triplet
$$(x+x',y+y',z+z'+xy').$$
The inverse of $(x,y,z)$ is
\begin{equation}\label{e:3.8}
(x,y,z)^{-1}=(-x,-y,-z+ xy).
\end{equation}
This is isomorphic but different from the ``exponential coordinate description'' discussed earlier where
$$
(u,v,w)\cdot (u',v',w')= (u+u',v+v',w+w'+\tfrac{1}{2}(uv'-u'v)).
$$
The map
\begin{equation}\label{e:3.6}
q:  \   (x,y,z) \to q(x,y,z)=(u,v,w)=(x,y,z-\tfrac{1}{2}xy)
\end{equation}
 provides the group isomorphism between these two descriptions.

Now, consider the group of diffeomorphisms $(\phi_t)_{t>0}$  (straight dilations in that system) given in the $(x,y,z)$ matrix-coordinates by
$$
\phi_t(x,y,z)=(t^ax,t^by,t^cz)
\quad \hbox{for some fixed } a,b,c>0.
$$
These are group diffeomorphisms for all $t>0$ if and only if $c=a+b$. They form
an approximate group dilation structure at infinity if and  only if $c\ge a+b$. When $c>a+b$,
$$(x,y,z)\bullet (x',y',z')=(x+x',y+y',z+z') $$
 and
 $$(x,y,z)_\bullet^{-1}=(-x,-y,-z) \neq (x,y,z)^{-1}.$$

 If we write down these same diffeomorphisms in the ``exponential coordinate" description $(u,v,w)$ they are given by the maps
 $$
 \psi_t(u,v,w) = q^{-1}\circ \phi_t \circ q (u,v,w)= (t^a u,t^b v, t^c w+ \tfrac{1}{2}(t^c -t^{a+b})uv).
 $$
In the $(u,v,w)$ global coordinate chart $\exp=\log =\mbox{id}$, and if we assume $c\ge a+b$,  the straight dilations
$$
\delta_t(u,v,w)=(t^au,t^bv,t^cw),\quad t>0,
$$
   give both an approximate Lie dilation structure and an associated approximate group dilation structure which are distinct
from the $\phi_t/\psi_t$ approximate group dilation structure even so they share the same differential at the identity. They lead to isomorphic limit group structures.
\end{exa}

\begin{exa}\label{E:3.15}
	 Consider the group
$$G=\left\{\begin{pmatrix} 1&x_{12}&x_{13}&x_{14} \\
0&1&x_{23}&x_{24}  \\
0&0&1& x_{34}\\
0&0&0&1\end{pmatrix} : x_{ij}\in \mathbb R\right\}$$
and the straight dilation structures associated with any tuple $$1/\alpha_{ij},\; ij=(i,j)\in\{12,13,14,23,24,34\}$$ so that
$$
\delta_t \left(\begin{pmatrix} 1&x_{12}&x_{13}&x_{14} \\
0&1&x_{23}&x_{24}  \\
0&0&1& x_{34}\\
0&0&0&1\end{pmatrix} \right)=\begin{pmatrix} 1&y_{12}&y_{13}&y_{14} \\
0&1&y_{23}&y_{24}  \\
0&0&1& y_{34}\\
0&0&0&1\end{pmatrix} ,
\quad  y_{ij}=t^{1/\alpha_{ij}} x_{ij}.
$$
Such a $(\delta_t)_{t>0}$ is a group dilation structure if and only if
$$
1/\alpha_{k\ell}=1/\alpha_{kj}+1/\alpha_{j\ell}
\quad \mbox{for all }  1\le k<j<\ell\le 4,
$$
that is,
$$ (1): \;1/\alpha_{13}=1/\alpha_{12}+1/\alpha_{23},  \qquad
(2): \;1/\alpha_{24}=1/\alpha_{23}+1/\alpha_{34}
$$
and $$(3):\;1/\alpha_{14}=1/\alpha_{12}+1/\alpha_{24},
\qquad (4):\;1/\alpha_{14} =1/\alpha_{13}+1/\alpha_{34}.
$$
The group $(\phi_t)_{t>0}$  is an approximate group  dilation structure at infinity if and only if
\begin{equation} \label{eq-dil}
1/\alpha_{k\ell}\ge 1/\alpha_{kj}+1/\alpha_{j\ell}  \quad \hbox{for all }
1\le k<j<\ell\le 4.\end{equation}

We now list all the possible Lie structures that appear as a limit of such an approximate group dilation structure on $G$.
\begin{itemize}
\item When equality holds in all of the inequalities  (\ref{eq-dil}), we have  $G_\bullet=G$.
\item When strict inequality holds in all of the inequalities (\ref{eq-dil}), we have $G_\bullet=\mathbb R^6$ (abelian).

\item When equations (1) and (2) are equalities then equations (3) and (4) become equivalent. Assume a strict inequality holds in (3) and (4).  Then the limit  $G_\bullet$ is
$$\left\{\left( \begin{pmatrix}1&x_{12}&x_{13}\\0&1&x_{23}\\0&1&1\end{pmatrix},\begin{pmatrix}1&x_{23}& x_{24}\\0&1&x_{34}\\0&0&1\end{pmatrix}, \begin{pmatrix}x_{14}\end{pmatrix}\right): x_{ij}\in \mathbb R\right\}.$$
Here multiplication for these triplets of matrices is matrix-coordinate by matrix-coordinate.  Note how the same $x_{23}$ appears in the first and second matrix-coordinates.

\item When strict inequality holds in both equations (1) and (2) and equality holds in both (3) and (4), then the limit $G_\bullet$
is $$\left\{\left( \begin{pmatrix}1&x_{12}&x_{13}&x_{14}\\0&1&0&x_{24}\\ 0&0&1&x_{34}\\0&0&0&1
\end{pmatrix}, \begin{pmatrix}x_{23}\end{pmatrix}\right): x_{ij}\in \mathbb R\right\}$$
(this is the direct product of the $5$ dimensional Heisenberg group $\mathbb H_5(\mathbb R)$ and a copy of $\mathbb R$).
\item When strict inequality holds in both equations (1) and (2) and equality holds in (3) but not in  (4) (resp. (4) but not in (3)), then the limit $G_\bullet$ is
$$\left\{\left( \begin{pmatrix}1&x_{12}&x_{14}\\0&1&x_{24}\\ 0&0&1
\end{pmatrix},\begin{pmatrix}x_{13}\end{pmatrix},\begin{pmatrix}x_{23}\end{pmatrix},\begin{pmatrix}x_{34}\end{pmatrix}\right): x_{ij}\in \mathbb R\right\}$$
(resp. exchange the roles of pairs $x_{12}, x_{24}$ and $x_{13}, x_{34}$). This is the direct product of a copy of $\mathbb H_3(\mathbb R)$ and $\mathbb R^3$).

\item When strict inequality holds in (1) (resp. (2)) and equality holds in (2)  (resp. (1)) then strict inequality must  
   hold   in  (3) (resp. (4)). 
   If equality holds in (4) (resp. (3)), the limit group is isomorphic to
$$ \left\{\left(  \begin{pmatrix}x_{12}\end{pmatrix}, \begin{pmatrix}1&0&x_{13}&x_{14}\\0&1&x_{23}&x_{24}\\ 0&0&1&x_{34}\\0&0&0&1
\end{pmatrix}\right): x_{ij}\in \mathbb R\right\}$$
(resp. exchange the roles of $x_{12}$ and $x_{34}$, the limit groups in both cases  are isomorphic).
\item When strict inequality holds in (1) (resp. (2)) and equality holds in (2)  (resp. (1)) and strict inequality holds in each of (3) and (4), the limit group is isomorphic to
$$\left\{\left( \begin{pmatrix}1&x_{23}&x_{24}\\0&1&x_{34}\\ 0&0&1
\end{pmatrix},\begin{pmatrix}x_{12}\end{pmatrix},\begin{pmatrix}x_{13}\end{pmatrix},\begin{pmatrix}x_{14}\end{pmatrix}\right): x_{ij}\in \mathbb R\right\}$$
(resp.  replace the triplet $(x_{23}, x_{24},x_{34}) $ with $(x_{12},x_{13}, x_{23})$ and the triplet $(x_{12},x_{13},x_{14})$ with $(x_{34},x_{24},x_{14})$). This is the direct product of a copy of $\mathbb H_3(\mathbb R)$ and $\mathbb R^3$).
\end{itemize}
\end{exa}

\section{Vague convergence and change of group law}\label{3-4weak}

\subsection{Vague convergence under rescaling}\label{S:4.1}

We consider a rather general situation pertaining to the problem we want to study.  We are given the following data:

\begin{enumerate}
\item A finitely generated torsion free nilpotent group $\Gamma $ given as a co-compact closed subgroup of a simply connected nilpotent Lie group $G$.
It is useful for our purpose to be more explicit and write $G=(\mathbb R^d, \cdot)$ where this coordinate system is a polynomial coordinate system as explained earlier.
\item  A probability measure $\mu$ on  $\Gamma$.
\item An approximate group dilation structure $(\delta_t)_{t>0}$ on $G$ with Lie group limit $G_\bullet=(\mathbb R^d,\bullet)$.
\end{enumerate}

\begin{defin}\label{def-admi}
We say that the approximate group dilation structure $(\delta_t)_{t>0}$ is admissible for $\mu$ if the family of measures
\begin{equation}\label{def-mut}\mu_t=t\delta_{1/t}(\mu) \quad \hbox{defined by } 
\mu_t(\phi):= t\int_{\mathbb R^d} \phi(\delta_{1/t}(u)) \mu(du)
\end{equation}
converge vaguely to a
  Radon measure $\mu_\bullet$ on $\mathbb R^d\setminus\{0\}$
  as $t\to \infty$.
Recall that, by definition, this means that, for any continuous function $\phi$ with compact support in $\mathbb R^d\setminus\{0\}$,
$$\lim_{t\ra \infty} \int  \phi(x) d\mu_t(x)    =
 \int \phi(x)d\mu_\bullet(x).$$
\end{defin}

\medskip

\begin{rem} Note the following identities:
$$\mu_t(A)=t\mu(\delta_t(A)) = t\sum_{y\in \Gamma}\1_{\delta_t A}(y)\mu(y)=t \sum_{x\in \delta^{-1}_t \Gamma} \1_A(x) \mu(\delta_t x)
$$
and
$$\int \phi(x) d\mu_t(x)= t\int   \phi (\delta^{-1}_t y) d\mu(y) = t\sum_{x\in \delta_t^{-1}\Gamma}\phi(x) \mu(\delta_tx).$$
\end{rem}

\begin{rem} The normalization by a factor of $t$ in  $\mu_t=t\delta_{1/t}(\mu)$ is less restrictive than it may first appear because of Remark \ref{scale}.  If there is an approximate Lie dilation structure $(\delta_t)_{t>0}$ (with limit law $\bullet$) such that  the measure $\mu_t=t^a\delta_t^{-1}(\mu)$ converges vaguely to $\mu_\bullet$ on $\mathbb R^d\setminus\{0\}$ then  the modified approximate Lie dilation structure $(\delta_{t^{1/a}})_{t>0}$ gives the same limit law $\bullet$ and is admissible for $\mu$. In this sense, the choice of the linear $t$ factor in the definition of $\mu_t$ amounts, more or less, to a scaling normalization.
 \end{rem}

 \begin{exa}\label{ex3-2}   Fix $\alpha\in (0,2)$ and let $\mu$ be the probability measure on $\mathbb Z\subset \mathbb R$ with $$\mu(k)= c_\alpha(1+|k|)^{-\alpha-1}.$$
 Let $\delta_t (x)=t^{1/\alpha} x$. Then $t \delta^{-1}_t(\mu) $ converges vaguely on $\mathbb R\setminus\{0\}$
 as $t\to \infty$
 to the measure $\mu_\bullet$ with density
 $c_\alpha |x|^{-\alpha-1}$ with respect to  the Lebesgue measure on $\mathbb R$.
 \end{exa}
 \begin{exa}\label{ex3-3}  Fix $\alpha\in (0,2)$ and $\beta \in (0,\alpha)$. Let  $\mu$ be the probability measure on $\mathbb Z^2$ given by
 $$
 \mu((x,y))= c(1+|x|+|y|)^{-\alpha-2}, \quad
  (x,y)\in \mathbb Z^2\subset \mathbb R^2.
  $$
 Let $\delta_t((x,y))=(t^{1/\alpha} x,t^{1/\beta}y).$  Then  $t\delta^{-1}_t(\mu)$ converges vaguely on $\mathbb R^2\setminus\{(0,0)\}$  as $t\to \infty$
 to the measure $\mu_\bullet(dxdy)= f_\bullet(x)dx\otimes \delta_0(dy) $ supported on the $x$-axis with $f_\bullet(x)= c' |x|^{-\alpha-1}$, where
 $c'= c \int_{\mathbb R} (1+u)^{-\alpha-2} du.$
 \end{exa}

 \begin{exa}\label{ex3-4}  On the Heisenberg group $\mathbb H_3(\mathbb Z)$ viewed as the group of matrix
 \begin{equation}\label{eq:3dimHB}\left\{ \begin{pmatrix} 1&x_1&x_3\\0&1&x_2\\0&0&1\end{pmatrix}: x_1,x_2,x_3\in \mathbb Z\right\},
\end{equation}
consider the measure
\begin{equation}\label{woob34}
\mu((x_1,x_2,x_3))=  \frac{c_\alpha }{\left(1+\sqrt{x_1^2+x_2^2+|x_3-x_1x_2/2|}\right)^{\alpha+ 4}},
\end{equation}
(note that this is a symmetric measure).
Consider an approximate Lie dilation structure $(\delta_t)_{t>0}$ of the form
$\delta_t((x_i)_1^3)=(t^{1/\gamma_i}x_i)_1^3$. For this to be an approximate Lie dilation structure, it must be that
$1/\gamma_3\ge 1/\gamma_1+1/\gamma_2$ which we assume.  For the  measure $t\delta^{-1}_t(\mu_t)$ to have a vague limit, it is necessary that
$1/\gamma_1\ge 1/\alpha, 1/\gamma_2\ge 1/\alpha$ and $1/\gamma_3\ge 2/\alpha$.  Note that the roles of $x$ and $y$ are the same so that we can assume for the sake of the computations described below that $1/\gamma_1\le 1/\gamma_2$.
\begin{enumerate}
\item  Assume that $1/\gamma_2\ge 1/\gamma_1 >1/\alpha$. Then $1/\gamma_3>2/\alpha$ and it is not hard to see that $t\delta^{-1}_t(\mu)$ converges vaguely to $0$  as $t\to \infty$.

\item  Assume $\gamma_1=\gamma_2=\gamma_3/2=\alpha$. Then $(\delta_t)_{t>0}$ is a group dilation structure and
$t\delta_t^{-1}(\mu)$ converges vaguely  as $t\to \infty$ to
$$  \mu_\bullet((dx_1,dx_2,dx_3))=  \frac{c_\alpha dx_1dx_2 dx_3}{\left(\sqrt{x_1^2+x_2^2+|x_3-x_1x_2/2|}\right)^{\alpha+ 4}}.$$\item Assume that   $1/\gamma_1=1/\gamma_2=1/\alpha$ and $1/\gamma_3>2/\alpha$. Then
$t\delta_t^{-1}(\mu) \Rightarrow \mu_\bullet$  as $t\to \infty$,
where
$$\mu_\bullet(dx_1dx_2dx_3)=  \frac{c'}{\left(\sqrt{x_1^2+x_2^2}\right)^{ \alpha+2}} dx_1dx_2 \otimes \delta_0(dx_3) $$
with $c'= 2c \int_0^\infty (1+s)^{-(2+\alpha/2)} ds$.

\item Assume that $1/\gamma_2> 1/\gamma_1=1/\alpha$. It follows that $1/\gamma_3\ge 1/\gamma_1+1/\gamma_2>2/\alpha$. In this case
$t\delta^{-1}_t(\mu)\Rightarrow \mu_\bullet$   as $t\to \infty$,
where
$$\mu_\bullet(dx_1dx_2dx_3)= c'{|x_1|^{-\alpha+1} }dx_1 \otimes \delta_0(dx_2)\otimes \delta_0(dx_3)
$$
with
$$c'= 2c  \int_{-\infty}^{\infty} \left( \int_0^{\infty}  \left(\sqrt{1+u^2 +v}\right)^{-(\alpha+4)}dv\right)du.
$$
 \end{enumerate}

 We provide details for the third case (the fourth case is similar).   Let $f$ be a continuous function with compact support in $ \mathbb R^3\setminus\{0\}$. We want to show that
\begin{align*}\lim_{t\to \infty}\int f(x) d\mu_t(x) =& \lim_{t\to \infty} t\sum_{x\in  \mathbb Z^3}    \frac{c_\alpha f(\delta_t^{-1}(x))}{ \left(1+\sqrt{x_1^2+x_2^2+|x_3-x_1x_2/2|}\right)^{\alpha + 4}} \\
  =&\int _{\mathbb R^3}  \frac{c' f(x) }{\left(\sqrt{x_1^2+x_2^2}\right)^{\alpha+2}} dx_1dx_2 \otimes \delta_0(dx_3)\\
  =& \int_{\mathbb R^2} \frac{c' f(x_1,x_2,0)}{\left(\sqrt{x_1^2+x_2^2}\right)^{\alpha+2}} dx_1dx_2 .\end{align*}
  Recall  that $\gamma_1=\gamma_2=\alpha$ and $2\gamma_3<\alpha$ and  write
 \begin{align*}&\lefteqn{ \int f(x) t\delta_t^{-1}(\mu(dx))=t\sum_{x\in \mathbb Z^3} f(\delta_t^{-1}(x))  \frac{c_\alpha}{\left(1+\sqrt{x_1^2+x_2^2+|x_3-x_1x_2/2|}\right)^{\alpha + 4}}}  \\ & =  t \sum_{z\in \delta_t^{-1}(\mathbb Z^3)} f(z)  \frac{c_\alpha}{\left(1+\sqrt{t^{2/\gamma_1}z_1^2+t^{2/\gamma_2}z_2^2+|t^{1/\gamma_3}z_3- t^{1/\gamma_1+1/\gamma_2}z_1z_2/2|}\right)^{\alpha + 4}}\\ &= t^{-4/\alpha}
 \sum_{z\in \delta_t^{-1}(\mathbb Z^3)} f(z)  \frac{c_\alpha}{\left(t^{-1/\alpha}+\sqrt{z_1^2+z_2^2+|t^{1/\gamma_3-2/\alpha}z_3- z_1z_2/2|}\right)^{\alpha + 4}}\\
 &=   t^{-4/\alpha}
 \sum_{z\in (t^{-1/\alpha}\mathbb Z)^2\times t^{-2/\alpha}\mathbb Z} \frac{c_\alpha f((z_1,z_2,0)) }{\left(t^{-1/\alpha}+\sqrt{z_1^2+z_2^2+|z_3- z_1z_2/2|}\right)^{\alpha + 4}} \\
 &\quad+  t^{-4/\alpha}
 \sum_{z\in (t^{-1/\alpha}\mathbb Z)^2\times t^{-2/\alpha}\mathbb Z} \frac{c_\alpha (f((z_1,z_2,t^{-1/\gamma_3+2/\alpha} z_3))- f((z_1,z_2,0))) }{\left(t^{-1/\alpha}+\sqrt{z_1^2+z_2^2+|z_3- z_1z_2/2|}\right)^{\alpha + 4}}.
   \end{align*}

 The first term is, essentially,  a (multivariate, generalized) Riemann sum of a uniformly continuous integrable function on $\mathbb R^3$ over the lattice
 $(t^{-1/\alpha}\mathbb Z)^2\times t^{-2/\alpha}\mathbb Z $ and, consequently, it  converges when $t$ tends to infinity to
 \begin{align*} \int_{\mathbb R^3}  \frac{ c_\alpha f(x_1,x_2,0) }{\left(\sqrt{x_1^2+x_2^2+|x_3- x_1x_2/2|}\right)^{\alpha + 4}}dx_1dx_2dx_3=
 \int_{\mathbb R^2}\frac{c'_\alpha f(z_1,z_2,0)}{(z_1^2+z_2^2)^{(\alpha+2)/2}} dx_1dx_2,\end{align*}
 where $c'_\alpha=2c_\alpha\int_0^\infty \frac{du}{(1+u)^{(\alpha+4)/2}}.$

The second term goes to $0$ when $t$ tends to $\infty$  because $f$ is uniformly continuous and $1/\gamma_3-2/\alpha>0$: for any $\eps>0$ there is a $T_\eps$ such that for all $t>T_\eps$,
$$
  | ((z_1,z_2,t^{-1/\gamma_3+2/\alpha} z_3))-f((z_1,z_2,0))|<\eps.
  $$
   This gives
 \begin{align*}\lefteqn{  t^{-4/\alpha}
 \sum_{ (t^{-1/\alpha}\mathbb Z)^2\times t^{-2/\alpha}\mathbb Z} \frac{c_\alpha |f((z_1,z_2,t^{-1/\gamma_3+2/\alpha} z_3))- f((z_1,z_2,0))| }{\left(\sqrt{z_1^2+z_2^2+|z_3- z_1z_2/2|}\right)^{\alpha + 4}}}&&\\
 &\le   \eps   t^{-4/\alpha}
 \sum_{(t^{-1/\alpha}\mathbb Z)^2\times t^{-2/\alpha}\mathbb Z} \frac{c_\alpha}{\left(\sqrt{z_1^2+z_2^2+|z_3- z_1z_2/2|}\right)^{\alpha + 4}} . \end{align*}
When $t$ tends to infinity, the limit of the right-hand side is
$$\eps  \int_{\mathbb R^3}  \frac{  c_\alpha } {\left(\sqrt{x_1^2+x_2^2+|x_3- x_1x_2/2|}\right)^{\alpha + 4}}dx_1dx_2dx_3.$$
As $\eps>0$ is arbitrary, this proves that
$$\lim_{t\to \infty}  t^{-4/\alpha}
 \sum_{ (t^{-1/\alpha}\mathbb Z)^2\times t^{-2/\alpha}\mathbb Z} \frac{c_\alpha (f((z_1,z_2,t^{1/\gamma_3} z_3))- f((z_1,z_2,0))) }{\left(t^{-1/\alpha}+\sqrt{z_1^2+z_2^2+|z_3- z_1z_2/2|}\right)^{\alpha + 4}} =0$$
 as desired.
 \end{exa}

\subsection{Vague convergence of jump measures and kernels} \label{S:4.2}

\begin{pro}\label{weaklimmeas} 
Let $\Gamma\subset G$ be a discrete co-compact subgroup of the simply connected nilpotent Lie group $G=(\mathbb R^d,\cdot)$.
Let $c(\Gamma,G)$ be the Haar volume of $G/\Gamma$ $($i.e., of a fundamental domain for $\Gamma$ in $G$$)$.
Let $\mu$ be a probability measure on
$\Gamma$,  $(\delta_t)_{t>0}$ be an approximate group dilation structure on $G$ which is admissible for $\mu$, 
 and let $\mu_t := t \delta_{1/t}(\mu) $.
Suppose that $ \mu_t$  converges vaguely  on $\mathbb R^d\setminus \{0\}$ to a Radon measure $\mu_\bullet$ 
 as $t$ tends to infinity.
Then, for any continuous and compactly supported function $\phi$ in $\mathbb R^d\times \mathbb R^d\setminus \Delta$,
 the positive Radon measure
$ J_t(dxdy)$ on $\mathbb R^d\times \mathbb R^d \setminus \Delta$ defined by
\begin{eqnarray*}
\lefteqn{\iint_{ \mathbb R^d\times \mathbb R^d \setminus \Delta} \phi(x,y)J_t(dxdy)= }&&\\
&&\hspace{.5in}  c(\Gamma,G)t \det (\delta_{1/t})\sum_{x,y\in \delta_{1/t}(\Gamma), \;x\neq y} \phi(x,y) \mu (\delta_t (x)^{-1}\cdot\delta_t(y))\end{eqnarray*}
converges vaguely as $t$ tends to infinity to the positive   Radon measure $J_\bullet$ defined on $(\mathbb R^d\times \mathbb R^d ) \setminus \Delta$ by
$$\iint_{\mathbb R^d\times \mathbb R^d \setminus \Delta} \phi(x,y)J_\bullet(dxdy)= \iint_{\mathbb R^d\times \mathbb R^d \setminus \Delta}\phi(x,x\bullet y) dx\mu_\bullet ( dy),$$
where $\bullet$ is the limit law  $x\bullet y=\lim_{t\ra \infty} \delta_t^{-1}(\delta_t(x)\cdot \delta_t(y))$ for the approximate Lie dilation structure $(\delta_t)_{t>0}$.
\end{pro}
\begin{rem}
Of course, in the group $G_\bullet=(\mathbb R^d,\bullet)$, we can write
$$\iint_{\mathbb R^d\times \mathbb R^d \setminus \Delta}\phi(x,x\bullet y) dx\mu_\bullet ( dy)= \iint_{\mathbb R^d\times \mathbb R^d \setminus \Delta}\phi(x,y) dx\mu_\bullet (x_\bullet^{-1}\bullet dy)$$
 (the inverse operation is in $(G,\bullet)$) so that $$J_\bullet(dxdy) = dx\mu_\bullet(x_\bullet^{-1}\bullet dy).$$
\end{rem}

\begin{rem} Note that the measure $J^1_t(dx)$ defined by
$$\int \phi(x) J^1_t(dx)=  c(\Gamma,G) \det (\delta_{1/t})\sum_{x
\in \delta_{1/t}(\Gamma)} \phi(x)$$
obviously converges to $\int \phi(x)dx$ as $t$ tends to infinity. That is, the vague limit of $J^1_t(dx)$ is the Lebesgue (=Haar) measure on $\mathbb R^d$.
\end{rem}

\begin{proof}   Observe that
\begin{equation}
x\cdot_t y= \delta_t^{-1}(\delta_t(x)\cdot \delta_t(y))\label{def-Gt}\end{equation} is a group law which  turns $\mathbb R^d$ into a Lie group $G_t=(\mathbb R^d, \cdot_t)$  (this group is actually isomorphic to $G$).  For any fixed  $x\in \delta_{1/t}(\Gamma)$, consider
\begin{align*}
t\sum_{y\in \delta_{1/t}(\Gamma) \setminus \{x\}}  \phi(x,y) \mu (\delta_t (x)^{-1}\cdot\delta_t(y))
&=  t\sum_{y\in \delta_{1/t}(\Gamma)  \setminus \{x\}}   \phi(x,y) \mu (\delta_t (x^{-1}\cdot_ty))\\
&=   t\sum_{y\in \delta_{1/t}(\Gamma)  \setminus \{x\}}    \phi(x,x \cdot _ty) \mu (\delta_t(y)).\end{align*}
Now, write
\begin{align*}
\lefteqn{\iint_{\mathbb R^d\times \mathbb R^d \setminus \Delta} \phi(x,y)J_t(dxdy) -\iint_{\mathbb R^d\times \mathbb R^d \setminus \Delta} \phi(x,y) J_\bullet(dxdy)} &\\
&=   \iint_{\mathbb R^d\times \mathbb R^d \setminus \Delta} \phi(x,y)J_t(dxdy) - \sum_{y\in \delta_{1/t}(\Gamma)
  \setminus\{e\}} \int_{ \mathbb R^d} \phi(x,x\bullet y) t\mu( \delta_t(y))  dx \\
&\quad+  \sum_{y\in \delta_{1/t}(\Gamma)  \setminus\{e\}} \int_{ \mathbb R^d}  \phi(x,x\bullet y) t\mu( \delta_t(y))  dx
-\iint_{\mathbb R^d\times \mathbb R^d \setminus \Delta} \phi(x,y) J_\bullet(dxdy) \\
&=   I_1(t)+I_2(t).
\end{align*}
To bound $|I_1|$, write  $c_t= c(\Gamma,G)\det (\delta_{1/t})$ and
\begin{align*}
|I_1(t)|= \left|\sum_{y\in \delta_{1/t}(\Gamma) \setminus \{e\}} t\mu(\delta_t(y)) \left( c_t\sum_{x\in \delta_{1/t}(\Gamma)} \phi(x,x\cdot_ty)
-\int \phi(x,x\bullet y)dx\right)\right|  .
\end{align*}

Note that $\phi$ is continuous and compactly supported in $\mathbb R^d\times \mathbb R^d \setminus \Delta$
and $x\cdot_ty$ converges (uniformly on compact sets) to $x\bullet y$.  It follows that there is a compact set $K=K_\phi$ in $\mathbb R^d\setminus \{0\}$
with the property that, for  any $\eps>0$, there is $T$ such that, for all $y\in \mathbb R^d$ and all $t>T$,
 $$ \left| c_t\sum_{x\in \delta_t^{-1}\Gamma} \phi(x,x\cdot_ty)
-\int \phi(x,x\bullet y)dx\right|\le \eps \1_K(y).$$
Also, there exist $C_K$ and $T'$ such that for all $t>T'$,   $t\mu(\delta_t (K)) \le  C_K$.
It follows that $|I_1(t)|\le \eps C_K$.  As for $|I_2(t)|$,  the fact that it converges to $0$ is a consequence of the vague convergence of
$t\delta_t^{-1}(\mu)$ to $\mu_\bullet$ on $\mathbb R^d\setminus \{0\}$.
\end{proof}

The jump kernel $J_t$ introduced above  is defined on $\mathbb  R^d\times \mathbb R^d\setminus\Delta$ and acts on functions
of $x,y\in \mathbb  R^d\times \mathbb R^d\setminus\Delta$.  It is useful to  consider also a related discrete jump kernel supported on
$$\Gamma_t\times \Gamma_t\setminus \Delta,$$
where $\Gamma_t=\delta_t^{-1}(\Gamma)$
(by abuse of notation, we use the letter $\Delta$ to demote the diagonal on $R\times R$ for any space $R$, e.g., $R=\mathbb R$ or $R=\Gamma_t$).
Note that $\Gamma_t$ is a co-compact subgroup of the group $G_t=(\mathbb R^d,\cdot_t)$ defined at (\ref{def-Gt}) and that $\delta_t$ provides a group isomorphism from $\Gamma_t$ onto $\Gamma$.  We equipped $\Gamma_t$ with the rescaled counting measure
\begin{equation} \label{def-mt}
m_t(A)= c(\Gamma,G) \det (\delta_t^{-1}) |A|, \quad  \mbox{where } |A|= \# A\end{equation}
 for any finite subset $ A\subset \Gamma_t$. On $\Gamma_t$, we consider the jump kernel measure $j_t$  defined by
\begin{equation}
j_t(x,y)= c(\Gamma,G)t \det (\delta_{1/t}) \mu (\delta_t (x)^{-1}\cdot\delta_t(y)), \quad
  (x, y)  \in \Gamma_t\times \Gamma_t\setminus\Delta.
\label{def-jt}
\end{equation}
  We now assume that  the probability measure $\mu$ on $\Gamma$ is  symmetric.
Then $j_t(x, y)$ is symmetric in $(x, y)$ and it gives arise  to  an associated symmetric
   Dirichlet form   in $L^2(\Gamma_t,m_t)$   with domain $\sF^{(t)}:=L^2(\Gamma_t,m_t)$  defined by
\begin{equation}\label{def-Et}
\mathcal E^{(t)}(u,v)= \frac{1}{2} \sum_{x,y\in \Gamma_t}
{(u(x)-u(y))(v(x)-v(y))} {j_t(x,y)},\quad u, v\in    \sF^{(t)}  .
\end{equation}
 The infinitesimal generator  of this Dirichlet form on $L^2(\Gamma_t,m_t)$  is
\begin{equation}\label{eq:gene-scale}
f\mapsto  -t(f-f*_{\Gamma_t}\delta_t^{-1}(\mu))
\end{equation}
on $\Gamma_t$.

Recall that $\Gamma_t \subset \mathbb R^d$. For each $x\in \mathbb R^d$, let $[x]_t\in \Gamma_t$  be the point closest to $x$ in the $\|\cdot\|$-norm (if there are more than two such points, we choose one arbitrary and fix it).
When needed, extend a function $ f$ on $\Gamma_t$ to a function $\tilde f $ on $\mathbb R^d$ by setting
$f(x)=f([x]_t)$ for each $x\in \mathbb R^d$. We say a family of functions $\{f_t: \Gamma_t\to \mathbb R\}_{t\ge 1}$
converges uniformly to a function $f$ on $\mathbb R^d$ if $\tilde f$ converges uniformly to $f$.

The following is an easy consequence of Proposition \ref{weaklimmeas} that relates to $j_t$. It is stated for continuous limit but it obviously holds as well for sequential limits based on an arbitrary sequence $t_k$ tending to infinity.
\begin{lem}\label{weaklim-rev}
Let $\{f_t: \Gamma_t\to \mathbb R\}_{t>0}$ {\rm (resp. $\{g_t: \Gamma_t\to \mathbb R\}_t$)}
be a family of continuous functions that converges uniformly to a continuous function $f$ {\rm (resp. $g$)} on
$\mathbb R^d$. Then, under the assumptions of Proposition \ref{weaklimmeas},  for any open set $U\subset \mathbb R^d\times \mathbb R^d\setminus \{(x,y):
 \|x_\bullet ^{-1}\bullet y\|_2\le \eta\}$ with $\eta>0$
whose closure is compact, it holds that
\begin{eqnarray*}
   &&   \lim_{t\to\infty}   \sum_{(x,y)\in (\Gamma_t \times \Gamma_t)\cap U} (f_t(x)-f_t(y))
(g_t(x)-g_t(y))j_t(x,y)  \\
&& =   \iint_U (f(x)-f(y))(g(x)-g(y))J_\bullet(dxdy).
\end{eqnarray*}
\end{lem}
\begin{proof} Set $\psi_t(x,y):=(f_t(x)-f_t(y))(g_t(x)-g_t(y))$ and $\psi(x,y):=(f(x)-f(y))(g(x)-g(y))$.
Then
\begin{eqnarray*}
&&\Big|\sum_{(x,y)\in (\Gamma_t \times \Gamma_t)\cap U} \psi_t(x,y)j_t(x,y)-
\iint_U \psi(x,y)J_\bullet(dxdy)\Big|\\
&&\le  \Big|\sum_{(x,y)\in (\Gamma_t \times \Gamma_t)\cap U} (\psi_t(x,y)-\psi(x,y))j_t(x,y)\Big|\\
&&~~+\Big|\iint_U (\psi(x,y)J_t(dxdy)-\psi(x,y)J_\bullet(dxdy))\Big|=:I_1+I_2.
\end{eqnarray*}

By Proposition  \ref{weaklimmeas},  $\sup_{t\ge 1} \sum_{(x,y)\in (\Gamma_t \times \Gamma_t)\cap U}
j_t(x,y)<\infty$.  It follows that $\lim_{t\to \infty} I_1=0$ because $\psi_t$ converges uniformly to $\psi$.
By the proof of Proposition \ref{weaklimmeas}
(and the fact that $U$ is compact in $\mathbb R^d\times \mathbb R^d\setminus \Delta$), $\lim_{k\to\infty} I_2=0$.
\end{proof}

 \section{Weak convergence of the processes} \label{sect-weakc}

\subsection{Assumption {\bf (A)}}\label{S:5.1}

In this section and next, we prove limit theorems involving
\begin{enumerate}
\item A finitely generated torsion free nilpotent group $\Gamma$ embedded as a
co-compact lattice in a simply connected nilpotent Lie group $G$;

\item A symmetric probability measure $\mu$ on $\Gamma$;

\item A polynomial coordinate system for $G=(\mathbb R^d,\cdot)$ and straight dilation structure
\begin{equation}\label{e:4.1a}
\delta_t,  \, t>0, \quad \delta_t((u_i)_1^d)=(t^{1/\beta_i} u_i)_1^d,
\quad  \beta_i\in (0,2), \quad i=1,\dots,d ,
\end{equation}
 which is an approximate group dilation structure for $G$ with limit group $G_\bullet=(\mathbb R^d,\bullet)$.
\end{enumerate}
The key hypothesis we will make that links together the   probability
measure $\mu$, the dilation structure $(\delta_t)_{t>0}$ and the limit group $G_\bullet$ is that
\begin{itemize}
\item[(A)] The straight dilation structure $(\delta_t)_{t>0}$ is admissible for the probability measure $\mu$, that is, the (positive) 
  measure $\mu_t= t\delta_{1/t}(\mu)$, $t\ge 1$, defined at (\ref{def-mut}) converges vaguely   to a  non-trivial Radon measure $\mu_\bullet$ on $\mathbb R^d\setminus\{0\}$
as $t$ tends to infinity.
\end{itemize}

  \begin{rem} \label{R:5.1} \rm
	\begin{enumerate}
	\item[(i)]   The Radon measure $\mu_\bullet$ appeared in (A) is on $\R^d\setminus \{0\}$ and is expressed
	under the global coordinate system we use for the nilpotent group $G$ and hence for $G_\bullet$.
	It induces a Radon measure of $G_\bullet$ through this global coordinate system.
	By abusing the notations, we use the same
	notation $\mu_\bullet$ for the induced measure on $G_\bullet$.

\item[(ii)] Under assumption (A),  it follows from the definition of $\mu_t$ that  $\mu_\bullet$ is a symmetric measure
 on $G_\bullet \setminus \{e\}$ and  has the following scaling property
  \begin{equation} \label{e:4.1}
 \delta_r (\mu_\bullet ) =r  \mu_\bullet \quad \hbox{for every } r>0;
\end{equation}
that is,  for any Borel measurable set $A\subset \R^d\setminus\{0\}$,
 $\mu_\bullet (A)= \mu_\bullet (A_\bullet^{-1})$,
  where $A_\bullet^{-1}:=\{x\in \R^d: x_\bullet^{-1} \in A\}$, and
 $$
 \mu_\bullet  (\delta_r^{-1} (A) ) =\mu_\bullet (\delta_{1/r}(A))= r  \mu_\bullet  (A) \quad \hbox{for every } r>0.
$$

\end{enumerate}
\end{rem}

We are most interested in the case the limit measure $\mu_\bullet$ is  
not supported on a proper closed connected subgroup of  $G_\bullet$. In that case, the condition that the exponents
 $\{\beta_i ,  1\leq i  \leq d\}$  for the straight dilation structure  $\{ \delta_t; t\geq 0\}$ of \eqref{e:4.1a}
are in $(0,2)$ means that the original measure $\mu$ must have some sort of heavy tail characteristics, i.e., $\mu$ has to be ``stable-like''.

\subsubsection*{Geometries on  $\mathbb R^d$ and $G_\bullet$}
Fix $\beta\ge \max_{1\le i\le d}\{\beta_i\}$.
By \cite{Hebish1990},
there is a norm $\|\cdot\|$ on $G_\bullet=(\mathbb R^d,\bullet)$ (this means that $\|x\bullet y\|\le \|x\|+\|y\|$ for all $x,y\in \mathbb R^d$, $\|x_\bullet^{-1}\|=\|x\|$ and $\|x\|=0$ if and only if $x=0$)
such that
\begin{equation} \label{e:4.2}
\|\delta_t(u)\|=t^{1/\beta}\|u\|
\quad \hbox{for every } t>0 \hbox{ and } u=(u_i)_1^d\in \mathbb R^d.
\end{equation} 
 This implies, of course, that there are {  constants  $c,C\in (0,\infty)$ such that
\begin{equation}\label{e:5.4}
c \max_{1\le i\le d}\{|u_i|^{\beta_i/\beta}\}\le \|u\|\le C \max_{1\le i\le d}\{|u_i|^{\beta_i/\beta}\}
\quad \hbox{for } u=(u_i)_1^d \in \R^d.
\end{equation} 
Note that $ \max_{1\le i\le d}\{|u_i|^{\beta_i/\beta}\}$ itself  is a norm on
 $(\mathbb R^d,+)$ but not necessarily on $G_\bullet=(\mathbb R^d,\bullet)$
(it may not be symmetric on $G_\bullet$ and only satisfies the triangle inequality up to a multiplicative constant in general).
Set
$$
B(r)=\left\{x\in \mathbb R^d: \|x\|<r \right\}.
$$
Obviously, we have
$$
\|\delta_t (u) \|=t^{1/\beta} \|u\|
\quad \hbox{and} \quad
\delta_t (B(r))= B(rt^{1/\beta}).
$$
This means that the volume (the Lebesgue measure) of $B(r)$  is
$$m(B(r))=m(\delta_{r^\beta}(B(1)))= m(B(1))\det(\delta_{r^\beta})= m(B(1))r^{\beta(\sum_{i=1}^d1/\beta_i)}.$$

Recall that $\mathbb R^d$ is also equipped with the Euclidean norm $\|u\|_2=\sqrt{\sum_1^d|u_i|^2}.$
Let
$$\beta_-=\min_{1\le i\le d} \beta_i  \quad \hbox{ and } \quad
\beta_+=\max_{1\le i\le d} \beta_i .
$$
From the definition, it is clear that
\begin{equation} \label{eq-normcomp2}
c \min\{\|u\|^{\beta /\beta_-},\|u\|^{\beta/\beta_+}\}\le \|u\|_2\le C \max\{\|u\|^{\beta /\beta_-},\|u\|^{\beta/\beta_+}\}.
\end{equation}
Similarly, for any $u\in \mathbb R^d$ with $\|u\|\le C_1 r^{1/\beta}$, we have
\begin{equation}
 c_2 \left(\frac{\|u\|_2}{r}\right)^{\beta_+/\beta}  \le   \frac{\|u\|}{r^{1/\beta}} \le  C_2 \left(\frac{\|u\|_2}{r}\right)^{\beta_-/\beta}.
\end{equation}

We will need the following version of Lemma \ref{lem2-6}   with respect to the norm $\| \cdot \|$.

  \begin{lem}\label{lem2-6norm}
 For any compact $K\subset \mathbb R^d$
  there is a constant $C_K$ such that, for any $x,y\in K$ and $t\ge 1$,
  $$\|   \delta_{1/t}\left(\delta_t(x)^{-1}\cdot\delta_t(y)\right)  \|\le C_K \|y-x\|^{\beta_-/\beta_+}$$
  and
  $$\|  \delta_{1/t}\left(\delta_t(x)^{-1}\cdot \delta_t(y)\right)  \|\le C_K \|x_\bullet^{-1}\bullet y\|^{\beta_-/\beta_+}.$$  \end{lem}

  \begin{proof}
  In view of \eqref{e:3.1a} and \eqref{e:3.4a}, the  function
  $(t,x,y)\mapsto  \delta_{1/t}\left( \delta_t(x)^{-1}\delta_t(y) \right)$ is a polynomial in
  $$(x,y)=(x_1,\dots,x_d,y_1,\dots, y_d)$$ with coefficients equal  to linear combination of powers of $t$ with exponents in $\mathbb R$.   By (\ref{keylim}), only  non-positive powers of $t$ appear.  The desired inequality follows
   from \eqref{eq-normcomp2}
    because this polynomial function equals $0$ when $x=y$.
 For the  second inequality, we first note
   from \eqref{eq-normcomp2} again  that for $x, y \in K$,
 $$\|    \delta_{1/t}\left(\delta_t(x)^{-1}\cdot \delta_t(y)\right)   \|\le C_K \|x-y\|_2^{\beta_-/\beta}$$
 and then observe that
  $\|x-y\|_2\le  C'_K\|x_\bullet^{-1} \bullet y\|^{\beta/\beta_+}$.
  \end{proof}

\subsection{Further hypotheses}\label{S:5.2} 

Under the general  circumstances described above, in order to obtain  limit theorems relating the random walk on $\Gamma$ driven by $\mu$ to the continuous time  left-invariant jump process on $G_\bullet$ associated with
  the jump measure
$J_\bullet$ of Proposition \ref{weaklimmeas}, we need several additional hypotheses which we now spell out in details.  One important feature of the various hypotheses described in this section is that they do not involve the precise limit behavior of $\mu_t$ as $t$ tends to infinity. In a non-technical sense, they are of a coarser, more robust nature.  In   Section \ref{S:10},  we will exhibit a large class of ``stable-like'' measures on $\Gamma$, all of which satisfy these hypotheses thanks to the results of \cite{SCZ-nil,CKSWZ1}.

\subsubsection*{The random walk on $\Gamma$ (regularity)}
A bounded function $u$ on $\Gamma$ is called $\mu$-harmonic in  a subset $U$ if it satisfies
$$u*\mu=u \quad \mbox{ in } U.$$
Consider the following basic regularity assumption regarding $\mu$-harmonic functions.  Note that we consider that $\Gamma$ as a subgroup of $G=(\mathbb R^d,\cdot)$ and use the $G_\bullet$-norm $\|\cdot\|$ to state this property.

\begin{itemize}
\item[(R1)] There are constants $C_1$ and $\kappa$  such that, for any bounded  function $u$ defined on $\Gamma$ and $\mu$-harmonic in $B(r)=\{x\in \mathbb R^d: \|x\|<r\}$, $r>0$,
and  all
  $x,y\in \Gamma \cap B(r/2)$,
we have
\begin{equation}\label{eq:harm}
|u(y)-u(x)|\le C_1 \|u\|_{\infty}\left(\frac{\|x^{-1}\cdot y\|}{r}\right)^{\kappa}.
 \end{equation}
\end{itemize}
\begin{rem}  For any fixed   $a>0$, changing $\|\cdot\|$ to
$\|\cdot \|^a$ (including in the definition of balls) amounts to  changing $\kappa$ to $\kappa/a>0$.
\end{rem}

 \subsubsection*{Exit time estimates}
We consider the following exit time hypotheses formulated   in terms   of the norm $\|\cdot\|$
  and the scaling exponent $\beta>0$ associated with it in \eqref{e:4.2}.
In particular, the balls appearing in the definition below are the balls $B(r)=\{x\in \mathbb R^d: \|x\|<r\}$, $r\ge0$, even so the exit
probability estimates below concern the random walk on $\Gamma$.
 \begin{itemize}
\item[(E1)] There exists $A>1$ such that the following holds:
for any $\eps\in (0,1)$, there exists $\gamma=\gamma(A,\eps)>0$ such that for
  any $r>0$,  we have
$$\mathbb P^x \left( \tau_{B(Ar)}\le \gamma r^{\beta} \right)\le \eps
\quad  \hbox{for all } x \in \Gamma \cap B(r) .
$$
\item[(E2)] There exists $0<C<\infty$ such that
  for any $r>0$, we have
$$\mathbb E^x \left[\tau_{B(r)} \right] \le C  r^{\beta}
\quad   \hbox{for all } x \in \Gamma \cap B(r)  .
$$
\end{itemize}
Here,   $\mathbb P^x$ and $\mathbb E^x$   refer to the random walk on $\Gamma$
 starting from $x$  driven by the probability measure $\mu$.

\medskip

\begin{rem} One  can esily check that Assumptions (E1)-(E2), together, are equivalent to $\mathbb E^x [\tau_{B(r)}]\asymp  r^{\beta}$   for any  $x\in \Gamma \cap B(r/2)$.
\end{rem}

\medskip

\begin{rem}
  For our limit theorems to hold,    the exponent $\beta >0$ in (E1) and (E2) {needs to be} the same exponent $\beta$ in \eqref{e:4.2}.
Thus,  in this context,  conditions (E1) and (E2) as well as condition (R1) are not only a condition on the measure $\mu$ (which determines the random walk $X_n$ on  $\Gamma$ and hence its harmonic functions) but also a condition on its comparability   with
 the dilation structure $(\delta_t)_{t>0}$, scaled measure $\mu_t=t\delta_{1/t}(\mu)$, and  norm $\|\cdot\|$ on $\mathbb R^d$.
 We expect the rescaled random  walks $(\delta_{1/k}(X_{[kt]}))_{t>0}$ to converge when $k$ tends to infinity
to a self-similar process $(Z_t)_{t>0}$ satisfying $ \delta_{1/s}(Z_{st})=Z_t$ for all $s,t>0$.
From the definition of  $\|\cdot\|$
  at \eqref{e:4.2},  the expected exit time  out of a ball of radius $r$ for this process should scale as $r^\beta$.  Moreover, the random  walk  exit time of the ball of radius $r$ is
$$
\tau_{B(r)}= \inf \{n: X_n\not\in B(r)\} =r^\beta \inf \left\{ n/r^\beta:  \delta_{1/r^\beta} (X_{ r^\beta(n/r^\beta)}) \not\in B(1) \right\}
$$
and we expect that, as $r$ tends to infinity,
$$
\inf \left\{ n/r^\beta:  \delta_{1/r^\beta} (X_{r^\beta(n/r^\beta)}) \not\in B(1) \right\} \to \inf\{s: Z_s\in B(1)\}$$
so that  $\mathbb E^e \left[ \tau_{B(r)} \right]$ should indeed behave as $r^\beta$.
\end{rem}

\subsubsection*{Tails properties for $J_t$ and $J_\bullet$}

We now discuss two related sets of hypotheses that are more technical but essential to obtain the desired results.  They concern the limit jump measure $J_\bullet$ and the  rescaled jump measures $J_t$  for large $t>0$. Theses hypotheses will have a natural flavor to anyone familiar with L\'evy processes  and Dirichlet forms.  They complement the vague convergence of $J_t$ to $J$ on $( \mathbb R^d\times \mathbb R^d) \setminus \Delta$.

Set
$$
B_\bullet(x,r)= x\bullet B(r)= \left\{y\in \mathbb R^d: \|x_\bullet^{-1}\bullet y\|<r \right\}.
$$
Concerning  the  limit Radon measure
$$
J_\bullet(dxdy)= dx\mu_\bullet(x_\bullet^{-1}\bullet dy)
$$
 on $(\mathbb R^d\times \mathbb R^d )\setminus \Delta$ from Proposition \ref{weaklimmeas},
  which is symmetric by Remark \ref{R:5.1},
 consider the hypothesis that
\begin{itemize}
\item[(T$\bullet$)] For any fixed compact set $K\subset \mathbb R^d$, \begin{eqnarray} \label{WCA2-A}
 \lim_{\eta\to 0}
 \iint_{\{(x, y)\in K\times K: \|x_\bullet^{-1}\bullet y\|_2\leq  \eta\}} \|x_\bullet^{-1}\bullet y\|_2 ^2
 \,  J_\bullet(dx,dy)&=0,\\\label{WCA1-B}
   \lim_{R\to \infty}
 \int_{x\in K}\int_{y\in B_\bullet(x,R)^c} J_\bullet(dx,dy) &=0.
\end{eqnarray}
\end{itemize}
Note that, because $J_\bullet(dxdy)=dx\mu_\bullet (x_\bullet^{-1}dy)$,
  where $\mu_\bullet$ is a Radon measure on $G_\bullet \setminus \{e\}=\R^d \setminus \{ 0\}$,
 condition (T$\bullet$)
  is equivalent to
\begin{equation} \label{e:5.9}
  \int _{G_\bullet} \min\{1,\|z\|_2^2\} \, \mu_\bullet(dz)<\infty ,
  \end{equation}
which is  \eqref{LM} for $\nu=\mu_\bullet$.

Under this hypothesis,   $J_\bullet$ is the jump measure of a symmetric bilinear form
\begin{equation}\label{def-Ebullet}
\sE_\bullet (u, v):= \frac1{2} \iint_{\mathbb R^d\times \mathbb R^d \setminus \Delta}
{(u(x)-u(y))(v(x)-v(y))} {J_\bullet (dx,dy)} \end{equation}
 on $\Lip_c(\mathbb R^d)$, which is the space of Lipschitz functions with compact support.
Moreover, this form is closable  in $L^2(G_\bullet;   dx )$ and its closure is a regular conservative Dirichlet form $(\sE_\bullet,\sF_\bullet)$ -- see, e.g., \cite[Example 1.2.4]{FOT} and \cite[Theorem 1.3]{GHM}.
 Hence by \cite[Corollary 6.6.6]{CF},
\begin{equation}\label{e:5.11}
\sF_\bullet=\left\{ u\in    (\sF_\bullet)_{\rm loc} \cap L^2(G_\bullet; dx): \sE_\bullet (u, u)<\infty \right\}.
\end{equation}

 Recall that by Lemma \ref{L:2.5}, the straight dilation $\{\delta_t, t>0\}$ is a group dilation structure for
the group $(G_\bullet, \bullet)$.
 Denote by $(\sL_\bullet, {\rm Dom}(\sL_\bullet))$  the  infinitesimal generator of  $(\sE_\bullet,\sF_\bullet)$ on $L^2(G_\bullet;   dx)$.  Under  the hypothesis (T$\bullet$), we have
 $C^2_b (\R^d)\cap L^2 (\R^d; m)\subset {\rm Dom}(\sL_\bullet)$
and
\begin{equation}\label{e:4.9}
\sL_\bullet f(x)= \lim_{\eps \to 0} \int_{\{z\in G_\bullet: \| z\| \geq \eps\} } ( f(x\bullet z) - f(x)) \mu_\bullet (dz)
\quad \hbox{for } f\in {\rm Dom}(\sL_\bullet).
\end{equation}
For $f\in   {\rm Dom}(\sL_\bullet)$  and $r>0$,  we have by \eqref{e:4.1} and \eqref{e:4.2}
that for $x\in G_\bullet$,
\begin{align}
\sL_\bullet (f\circ \delta_r)(x)
&=  \lim_{\eps \to 0} \int_{\{z\in G_\bullet: \| z\| \geq \eps\} } ( f( \delta_r( x\bullet z)) - f(\delta_r(x))) \mu_\bullet (dz)
\nonumber \\
&=  \lim_{\eps \to 0} \int_{\{z\in G_\bullet: \| w\| \geq \eps\} } ( f( \delta_r(x) \bullet \delta_r (z) ) - f(\delta_r(x)))    \mu_\bullet   (dz)
\nonumber \\
&=  \lim_{\eps \to 0} \int_{\{z\in G_\bullet: \| w\| \geq \eps\} }  ( f( \delta_r(x) \bullet w) - f( \delta_r(x)) ) (\delta_r \mu_\bullet ) (dw)
\nonumber \\
&=  r \lim_{\eps \to 0} \int_{\{z\in G_\bullet: \| w\| \geq \eps\} } ( f( \delta_r(x) \bullet w) - f( \delta_r(x)) )   \mu_\bullet  (dw)
\nonumber \\
&=  r \sL_\bullet f (\delta_r (x)) .    \label{e:4.10}
\end{align}
In particular, we have for $f\in   C^2_c(\R^d)$  and $r>0$,
\begin{equation}\label{e:4.11}
\sL_\bullet (f\circ \delta_r)(e) = r \sL_\bullet f (e) .
\end{equation}

\medskip

\begin{rem}\label{R:4.5}
Under  the hypothesis (T$\bullet$), let $X^\bullet$ be the symmetric Hunt process on $G_\bullet$
associated with the regular Dirichlet form $(\sE_\bullet,\sF_\bullet)$ on $L^2(G_\bullet,  dx)$; see \cite{CF, FOT}.
In view of \eqref{e:4.9},  $X^\bullet$ has stationary independent increment property; that is, for any $t>s \geq 0$,
$(X^\bullet_s)_\bullet^{-1} \bullet X_t^\bullet$ is independent of $\sigma (X^\bullet_r; r\leq s)$ and has the same distribution
as $(X^\bullet_0)_\bullet^{-1} \bullet X_{t-s}^\bullet$.  Thus $X^\bullet$ can be refined to start from every point in $G_\bullet$.
Moreover, it follows from \eqref{e:4.10} that  if $X^\bullet_0=e$, then
\begin{equation}\label{e:4.12}
\{ \delta_r (X^\bullet_t); t\geq 0\} \hbox{ has the same distribution as }  \{ X^\bullet_{r t}; t\geq 0\}.
\end{equation}
 We know from Lemma \ref{L:2.5} the straight dilations $\{\phi_t, \, t>0\}$ form a group dilation structure for the nilpotent group $(G_\bullet, \bullet)$. Thus in the terminology of \cite[p.170]{KunBR},
the L\'evy process $X^\bullet$ is stable with respect to the dilations  $\{\phi_t, \, t>0\}$.
 \end{rem}

\medskip

\begin{rem}\label{R:4.6}
 \rm In the terminology of \cite[p.31]{Gl},  the scaling property
  \eqref{e:4.11}     says that
  the generating functional  $f\to \sL_\bullet f(e)$
is a kernel of order $\beta_+ = \max_{1\leq i\leq d} \beta_i$.
Observe that in \cite{Gl}, the exponents  $\{d_j, 1\leq j\leq d\} $ for the
straight dilation structure   $\{\delta_t, t>0\}$ are our $\{1/\beta_j,  1\leq j\leq d\}$ and the smallest
$d_j$ there (which corresponds our $\beta_+$, the largest of $\beta_j$) is normalized to 1; see Remark \ref{scale}
for the procedure of doing such a normalization.
Note also that the norm $|\cdot |$ defined on \cite[(1.1)]{Gl} is comparable to our norm $\| \cdot \|$.
\end{rem}

\medskip

Regarding  the scaled jump kernel $J_t$, consider the property
\begin{itemize}
\item[(T$\Gamma$)] For any fixed compact set $K\subset \mathbb R^d$,
 \begin{eqnarray} \label{WCA2*}
\hskip -0.5truein  \lim_{\eta\to 0}\limsup_{t\to \infty}
 \iint_{\{(x, y)\in K^2 : \|x_\bullet^{-1}\bullet y\|_2 \leq  \eta\}} \|x_\bullet^{-1}\bullet y\|_2^2 J_t(dx,dy)&=0,\\ \label{WCA1*}
   \lim_{R\to \infty}\limsup_{t\to \infty}
\int_{x\in K} \int_{y\in B_\bullet(x,R)^c} J_t(dx,dy) &=0.
\end{eqnarray}
\end{itemize}

\begin{rem}\label{rem4-4}
 In the estimates \eqref{WCA2-A} and \eqref{WCA2*}, it is crucial to use the norm $\|\cdot\|_2$ in the integrant in order to measure the strength of small jumps allowed by theses jump kernels in a classical fashion.
In the estimates \eqref{WCA1-B} and \eqref{WCA1*}, it is natural to use the norm $\|\cdot\|$
due to the scaling property of $\{\delta_t\}$, but we may also use $\|\cdot\|_2$ if desired because
  of \eqref{eq-normcomp2}.
\end{rem}

\begin{rem}\label{R:4.8}
  Recalling \eqref{e:5.9},
one can check that conditions  (A), (T$\bullet$) and (T$\Gamma$) combined  
 are equivalent to
the following condition:
\begin{itemize}
\item[(A$'$)] The straight dilation structure $(\delta_t)_{t>0}$ is admissible for the probability measure $\mu$ in the sense that
the finite positive measure $(\| z\|_2^2 \wedge 1)  \, \mu_t (dz)$ converges weakly on $\mathbb R^d\setminus\{0\}$ to a finite measure $(\| z\|_2^2 \wedge 1) \, \mu_\bullet$
as $t$ tends to infinity, where $\mu_t$ is the measure  defined at (\ref{def-mut}).
\end{itemize}
 \end{rem}

Note that Examples \ref{ex3-2}-\ref{ex3-4} satisfy any of these conditions (A), (T$\bullet$),  (T$\Gamma$), (R1)-(R2) and (E1)-(E2).

\subsection{Weak convergence} \label{S:5.3}

Throughout this subsection, we generally assume that  (A)-(R1)-(E1)-(E2) and (T$\bullet$)-(T$\Gamma$) are all satisfied even so we will list exactly which properties are used for different results stated in this subsection.

Because of assumptions (A) and (T$\bullet$), we can consider
 the (continuous time) Markov semigroup of operators
 $$\{P_{\bullet,s}\}_{s\ge 0}$$ corresponding to
$(\sE_\bullet,\sF_\bullet)$ at (\ref{def-Ebullet}).  Let
  $\{U_\bullet^\lam;  \lambda >0\}$, and $\{\P_\bullet^x; x\in G_\bullet\}$
 be the resolvent, and probabilities
corresponding to the   regular Dirichlet form $(\sE_\bullet, \sF_\bullet) $ on $L^2(G_\bullet; dx)$.
Our goal is to prove that the continuous time conservative Markov process associated with this
  regular Dirichlet form  is the limit of the properly rescaled discrete time random walk on $\Gamma$
  driven by the probability   measure $\mu$.   We let $(X_n)_{n\ge 0}$
denote this random walk. Assumptions (R1) and (E1)-(E2) are assumptions regarding the behavior of this discrete time random walk on $\Gamma$.

Fix an arbitrary sequence of positive reals $\{T_k\}$ that goes to $\infty$.
We write $(\sE^{(k)},\sF^{(k)})$ for $(\sE^{(T_k)},\sF^{(T_k)})$
   defined by \eqref{def-Et}   with $T_k$ in place of $t$ there, which corresponds to the
  rescaled discrete time process by
 $$(X_n^{(k)}
:=\delta_{T_k}^{-1}(X_{T_kn}))_{n\in  (1/T_k) {\mathbb N}\cup\{0\}}.
$$
  Note that this is just  discrete time random walk  on $\Gamma_{T_k}$
where time has been rescaled linearly according to the scaling sequence $T_k$.
Let
 $\{ P_{n}^k;  n\in (1/T_k ) {\mathbb N} \cup\{0\} \}$,  $\{U^\lam_k; \lambda>0\}$, and $\{\P^x_k; x\in \Gamma_{T_k}\}$
  be the associated semigroup, resolvent, and probabilities.
For $t\ge 0$, we write
\begin{equation}\label{def-hat}
\hat X_t^{(k)}:=\delta_{T_k}^{-1}(X_{[T_kt]})=X_{[T_kt]/T_k}^{(k)}, \quad
\hat P_t^k:=P_{[T_kt]/T_k}^k
\end{equation}
and denote the corresponding probabilities by
 $\{\hat \P^x_k; x\in \Gamma_{T_k}\}$.
So  for $x,y\in \Gamma_{T_k}$, and $n=m/T_k, \,m\in \mathbb N\cup \{0\}$,
 $$
 \P^x_k( X_n^{(k)}=y) =  \mu^{(m)}(  \delta_{T_k}(x)^{-1}   \cdot \delta_{T_k}(y)),
 $$
and for $x,y\in \Gamma_{T_k}$, and $t>0$,
  \begin{equation}\label{e:4.15}
 \hat \P^x_k( \hat X_t^{(k)}=y) =  \mu^{([tT_k])}(  \delta_{T_k}(x)^{-1}
 \cdot \delta_{T_k}(y)).
   \end{equation}

   For  a constant $M_0>0$, let    $\bD([0,M_0],\R^d)$ be
the space of right continuous functions on $[0,M_0]$ having left
limits and taking values in $\R^d$ that is equipped with
  the Skorohod ${\cal J}_1$ topology.
 Our goal is to prove the following theorem.
 Recall that $\beta_i\in (0, 2)$, $1\leq i\leq d$, are  the parameters in \eqref{e:4.1a} for the straight dilation structure
 $\{\delta_t; t>0\}$
 and $\beta_+ =\max \{\beta_i: 1\leq i \leq d\}$.

\begin{theo}\label{WT1}
Referring to the setup and notation introduced above,
assume that \eqref{e:4.2},  {\rm (A)-(R1)-(E1)-(E2)} and {\rm (T$\bullet$)-(T$\Gamma$)} are all satisfied  with the same exponent $\beta>0$. Then
\begin{enumerate}
 \item[\rm (i)]  The symmetric Hunt process $X^\bullet$ associated with
   the regular Dirichlet form
 $(\sE_\bullet, \sF_\bullet)$ on $L^2(G_\bullet; dx)$ is a L\'evy process on $G_\bullet$.
  The L\'evy process $  X^\bullet_t$ has a  bounded,  strictly positive,    jointly continuous transition density function
 $p(t, x, y)=p(t,   x_\bullet^{-1} \bullet y)$ with respect to $dy$ that has the following properties.
 \begin{description}
 \item[\rm (a)]   Let
    $\gamma_0:= \sum_{i=1}^d 1/\beta_i$.  For every  $ (t,x)\in (0,\infty)\times G_\bullet $,
  \begin{equation}\label{e:4.18}
  p(t,x)=    t^{  - \gamma_0}  p(1,\delta_{1/t}(x))   .
  \end{equation}
   In particular, there is a constant $C_1>0$ so that $ p(t, x) \leq C_1    t^{- \gamma_0}$
   for every $(t,x)\in (0,\infty)\times G_\bullet $.

 \item[\rm (b)] For every $\gamma \in (0, \beta_+ \wedge 1)$, there is a constant $C_2>0$ so that
 \begin{equation}\label{e:4.19}
 | p(1, x) - p(1, y)| \leq C_2 \| x^{-1}_\bullet \bullet y\|^\gamma
 \quad \hbox{for   } x, y \in G_\bullet.
  \end{equation}

  \item[\rm (c)] For every $\beta  \in (0, \beta_+)$, there is a constant $C_3>0$ so that
   for   every $ (t,x)\in (0,\infty)\times G_\bullet $,
  \begin{equation} \label{e:4.20}
  p(t, x) \leq \min \left\{ C_1 t^{ - \gamma_0}, \, C_3  \frac{    t^{\beta/\beta_+}} { \| x\|^{d+\beta} }\right\} .
    \end{equation}

   \end{description}

\item[\rm (ii)] For any bounded continuous function $f$ on $\mathbb R^d$,
$\hat P_s^k f$ converges uniformly on compacts to
$P_{\bullet,s}f$. Furthermore, for each $M_0>0$,  for every  $x\in \mathbb R^d$, $\hat \P^{[x]_k}_k$ converges weakly to $\P_\bullet^x$ on the space $\bD ([0,M_0],\R^d)$.

    \end{enumerate}
\end{theo}

 \medskip

   \begin{rem} \label{R:4.10} \rm
 Note that under its conditions,   Theorem \ref{WT1} in particular implies that
 the L\'evy process $X^\bullet$ is always non-degenerate in the sense that it has a strictly positive
 convolution density kernel
 $p(t, x)$ with respect to the Haar measure $dx$ on $G_\bullet$. Consequently,
 the support of its L\'evy measure $\mu_\bullet$ generates the whole group $G_\bullet$.
  \end{rem}

\subsection{Proof of Theorem \ref{WT1}}\label{S:5.4} 

In this subsection, we prove Theorem \ref{WT1}.
The main part of the argument is based on Section 4 of \cite{BKK}.
Similar arguments for discrete setting
(including a diffusion term in the limit) are given in \cite[Theorem 5.5]{BKU}.

Recall that
$X_n^{(k)}=\delta_{T_k}^{-1}(X_{T_kn})$, $n \in T_k^{-1}{\mathbb N}\cup\{0\} $. We first state  a lemma that is an easy consequence of rescaling, and assumptions   (R1)-(E1)-(E2),
and Lemma \ref{lem2-6norm} (with $\phi_t=\delta_t$).
For $x_0\in \Gamma_{T_k}$, let
$$B_{T_k}(x_0,r)=x_0 \cdot_{T_k} B(r).$$
Note that this is different from $B_\bullet(x_0,r)=x_0\bullet B(r)$ which we have used earlier. Also,
$y\in B_{T_k}(x_0,r)$ if and only if   $ \delta_{1/T_k}(\delta_{T_k}(x_0)^{-1})
\cdot_{T_k} y\in B(r)$ (i.e., we have to take the inverse of $x_0$ in $(\Gamma_{t},\cdot_t)$
 with $t=T_k$).

\begin{lem}\label{keypro}
\begin{enumerate}
\item[\rm (i)]  Assume   {\rm (E1)}. Then, there exists $A>1$ such that the following holds:
for any $\eps\in (0,1)$, there exists $\gamma=\gamma(A, \eps)>0$ such that for all
$k\ge 1$, $x_0\in \Gamma_{T_k}$, $r\in (0,1)$ and $x\in B_{T_k}(x_0,r)\cap \Gamma_{T_k}$,
\[
\bP^{x}_k\left( \tau_{B_{T_k}( x_0,Ar)}(X^{(k)})\le \gamma r^{\beta}\right)\leq \eps.
\]

\item[\rm (ii)] Under {\rm (E2)}, there exists $c_{1}>0$ such that the following
hold for all $k\ge 1$,  $x_0\in \Gamma_{T_k}$, $r\in (0,1)$, and all
$x\in B_{T_k}(x_0,r)\cap \Gamma_{T_k}$,
\[
\bE^{x}_k\left[ \tau_{B_{T_k}(x_0,r)}(X^{(k)})\right]\leq c_{1} r^{\beta}.
\]

\item[\rm (iii)]  Under {\rm (R1)}, there exists
$\kappa\in(0,\infty)$ such that, for any compact set $K\subset \mathbb R^n$,  there is
$c_{2,K}>0$  for which, for  any $k\ge 1$, $x_0\in K \cap \Gamma_{T_k}$ and $r\in (0,1)$,
if $h_k$ is bounded in
$\Gamma_{T_k}$ and harmonic with respect to $X^{(k)}$ in
a ball $B_{T_k}(x_0 , r)\cap \Gamma_{T_k}$ then, for $ x,y\in B_{T_k} (x_0,r/2)\cap K\cap \Gamma_{T_k}$,
\[|h_k(x)-h_k(y)|\leq c_{2,K}
\left(\frac{\|x_\bullet^{-1}\bullet y\|^{\beta_-/\beta_+}}{r}\right)^\kappa
\|h_k\|_{\infty}.
\]
\end{enumerate}
\end{lem}

 \proof  In view of the scaling property \eqref{e:4.2} of  the norm $\| \cdot \|$ on $G_\bullet$,
 properties (i) and (ii)   are just reformulation of conditions (E1) and (E2), respectively,  under   the approximate dilation
$\delta_{T_k}$.

(iii) follows from condition (R1) under the approximate dilation
$\delta_{T_k}$   and Lemma  \ref{lem2-6norm}.
\qed

    \medskip

Recall that for $\lambda >0$, the resolvent   $U_k^\lam$ is given by
\begin{align*}
U_k^\lam f(x)&= (\lambda I - T_k( P  -I))^{-1}f(\delta^{-1}_{T_k}(x))\\
&=  (T_k+\lambda)^{-1}\sum_{n=0}^\infty \Big(\frac 1{1+\lambda T_k^{-1}}\Big)^n
  P^n f(\delta^{-1}_{T_k}(x))
\quad \hbox{for } x\in \Gamma_{T_k}=\delta_{T_k}(\Gamma),
\end{align*}
where   $P$  is the transition matrix for
  the random walk $\{X_n\}_n$ on $\Gamma$.

\medskip

The following proposition is based on \cite[Proposition 2.4]{CCK} (see also \cite[Proposition 3.3]{BKK}). We outline the proof for the reader's convenience.

\begin{pro} \label{PR3}
Under   {\rm (R1)} and {\rm (E2)},
 for any compact set $K$, there exist $C_{\lam, K} \in(0,\infty)$ and $\gamma \in(0,  (\beta \wedge \kappa)/2]$ such that the following holds for any bounded function $f$ on $\Gamma_{T_k}$,
for any $k\ge 1$ and any $x,y\in K\cap \Gamma_{T_k}$
with $\|x_\bullet^{-1}\bullet y\|\le 1$,
\bee\label{cbe000}
|U_k^\lam f(x)-U_k^\lam f(y)|\leq C_{\lam, K}\|x_\bullet^{-1}\bullet y\|^{\gamma}\|f\|_{\infty}.
\eee
In particular, we have
\bee\label{cbe00110}
\lim_{\delta\to 0}\sup_{k\geq 1}\sup_{\substack{x,y\in K\cap \Gamma_{T_k}:\\\|x_\bullet^{-1}\bullet y\|<\delta}}|U_k^\lam f(x)-U_k^\lam f(y)|=0.
\eee
\end{pro}
\proof   Recall the notation  $B_{T_k}(z,r)=z\cdot_{T_k}B(r)$.
Let  $x,y\in  K\cap\Gamma_{T_k}$ and let $r\in (0,1]$ be such that $\|x_\bullet^{- 1}\bullet y\|\le r$.  By Lemma \ref{lem2-6norm}, $y\in B_{T_k}(x,\rho)$,
$\rho=C_Kr^{\beta_-/\beta_+}$.
Set $\tau^k_r:=\tau_{B_{T_k} (x,2\rho)}(X^{(k)})$.
In what follows the constant $C_K$ depend only on $K$ and can change from line to line.
 By the strong Markov property,
\begin{align*}
U_k^\lam f(x)
&= (T_k+\lambda)^{-1}\bE_{k}^x\Big[\sum_{{n\in {\mathbb N}\cup \{0\}}\atop{n\in [0,\tau_r^kT_k]}}
\Big(\frac 1{1+\lambda T_k^{-1}}\Big)^n P^nf(\delta^{-1}_{T_k}(x))\Big]\\ &\quad
+\bE_{k}^x\Big[\Big(\frac 1{1+\lambda T_k^{-1}}\Big)^{\tau_r^kT_k}U_k^\lam f(X^{(k)}_{\tau^k_r})
\Big] \\
&= (T_k+\lambda)^{-1}\bE_{k}^x\Big[\sum_{{n\in {\mathbb N}\cup \{0\}}\atop{n\in [0,\tau_r^kT_k]}}
\Big(\frac 1{1+\lambda T_k^{-1}}\Big)^n P^nf(\delta^{-1}_{T_k}(x))\Big]\\
&\quad
+\bE_k^x \left[ \Big(\Big(\frac 1{1+\lambda T_k^{-1}}\Big)^{\tau_r^kT_k}-1\Big)U_k^\lam f(X^{(k)}_{\tau^k_r}) \right]
 + \bE_k^x \left[ U_k^\lam f(X^{(k)}_{\tau^k_r}) \right]\\
 &=:  I_1+I_2+I_3,
\end{align*}
and similarly when $x$ is replaced by $y$. Because of  Lemma \ref{keypro}(ii)
and the fact that $\|   P f\|_\infty \le \|f\|_\infty$, we have
\[|I_1|\leq  \frac{ T_k}{ T_k+\lambda}  \bE_k^x \left[  \tau^k_r \right]   \|f\|_{\infty}\leq
   c_1 r^{\zeta}\|f\|_{\infty},\quad
\hbox{where } \zeta :=\beta\beta_-/\beta_+.\]
Note that
\[\|U_k^\lam f\|_\infty
\le (T_k+\lambda)^{-1}\frac 1{1-\frac 1{1+\lambda T_k^{-1}}}\|f\|_{\infty}
= \lambda^{-1} \|f\|_\infty.
\]
Using this and applying $1-e^{-s}\le s$, $s\ge 0$, with
$s=\tau^k_rT_k \log (1+\lambda T_k^{-1})$, we have
$$
|I_2|  \leq   \bE_k^x \left[ \tau^k_r\right]  T_k \log (1+\lambda T_k^{-1})
\|U_k^\lam f\|_{\infty}
 \leq    \bE_k^x  \left[ \tau^k_r\right] T_k \lambda T_k^{-1}
\lambda^{-1}\| f\|_{\infty}\le c_1r^{\zeta}\|f\|_{\infty}.
$$
Similar statements also hold when $x$ is replaced by $y$. So,
\begin{eqnarray}
\lefteqn{\left|U_k^\lam f(x)-U_k^\lam f(y)\right|\leq}&&\nonumber\\
&& c_1
r^{\zeta}\|f\|_{\infty}+\left|\bE_k^x\left[ U_k^\lam f(X^{(k)}_{\tau^k_r})\right]
-\bE_k^y \left[U_k^\lam f(X^{(k)}_{\tau^k_r})\right]\right|. \label{C05}\end{eqnarray}
But $z\to \bE_k^z \left[U_k^\lam f(X^{(k)}_{\tau^k_r})\right]$ is bounded in $\Gamma_{T_k}$ and harmonic in
$B_{T_k} (x,2\rho)\cap \Gamma_{T_k}$.
By Lemma \ref{keypro}(iii), for $y\in B_{T_k}(x,\rho)$, the second term in \eqref{C05} is bounded by
$$C_K(\|x_\bullet^{-1}\bullet y\|^{\beta_-/\beta_+}
/r^{\beta_-/\beta_+})^\kappa
\|U_k^\lam f\|_{\infty}.$$
So using $\|U_k^\lam f\|_{\infty}\leq  \lam^{-1} \|f\|_{\infty}$ again,
for $y\in B_\bullet (x,\rho)\cap \Gamma_{T_k}$ we have
\bee\label{cbe01}
\left|U_k^\lam f(x)-U_k^\lam f(y)\right|\leq  C_K
\left(r^{\beta\beta_-/\beta_+}+\lam^{-1}\left(\frac{\|x_\bullet^{-1}\bullet y\|^{\beta_-/\beta_+}}{r^{\beta_-/\beta_+}}
\right)^{\kappa}\right)\|f\|_{\infty}.
\eee
Now choose $r$ such that $r=\|x_\bullet^{- 1}
\bullet y\|^{1/2}$ (then $\|x_\bullet^{-1}\bullet y\|=r^{2}\le r\le 1$).
For this choice of $r$, we obtain
\begin{align*}\lefteqn{
\left|U_k^\lam f(x)-U_k^\lam f(y)\right|}  \\
&\leq  C_K \left(\|x_\bullet^{-1}\bullet y\|^{\beta\beta_-/(2\beta_+)}+\lam^{-1}\|x_\bullet^{- 1}\bullet y\|^{\kappa \beta_-/(2\beta_+)}\right)\|f\|_{\infty}
\\
&\leq  C_K(1+\lam^{-1})\|x_\bullet^{-1}\bullet y\|^{\gamma}\|f\|_{\infty},  \end{align*}
where $\gamma=\min\left\{\frac{\beta\beta_-}{2\beta_+},\frac{\kappa \beta_-}{2\beta_+}\right\}
 \in (0, (\beta \wedge \kappa)/2].$
\qed

\medskip

The   first part of the
 next proposition is based on \cite[Proposition 2.8]{CCK} (see also \cite[Proposition 6.2]{BKu08} and \cite[Section 6]{BaBa}).
  In the following, $m$ denotes the Lebesgue measure on $\R^d$.

\begin{pro}\label{clttight}
Assume   {\rm (A)-(R1)-(E1)-(E2)}.
For every subsequence $\{k_j\}$, there exist a sub-subsequence $\{k_{j(l)}\}$ and
 a conservative  $m$-symmetric Hunt  process $(\wt X, \wt \bP^x, x\in \mathbb R^d)$,
  which is a L\'evy process on $(G_\bullet, \bullet)$,
  such that for every $x_{k_{j(l)}}\to x$,
    $\hat \bP_{k_{j(l)}}^{x_{k_j(l)}}$   converges weakly in $\D ([0,\infty), \mathbb R^d)$ to $\wt \bP^x$.
  Moreover, the resolvents of the conservative Hunt process $\wt X$  map bounded functions on $\R^d$ into bounded
 local H\"older continuous functions on $\R^d$ and so for each $t>0$, $\wt X_t$ has a  transition density function
 $p(t, x, y) = p(t, e, x_\bullet^{-1} \bullet y)$ with respect to $dy$.
  \end{pro}

 \proof  For simplicity, denote by the subsequence $\{k_j\}$ by $\{k\}$.
Let $T_0>0$ an arbitrary constant and $x_k\in \Gamma_k$.     For any stopping time $\eta_k$ of $X^{(k)}$ that is bounded by $T_0$
and any positive constant $\delta_k\to 0$,
it follows from Proposition \ref{keypro}(i) and the strong Markov property of $X^{(k)}$ that for any $\eps >0$,
\begin{align*}
& \limsup_{k\to \infty} \P^{x_k}_k \left ( \| \delta_{1/k}(  \delta_{k} ((X^{(k)}_{\eta_k  })^{-1})  )\cdot_k
   X^{(k)}_{\eta_k + \delta_k} \| > \eps \right) \\
&\leq    \limsup_{k\to \infty}  \E^{x_k}_k \left[     \P^{X^{(k)}_{\eta_k  }}_k ( \tau_{B_k (X_0^{(k)}, \eps)}
 < \delta_k) \right]    =0.
\end{align*}
Thus by  \cite{Al78},
  the probability laws   $\{\hat \bP_k^{x_k}; k\geq 1\}$  are tight on  $\D ([0, T_0), \mathbb R^d)$.
Under conditions (R1)-(E1)-(E2), the proof of the first part of this proposition  (on weak convergence)
 is then similar to that of  \cite[Proposition 2.8]{CCK}, modulo   modifying the arguments for
continuous time processes there to discrete time processes, so we omit this part of the    proof.
 Since $X^{(k)}$ has stationary independent increments on $\Gamma_k$, so does $\wt X$  on $(G_\bullet, \bullet)$.

   That the resolvents of $\wt X$
maps  bounded functions on $\R^d$ into bounded H\"older  continuous functions on $\R^d$ follows readily from
Proposition \ref{PR3}.   For $\lambda >0$, denote by $\wt U^\lambda$ the $\lambda$-resolvent of $\wt X$.
For any Borel measurable set $A\subset \R^d$ having $m(A)=0$, by the $m$-symmetry and conservativeness of  $\wt X$, we have
$ \int_{\R^d}  \wt U^\lambda (x, A) m(dx) =  \lambda^{-1} m(A)=0 $
for every $\lambda>0$.
As $  \wt U^\lambda (x, A)$ is continuous in $x\in \R^d$, we conclude that $\wt U^\lambda (x, A)=0$ for every $x\in \R^d$.
By \cite[Theorem 4.2.4]{FOT}, this   implies that   the law of $X_t$ under $\wt \P^x$ is absolutely continuous with respect to $m$
for each  $t>0$ and $x\in \R^d$.   Denote its density by $p(t, x, y)$.
  By the L\'evy property of $X^\bullet$,  we have $p(t, x, y) = p(t, e,  x_\bullet^{-1} \bullet y)$.
  \qed

\bigskip

\noindent {\bf Proof of Theorem \ref{WT1}.}
 In view of Proposition \ref{clttight}, it suffices to show  that the Dirichlet form
   in $L^2(G_\bullet; dx)$
 of the  conservative  $m$-symmetric process $\wt X$ in Proposition \ref{clttight} is $(\sE_\bullet, \sF_\bullet)$
  and establish (i).
As in the proof of Proposition \ref{PR3},
 we know that any subsequence $\{k_j\}$ has a further
subsequence $\{k_{j_l}\}$  such that $U^\lam_{k_{j_l}}f$ converges uniformly
on compacts whenever  e $\lam >0$ and $f$ is bounded and continuous   on $\R^d$.

Now suppose we have a subsequence $\{k'\}$ such that the $U^\lam_{k'} f$   on $\Gamma_{T_{k'}}$
are equi-continuous and converge uniformly on compacts whenever  $\lam >0$ and $f$ is bounded and
continuous with compact support  on $\R^d$. Fix  $\lam >0$ and such an $f$,  and let
$H\in C_b (\R^d)$  be the limit
of $U^\lam_{k'}f$.  We will show
 that $H\in \sF_\bullet$ and
\bee\label{WCE1}
 \sE_\bullet  (H,g)=\angel{f,g}-\lam\angel{H,g}
\eee
whenever $g$ is a Lipschitz function  on $\R^d$
with compact support, where
 $(\sE_\bullet, \sF_\bullet )$ is the Dirichlet
form of \eqref{def-Ebullet}
and
$\angel{\cdot,\cdot}$ is the $L^2$-inner product with respect to the Lebesgue measure $m $ on $\R^d$.
This will prove that $H$ is the
$\lam$-resolvent of $f$ with respect to
 $(\sE_\bullet, \sF_\bullet)$ in $L^2(\R^d; dx)$,
that is, $H=U^\lam f$.
We can then conclude that the full sequence $U^\lam_k f$ converges to
$U^\lam f$ whenever $f$ is bounded and continuous with compact support.
The assertions about the convergence of $P_t^k$ and $\P^x_k$ then follow
by Proposition \ref{clttight}.

So we need to prove $H$ satisfies \eqref{WCE1}. We drop the primes for
legibility. We know
\bee\label{WCE3}
\sE^{(k)}(U^\lam_k f, U^\lam_k f)=\angel{f,U^\lam_k f}_{ L^2(\Gamma_{T_k}, m_{T_k})}-\lam \angel{U^\lam_k f,
U^\lam_k f}_{  L^2(\Gamma_{T_k}, m_{T_k})},
\eee
 where for $t>0$, $m_t$ is the measure  on $\Gamma_t$ defined by \eqref{def-mt}.
 Since
$$
\norm{U^\lam_k f}_{ L^2(\Gamma_{T_k}, m_{T_k})}
\leq (1/\lam) \norm{f}_{ L^2(\Gamma_{T_k}, m_{T_k})},
$$
we have by the Cauchy-Schwarz inequality that
$$\sup_k \sE^{(k)} (U^\lam_k f, U^\lam_k f)
 \leq \sup_k \lambda^{-1} \|f\|^2 _{L^2(\Gamma_{T_k}, m_{T_k})}
\leq   c<\infty.$$

 Set  $B_2(r)=\{x: \|x\|_2<r\}$.
  Since $U^\lam_{k}  f$  converge uniformly to $H$ on $\ol{B_2( 1/\eta )}$ for  every
$\eta\in (0,1)$, it follows from Lemma \ref{weaklim-rev} that
  \begin{align*}
&\iint_{D_\eta} (H(y)-H(x))^2 J_\bullet(dx,dy)\\
&\leq \limsup_{   k  \to \infty} \sum_{(x,y)\in (\Gamma_{T_ {k}} \times \Gamma_{T_ {k}})\cap D_\eta}
(U^\lam_ {k} f(x)-U^\lam_k f(y))^2 j_ {  k}(x,y)\\
&\leq \limsup_{k\to \infty} \sE^{( {  k})}(U^\lam_ {  k} f, U^\lam_ {  k} f)\leq c<\infty,
\end{align*}
 where $D_\eta:=\{(x,y)\in B_2(e,\eta^{-1})\times B_2(e,\eta^{-1}): \eta
   <  \|x_\bullet^{-1}\bullet y\|_2\le \eta^{-1}\}$.
Letting $\eta\to 0$,  we have
$$
\iint_{ \R^d\times \R^d \setminus \Delta} (H(y)-H(x))^2 J_\bullet(dx,dy) \leq c <\infty.
$$
 Since $H\in C_b(\R^d)$, the above in particular implies that $H\in (\sF_\bullet)_{\rm loc}$.
 Note  that by Fatou's lemma, $H\in L^2(\R^d; dx)$ as it is the pointwise limit of $U^\lam_k f$.
Thus we conclude from \eqref{e:5.11} that
\bee\label{WCE2}
H\in \sF_\bullet \quad \hbox{with} \quad  \sE_\bullet  (H,H)<\infty.
\eee

Fix a Lipschitz function $g$ on $\mathbb R^d$
with compact support, and choose  $r_0>0$   large enough so that
the support of $g$ is contained in   the $L^2$-ball   $B_2(e,   r_0 )$. Then,
setting $H_{\geq \eta^{-1}}:=\{\|x_\bullet^{-1}\bullet y\|_2\geq \eta^{-1}\}$,
\begin{align*}
\Big| \sum_{(x,y)\in (\Gamma_{T_k} \times \Gamma_{T_k})\cap H_{\geq \eta^{-1}}}  &
(U^\lam_k f(y)-U^\lam_k f(x))(g(y)-g(x)) j_k(x,y)\Big|\\
&\leq \Big( \sum_{(x,y)} (U^\lam_k f(y)-U^\lam_k f(x))^2 j_k(x,y)\Big)^{1/2}\\
&\quad\times \Big(\sum_{(x,y)\in (\Gamma_{T_k} \times \Gamma_{T_k})\cap H_{\geq \eta^{-1}}} (g(y)-g(x))^2 j_k(x,y)\Big)^{1/2}.
\end{align*}
The first factor is $(\sE^{(k)}(U^\lam_k f, U^\lam_kf))^{1/2}$, while the
second factor is bounded by
$$
\sqrt 2 \,\norm{g}_\infty
\Big(\int_{B_2(e,   r_0 )}\int_{\|x_\bullet^{- 1}\bullet y\|_2\geq \eta^{-1}} J_k(dx,dy)\Big)^{1/2},
$$
which, in view of \eqref{WCA1*} in (T$\Gamma$), will be small if $\eta$ is small.
Similarly, setting $H_{\leq \eta}:=\{\|x_\bullet^{-1}\bullet y\|_2\leq \eta\}$, it holds that
\begin{align*}
\Big| \sum_{(x,y)\in (\Gamma_{T_k} \times \Gamma_{T_k})\cap H_{\leq \eta}}  &
(U^\lam_k f(y)-U^\lam_k f(x))(g(y)-g(x)) j_k(x,y)\Big|\\
&\leq \Big(  \sum_{(x,y)} (U^\lam_k f(y)-U^\lam_k f(x))^2 j_k(x,y)\Big)^{1/2}\\
&\qq\times \Big(\sum_{(x,y)\in (\Gamma_{T_k} \times \Gamma_{T_k})\cap H_{\leq \eta}}
(g(y)-g(x))^2 j_k(x,y)\Big)^{1/2}.
\end{align*}
The first factor is as before, while the second is bounded by
$$\norm{g}_{\rm Lip} \Big(\int_{B_2(e, r_0 )} \int_{\|x_\bullet^{-1}\bullet y\|_2\leq \eta}\|x_\bullet^{-1}\bullet y\|^2_2
J_k(dx,dy)\Big)^{1/2},$$
where
\[
\norm{g}_{\rm Lip}:=
\sup_{x,y\in \mathbb R^d}\frac{|g(x)-g(y)|}{\|x_\bullet^{-1}\bullet y\|_2}<\infty.
\]
In view of \eqref{WCA2*} in (T$\Gamma$), the second factor will be small if $\eta$ is small.
Similarly, using \eqref{WCE2}, we have
$$\Big|\iint_{\|x_\bullet^{-1}\bullet y\|_2\notin (\eta,\eta^{-1})}
(H(y)-H(x))(g(y)-g(x)) J_\bullet(dx,dy)\Big|$$
will be small if $\eta$ is taken small enough, due to Remark \ref{rem4-4}.

Note that $U^\lam_kf$ are
equi-continuous and converge to $H$ uniformly on compacts, and $g$ is a compactly supported function.
For $\eta>0$ we have by Lemma \ref{weaklim-rev},
\begin{align*}
&\sum_{(x,y)\in (\Gamma_{T_k} \times \Gamma_{T_k})\cap \{{\|x_\bullet^{-1}\bullet y\|_2\in (\eta, \eta^{-1})}\}}
(U^\lam_kf(y)-U^\lam_kf(x))(g(y)-g(x))j_k(x,y)\\
&\to \iint_{\|x_\bullet^{-1}\bullet y\|_2\in (\eta, \eta^{-1})} (H(y)-H(x))(g(y)-g(x))J_\bullet(dx,dy).
\end{align*}
It follows that
\bee\label{WCE4}
\lim_{k\to \infty} \sE^{  (k)  }(U^\lam_kf,g) =    \sE_\bullet  (H,g).
\eee
But  as $k\to \infty $,
$$
\sE^{  (k)}(U^\lam_kf,g)=\angel{f,g}_{  L^2(\Gamma_{T_k}, m_{T_k})}
-\lam\angel{U^\lam_kf,g}_{ L^2(\Gamma_{T_k}, m_{T_k})} \to
\angel{f,g}-\lam\angel{H,g}.
$$
Combining   this  with \eqref{WCE4} proves \eqref{WCE1}.
  This proves that $\wt X$ has the same distribution as the L\'evy process $X^\bullet$  associated with
 the regular Dirichlet form $(\sE_\bullet, \sF_\bullet)$
 on $L^2(G_\bullet; m)$, which in particular establishes  part (ii) of the theorem   by  Proposition \ref{clttight}.

  We next show part (i) of the theorem.  By Proposition \ref{clttight}, $X^\bullet_t$ has transition density function
  $p(t, x^{-1}_\bullet \bullet y)$ with respect to   the Lebesgue measure $dy$ on $G_\bullet$.
    By Remark \ref{R:4.6},  the generating functional
   $f\mapsto \sL f(e) $ is a kernel of order $\beta_+$. Thus by   \cite[Theorem 2.2]{Gl},
     $p(t, x)$ is square-integrable
   for every $t>0$ and so
   $c_t:=p(t, e) = \int_{\R^d} p(t/2, x)^2 dx <\infty$.  By the Cauchy-Schwarz inequality, for any $x\in G_\bullet$,
   \begin{align*}
   p(t, x) &= p(t, e, x) =\int_{\R^d} p(t/2, e, z) p(t/2, z, x) dz \\
&   \leq  \| p(t/2, e, \cdot)\|_2 \,    \| p(t/2,  \cdot, x)\|_2 \leq c_t.
   \end{align*}
That is, $p(t, x)$ is bounded on $G_\bullet$ for every $t>0$.
Property \eqref{e:4.18} follows from the self-similarity property \eqref{e:4.12} of $X^\bullet$,
and H\"older regularity \eqref{e:4.19} follows from \cite[Corollary 3.12]{Gl}.
The joint continuity of $p(t, x, y)= p(t, x^{-1}_\bullet \bullet y)$ in $(t, x, y)$ follows from the scaling property \eqref{e:4.18}
and the H\"older continuity of $p(t, x)$ in $x$.
 Note that  $p(t, e) = \int_{G_\bullet} p(t/2, y)^2 dy>0$ and $\lim_{t\to \infty} \delta_{1/t}(x)=e$ uniformly
on every compact subset of $G_\bullet$. Thus by \eqref{e:4.18},
for any $n\geq 1$, there is some $t_n>0$ so that $p(t, x)>0$ for every $(t, x)\in (0, t_n] \times B(0, n)$.
It then follows from the Chapman-Kolmogorov equation that $p(t, x)>0$ for every $(t, x) \in (0, \infty) \times G_\bullet$.

For  any $\beta \in (0, \beta_+)$, by  \cite[Theorem 5.1]{Gl}, there is a constant $C_3$ so that
$p(t, x) \leq  C_3    t^{\beta/\beta_+}  / \| x\|^{d+\beta }$ for every $t>0$ and $x\in G_\bullet$.
Together with \eqref{e:4.18}, it  gives the estimate \eqref{e:4.20}.
This establishes part (i) of the theorem, and thus completes the proof of the theorem.
   \qed

\section{Local  limit theorem} \label{sec-LCLT}

\subsection{Assumption {\bf (R2)}}\label{S:6.1}

In this section, we discuss the local limit theorem for $(X^{(k)}_{nT_k^{-1}})_{n\in \mathbb N\
\cup\{0\}}$
based on \cite[Theorem 1]{CH} and \cite[Theorem 4.5]{CKW-JFA} (c.f. \cite[Section 4]{BH} for the case the limit
heat kernel is Gaussian).

For this purpose, we introduce an additional hypothesis, (R2),  which reads as follows. Let $\mu^{(n)}$ be the $n$-th convolution power of the probability measure $\mu$ on $\Gamma$. This is the law at time $n$ of the random walk driven by $\mu$, started at the identity element on $\Gamma$.
\begin{itemize}
\item[(R2)] There are positive   constants $C_2 >0$ and $\beta >0$  such that, for all $n,m\in \mathbb N$  and  $x,y\in \Gamma$,
 \begin{equation}\label{eq:3-29-1}
|\mu^{(n+m)}(xy)-\mu^{(n)}(x)|\le \frac{  C_2}{V( n^{1/\beta} )}  \left(\frac{m}{n+1}+\sqrt{\frac{\|y\|^{\beta}}{n+1}} \right),
\end{equation}
 where $V(r):=\sharp \{g\in \Gamma: \|g\|   <   r\}$.
   \end{itemize}

\medskip

  For our local limit theorem to hold,   the exponent $\beta>0$  in (R2) should be the same as those in \eqref{e:4.2} and in (E1)-(E2).
  We start with verifying the needed convergence of the volume of appropriate balls.

\subsection{Convergence of volume}\label{S:6.2}

Recall that  $\Gamma_t:=\delta_t^{-1}(\Gamma)=\delta_{1/t}(\Gamma)$,
 $$
 B(r)=\{x: \|x\|<r\} \quad
 \hbox{and}\quad   B_\bullet(x,r)=x\bullet B(r)=\left\{y\in \mathbb R^d: \|x_\bullet^{-1}\bullet y\|< r \right\}.
 $$

 Recall also that $m$ is the Lebesgue measure on $\mathbb R^d$ and
$$m_t(A)=c(\Gamma,G) \det(\delta_t^{-1})\#A$$ for any finite subset $A$ of  $\Gamma_t$, where $c(\Gamma,G)$ is given in Proposition \ref{weaklimmeas} (see also below). We need the following lemma.
\begin{lem} \label{lem-approx-meas}
For all  $x\in \mathbb R^d, r\ge 0$,
\begin{equation}\label{eq:approx-meas} \lim_{t\to \infty}
m_t(B_\bullet(x,r)\cap \Gamma_t)= m(B_\bullet(x,r)).
\end{equation}
\end{lem}

\begin{proof} Fix $x\in \mathbb R^d$ and $r\ge 0$.  Recall that $B_\bullet (x,r)=x\bullet B(r)$ and  $$\delta_t(x\bullet B(r))=\delta_t(x)\bullet B(rt^{1/\beta})=B_\bullet(\delta_t(x),rt^{1/\beta}),$$
so that
$B_\bullet(x,r)\cap \Gamma_t$ is the finite set of all points $y\in \mathbb R^d$ such that
$$z=\delta_t(y)\in  B_\bullet(\delta_t(x),rt^{1/\beta})    \cap \Gamma .$$
  Let $\mbox{dist}_{\bullet}$ be a left-invariant Riemannian metric on the Lie group  $G=(\mathbb R^d,\cdot)$ and take the Voronoi cell for the discrete subgroup $\Gamma$:
$$
U= \left\{x\in \mathbb R^d: \mbox{dist}_\bullet(x,e)=  \min_{\gamma\in \Gamma}\mbox{dist}_\bullet(x,\gamma) \right\}
$$ 
so that $$\mathbb R^d=\bigcup_{\gamma\in \Gamma} \gamma \cdot U$$
and
  $ (\gamma \cdot U )\cap ( \gamma'\cdot U) \subset \partial U$ and thus $m( (\gamma \cdot U) \cap
( \gamma'\cdot U) )=0$ for any $\gamma\neq \gamma' \in \Gamma$.
 Note that this definition is based on the law $\cdot$ of the
  Lie group  $G=(\mathbb R^d,\cdot)$ and its   closed
subgroup $\Gamma$, not on the rescaled limit law $\bullet$.
 Since $m(\partial U)=0$,  by  definition, $c(\Gamma,G)=m(U)$. For any  $S\subset \mathbb R^d$,  we have
$$c(\Gamma,G)\#\{z\in \Gamma \cap S\}\le  m( S\cdot U ).$$
In particular, for $S=B_\bullet(\delta_t(x),rt^{1/\beta})=\delta_t(B_\bullet(x,r))$,
$$c(\Gamma,G)\#\{z\in \Gamma \cap B_\bullet(\delta(x),rt^{1/\beta})\}\le  m(\delta_t [\delta_{1/t} (\delta_t (B_\bullet(x,r)) \cdot \delta_t(\delta_{1/t}(U) ))]).$$
 Note that since $\Gamma$ is a co-compact closed subgroup of $G$, $U$ is bounded and closed and hence   compact.
Consequently, $\delta_{1/t} (U)$ converges uniformly to $\{e\}$ as $t\to \infty$.
 By the uniform convergence of the product $\cdot_t$ to $\bullet$ on compact sets (e.g., see Lemma \ref{lem2-6norm}), for any fixed $\eps>0$, there exists a constant  $T>0$ large enough such that,  for all $t>T$, the set $ [\delta_{1/t} (\delta_t (B_\bullet(x,r)) \cdot \delta_t(\delta_{1/t}(U) ))]$ is contained in an $\eps$ neighborhood for the norm $\|\cdot\|$ in the group $(G,\bullet)$ of the set  $B_\bullet (x,r)\bullet \delta_{1/t}(U)$.
This means that, for $t$ large enough,
$$ \delta_{1/t} (\delta_t (B_\bullet(x,r)) \cdot \delta_t(\delta_{1/t}(U) )) \subset B_\bullet (x,r+2\eps).
$$
 Hence,
 \begin{align*}
 \det(\delta_{1/t}) c(\Gamma,G) \#\{z\in \Gamma \cap B_\bullet(\delta(x),rt^{1/\beta})\}&\le  \det(\delta_{1/t}) m(\delta_t ( B_\bullet (x,r+2\eps)))\\
&= m(B_\bullet(x,r+2\eps)).\end{align*}
Take the limsup in $t  \to \infty  $, note that  $m(B_\bullet(x,r+2\eps))=c(r+2\eps)^{\sum_1^d \beta/\beta_i}$, and let $\eps$  tend  
to $0$, to obtain
$$\limsup_{t\to \infty} m_t(B_\bullet(x,r)\cap \Gamma_t)\le  m(B_\bullet(x,r)).$$
To prove the complementing inequality, namely,
$$\liminf_{t\to \infty} m_t(B_\bullet(x,r)\cap \Gamma_t)\ge  m(B_\bullet(x,r)),$$
we use the same line of reasoning as above to see that, for any fixed $\eps>0$ and all $t$ large enough,
$$
B_\bullet (\delta_t(x),(r-2\eps)t^{1/\beta}) \subset \bigcup_{\gamma\in B_\bullet(\delta_t(x),rt^{1/\beta})} \gamma \cdot U.
$$
From this, it follows that, for all $t$ large enough,
$$m(B_\bullet (x,r-2\eps)) \le  \det(\delta_{1/t}) c(\Gamma,G) \#\{z\in \Gamma \cap B_\bullet(\delta(x),rt^{1/\beta})\}.$$
The desired lower bound follows.
\end{proof}

\subsection{Statement and proof of the LLT}\label{S:6.3}

Given an arbitrary sequence of positive reals $T_k$ tending to infinity and $t>0$, let
$
\hat \mu^{(t)}_k
$
be the probability distribution of $(\hat X^{(k)}_{t})_{t>0}$, i.e.,
$$\hat \mu^{(t)}_k(x)=   \P^e(\hat X^{(k)}_{t}=x)=\mu^{[tT_k]}(\delta_{T_k}(x)),\quad x\in \Gamma_{T_k}.$$
Recall that for each $x\in \mathbb R^d$, $[x]_k\in \Gamma_{T_k}$ is the point closest to $x$ in the $\|\cdot\|$-norm.

  We know from Theorem \ref{WT1} that the L\'evy process $X^\bullet$  s 
    corresponding  to $(\sE_\bullet,\sF_\bullet)$
has a jointly continuous convolution kernel
$$
(t,x)\mapsto p_\bullet(t,x) =  t^{- \gamma_0}  p_\bullet (1,\delta_{1/t}(x))  $$
 with
$t >   0$,  $x\in \mathbb R^d$.

\begin{theo}[{\bf Local limit theorem}]\label{localCLT}
Assume \eqref{e:4.2}, {\rm (A)-(R1)-(R2)-(E1)-(E2)} and {\rm (T$\bullet$)-(T$\Gamma$)}   with the same exponent $\beta>0$.
 Then, for any $U_2>U_1>0$ and $r>1$,
$$
\lim_{k\to\infty}\sup_{{x\in  \R^d  : \|x\|\le r}}\,\sup_{t\in [U_1,U_2]}
\left|\det (\delta_{T_k}) \mu^{([tT_k])}_k(\delta_{T_k}([x]_k))-p_\bullet(t,x)  \right|  =0.
$$
\end{theo}

\begin{proof} We adopt the notations in \cite{CH}.
Let $E=\mathbb R^d$ with $d_E(x,y)=\|x_\bullet^{- 1}\bullet y\|$,  and $G^k=\delta_{T_k}^{-1}(\Gamma)\subset \mathbb R^d=E$
with the same distance  $d_{G^k}(x,y)=\|x_\bullet^{-1}\bullet y\|$.
(Note that $d_{G^k}(\cdot,\cdot)$ is a graph distance on $G^k$ in \cite{CH}. However, the proof of \cite[Theorem 1]{CH}
works for any distance on $G^k$.)
 Then, conditions (a) and (b) in \cite[Assumption 1]{CH} hold with $\alpha(k)=1$.
 Let $\nu=m$ and $\nu^k=m_{T_k}$. Then by \eqref{eq:approx-meas}, (c) in \cite[Assumption 1]{CH} holds with
 $\beta(k)=\det (\delta_{T_k})$.
Set
$$
q^{k}_t(x)=\hat{\mu}^{(t)}_k(x) \quad \mbox{ and } \quad
q_t(\cdot)=p_\bullet(t,\cdot).
$$
It suffices to prove that the conclusion of \cite[Theorem 1]{CH} holds for
 $q^{k}_t(x)$.
We now check that (d) in \cite[Assumption 1]{CH} holds.
 Let $U_0>0$ be a fixed constant. By Theorem \ref{WT1},
for every bounded and continuous function $f$ on $\R^d$, $t\in    (  0,U_0]$
and $x\in \mathbb R^d$, we have
 \begin{equation}\label{eq:3-qwq}
 \lim_{k\rightarrow \infty}
 \left|\hat \bE_{k}^{[x]_{k}}\big[f(\hat X^{(k)}_t)\big]-\int_{\mathbb R^d} f(z)q_t(
 x_\bullet^{-1}\bullet z)\,dz\right|=0.
 \end{equation}
 We need to  prove that this convergence is uniform in $t$ over any compact time   
   interval in $(0,\infty)$.
 This would easily     follow 
  if we could prove the equi-uniform continuity of the   
    function $t\mapsto \hat \bE_{k}^{[x]_{k}}\big[f(\hat X^{(k)}_t)\big]$ on compact time intervals. However, because we are dealing with what is essentially a discrete time process, these functions are not even continuous. Nevertheless, condition  (R2)   says
   that, for all  non-negative integers $n,m$ and all $x,z\in \Gamma$ (inverse and multiplication are in $\Gamma$), we have
 \begin{equation}\label{eq:3-29-2}
|\mu^{(n+m)}(z)-\mu^{(n)}(x)|\le C_2 \left(\frac{m}{n+1}+\frac{\|x^{-1}\cdot z\|^{\beta/2}}{\sqrt{n+1}} \right).
\frac 1{V(n^{1/\beta})}.\end{equation}
It follows that, for $0<s<t$,
\begin{align*}
& \Big|\mu^{([tT_k])}([\delta_{T_k}([x]_k)]^{-1}\cdot \delta_{T_k}(y))-\mu^{([sT_k])}([\delta_{T_k}([x]_k)]^{-1}\cdot \delta_{T_k}(y))\Big|\\
&\le   C_2\frac{[T_k(t-s)]  +1  }{[T_ks]+1}\frac 1{V([T_ks]^{1/\beta})}
\le  C_2   \frac{t-s  +T_k^{-1}}{s}  \frac 1{V([T_ks]^{1/\beta})}.
\end{align*}
For any fixed time interval $[U_1,U_2]$, $0<U_1<U_2$, this is a version of  ``equi-uniform continuity,"
 call it ``equi-uniform continuity modulo $T_k^{-1}$."
   Together with the fact that $t\mapsto \int_{\mathbb R^d} f(z)q_t(x_\bullet^{-1}\bullet z)\,dz$ is  uniformly continuous  for $t\in [U_1,U_2]$,
and  \eqref{eq:3-qwq}, this equi-uniform continuity modulo $T_k^{-1}$ yields
\begin{equation}\label{p4-2-1}
 \lim_{k\rightarrow \infty} \sup_{t\in [U_1,U_2]}
\left|\hat\bE_{k}^{[x]_{k}}\big[f(\hat X^{(k)}_t)\big]-\int_{\mathbb R^d} f(z)q_t(x_\bullet^{-1}\bullet z)\,dz\right|=0.
 \end{equation}
 By the joint continuity of $q_t(x)$, we have
 $\displaystyle\int_{\partial B(x_0,r)}q_t(x_\bullet^{- 1}\bullet z)\,dz=0$
 for every $x,x_0\in E$ and $r>0$. Hence, \eqref{p4-2-1} yields that
 $$ \lim_{k\rightarrow \infty} \sup_{t\in [U_1,U_2]}
 \left|{\mathbb P}_{k}^{[x]_{k}}\big(X^{(k)}_{[T_kt]/T_k}
 \in B(x_0,r)\big)-\int_{B(x_0,r)}
 q_t(x_\bullet^{-1}\bullet z)\,dz\right|=0, $$
and (d) in \cite[Assumption 1]{CH} is satisfied with $\gamma(k)=T_k$.

On the other hand, by \eqref{eq:3-29-2} again,
we have   for $x, z\in  B(2r) \cap \delta_{T_k}^{-1}(\Gamma)$,
\begin{align*}
\det (\delta_{T_k})|q^k_{[T_kt]}(z)-q^{k}_{[T_kt]}(x)|
&= \det (\delta_{T_k})|\mu^{([T_kt])}(\delta_{T_k}(z))-\mu^{([T_kt])}(\delta_{T_k}(x))|\\
&\le  C_2 \frac{\|\delta_{T_k}(x)^{-1}\cdot \delta_{T_k}(z)\|^{\beta/2}}{\sqrt{T_kt}}
\frac {\det (\delta_{T_k})}{V((T_kt)^{1/\beta})}\\
&\le   C_3 \|x_\bullet^{-1}\bullet z\|^{\beta_-\beta/(2\beta_+)}
t^{-1/2-\gamma}  T_k^{-1/2} .
\end{align*}
For  the last inequality, we have used
Lemma \ref{lem2-6norm}, and the fact that
$V(t^{1/\beta})\asymp \det (\delta_{t})\asymp t^\gamma$ for  $\gamma=\sum_1^d1/\beta_i>0$.
 Hence it holds that
 for any $0<U_1<U_2$, $r>0$,
$\delta\in (0,r]$ and $k\ge1$,
$$ \sup_{x\in B_{G^k}(0,r),\atop d_{G^k}(z,x)\le  \delta}\!\!\,
\sup_{t\in [U_1,U_2]}
\det (\delta_{T_k})|q^k_{[T_kt]}(z)-q^{k}_{[T_kt]}(x)|\le
  C_4  \frac{\delta^{\beta_-\beta/(2\beta_+)}}
{U_1^{1/2+\gamma}}.$$
Taking $\lim_{\delta\to 0}\limsup_{k\to \infty}$, we obtain
\cite[Assumption 2]{CH}.
Therefore, the desired assertion follows from \cite[Theorem 1]{CH}. \end{proof}

\begin{rem} \label{R:5.3}\rm
It is well known (see, e.g.,  \cite{CKS87, CKKW21})  that if a Nash's inequality holds for $(\sE_\bullet, \sF_\bullet)$, then $(\sE_\bullet, \sF_\bullet)$ has transition density function $p(t, x, y)$ that is bounded for each $t>0$.
It follows from Theorem  \ref{WT1}, $p(t, x, y)$ is jointly  locally H\"older continuous in $(x, y)$ for each fixed $t>0$.
Since $X^\bullet$ is a L\'evy process on $(G_\bullet,  \bullet)$ and $\{\delta_t; t>0\}$ is a group dilation structure for
$(G_\bullet,  \bullet)$, we have
$$
p(t, x, y) = p(t, e, x_\bullet^{-1} \bullet y) = {\rm det} (\delta_t^{-1}) p(1, e,  \delta_t^{-1} (x_\bullet^{-1} \bullet y)) .
$$
Define $p_\bullet (t, x) =  {\rm det} (\delta_t^{-1}) p(1, e,  \delta_t^{-1} (x_\bullet^{-1} \bullet y))$.
Then $p_\bullet (t, x)$ is jointly continuous in $(t, x)$, symmetric in $x\in G_\bullet$
 (that is, $p_\bullet (t, x)= p_\bullet (t, x_\bullet^{-1}) )$ and is the convolution kernel for $(\sE_\bullet, \sF_\bullet)$.
Moreover,  for any $t>0$ and $x\in G_\bullet$,
$$
p_\bullet (t, x)= {\rm det} (\delta_t^{-1}) p_\bullet (1,   \delta_t^{-1} (x ))
= {\rm det} (\delta_{1/t}) p_\bullet (1,   \delta_{1/t} (x ))   .
$$
\qed
\end{rem}

 \section{Symmetric L\'evy processes on nilpotent  groups} \label{S:7}

\subsection{The problem of identifying the limit process}\label{S:7.1}

Theorem \ref{WT1} gives the functional central limit theorem for a class of random walks
on simply connected nilpotent groups driven by probability measures $\mu$, which  are the distributions
of  the one-step increments of the random walks. However, the limit symmetric L\'evy process $X^\bullet$ is characterized in an abstract way by
a non-local pure jump Dirichlet form $(\sE_\bullet, \sF_\bullet)$ on $L^2( G_\bullet; dx)$ of the form \eqref{def-Ebullet} with $J_\bullet (dx, dy)=dx \mu_\bullet (x_\bullet^{-1} \bullet dy)$.
A natural question is whether we can use Theorem \ref{WT1} to give explicit limit theorems
in concrete examples  as those studied in Examples \ref{E:1.4} and \ref{E:1.5}.
This amounts to ask whether we can explicitly identity or describe the L\'evy process $X^\bullet$ in concrete cases. These are the questions we are going to address in this section and the answer is affirmative.
 In fact, we will do this in a more general context for any symmetric L\'evy measure $\mu_\bullet$
on any simply connected nilpotent group $G_\bullet$; that is, $(G_\bullet, \bullet)$ does not need to be the
limit group obtained from a simply connected nilpotent group $G$ through an approximate group dilation
structure $\{\phi_t; t>0\}$ on $G$, and $\mu_\bullet$ does not need to be the
  weak limit of $\mu_t=t \delta_{1/t} (\mu)$ of some symmetric probability measure $\mu$
	on a discrete subgroup $\Gamma$ of $G$ as in condition (A).
	This is achieved in Theorem \ref{T:7.3}.
Then we use this concrete description to illustrate our convergence theorem, Theorem \ref{WT1},
 by revisiting  Example \ref{E:1.5} and presenting several more examples
 through this approach, without using
the limit results for operator stable processes on $\R^d$, Propositions \ref{P:1.1} and \ref{P:1.3},
 from the literature.

\subsection{Symmetric L\'evy processes and their approximations}\label{S:7.2}

   Let  $\group$ be any simply connected nilpotent group.
As mentioned in Subsection \ref{S:3.1}, there is a global polynomial coordinate system on
 $\group$  satisfying   \eqref{triangular}-\eqref{e:3.4a}. Unless mentioned otherwise, this is the default coordinate system we use on $\group$ in this section. Through this global system
$\Phi: \R^d \to \group$ with $\Phi (0) =e$,
$\group$ can be identified with $\R^d$ and $dx$ is a Haar measure for $\group$.
The coordinate system
 also induces a  function on $\group$: $\| \sigma \|_2 := \| \Phi^{-1} (\sigma )\|_2$
for $\sigma \in \group$, where $\| \Phi^{-1} (\sigma )\|_2$ is the Euclidean norm of
$\Phi^{-1}(\sigma ) \in \R^d$.
As we already see from Subsection \ref{S:3.1}, there are many choices of the global coordinate systems
for $\group$. One of the commonly used coordinate system is the exponential map.
However, sometimes it is more convenient or more natural to use other coordinate systems,
for example, matrix coordinates in the Heisenberg group case.
Thus with this in mind,  we do not fix a particular choice of
the polynomial coordinate systems, except for the assumption that
 \eqref{triangular}-\eqref{e:3.4a}.

	Let $\nu$ be
 any non-zero symmetric L\'evy measure on $\group$; that is, $\nu$ is a    non-negative
   Borel measure on $\group$ satisfying
$0< \int_{\group} (1\wedge \| x\|_2^2) \nu (dx) <\infty$ and
 $\nu (A)= \nu  (A^{-1})$
for any $A\subset \group \setminus \{e\}$, where $A^{-1}= \{x\in \group: x^{-1} \in A\}$.  Note that we do not impose any additional conditions on $\nu$.
 Define
   \begin{equation}\label{e:7.1}
\sE (u, v):= \frac1{2} \iint_{ \group \times \group\setminus \Delta}
{(u(x z )-u(x))(v(x  z )-v(x))} dx \nu (dz) ,
\end{equation}
and $\F $ is the closure of  $\Lip_c( \group)$, which is the space of Lipschitz functions on $\group$ with compact support, with respect to the norm $\sqrt{ \sE  (u, u)+ \int_{\group} u(x)^2 dx}  $. Here $xz$ is the group multiplication of two elements $x, z\in N$.
Let $X$ be the symmetric Hunt process associated with the regular Dirichlet form
$(\sE , \sF )$ on $L^2(\group; dx)$; cf.  \cite{CF, FOT}.
Note that in this subsection, as mentioned above,
 we do not assume the L\'evy measure $\nu$ on $\group$ generates $N$.

\begin{lem} The Hunt process $X$ is a L\'evy process on $\group$.
 \end{lem}

 \proof  For each fixed  $\sigma \in \group
 \backslash\{e\}$, the process
 $Y^\sigma=\{Y^\sigma_t,t\ge0\}$, with $Y_t^\sigma=\sigma       X_t    $ for any $t>0$,
 is a symmetric Hunt process on $\group$  as $dx$ is a left Haar measure on $\group$
  and its transition semigroup
  $$
  P^\sigma_t f(x)= \E \left[ f(Y^\sigma_t) | Y^\sigma_0=x \right]
	= \E  \left[ f (   \sigma   X_t    ) | X_0= \sigma^{-1}   x \right]
  = (P_t f_\sigma )(\sigma^{-1}  x),
  $$
   where $f_\sigma (\eta):= f(\sigma  \eta).$
  Thus
  \begin{align*}
  &  \lim_{t\to 0} \frac1t (f- P^\sigma_t f, f)_{L^2(\group; dx)} \\
  &=   \lim_{t\to 0} \frac1t  \int_{\group} ( f_\sigma (\sigma^{-1}  x) -  (P_t f_\sigma )(\sigma^{-1}  x))  f_\sigma (\sigma^{-1}  x) dx \\
  &=  \lim_{t\to 0} \frac1t  \int_{\group} ( f_\sigma (y) -  (P_t f_\sigma )(y))  f_\sigma (y) dy \\
  &=  \frac12     \iint_{N \times N \setminus \Delta}
  (f_\sigma(x  z  ) -f_\sigma (x ) )^2\nu (dz) dx \\
&=  \frac12     \iint_{N \times N \setminus \Delta}
  (f (x  z  ) -f (x ) )^2\nu (dz) dx =\sE (f, f).
  \end{align*}
   This shows that $Y^\sigma$ is
   a symmetric Hunt process associated with the Dirichlet form
    $(\sE, \FF)$ on $L^2(\group; dx)$ and so it has the same distribution as $X$. In other words,
   $\{\sigma   X_t; t\geq 0\}$ with $X_0=x\in \group$ has the same distribution as
 $\{X_t; t\geq 0\}$ starting from $\sigma  x$.  This combined with the Markov property of $X$
 shows that $X$ is a symmetric L\'evy process on $\group$.
 \qed

 \medskip

 We next investigate how the L\'evy process $X$ is determined by $\nu$ in a more explicit way; that is, given a symmetric  L\'evy measure $\nu$ on $\group$,
 how to construct or approximate its corresponding
symmetric L\'evy process $X$ in a concrete way.
We will show in Theorem \ref{T:7.3}
that $X$ can be approximated by a sequence of random walks on $\group$
whose one-step increments are from the small increments of a common L\'evy process $Z$ on $\R^d$
through the identification of the global coordinate system $\Phi$.
The key is to identify the L\'evy measure and the drift of the L\'evy process $Z$ on $\R^d$.
Our approach uses Hunt's characterization for L\'evy processes on Lie groups and Kunita's triangular array type limit result for random walks on Lie groups, which we recall in Theorem \ref{T:7.2}.

We identify each element $\sigma \in \group$ with its global coordinate
$$
\Phi^{-1}(\sigma) =:x =(x_1, \ldots, x_d) \in \R^d.
$$
For each $1\leq j\leq d$, let $\sX_j$ be the left-invariant vector field  in
the Lie algebra   $\mathfrak g$   of the group $\group$
\  at $e$
determined by the coordinate function $x\mapsto \sX_j$;
that is, for any $C^2$ function $f(x)$ on $\group=\R^d$,
$$
(\sX_j f) (e)= \frac{\partial f(x)}{\partial x_j} \Big|_{x=0}.
$$
 These vector fields  $(\sX_1, \ldots, \sX_d)$ form
 a natural base of $ \mathfrak g$ at $e$.
On the other hand, it  is well known that the simply connected nilpotent group $\group$
admits an exponential map of first type from  its Lie algebra $\mathfrak g =\R^d$
to $\group$ which is surjective.
Under its exponential coordinates exp: $\mathfrak g\to \group$ (of first type), $x^{-1}=-x$.
Let  $\{x^1, \ldots, x^d\}$ be the
exponential coordinate of $\sigma \in \group$ with respect to the base
$\{\sX_1, \ldots,\sX_d\}$; that is, $\exp (\sum_{j=1}^d x^j \sX_j) =\sigma$.
Note that $x^j(\sigma^{-1})=-x^j(\sigma)$ and $\sX_i x^j =\delta_{ij}$.
Let $\| \cdot \|$ be the norm on $\group$ defined by
\begin{equation}\label{e:symnorm}
	\| \sigma \|:= \Big(\sum_{j=1}^d (x^j(\sigma))^2 \Big)^{1/2}
\end{equation}
 in terms of the exponential coordinates of $\sigma \in \group$.
Note that the norm $\| \cdot \|$ is symmetric on the group $\group$ in the sense
that $\| \sigma^{-1}\| = \| \sigma \|$ for any $\sigma \in \group$.

Denote by $\sC$ the space of   real-valued functions on $\group$ that are continuous and have limit at infinity,  and $\sC^2$ be the space of $C^2$ functions $f$ on $\group$ so that $f, \sX_k f$ and $\sX_k \sX_jf $ are all in $\sC$.
Let $\psi\in \sC^2$ be such that $\psi >0$ on $G\setminus \{e\}$, $\psi (\eta) \asymp \sum_{j=1}^d x^j (\eta)^2$ near $e$,
and $\lim_{\eta \to \infty } \psi (\eta)>0$.  Note that in view of \eqref{Lip1}-\eqref{Lip2},
$$
\psi (\eta) \asymp 1\wedge \| \eta \|^2
\asymp 1\wedge \| \eta \|_2^2
\quad \hbox{for } \eta \in \group.
 $$

\medskip

We  recall the following triangular array type limit result on $\group$
 from \cite{KunCornell}, which in fact holds for any Lie   group.

\begin{theo}[Theorem 3 of  \cite{KunCornell}] \label{T:7.2}
In the above setting,
suppose the following hold.
\begin{enumerate}
\item[\rm (i)] For each $n\geq 1$,  $k\mapsto S^{(n)}_k =\xi_{n, 1}\cdots \xi_{n, k}$ is a discrete time random walk on the Lie group $\group$, where $\{\xi_{n, k}; k\geq 1\}$ are i.i.d $\group$-valued random variables having distribution $\nu_n$.

\item[\rm (ii)] As $n\to \infty$, the measure $n \nu_n$ converges vaguely to a measure $\nu$ on $\group\setminus \{e\}$ satisfying
$\int_{\group\setminus \{e\}} \psi (x) \nu (dx)<\infty$.

\item[\rm (iii)] For $\eps >0$,  let
$$
U_\eps:=\{\eta \in \group: \| \eta \|<\eps \}
=  \Big\{\eta \in \group: \sum_{j=1}^d x^j (\eta)^2 <\eps^2 \Big\}
$$
be an $\eps$-neighborhood of $e$ in $\group$.  For each $\eps >0$,
$$
\lim_{n\to \infty} n \int_{U_\eps} x^i(\eta)  x^j (\eta) \nu_n (d\eta) =:a^{(\eps)}_{ij}
\ \hbox{ exists}.
$$
Clearly, $(a_{ij}^{(\eps)})$ is symmetric and non-negative definite, which decreases to $(a_{ij})$ as $ \eps \to 0$.

  \item[\rm (iv)]    For each $\eps >0$,
$\lim_{n\to \infty} n \int_{U_\eps} x (\eta)   \nu_n (d\eta) =:b^{(\eps)} \in \R^d$ exists.
\end{enumerate}

 Take $\eps >0$  so that $\partial U_\eps$ has zero $\nu$-measure. Define
$$
b=b_\eps + \int_{U_\eps^c} x(\eta) \nu (d\eta),
$$
 whose value is independent of the choice of $\eps$.
Then for each $T>0$,
   $\{Z^{(n)}_t:= S^{(n)}_{[nt]}; t\in [0, T]\}$
 converges weakly in the Skorokhod space $\bD ([0, T]; \group)$ as $n\to \infty$   to a L\'evy process $Z=\{Z_t; t\in [0, T]\}$ on $\group$,
 whose generator is characterized by
 \begin{align}\label{e:hunt}
 \sL f (\eta) &= \frac12 \sum_{i,j=1}^d a_{ij} \sX_i\sX_j f (\eta) + \sum_{i=1}^d b_i \sX_i f(\eta) \nonumber \\
 &\quad + \int_{\group\setminus \{e\}} ( f(\eta \sigma) - f(\eta) -\sum_{i=1}^d x^i (\sigma) \sX_i f(\tau)) \nu ( d \sigma)
 \end{align}
  for any $f\in \sC^2$.
  \end{theo}

\bigskip

Let  $\{\varphi_1 (\sigma), \cdots, \varphi_d (\sigma)\}$ be $\sC^2$  functions on $\group$
such that under
    exponential coordinates for $\sigma ={\rm exp} (\sum_{j=1}^d x^j \sX_j)
  \in \group$, $\varphi_j (\sigma)$ is an odd increasing function of $x^j$ with  $\varphi_j(\sigma)=x^j$
  for $x^j\in (-1, 1)$. Since $\group$ is also identified with $\R^d$ through the global coordinate system $\Phi$ satisfying \eqref{e:3.4a}	mentioned above, sometimes we also write $\varphi_j (x)$ for $\varphi_j (\sigma)$ through this global	coordinate system $\Phi$. Since $\nu$ is a symmetric measure on $\group$ and
	$\varphi_j (\sigma)=-\varphi_j (\sigma^{-1})$ for any $\sigma \in \group$,
	we have for every $1\leq j\leq d$ and every $r>0$,
\begin{equation} \label{e:7.2a}
\int_{\{\sigma \in \group: \| \sigma \|>r\}} \varphi_j (\sigma) \nu (d\sigma)
=0.
\end{equation}

\medskip

Through the identification of $\group$ with $\R^d$ under the global coordinate system $\Phi$,
the L\'evy measure $\nu$ can also be viewed
as a L\'evy measure on the Euclidean space $\R^d$.
More precisely, let $\bar \nu$ be the Radon measure on $\R^d\setminus \{0\}$ defined by
\begin{equation}\label{e:meas}
\bar \nu (A):= \nu (\Phi (A))
\quad \hbox{for any } A\in \sB (\R^d \setminus \{0\}).
\end{equation}
Note that $\bar \nu$ satisfies $\int_{\R^d} (1\wedge \|z\|_2^2) \bar \nu (dz) <\infty$
and thus is a L\'evy measure on $\R^d$.
However, we point out that even though $\nu$ is a symmetric L\'evy measure on $\group$,  $\bar \nu$ may not be a symmetric measure on $\R^d$;
 see Examples \ref{E:7.6} and \ref{E:7.7}(i).
It is not hard to see or guess  that the L\'evy process $Z$ on $\R^d$ that will be used to approximate the L\'evy process $X$ on $\group$ should have L\'evy measure $\bar \nu$, however
in general it also need a proper drift correction term. For this, define
for $1\leq j\leq d$,
\begin{equation}\label{e:7.3}
\bar b_j = \int_{\{z\in \group: \| z\|_2  \leq 1\}} (z_j -   \varphi_j (z))  \nu (dz)
- \int_{\{z\in \group: \| z\|_2  > 1\}} \varphi_j (z)   \nu (dz) ,
\end{equation}
 where $(z_1, \ldots, z_d)= \Phi^{-1}(z)$ is the coordinates of $z\in \group$
under the global coordinate system $\Phi$.
Recall that the coordinates of $z\in \group$ under the exponential coordinate system
is denoted by $(z^1, \ldots, z^d)$.
Observe that the   integral in \eqref{e:7.3} is well-defined and is finite, because
$$
\frac{\partial \varphi_j}{\partial z_i}\Big|_{z=0} = \sX_i \varphi_j (z)\big|_{z=0}
= \sX_i z^j\big|_{z=0}  =\delta_{ij}
$$
 and  so
$$
 |z_j - \varphi_j (z)| = |z_j - z^j (z)| \leq c \|z\|_2^2  \quad \hbox{for } \|z\|_2\leq 1.
	$$
In view of \eqref{e:7.2a}, we can rewrite \eqref{e:7.3} as
\begin{equation}\label{e:7.6a}
\bar b_j = \lim_{r\to 0}  \int_{\{z\in \group: \| z\|_2  \leq 1
\hbox{ \small and } \| z\| \geq r  \}}  z_j \,   \nu (dz) .
\end{equation}

Since  the L\'evy measure $\nu$ is symmetric
on $\group$,
we have from \eqref{e:7.6a} that
\begin{align}\label{e:7.4a}
	\bar b_j   =&  \frac12
		\int_{\{z\in \group: \,  \| z\|_2\leq 1 \hbox{ \small and }  \| z^{-1}\|_2 \leq 1 \} } (z_j +(z^{-1})_j)  \nu (dz)  \nonumber \\
		& + 	\int_{\{z\in \group: \| z\|_2\leq 1  \hbox{ \small and }
		\| z^{-1}\|_ 2 >1 \} }  z_j   \nu (dz) .
\end{align}
	Here $z^{-1}$ denotes the group inverse of $z\in \group$,
	and $( z^{-1})_j$ is the $j$-th coordinate of the element $z^{-1} \in \group$
	under the original global coordinate system $\Phi$.
	Note that both  integrals in \eqref{e:7.4a} are absolutely convergent.
	This is because under the global coordinate system $\Phi$, we know from \eqref{e:3.4a} that
$$
z^{-1} = -z + (0, \bar q_2(z_1), \ldots,   \bar q_{ d}(z_1, \ldots, z_{d-1})),
$$
where for $2\leq j\leq d$, $\bar q_j(z_1, \cdots, z_{j-1})$ is polynomial having no constant
and first order terms. Thus on any compact set $K\subset N$,
there is a constant $C_K>0$ so that
\begin{equation}\label{e:7.5a}
\| z+z^{-1}\|_2 \leq C_K \| z\|_2^2 \quad \hbox{for every } z\in K,
\end{equation}
and $\{z\in \group: \,  \| z\|_2 <  1 \hbox{ \small and }  \| z^{-1}\|_2  <  1 \}$
is an open neighborhood of $e\in  N$.  Since $\int_{N\setminus \{e\}} (1\wedge \| z\|_2^2)
\nu (dz) <\infty$, both integrals in \eqref{e:7.4a} are absolutely convergent.
	 In general, the constant vector $\bar b:=(b_1, \ldots, \bar b_d)$ may not be zero.
	However, if $\Phi$ is the exponential coordinate system,
	then $\bar b_j=0$ for every $1\leq j\leq d$ as $z^{-1}=-z$ for every $z\in \group$
	and $\| z\|_2=\|z\|$.

	\medskip
	
Let $Z:=\{Z_t:t\ge0\}$ be the L\'evy process on the Euclidean space $\R^d$ with L\'evy triplet $(0, \bar b, \bar \nu)$,
see \eqref{e:1.1}, where
$\bar b=(\bar b_1, \cdots, \bar b_d)$. In other words,
\begin{equation} \label{e:7.6}
Z_t =\bar bt + \int_0^t {\mathbbm 1}_{\{\|z\|_2 \leq 1\}} z  \left( N(ds, dz)- ds
 \bar \nu  (dz) \right) + \int_0^t {\mathbbm 1}_{\{\|z\|_2> 1\} }z N(ds, dz),
\end{equation}
where $N(ds, dz)$ is the Poisson random measure on $[0, \infty) \times \R^d$ with intensity measure
$ds \bar  \nu (dz)$.

Recall that $\Phi: \R^d \to \group$ is the global polynomial coordinate system for the
simply connected nilpotent group $\group$.
Most of the time, we identify $x\in \R^d$ with $\sigma:=\Phi (x)\in \group$ and use the notations interchangeably.
In the next theorem and its proof, to be absolutely clear,
 we explicitly use the notation $\Phi (x)$
for emphasis when $x\in \R^d$ is viewed as an element in the group $\group$.

\begin{theo} \label{T:7.3}
Let $Z$ be the L\'evy process on $\R^d$ with
 $Z_0=0$,  L\'evy measure $\bar \nu$ of \eqref{e:meas}
and drift $\bar b$ of \eqref{e:7.4a}.
 For each $T>0$,  the random walk
  \begin{equation}\label{e:7.2}
  Z^{(n)}_t :=  \Phi(  Z_{1/n})  \,    \Phi ( Z_{2/n} -  Z_{1/n})     \cdots
         \Phi  ( Z_{[nt]/n} -  Z_{([nt]-1)/n})
 \end{equation}
on $\group$
   converges weakly in the Skorokhod space $\D([0, T]; \group)$ as $n\to \infty$   to the  left-invariant
Hunt process $ \{ (Y_0)^{-1}  Y_t; \, t\in [0, T]\}$ on $\group$.
The Hunt process $Y$ has the same distribution as the symmetric L\'evy process $X$
on $\group$ having L\'evy measure $\nu$ determined by the Dirichlet form  $(\sE, \sF)$
	of \eqref{e:7.1} on $L^2(\group; dx)$.
\end{theo}

\proof     By Ito's formula,  for any $f\in C^2_b (\R^d)$,
\begin{align*}
&  f (Z_t) -f (Z_0)  \\
&=  \int_0^t \bar b\cdot \nabla f(Z_s) ds + \int_0^t \int_{\{\| z\|_2 \leq 1\}} \left( f(Z_{s-} + z) -f(Z_{s-}\right)) \left( N(ds,dz) -ds   \bar \nu  (dz)\right) \\
&\quad+  \int_0^t \int_{\{\| z\|_2 >1\}} \left( f(Z_{s-} + z) -f(Z_{s-}\right)) N(ds,dz) \\
&\quad+ \int_0^t \int_{\{\| z\|_2 \leq 1\}} \left( f(Z_{s} + z) -f(Z_{s}) - \nabla f (Z_s) \cdot z \right)
 \bar \nu  (dz) ds   .
\end{align*}
Thus
\begin{align}\label{e:7.8a}
& \E f(Z_t)-f(0) \nonumber \\
&=  \E \int_0^t {\bar b}\cdot \nabla f(Z_s) ds+ \E   \int_0^t \int_{\{\| z\|_2 >1\}} \left( f(Z_{s-} + z) -f(Z_{s-}\right) N(ds,dz)  \nonumber\\
&\quad+ \E \int_0^t \int_{\{\| z\|_2 \leq 1\}} \left( f(Z_{s} + z) -f(Z_{s}) - \nabla f (Z_s) \cdot z \right)
  \bar \nu  (dz) ds   \nonumber\\
&=  \E \int_0^t {\bar b}\cdot \nabla f(Z_s) ds+
\E   \int_0^t \int_{\{ \| z\|_2 >1 \}} \left( f(Z_{s-} + z) -f(Z_{s-}\right)  \bar \nu  (dz) ds  \nonumber\\
&\quad+ \E \int_0^t \int_{\{\| z\|_2 \leq 1\}} \left( f(Z_{s} + z) -f(Z_{s}) - \nabla f (Z_s) \cdot z \right)    \bar \nu  (dz) ds
 \nonumber\\
&=  \E   \int_0^t \int_{\R^d} \left( f(Z_{s-} + z) -f(Z_{s-})
-  \nabla f (Z_s) \cdot \varphi (\Phi (z)) \right)   \bar \nu  (dz) ds ,
\end{align}
where $\varphi (\sigma):=(\varphi_1(\sigma), \ldots, \varphi_d(\sigma))$ for $\sigma \in \group$, and
the last equality is due to the definition of $\bar b$.

For $f_0\in \sC^2$ on $\group$ with $f_0(e)=0$ and $X_j f_0(e)=0$ for $1\leq j\leq d$, the function
$f:=f_0\circ \Phi$ is $C^2_b$ on $\R^d$ with $f(0)=0$  and $\nabla f(0)=0$.
Applying \eqref{e:7.2a} and \eqref{e:7.8a} to this $f$, we have by
the dominated convergence theorem that
\begin{equation}\label{e:cond1}
\lim_{t\to 0} \frac1t \E f_0 (\Phi (Z_t))=
\lim_{t\to 0} \frac1t \E f(Z_t) = \int_{\R^d \setminus \{0\}} f(z)   \bar \nu  (z)
=\int_{\group \setminus \{e\}} f_0 (\sigma) \nu (d \sigma).
\end{equation}
If  we denote the law of $\Phi (Z_t)$ on $\group$ by $\wt \nu_t$, then  the above in particular implies that
$t^{-1} \wt \nu_t$ converges vaguely to $ \nu  $
on $\group  \setminus \{e\} $ as $t\to 0$.

Since $\varphi_j$ is an odd function on $\group$,
taking $f_0 =\varphi_j$ in \eqref{e:cond1}  in particular yields that
 \begin{equation}\label{e:7.12}
\lim_{t\to 0} \frac1t \E \varphi_j (\Phi (Z_t)) = 0
\quad \hbox{for every } 1\leq j \leq d.
\end{equation}
On the other hand, since $\nu$ is a symmetric measure on $\group$ and $\varphi$ is an odd $\R^d$-valued function on $\group$, we have from \eqref{e:7.8a} that for any $f\in C_b^2(\R^d)$,
$$
 \E f(Z_t)-f(0) = \E \int_0^t \int_{\R^d} \left( f(Z_{s-} + z^{-1}) -f(Z_{s-})
+ \nabla f (Z_s) \cdot \varphi (\Phi (z)) \right)   \bar \nu  (dz) ds
 $$
and so
\begin{equation}\label{e:7.7}
	\E f(Z_t)-f(0) =\frac12 \E \int_0^t \int_{\R^d} \left( f(Z_s+z) +
	f(Z_{s-} + z^{-1}) -2 f(Z_{s-})
 \right)   \bar \nu  (dz) ds .
\end{equation} 	
Note that by \eqref{e:7.5a},
$$ \int_{\R^d} |f(z)+f(z^{-1}) -2 f(0)| \, \bar \nu (dz)<\infty.
$$
It follows from \eqref{e:7.7} and the dominated convergence theorem that for any $f\in C^2_b(\R^d)$,
\begin{equation}\label{e:7.11}
\lim_{t\to 0} \frac1t \E [f(Z_t) -f (0)] =
\frac12 \int_{\R^d \setminus \{0\}} \left(f(z) + f(z^{-1}) - 2 f(0)\right)   \bar \nu  (z).
\end{equation}

 For $\eps \in (0, 1)$, define
$$
U_\eps = \left\{\sigma \in \group: \| \sigma \|<\eps \right\}
= \Big\{ \sigma \in \group:  \sum_{i=1}^d  x^i (\sigma)^2 < \eps^2 \Big\},
$$
where $\| \sigma \|$ is the symmetric norm of $\sigma \in N$ as defined by \eqref{e:symnorm}
and $(x^1 (\sigma), \ldots x^d(\sigma) )$ is the exponential coordinates of $\sigma \in N$.
By the Lipschitz equivalents \eqref{Lip1}-\eqref{Lip2} between $\Phi$ and the exponential coordinates,
there is a constant $\lambda_0\geq 1$ so that
$$
\lambda_0^{-1} \| \Phi^{-1}(\sigma ) \|_2 \leq \| \sigma \| \leq \lambda_0 \| \Phi^{-1}(\sigma )\|_2
\quad \hbox{for } \sigma \in \group \hbox{ with } \| \sigma \|\leq 1.
$$
Consequently,
$$
U_\eps \subset \left\{ \sigma \in \group: \| \Phi^{-1}(\sigma ) \|_2 <\lambda_0 \eps \right\}
\quad \hbox{for every } \eps \in (0, 1).
$$
For  $\eps \in (0, 1)$, let $f_\eps \in C^2_c(\R^d)$ so that $f_\eps (z)=\| z\|_2^2$ for $| z\|_2 < \lambda_0\eps$,
$f_\eps (z)=0$ for $\| z\|_2  \geq 2 \lambda_0 \eps$,
$0\leq f_\eps(z) \leq 2 \lambda_0^2 \eps^2$ and $|Df_\eps|+|D^2 f_\eps|\leq C$ for some constant $C>0$ independent of $\eps$.
 Then we have by \eqref{e:7.11} and  the Taylor expansion, 
 that
 \begin{align} \label{e:cond2}
&  \limsup_{t\to 0} \frac1t \E  \left[ {\mathbbm 1}_{\{ \Phi (Z_t)\in U_\eps \}}
	\| \Phi (Z_t) \|^2  \right]  \nonumber  \\
&\leq  	 \limsup_{t\to 0} \frac{\lambda_0^2}t \E  \left[ {\mathbbm 1}_{\{ \|
 {Z_t} \|_2 \leq  \lambda_0 \eps \}}
	\| Z_t\|_2 ^2  \right]  \nonumber   \\
&\leq \limsup_{t\to 0} \frac{\lambda_0^2}t  \left(  \E f_\eps (Z_t) -f_\eps (Z_0)  \right)
\nonumber   \\
&=  \lambda_0^2 \int_{\R^d} f_\eps (z) \bar \nu (dz) \nonumber  \\
  &\leq     \lambda_0^2  \|D^2 f_\eps\|_\infty  \int_{\{\| z\|_2 \leq 2 \lambda_0 \eps\}}
	\|z\|^2_2  \, \bar \nu  (dz)   ,
\end{align}
 which tends to 0 as $\eps \to 0$.

     For $\eps \in (0, 1)$ so that $\partial U_\eps$ has zero $\nu$-measure, it follows from
		\eqref{e:cond1} that
		\begin{equation}\label{e:cond3}
		\lim_{t\to  0} \frac1t \int_{U_\eps^c} \varphi_j (z) \nu_t (dz)
		=  \int_{U_\eps^c} \varphi_j (z) \nu (dz) =0 .
 \end{equation}
This together with \eqref{e:7.12} shows that
\begin{equation}\label{e:cond4}
  \lim_{t\to  0} \frac1t \int_{U_\eps} \varphi_j (\sigma )   \nu_t (d\sigma)=0.
\end{equation}
  Properties \eqref{e:cond1}, \eqref{e:cond2} and \eqref{e:cond3}-\eqref{e:cond4} show that
	the conditions of Theorems \ref{T:7.2} are all satisfied for the sequence of random walks on $\group$ whose one-step increment distributions are
	$\nu_n:=\wt \nu_{1/n}$  for $n\in \N$  with $(a_{ij})=0$ and $b=0$.
  Thus for each $T>0$,  the random walk
  $$
  Z^{(n)}_t :=   \Phi(  Z_{1/n}) \, \Phi ( Z_{2/n} -  Z_{1/n})     \cdots
          \Phi ( Z_{[nt]/n} -  Z_{([nt]-1)/n})
 $$
   converges weakly in the Skorokhod space $\bD ([0, T]; G)$ as $n\to \infty$   to a symmetric
 L\'evy process $ Y=\{ Y_t; t\in [0, T]\}$ on $\group$ with L\'evy measure $\nu$
 in the following sense: Denote by  $({\mathcal L} , {\mathcal D} ({\mathcal L}))$
 the infinitesimal generator of $Y$.   Then
 $\sC^2\subset {\mathcal D} ({\mathcal L})$ and for any $f\in \sC^2$,
 $$
 {\mathcal L} f (\sigma) = \int_{\group \setminus \{e\}} \Big(f(\sigma   z) -f (\sigma)
-\sum_{j=1}^d  \varphi_j (z)
 \sX_i f(\sigma)  \Big)  \nu (dz).
 $$

 We next show that $Y$ has the same distribution as the L\'evy process $X$ on $\group$
  defined through the Dirichlet form $(\sE, \sF)$
	of \eqref{e:7.1} on $L^2(\group; dx)$.
  Denote by ${\mathcal L}^0$ the $L^2$-generator
 of  the symmetric L\'evy process $X$.
 It is easy to check by definition (cf. \cite{CF, FOT})
 that $\sC^2\subset {\mathcal D} ({\mathcal L}^0)$ and
  \begin{align*}
{\mathcal L}^0  f (\sigma) &= 
{\rm p.v.} \int_{\group \setminus \{e\}} \left(f(\sigma   z) -f (\sigma)   \right)  \nu (dz)\\
&= \int_{\group \setminus \{e\}} \left(f(\sigma   z) -f (\sigma) -\sum_{i=1}^d \varphi_j (z)
 \sX_i f(\sigma)  \right)  \nu (dz)\\
 &=  {\mathcal L} f(\sigma).
 \end{align*}
By the uniqueness of infinitesimal generator characterization of L\'evy processes on $\group$
(see, e.g., \cite[Theorem 1]{KunCornell} due to Hunt),
we conclude that the L\'evy processes $X$ and $Y$ have the same law.
This completes the proof of the theorem.
\qed

\begin{rem} \rm

When $\Phi$ is the exponential coordinate system of the first type for $\group$,
Theorem \ref{T:7.3} follows from Theorem 4.2 and the proof of Theorem 4.1 of \cite{KunBR}.
The main point of Theorem \ref{T:7.3} is that it is valid  for any global coordinate system
$\Phi$ of $\group$, not just the exponential coordinate system of first type.
 This is important in applications as many times it is more natural or convenient to work
in other global coordinate systems such as the matrix coordinate system for Heisenberg groups.
In theory, one could  translate the global coordinate system into exponential coordinate system,
apply Kunita's result in exponential coordinate system, and then translate the results
 back to the original global coordinate system. But this is not always easy to carry out
and it needs to be performed on a case by case basis.
Interested reader may try the following two
exercises.

\medskip
 
\begin{exer} \label{Exer:3} \rm 
Let $\group$ be the continuous Heisenberg group $\HH_3 (\R)$
and $\nu$  be a L\'evy measure on $\HH_3(\R)$ whose expression under the matrix coordinate $(x, y, z)$ is given by
$$
 \bar \nu (dx, dy, dz)= \frac{\kappa_1}{|x|^{1+\alpha_1}} dx \otimes \delta_0 (dy) \otimes \delta_0 (dz)
+ \frac{\kappa_2}{|y|^{1+\alpha_2}} \delta_0 (dx)    \otimes dy \otimes \delta_0 (dz)
$$
for some  positive constants $ \alpha_i \in (0, 2)$ and $\kappa_i>0$, $i=1, 2$.
What is the expression of $\nu$ in the exponential coordinates
$(x^1, x^2, x^3)$ of the first type for $\HH_3 (\R)$?
 The group isomorphism between the matrix coordinate system and the exponential coordinate system
on $\HH_3 (\R)$ is given in \eqref{e:3.6}.
\end{exer} 

\medskip

 \begin{exer} \label{Exer:4} \rm 
Repeat Exercise \ref{Exer:3} with the L\'evy measure $\nu$ on $\group$ being replaced
in the matrix coordinate system by
$$
 \bar \nu (dx, dy, dz)= \frac{\kappa_1}{(|x|^2+|y|^2)^{1+\beta_1}} dx \otimes dy \otimes \delta_0 (dz)
+ \frac{\kappa_2}{(|y|^2+|z|^2)^{1+\beta_2}} \delta_0 (dx)    \otimes dy \otimes dz
$$
for some  positive constants $ \beta_i \in (0, 1)$ and $\kappa_i>0$, $i=1, 2$.
\end{exer}
\end{rem}

\subsection{Examples} \label{S:7.3}

To illustrate the main results of this work,
in this subsection, we first revisit Example \ref{E:1.5} of random walks on the Heisenberg  group $\HH_3 (\Z)$.
Here, we will not use the limit results for operator stable processes from the literature;
that is, we will not use Propositions \ref{P:1.1} and \ref{P:1.3}.
We will use instead  Theorems  \ref{WT1} and \ref{T:7.3} developed in this monograph.
  We will then present some more  examples.

\medskip

\noindent {\bf Exampe  \ref{E:1.5}} (revisited)
We use the matrix coordinate system $\Phi$ on the discrete Heisenberg group $\HH_3 (\Z)$, through which
it is identified with $\Z^3$. Denote by $e_1, e_2$ and $e_3$ the elements in $\HH_3(\Z)$ that
has matrix coordinates $(1, 0, 0)$, $(0, 1, 0)$ and $(0, 0, 1)$, respectively.
Recall that
$\mu_{ \boldsymbol \alpha}$ is the probability measure on $\HH_3 (\Z)=\Z^3$ given by
$$
\mu_{ \boldsymbol \alpha}(g)= \sum_{i=1}^3 \sum_{n\in \Z}\frac{\kappa_i}{(1+|n|)^{1+\alpha_i}}\1_{\{e_i^n\}}(g),\quad g\in \HH_3 (\Z)  ,
$$
where $0<\alpha_j<2$ and $\kappa_j$, $1\leq j\leq 3$, are positive constants.
The measure $\mu_{ \boldsymbol \alpha}$ is in $\mathcal{SM}$ on $\HH_3 (\Z)$ and
the matrix  coordinate system $\Phi$ is an exponential coordinate system of the second kind described in Section \ref{S:9.5}
adapted to the measure $\mu_{\boldsymbol \alpha}$. 
  The dilation structures  $\{\delta_t; t>0\}$ considered below in this example are
straight approximate group dilations of \eqref{def-delta}.
So by Section \ref{S:10} below,    the conditions (R1)-(R2), (E1)-(E2), (T$\bullet$) and (T$\Gamma$)
are automatically satisfied for $\mu_{ \boldsymbol \alpha}$ and these $\{\delta_t; t>0\}$.
For simplicity, we write $\mu$ for $\mu_{ \boldsymbol \alpha}$.
Let $\{\xi_k =(\xi_k^{(1)}, \xi_k^{(2)}, \xi_k^{(3)}); k\geq 1\}$ be a
sequence of i.i.d random variables taking values in $\HH_3 (\Z)$ of distribution $\mu$. Then
$$
S_n=S_0\cdot \xi_1 \cdot \ldots \cdot \xi_n, \quad n=0, 1, 2, \ldots
$$
 defines a random walk on the Heisenberg  group $\HH_3 (\Z)$.
Write $S_n$ as $(X_n, Y_n, Z_n)$.

\begin{enumerate}

\item[(i)]  If $1/\alpha_3< 1/\alpha_1 +1/\alpha_2$, we consider straight dilation structure $\{\delta_t; t>0\}$
in matrix coordinates:
$$
\delta_t (x, y, z) = \left(t^{1/\alpha_1}x,  \,  t^{1/\alpha_2} y,  \, t^{(1/\alpha_1) +(1/\alpha_2)}z \right).
$$
 In this case,  $\{\delta_t; t>0\}$ is a straight group dilation structure
for  the limit group $(G_\bullet, \bullet)$,
and  $(G_\bullet, \bullet)$ is the continuous Heisenberg group $\HH_3 (\R)$.
   It is easy to check that
 $ t\delta_{1/t} (\mu)$ converges vaguely on $\R^3\setminus \{0\}$ to $\bar \mu_{\bullet}(dx, dy, dz)$ as $t\to \infty$, where
$$
 \bar \mu_{\bullet}(dx, dy, dz)= \frac{\kappa_1}{|x|^{1+\alpha_1}} dx \otimes \delta_0 (dy) \otimes \delta_0 (dz)
+ \frac{\kappa_2}{|y|^{1+\alpha_2}} \delta_0 (dx)    \otimes dy \otimes \delta_0 (dz).
$$
The measure $\bar \mu_\bullet$ defines a L\'evy measure measure $\mu_\bullet$
on the continuous Heisenberg group
$(G_\bullet, \bullet)$ through the matrix coordinate system; see Remark \ref{R:5.1}(i).
 In other words,
$\bar \mu_\bullet$ is the pull-back measure of $\mu_\bullet$ under the matrix coordinate system.
When there is no danger of  confusions,
we simply use the same notation $\mu_\bullet$ for $\bar \mu_\bullet$.
Thus by Theorem \ref{WT1}, for any $T>0$, the rescaled random walk on $\HH_3 (\Z)$
in matrix coordinates
$$
\Big\{ \left(  n^{-1/\alpha_1} X_{[nt]},  \,  n^{-1/\alpha_2} Y_{[nt]},  \, n^{-1/\alpha_1 -1/\alpha_2}   Z_{[nt]} \right) ; t\in [0, T] \Big\}
$$
converges weakly in the Skorohod space $\D([0, T]; \R^3)$ to a L\'evy process $X^\bullet$ on
$(G_\bullet, \bullet)$
with L\'evy measure $\mu_\bullet$
as $n\to \infty$.
We next identify the L\'evy process $X^\bullet$ in the matrix coordinate system of $(G_\bullet, \bullet)$ by using Theorem \ref{T:7.3}.

By \eqref{e:3.8} and \eqref{e:3.6}, in matrix coordinates  $ (x, y, z)$ for $\sigma \in\G_\bullet$,
\begin{equation}\label{e:7.13}
 \sigma^{-1}_\bullet = (-x, -y, -z +xy) \quad \hbox{and} \quad
\| \sigma \|= \sqrt{x^2+y^2 + (z+\tfrac12 xy)^2} .
\end{equation}
So $\sigma + \sigma^{-1}_\bullet = (0, 0, xy)$.
On the support of $\mu_\bullet$, since $xy=0$,  we have
$\| \sigma \| =\| \sigma \|_2$ and
$\sigma+ \sigma^{-1}_\bullet = (0, 0, 0)$,
that is, $\sigma_\bullet^{-1}= -\sigma$.
Hence denoting the matrix coordinates for $\sigma\in G_\bullet$ by $(\sigma_1, \sigma_2, \sigma_3)\in \R^d$, it follows from \eqref{e:7.4a}  that for every $1\leq i\leq 3$,
$$
b_j =\frac12 \int_{\{\sigma \in G_\bullet: \|\sigma\|_2 \leq 1 \}}
\left(\sigma_i + (\sigma^{-1}_\bullet)_i \right) \mu_\bullet (d\sigma)=0.
$$
Let $\bar X$ be a symmetric $\alpha_1$-stable process on $\R$ with L\'evy measure $\kappa_1 |z|^{-(1+\alpha_1)}$ and $\bar Y$ be a symmetric $\alpha_2$-stable process on $\R$ with L\'evy measure $\kappa_2 |z|^{-(1+\alpha_2)}$
independent of $X$. Then
$$
X^\circ =(\bar X, \bar Y, 0)
$$
 is a driftless L\'evy process on $\R^3$ with L\'evy measure $\bar \mu_\bullet$ corresponding to \eqref{e:7.6}. By Theorem \ref{T:7.3}, $X^\bullet$ is the weak limit
on  $\bD ([0, T]; G_\bullet)=\D([0, T]; \R^3)$ of
  $$
  X^{\bullet , n}_t :=  \Phi(  X^{\circ}_{1/n}) \bullet  \Phi (X^{\circ}_{2/n} -  X^{\circ}_{1/n}) .  \bullet  \cdots
      \bullet  \Phi   ( X^{\circ}_{[nt]/n} -  X^{\circ}_{([nt]-1)/n}) .
 $$
 Note that in matrix coordinate system on $G_\bullet$,
$$
X^{\bullet , n}_t =\left( \bar X_{[nt]/n}, \, \bar Y_{[nt]/n}, \, \sum_{k=1}^{[nt]} \bar X_{(k-1)/n}
 \left(\bar Y_{k/n}- \bar Y_{(k-1)/n} \right) \right), \quad t\geq 0,
$$
which converges weakly in the Skorohod space $\D([0, T]; \R^3)$ equipped with ${\cal J}_1$-topology
 to $\{ ( \bar X_t, \bar Y_t, \int_0^t \bar X_{s-} d \bar Y_s); t\in [0, T]\}$; see, e.g., \cite[Theorem 7.19]{KP}.
 This shows that $\{X^\bullet_t; t\in [0, T]\}$ under the matrix coordinate system of
$(G_\bullet, \bullet)$ has the same distribution
as $\{ (\bar X_t, \bar Y_t, \int_0^t \bar X_{s-} d \bar Y_s); t\in [0, T]\}$.

\item[(ii)] If $1/\alpha_3= 1/\alpha_1 +1/\alpha_2$,
 we consider straight dilation structure $\{\delta_t; t>0\}$
in matrix coordinates:
$$
\delta_t (x, y, z) =\left( t^{1/\alpha_1}x,  \,  t^{1/\alpha_2} y,  \, t^{1/\alpha_3}z\right).
$$
As mentioned in (i),
   $\{\delta_t; t>0\}$ is a straight group dilation structure
 the limiting group structure $(G_\bullet, \bullet)$ is the continuous Heisenberg group $\HH_3 (\R)$.
It is easy to check in this case that
 $ t\delta_{1/t} (\mu)$ converges vaguely on $\R^3\setminus \{0\}$ to $\bar \mu_{\bullet}(dx, dy, dz)$ as $t\to \infty$, where
\begin{align*}
&  \mu_{\bullet}(dx, dy, dz) \\
&=  \frac{\kappa_1}{|x|^{1+\alpha_1}} dx \otimes \delta_0 (dy) \otimes \delta_0 (dz)
+ \frac{\kappa_2}{|y|^{1+\alpha_2}} \delta_0 (dx) \otimes dy \otimes \delta_0 (dz)\\
&\quad +  \frac{\kappa_3}{|z|^{1+\alpha_3}} \delta_0 (dx)   \otimes \delta_0 (dy)  \otimes dz.
\end{align*}
The above measure  $\bar \mu_\bullet$ is the expression of a symmetric L\'evy measure $\mu_\bullet$ under the
matrix coordinate system on continuous Heisenberg group $(G_\bullet, \bullet)$.
By Theorem \ref{WT1}, for any $T>0$, the rescaled random walk on  $\HH_3 (\Z)$ in matrix coordinates
$$
\Big\{ \left(  n^{-1/\alpha_1} X_{[nt]},  \,  n^{-1/\alpha_2} Y_{[nt]},  \, n^{-1/\alpha_3}   Z_{[nt]} \right) ; t\in [0, T] \Big\}
$$
converges weakly in the Skorohod space $\D([0, T]; \R^3)$ to a symmetric
L\'evy process $X^\bullet$ on $G_\bullet$ with L\'evy measure $\mu_\bullet$
as $n\to \infty$.

To identify the L\'evy process $X^\bullet$ on $(G_\bullet, \bullet)$ in matrix coordinates
$(x, y, z)$ of $\sigma \in G_\bullet$,
note that by \eqref{e:7.13},
 since $xy=0$ on the support of $\mu_\bullet$,  $\| \sigma \| =\| \sigma \|_2$ and
$\sigma_\bullet^{-1}= -\sigma$.
Thus we have by \eqref{e:7.4a} that
$$
b_j = \frac12 \int_{\{\sigma \in G_\bullet: \|\sigma\|_2 \leq 1 \}}
\left(\sigma_i + (\sigma^{-1}_\bullet)_i \right) \mu_\bullet (d\sigma)
=0  \quad \hbox{for  every } 1\leq i\leq 3 .
$$
Let $\bar X$, $\bar Y$ and $\bar Z$ be independent one-dimensional symmetric $\alpha_1$-, $\alpha_2$- and $\alpha_3$-stable processes
 with L\'evy measure $\kappa_i |z|^{-(1+\alpha_i)}$, $1\leq i\leq 3$.  Then
$$
X^\circ =(\bar X, \bar Y, \bar Z)
$$
 is a driftless L\'evy process on $\R^3$ with L\'evy measure $\mu_\bullet$ corresponding to \eqref{e:7.6}. By Theorem \ref{T:7.3}, $X^\bullet$ is the weak limit
on  $\D([0, T]; G_\bullet)=\D([0, T]; \R^3)$ of
  $$
  X^{\bullet , n}_t :=  \Phi(   X^{\circ}_{1/n}) \bullet  \Phi (X^{\circ}_{2/n} -  X^{\circ}_{1/n}) .  \bullet  \cdots
      \bullet  \Phi   ( X^{\circ}_{[nt]/n} -  X^{\circ}_{([nt]-1)/n}) .
$$
 In this case, in matrix coordinates,
$$
X^{\bullet , n}_t =\left( \bar X_{[nt]/n}, \, \bar Y_{[nt]/n}, \,
\bar Z_{[nt]/n} +\sum_{k=1}^{[nt]} \bar X_{(k-1)/n} \left(\bar Y_{k/n}- \bar Y_{(k-1)/n} \right) \right),
$$
which converges weakly in the Skorohod space $\D([0, T]; \R^3)$ equipped with ${\cal J}_1$-topology
 to $\{ ( \bar X_t, \bar Y_t, \bar Z_t+ \int_0^t \bar X_{s-} d \bar Y_s); t\in [0, T]\}$.
 This shows that $\{X^\bullet_t; t\in [0, T]\}$ in the matrix coordinate system of $(G_\bullet, \bullet)$
has the same distribution
as $\{ ( \bar X_t, \bar Y_t, \bar Z_t+ \int_0^t \bar X_{s-} d \bar Y_s); t\in [0, T]\}$.

\item[(iii)] If $1/\alpha_3> 1/\alpha_1 +1/\alpha_2$,
 we consider straight dilation structure $\{\delta_t; t>0\}$
in matrix coordinates:
$$
\delta_t (x, y, z) = \left( t^{1/\alpha_1}x,  \,  t^{1/\alpha_2} y,  \, t^{1/\alpha_3}z
\right).
$$
In this case, we see from Example \ref{E:3.4} that the limit group structure
$(G_\bullet, \bullet)$ is just the additive $\R^3$.
It is easy to check that
 $ t\delta_{1/t} (\mu)$ converges vaguely on $\R^3\setminus \{0\}$ to $\mu_{\bullet}(dx, dy, dz)$ as $t\to \infty$, where
\begin{align*}
 &  \bar \mu_{\bullet}(dx, dy, dz) \\
&=  \frac{\kappa_1}{|x|^{1+\alpha_1}} dx \otimes \delta_0 (dy) \otimes \delta_0 (dz)
+ \frac{\kappa_2}{|y|^{1+\alpha_2}} \delta_0 (dx) d  \otimes dy \otimes \delta_0 (dz)\\
&\quad +  \frac{\kappa_3}{|z|^{1+\alpha_3}} \delta_0 (dx)   \otimes \delta_0 (dy)  \otimes dz.
\end{align*}
Note that the matrix coordinate system on $\R^3$ is the identity map so the induced L\'evy measure
$\mu_\bullet$ on the abelian group $(G_\bullet, \bullet)$ is just $\bar \mu_\bullet$ itself.
By Theorem \ref{WT1}, for any $T>0$, the rescaled random walk on $\HH_3 (\Z)$ in matrix coordinates
$$
\Big\{ \left(  n^{-1/\alpha_1} X_{[nt]},  \,  n^{-1/\alpha_2} Y_{[nt]},  \, n^{-1/\alpha_3}   Z_{[nt]} \right) ; t\in [0, T] \Big\}
$$
converges weakly in the Skorohod space $\D([0, T]; \R^3)$ to a symmetric
L\'evy process $X^\bullet$ on $G_\bullet$ with L\'evy measure $\mu_\bullet$
as $n\to \infty$. Since $(G_\bullet, \bullet)$ is
$(\R^3, +)$,  we conclude directly
that $X^\bullet$ has the same distribution as $X^\circ =(\bar X, \bar Y, \bar Z)$, where
$\bar X$, $\bar Y$ and $\bar Z$ are independent one-dimensional symmetric $\alpha_1$-, $\alpha_2$- and $\alpha_3$-stable processes
 with L\'evy measure $\kappa_i |z|^{-(1+\alpha_i)}$, $1\leq i\leq 3$.
\end{enumerate}

 We next present a few more examples.

\begin{exa}
Let $ \mu =\frac{1}{2}(\mu_1+\mu_2)$ be the probability measure on $\HH_3 (\Z)=\Z^3$ with
$$
\mu_1(x,y,z)=\frac{ c_1}{ (1+|x|+|z|)^{2+\alpha_1}}
\quad \hbox{and} \quad
\mu_2(x,y,z)=\frac{c_2}{ (1+|y|+|z|)^{2+\alpha_2}},
$$
where $0<\alpha_1,  \alpha_2<2$ and $c_j>0$, $j=1, 2$, are positive constants.
The measure $\mu $ is again in $\mathcal{SM}$ on $\HH_3 (\Z)$.
Let $\{\xi_k =(\xi_k^{(1)}, \xi_k^{(2)}, \xi_k^{(3)}); k\geq 1\}$ be a sequence of
i.i.d random variables taking values in $\HH_3 (\Z)$ of distribution $\mu$. Then
$$
S_n=S_0 \cdot  \xi_1 \cdot \ldots \cdot  \xi_n, \quad n=0, 1, 2, \ldots
$$
 defines a random walk on the Heisenberg  group $\HH_3 (\Z)$.
Write $S_n$ as $(X_n, Y_n, Z_n)$.

We consider straight dilation structure $\{\delta_t; t>0\}$
in matrix coordinates:
$$
\delta_t (x, y, z) =(t^{1/\alpha_1}x,  \,  t^{1/\alpha_2} y,  \, t^{1/\alpha_1 +1/\alpha_2}z).
$$
This  dilation structure  $\{\delta_t; t>0\}$  is a
straight   group dilation  of \eqref{def-delta} adapted to the measure $\mu$
so   the  limiting group $(G_\bullet, \bullet)$
 is the continuous Heisenberg group $\HH_3 (\R)$.
Since the matrix  coordinate system $\Phi$ is an exponential coordinate system of the second kind described in Section \ref{S:9.5},  by Section \ref{S:10} below,    the conditions (R1)-(R2), (E1)-(E2), (T$\bullet$) and (T$\Gamma$)
are automatically satisfied for $\mu $ and   $\{\delta_t; t>0\}$.
   It is easy to check (cf. Example \ref{ex3-3}) that
 $ t\delta_{1/t} (\mu)$ converges vaguely on $\R^3\setminus \{0\}$ to
$\bar \mu_{\bullet}(dx, dy, dz)$ as $t\to \infty$, where
$$
 \bar \mu_{\bullet}(dx, dy, dz)= \frac{\kappa_1}{|x|^{1+\alpha_1}} dx \otimes \delta_0 (dy) \otimes \delta_0 (dz)
+ \frac{\kappa_2}{|y|^{1+\alpha_2}} \delta_0 (dx)    \otimes dy \otimes \delta_0 (dz).
$$
Here $\kappa_i = c_i \int_{\R} (1+|u|)^{-(2+\alpha_i)} du$ for $i=1, 2$.
The measure $\bar \mu_\bullet$ induces a L\'evy measure $\mu_\bullet$ on $(G_\bullet, \bullet)$ via the matrix coordinate system $\Phi$. In part (i) of Example \ref{E:1.5} (revisited), we have already identified
the symmetric L\'evy process $X^\bullet$ on the continuous Heisenberg group $(G_\bullet, \bullet)$.
Thus it follows from Theorem \ref{WT1} that, for any $T>0$, the rescaled random walk on $\HH_3 (\Z)$
in matrix coordinates
$$
\Big\{ \left(  n^{-1/\alpha_1} X_{[nt]},  \,  n^{-1/\alpha_2} Y_{[nt]},  \, n^{-1/\alpha_1 -1/\alpha_2}   Z_{[nt]} \right) ; t\in [0, T] \Big\}
$$
converges weakly in the Skorohod space $\D([0, T]; \R^3)$ to
$\{ (\bar X_t, \bar Y_t, \int_0^t \bar X_{s-} d \bar Y_s); t\in [0, T]\}$
on the continuous Heisenberg group $\HH_3(\R)$ in matrix coordinates,
where $\bar X$ and $\bar Y$ are independent one-dimensional symmetric
$\alpha_1$- and $\alpha_2$-stable processes, respectively.

\end{exa}

\begin{exa}\label{E:7.6}
	Let $\mu$ be the probability measure on $\HH_3 (\Z)=\Z^3$ with
	$$
	\mu (x,y,z)=\frac{ c}{ (1+\sqrt{x^2+y^2 +|z-xy|} )^{4+\alpha}},
	$$
	where $0<\alpha  <2$ and $c>0$ are positive constants.

The measure $\mu $ is again in $\mathcal{SM}$ on $\HH_3 (\Z)$.
Let $\{\xi_k =(\xi_k^{(1)}, \xi_k^{(2)}, \xi_k^{(3)}); k\geq 1\}$ be a  sequence of i.i.d random variables taking values in $\HH_3 (\Z)$ of distribution $\mu$. Then
$$
S_n=S_0 \cdot  \xi_1 \cdot \ldots \cdot  \xi_n, \quad n=0, 1, 2, \ldots
$$
defines a random walk on the Heisenberg  group $\HH_3 (\Z)$.
Write $S_n$ as $(X_n, Y_n, Z_n)$.

Consider the dilation
$$
\delta_t((x, y, z))=(t^{1/{\alpha}} x, t^{{1}/{\alpha} } y,
t^{{2}/{\alpha }}z) .
$$
This  dilation structure  $\{\delta_t; t>0\}$  is a
straight   group dilation  of \eqref{def-delta} adapted to the measure $\mu$
so  the limiting group structure $(G_\bullet, \bullet)$ is the continuous Heisenberg group $\HH_3 (\R)$.
 By Section \ref{S:10} below,
 the conditions (R1)-(R2), (E1)-(E2), (T$\bullet$) and (T$\Gamma$) are automatically satisfied for
 $\mu$ and $\{\delta_t, t>0\}$.
It is easy to check in this case that
 $ t\delta_{1/t} (\mu)$ converges vaguely on $\R^3\setminus \{0\}$ to $\bar \mu_{\bullet}(dx, dy, dz)$
 as $t\to \infty$, where
$$
\bar \mu_\bullet (dx)= \frac{c}{ (\sqrt{  x^2+ y^2+|z-x y| }) ^{4+\alpha }} dx dy dz.
$$
The measure $\bar \mu_\bullet$,
 though itself is not symmetric on $\R^3$,
 induces a symmetric L\'evy measure $\mu_\bullet$ on $(G_\bullet, \bullet)$ via the matrix coordinate system $\Phi$.
Thus by Theorem \ref{WT1}, for any $T>0$, the rescaled random walk on $\HH_3 (\Z)$
in matrix coordinates
$$
\Big\{ \left(  n^{-1/\alpha} X_{[nt]},  \,  n^{-1/\alpha} Y_{[nt]},  \, n^{-2/\alpha}   Z_{[nt]} \right) ; t\in [0, T] \Big\}
$$
converges weakly in the Skorohod space $\D([0, T]; \R^3)$ to a
purely discontinuous symmetric L\'evy process $X^\bullet$ on
$(G_\bullet, \bullet)$
with L\'evy measure $\mu_\bullet$
as $n\to \infty$.
We next identify the L\'evy process $X^\bullet$ in the matrix coordinate system of $(G_\bullet, \bullet)$ by using Theorem \ref{T:7.3}.

Recall  that for $\sigma =(x, y, z)\in \HH_3 (\bR)$,
$$
\sigma_\bullet^{-1}= (-x, -y, -z+xy) \quad \hbox{and} \quad
\sigma +\sigma_\bullet^{-1}=(0, 0, xy).
$$
Since  $\| \sigma\|_2$, $ \|\sigma_\bullet^{-1}\|_2$
and $\bar \mu_\bullet$  are invariant under the transformations
$(x, y, z)\mapsto (-x, y, -z)$ and $(x, y, z)\mapsto (x, -y, -z)$, we have by
\eqref{e:7.4a} that $\bar b_j=0$ for every $1\leq j\leq 3$. Let $X^\circ=(\bar X, \bar Y, \bar Z)$ be the L\'evy process on $\R^3$ with L\'evy triplet $(0, 0, \bar \mu_\bullet)$.
 Note that the L\'evy process $X^\circ$ is not symmetric on $\R^3$ as its L\'evy measure
$\bar \mu_\bullet$ is not symmetric on $\R^3$.
We conclude from Theorem \ref{T:7.3} with the same calculation  as that in part (ii) of  Example \ref{E:1.5}(revisited) that, in matrix coordinate $\Phi$,
$$
X^\bullet_t= \Big(\bar X_t, \bar Y_t, \bar Z_t + \int_0^t \bar X_{s-} d\bar Y_s \Big)
\quad \hbox{for } t\geq 0.
$$
\end{exa}

\begin{exa} \label{E:7.7}
	Consider the group $\mathbb U_4(\mathbb Z)$ of $4$ by $4$ upper-triangular matrices
with diagonal entries equal to $1$ given in Example \ref{E:3.15}.
 That is,
$$
G=\mathbb U_4(\Z) =\left\{\begin{pmatrix} 1&x_1&x_4&x_6 \\
	0&1&x_2 &x_5  \\
	0&0&1& x_3 \\
	0&0&0&1\end{pmatrix} : x_i\in \mathbb R\right\}.
	$$
In matrix coordinates, $\mathbb U_4(\R)$ is $\mathbb R^6$ with multiplication
$(x_i)_1^6(y_i)_1^6=(z_i)_1^6$ given by
$$
z_i=
\begin{cases}
 x_i+y_i \quad & \hbox{ for } i=1,2,3,\\
 x_4+y_4 +x_1 y_2 & \hbox{ for } i=4,\\
x_5+y_5 +x_2 y_3 & \hbox{ for } i=5,\\
x_6+y_6 + x_1 y_5+x_4y_3 &\hbox{ for } i=6.
\end{cases}
$$
 This matrix  coordinate system $\Phi$ is an exponential coordinate system of the second kind described in Section \ref{S:9.5}; see Example \ref{exa:PhiPsi}.

We consider two cases.

\medskip

(i) Let $\mu = \frac{1}{2}(\mu_1+\mu_2)$ be the probability measure on $\mathbb U_4 (\Z)=\Z^6$ with
\begin{equation}\label{e:7.19a}
\mu_1((x_i)_1^6)=\frac{ c_1}{ (1+\sqrt{x_1^2+ x_2^2+|x_4-x_1x_2|})^{4+\alpha_1}} \otimes \1_{(0,0,0)}(x_3,x_5,x_6)
\end{equation}
and
\begin{equation}\label{e:7.20a}
\mu_2((x_i)_1^6)=\frac{c_2}{ (1+ \sqrt{x_3^2+x_5^2+x_6^2})^{3+\alpha_2}}
\otimes \1_{(0,0,0)}(x_1,x_2,x_4),
\end{equation}
where $0<\alpha_1,  \alpha_2<2$ and $c_1,c_2$ are appropriate positive normalizing constants.

The measure $\mu $ is in $\mathcal{SM}$ on $\mathbb U_4 (\Z)$.	
Let $\{\xi_k =(\xi_k^{(1)}, \xi_k^{(2)}, \xi_k^{(3)}, \xi_k^{(4)}, \xi_k^{(5)}, \xi_k^{(6)}); k\geq 1\}$ be  a sequence of i.i.d random variables taking values in $\mathbb U_4 (\Z)$ of distribution $\mu$. Then
$$
S_n=S_0 \cdot  \xi_1 \cdot \ldots \cdot  \xi_n, \quad n=0, 1, 2, \ldots
$$
defines a random walk on $\mathbb U_4 (\Z)$.
Write $S_n$ as
$$
(X^{(1)}_n, X^{(2)}_n, X^{(3)}_n, X^{(4)}_n, X^{(5)}_n, X^{(6)}_n).
$$

Consider the dilation
$$
\delta_t((x_i)_1^6)=(t^{{1}/{\alpha_1}} x_1, t^{{1}/{\alpha_1} } x_2, t^{{1}/{\alpha_2}}x_3,t^{{2}/{\alpha_1}}x_4,
t^{{1}/{\alpha_1}+{1}/{\alpha_2}}x_5,
t^{{2}/{\alpha_1}+{1}/{\alpha_2}}x_6).
$$
As noted in Example \ref{E:3.15},
this is a group dilation structure so the limit
group $(G_\bullet, \bullet )$ is $\mathbb U_4(\R)$. 
It is in fact the straight   group dilation  of \eqref{def-delta} adapted to the measure $\mu$
Thus  by Section \ref{S:10} below,    the conditions (R1)-(R2), (E1)-(E2), (T$\bullet$) and (T$\Gamma$)
are automatically satisfied for $\mu$ and   $\{\delta_t, t>0\}$.

\medskip

The measure $\mu_t=t\delta_t(\mu)$ has vague limit $\bar \mu_\bullet$ as $t \to \infty$
given by
\begin{align*}
\bar \mu_\bullet (dx) =&  \frac{c_1}{ 2 ( \sqrt{x_1^2+ x_2^2+|x_4-x_1x_2|})^{4+\alpha_1}} dx_1dx_2dx_4\otimes \delta_{(0,0,0)}(dx_3,dx_5,dx_6)\\
&+\frac{c'_2}{ 2 |x_3|^{1+\alpha_2} } dx_3  \otimes \delta_{(0,0,0,0,0)}(dx_1,dx_2,dx_4,dx_5,dx_6) .
\end{align*}
 Note that though the measure $\bar \mu_\bullet$ is not symmetric on $\R^6$, it
induces a symmetric L\'evy measure $\mu_\bullet$ on $(G_\bullet, \bullet)$ through the matrix coordinate system
$\Phi$.

By a similar reasoning as in the previous example, one can check that the drift $\bar b$ defined by
\eqref{e:7.4a} is a zero vector in $\R^6$.
Let
$$
X^\circ =\left(\bar X^{(1)}, \bar X^{(2)},  \bar X^{(3)}, \bar X^{(4)}, \, 0, \, 0 \right),
$$
 be the L\'evy process on $\R^6$ with L\'evy triplet $(0, 0, \bar \mu_\bullet )$.
Note that
$(\bar X^{(1)}, \bar X^{(2)}, \bar X^{(4)})$
is a L\'evy process on $\R^3$ with L\'evy triplet
$\Big(0, 0,  \frac{c_1}{ 2(\sqrt{x_1^2+ x_2^2+|x_4-x_1x_2|})^{4+\alpha_1}} dx_1dx_2dx_4 \Big)$
and $\bar X^{(3)}$ is a one-dimensional symmetric $\alpha_2$-stable process with L\'evy measure
$\frac{c'_2}{ 2 |x_3|^{1+\alpha_2} } dx_3$
independent of $(\bar X^{(1)}, \bar X^{(2)}, \bar X^{(4)})$.
Thus we have by Theorem \ref{WT1}, for any $T>0$, the rescaled random walk $\{ \delta_{1/n} (S_{[nt]}); t\in [0, T])$ on $\mathbb U_4 (\Z)$
converges weakly in the Skorohod space $\D([0, T]; \R^6)$ to a purely discontinuous symmetric
L\'evy process $X^\bullet$ on
$(G_\bullet, \bullet)$
with L\'evy measure $\mu_\bullet$
as $n\to \infty$.

We next identify the L\'evy process $X^\bullet$ in the matrix coordinate system of $(G_\bullet, \bullet)$ by using Theorem \ref{T:7.3}, through the fact that $X^\bullet$ is the weak limit of
 $$
  X^{\bullet , n}_t :=  \Phi(   X^{\circ}_{1/n}) \bullet  \Phi (X^{\circ}_{2/n} -  X^{\circ}_{1/n}) .  \bullet  \cdots
      \bullet  \Phi   ( X^{\circ}_{[nt]/n} -  X^{\circ}_{([nt]-1)/n}) .
$$
When $\alpha_1\leq \alpha_2$, $(G_\bullet, \bullet)$ is $\mathbb U_4 (\R)$.
By a similar reasoning as in previous examples, we conclude from Theorem \ref{T:7.3}
that in matrix coordinates, the symmetric L\'evy process
$X_t^\bullet$ on $\mathbb U_4 (\R)$ has the following six coordinates:
$$\bar X^{(1)}_t, \,  \bar X^{(2)}_t, \, \bar X^{(3)}_t, \,
\bar X^{(4)} +\int_0^t \bar X^{(1)}_{s-} d\bar X^{(2)}_s,
\,  \int_0^t \bar X^{(2)}_{s-} d\bar X^{(3)}_s,$$
and
$$
 \int_0^t \bar{X}^{(1)}_{s_-}\bar{X}^{(2)}_{s_-}d\bar{X}^{(3)}_s +
 \int_0^t \left( \bar{X}^{(4)}_{r_-}+\int_{[0, r)} \bar X^{(1)}_{s-} d\bar X^{(2)}_s \right)
d \bar X^{(3)}_r .
$$
Note that L\'evy processes are semimartingales so the above stochastic integrals are all well defined.

\medskip

(ii) Now let
$\mu = \frac13 (\mu_1+\mu_2+\mu_3)$ be the probability measure on $\mathbb U_4 (\Z)=\Z^6$ with  $\mu_1$ and $\mu_2$ given by
\eqref{e:7.19a} and \eqref{e:7.20a} with $\alpha_1=\alpha_2 \in (0, 2)$, and
 $$
 \mu_3 ((x_i)_1^6) =\frac{ c_3}{ (1+\sqrt{x_4^2+ x_5^2+x_6^2})^{3+\alpha_3}} \otimes \1_{(0,0,0)}(x_1,x_2,x_3)
$$
for some $\alpha_3 \in (0, \alpha_1/2)$.
 This measure $\mu $ is in $\mathcal{SM}$ on $\mathbb U_4 (\Z)$.
Let
$$
\left\{\xi_k =\left(\xi_k^{(1)}, \xi_k^{(2)}, \xi_k^{(3)}, \xi_k^{(4)}, \xi_k^{(5)}, \xi_k^{(6)} \right); k\geq 1 \right\}
$$
be a  sequence of   i.i.d random variables taking values in $\mathbb U_4 (\Z)$ of distribution $\mu$. Then
$$
S_n=S_0 \cdot  \xi_1 \cdot \ldots \cdot  \xi_n, \quad n=0, 1, 2, \ldots
$$
defines a random walk on $\mathbb U_4 (\Z)$.

Consider the dilation
$$
\delta_t((x_i)_1^6)= \left(t^{{1}/{\alpha_1}} x_1, \, t^{{1}/{\alpha_1} } x_2, \,
t^{{1}/{\alpha_1}}x_3, \, t^{ {1}/{\alpha_3}}x_4, \,
t^{ {1}/{\alpha_3}}x_5, \,
t^{1/\alpha_1 + 1/ \alpha_3  } x_6 \right),
$$
which  is a  straight group approximate dilation  of \eqref{def-delta} adapted to the measure $\mu$.
As noted in  Example \ref{E:3.15} (the fourth bullet case),
this is an approximate   group dilation structure for $\mathbb U_4 (\Z)$ and the group law
$\bullet$ of the limit  group  $(G_\bullet, \bullet)$ is the direct product of the 5 dimensional Heisenberg group
$\HH_5 (\R)$ and a copy of $\R$, that is,
\begin{eqnarray*}
 \lefteqn{(x_i)_1^6\bullet(y_i)_1^6}&&\\&=&(x_1+y_1, x_2+ y_2, x_3+y_3, x_4+y_4,  x_5+y_5,
x_6+y_6+x_1 y_5 + x_4 y_3 ).
\end{eqnarray*}
Clearly, $(G_\bullet, \bullet)$ is different from from $\mathbb U_4(\R)$.
Since the measure $\mu $ is in $\mathcal{SM}$ on $\mathbb U_4 (\Z)$,
  the conditions (R1)-(R2), (E1)-(E2), (T$\bullet$) and (T$\Gamma$) are automatically satisfied for $\mu$ and $\{\delta_t, t>0\}$ by Section \ref{S:10} below.

The measure $\mu_t=t\delta_t(\mu)$ has vague limit $\bar \mu_\bullet$ as $t \to \infty$
given by
\begin{align*}
	\bar \mu_\bullet (dx) =& \frac{\kappa_1}{ (x_1^2+ x_2^2)^{(2+\alpha_1)/2}} dx_1dx_2 \otimes \delta_{(0, 0, 0,0)}(dx_3, dx_4, dx_5,dx_6)\\
	& +\frac{\kappa_2}{  |x_3|^{1+\alpha_1} } dx_3  \otimes \delta_{(0,0,0,0,0)}(dx_1,dx_2,dx_4,dx_5,dx_6) \\
	& +\frac{\kappa_3}{ (x_4^2+ x_5^2)^{(2+\alpha_3)/2} } dx_4 dx_5  \otimes \delta_{(0,0,0,0)}(dx_1,dx_2,dx_3, dx_6) .
\end{align*}
It induces a symmetric L\'evy measure $\mu_\bullet$ on $(G_\bullet, \bullet)$ through the matrix coordinate system $\Phi$.
It is easy to see that the drift  $\bar b$ defined by \eqref{e:7.4a} is a zero vector in $\R^6$
and the L\'evy process $X^\circ$ on $\R^d$ with L\'evy triplet $(0, 0, \bar \mu_\bullet)$
is
$$
X^\circ_t= \left(\bar X^{(1)}_t, \bar X^{(2)}_t, \bar X^{(3)}_t,  \bar X^{(4)}_t, \bar X^{(5)}_t, \, 0 \right),
$$
where $( \bar X^{(1)}, \bar X^{(2)})$ is a two-dimensional isotropic $\alpha_1$-stable process,
$\bar X^{(3)}$ is an independent one-dimensional $\alpha_1$-stable process,
and $( \bar X^{(4)}, \bar X^{(5)} )$ is a two-dimensional isotropic $\alpha_3$-stable process
that is independent of $(\bar X^{(1)}, X^{(2)}, X^{(3)})$.
In a similar way as in previous examples, we can conclude from Theorems \ref{WT1} and \ref{T:7.3} that for any $T>0$, the rescaled random walk
$\left\{ \delta_{1/n} (S_{[nt]}); t\in [0, T] \right\}$ on $\mathbb U_4 (\Z )$
converges weakly in the Skorohod space $\D([0, T]; \R^6)$ to a purely discontinuous symmetric
L\'evy process $X^\bullet$ on
$(G_\bullet, \bullet)$
with L\'evy measure $\mu_\bullet$
as $n\to \infty$, which in the matrix coordinate system is given by
$$
\left(\bar X^{(1)}_t,  \ \bar X^{(2)}_t, \  \bar X^{(3)}_t, \,  \ \bar X^{(4)}_t, \  \bar X^{(5)}_t,
\,
\int_0^t \bar{X}^{(1)}_{s_-} d\bar{X}^{(5)}_s +
\int_0^t \bar X^{(4)}_{s-} d\bar X^{(3)}_s  \right).
$$
\end{exa}

\medskip

\begin{exa} \label{E:7.8}
	In Example \ref{E:7.7},  now consider the probability measure
  $\mu = \frac{1}{3}(\mu_1+\mu_2+\mu_3)$  on $\mathbb U_4 (\Z)=\Z^6$, where
$$
\mu_1((x_i)_1^6)=\frac{ c_1}{ (1+\sqrt{x_1^2+ x_4^2+x_6^2|})^{3+\alpha_1}} \otimes \1_{(0,0,0)}(x_2,x_3,x_5),
$$
$$
\mu_2((x_i)_1^6)=\frac{c_2}{ (1+ |x_2|)^{1+\alpha_2}}
\otimes \1_{(0,0,0, 0, 0)}(x_1,x_3, x_4, x_5, x_6),
$$
and
$$\mu_3((x_i)_1^6)=\frac{c_3}{ (1+ \sqrt{x_3^2+x_5^2+x_6^2})^{3+\alpha_3}}
\otimes \1_{(0,0,0)}(x_1,x_2,x_4),
$$
where $0<\alpha_1,  \alpha_2, \alpha_3<2$ and $c_1,c_2, c_3$ are appropriate positive normalizing constants.

The measure $\mu $ is in $\mathcal{SM}$ on $\mathbb U_4 (\Z)$.
Let
$$
\left\{\xi_k =(\xi_k^{(1)}, \xi_k^{(2)}, \xi_k^{(3)}, \xi_k^{(4)}, \xi_k^{(5)}, \xi_k^{(6)}); \, k\geq 1\right\}
$$
 be a  sequence of i.i.d random variables taking values in $\mathbb U_4 (\Z)$ of distribution $\mu$. Then
$$
S_n=S_0 \cdot  \xi_1 \cdot \ldots \cdot  \xi_n, \quad n=0, 1, 2, \ldots
$$
defines a random walk on $\mathbb U_4 (\Z)$.

Consider the dilation
$$
\delta_t((x_i)_1^6)= \left(t^{ {1}/{\alpha_1}} x_1, t^{ {1}/{\alpha_2} } x_2, t^{{1}/{\alpha_3}}x_3,t^{ {1}/{\alpha_1}+ 1/{\alpha_2}}x_4,
t^{{1}/{\alpha_2}+ {1}/{\alpha_3}}x_5,
t^{ {1}/{\alpha_1}+ {1}/{\alpha_2}+ {1}/{\alpha_3}}x_6 \right).
$$
As noted in Example \ref{E:3.15},
this is a group dilation structure so the limiting
group $G_\bullet$ is $\mathbb U_4(\R)$.
It is  in fact a  straight group   dilation  of \eqref{def-delta} adapted to the measure $\mu$.

Since measure $\mu $ is in $\mathcal{SM}$ on $\mathbb U_4 (\Z)$,
 the conditions (R1)-(R2), (E1)-(E2), (T$\bullet$) and (T$\Gamma$) are again automatically satisfied for $\mu$ and $\{\delta_t, t>0\}$ by Section \ref{S:10} below.
The measure $\mu_t=t\delta_t(\mu)$ has vague limit $\bar \mu_\bullet$ as $t \to \infty$
given by
\begin{align} \label{e:7.20}
\bar \mu_\bullet (dx) =&
 \frac{\kappa_1}{|x_1|^{1+\alpha_1}} dx_1 \otimes \delta_{(0, 0, 0,0,0)}(dx_2, dx_3, dx_4, x_5,dx_6)
\nonumber \\
& +\frac{\kappa_2}{|x_2|^{1+\alpha_2} } dx_2  \otimes \delta_{(0,0,0,0,0)}(dx_1,dx_3,dx_4,dx_5,dx_6)
 \\
&  +\frac{\kappa_3}{|x_3|^{1+\alpha_3} } dx_3  \otimes \delta_{(0,0,0,0,0)}(dx_1,dx_2,dx_4,dx_5,dx_6)
.  \nonumber
\end{align}
It induces a symmetric L\'evy measure $\mu_\bullet$ on $(G_\bullet, \bullet)$ through the matrix coordinate system
$\Phi$. It is easy to see that the drift  $\bar b$ defined by \eqref{e:7.4a} is a zero vector in $\R^6$
and the L\'evy process $X^\circ$ on $\R^d$ with L\'evy triplet $(0, 0, \bar \mu_\bullet)$
is
$$
X^\circ_t= \left(\bar X^{(1)}_t, \bar X^{(2)}_t, \bar X^{(3)}_t, \, 0,  \, 0, \, 0 \right),
$$
 where $\bar  X^{(i)}$ are
one-dimensional symmetric $\alpha_i$-stable processes with L\'evy measure $\kappa_i |z|^{-1-\alpha_i}dz$
for $1\leq i\leq 3$, independent to each other.
In a similar way as in previous examples, we can conclude from Theorems \ref{WT1} and \ref{T:7.3} that for any $T>0$, the rescaled random walk
$\left\{ \delta_{1/n} (S_{[nt]}); t\in [0, T] \right\}$ on $\mathbb U_4 (\Z )$
converges weakly in the Skorohod space $\D([0, T]; \R^6)$ to a purely discontinuous symmetric
L\'evy process $X^\bullet$ on
$(G_\bullet, \bullet)$
with L\'evy measure $\mu_\bullet$
as $n\to \infty$, which in the matrix coordinate system is given by
$$
\bar X^{(1)}_t,  \ \bar X^{(2)}_t, \  \bar X^{(3)}_t, \, \int_0^t \bar X^{(1)}_{s-} d\bar X^{(2)}_s,
\, \int_0^t \bar X^{(2)}_{s-} d\bar X^{(3)}_s,$$ and
$$
\int_0^t \bar{X}^{(1)}_{s_-}\bar{X}^{(2)}_{s_-}d\bar{X}^{(3)}_s +
\int_0^t\left( \int_{[0, r) } \bar X^{(1)}_{s-} d\bar X^{(2)}_s \right)
d \bar X^{(3)}_r .
$$

\medskip

In the above, if we  replace $\mu_2$ by

$$
\mu'_2((x_i)_1^6)=\frac{c_2}{ (1+ \sqrt{x_2^2+ x_3^2+x_5^2})^{3+\alpha_2}}
\otimes \1_{(0,0,0)}(x_1,x_4,x_6),
$$
the measure $\mu $ is again an $\mathcal{SM}$ measure on $\mathbb U_4 (\Z)$
and $\mu_t=t\delta_t(\mu)$ converges vaguely to the same $\bar \mu_\bullet$
of \eqref{e:7.20} as $t \to \infty$.
Thus for any $T>0$,
the rescaled random walk $\{ \delta_{1/n} (S_{[nt]}); t\in [0, T])$ on $\mathbb U_4 (\Z)$
converges weakly in the Skorohod space $\D([0, T]; \R^6)$ to the same purely discontinuous symmetric
 L\'evy process on $\mathbb U_4(\R)$.
\end{exa}

\begin{rem}
Since condition (R2) is   satisfied, the local limit theorem,
Theorem \ref{localCLT}, holds as well  for all the examples in this subsection.
\end{rem}

\section{Measures in $\mathcal{SM}(\Gamma)$ and their geometries} \label{sec-stab}

In this section, we define the set
 $\mathcal{SM}(\Gamma)$ of probability measures on $\Gamma$
 and the associated geometries. Here $\Gamma$ is a torsion free finitely generated nilpotent group although the basic definitions below make sense in a more general context.  In the next section, we show how measures in this set provide a wide variety of examples to which the limit theorems obtained earlier apply.

\subsection{Probability measures in $\mathcal{SM}$ and $\mathcal{SM}_1$}\label{S:8.1}

Consider a subgroup $H\subset \Gamma$. Because $\Gamma$ is nilpotent, $H$ is automatically finitely generated and we equip $H$  with a finite symmetric generating set $S$ and the associated word-length $|\cdot|_{S}$. Let $\alpha\in (0,2)$.
Let $\mathcal{SM}^{\alpha}_H(\Gamma)$ be the set of all symmetric probability measures $\nu$ on $\Gamma$ which are supported on $H$ and satisfy
$$\nu(g) \asymp   \frac{1}{(1+|g|_S)^\alpha V_{H,S}(|g|_S)}\1_H(g),$$
where
$V_{H,S}$ is the volume growth function of the pair $(H,S)$ and
the notation $\nu \asymp \mu$ indicates that there are constants $0<c\le C<\infty$, which  may depend on $\nu$ and $\mu$, so that 
$$
   c\mu(g)\le  \nu(g)\le C\mu(g) 
\quad \hbox{for every } g\in \Gamma. 
$$
   Note that since $H$ is a subgroup of the finitely generated nilpotent group, its
 volume growth is polynomial and there is an integer $d_H$ such that $V_{H,S}(k)\asymp k^{d_H}$, $k=1,2,\dots$.
 Note also that the set $\mathcal{SM}^\alpha_H(\Gamma)$ does not depend on the choice of the generating $S$ for $H$.
These symmetric probability measures are the basic building blocks of the set $\mathcal{SM}(\Gamma)$ which we now define.

\begin{defin}[$\mathcal{SM}(\Gamma)$] \label{def-Stab}The set $\mathcal{SM}(\Gamma)$ is the set of all finite convex combinations $\mu$ of probability measures belonging to the union
$$\bigcup_{\alpha\in (0,2)}\bigcup_{ H: \,\mbox{\small subgroup of } \Gamma} \mathcal{SM}^{\alpha}_H(\Gamma)$$
such that the  support of $\mu$ generates $\Gamma$.
\end{defin}

\begin{defin}[$\mathcal{SM}_1(\Gamma)$]\label{def-Stab1}The subset $\mathcal{SM}_1(\Gamma)$ of $\mathcal{SM}(\Gamma)$ is the set of all finite convex combinations $\mu$
of probability measures in
$$\bigcup_{\alpha\in (0,2)}\bigcup_{ h\in \Gamma} \mathcal{SM}^{\alpha}_{\langle h\rangle}(\Gamma)$$
such that the support of $\mu$ generates $\Gamma$. That is, $\mu\in \mathcal{SM}_1(\Gamma)$  if it is the finite convex combinations  of stable-like symmetric probability measures supported on a finite collection of subgroups of $\Gamma$,
$ \langle h_i \rangle  $,
$1\le i\le k$, with the property that  $\{h_i^{\pm 1}, 1\le i\le k\}  $ generates $\Gamma$.
\end{defin}

\begin{defin}[$\mathcal{SM}^\alpha(\Gamma)$, $\alpha\in (0,2)$] \label{def-Stabalpha}
For each $\alpha\in (0,2)$, the subset $\mathcal{SM}^\alpha(\Gamma)$ of $\mathcal{SM}(\Gamma)$ is the set of all finite convex combinations $\mu$ of probability measures in
$$\bigcup_{ H   :   \,\mbox{\small subgroup of } \Gamma} \mathcal{SM}^{\alpha}_{H}(\Gamma)$$ such that the support of $\mu$ generates $\Gamma$.
\end{defin}

So, any probability measure $\mu$ in $\mathcal{SM}(\Gamma)$ as the form
\begin{equation}\label{mu} \mu=\sum_{i=1}^k p_i \mu_{H_i,\alpha_i},\end{equation}
where $\alpha_i \in (0,2)$, $p_i> 0$, $\sum_{i=1}^kp_i=1$, each $H_i$ is a subgroup of $\Gamma$,  and $\mu_{H_i,\alpha_i}$ is a probability measure in $\mathcal{SM}^{\alpha_i}_{H_i}(\Gamma)$. In addition, $\Gamma= \langle  H_1,\dots,H_k \rangle $.
The  typical measures in $\mathcal{SM}_1(\Gamma) $ have the more explicit form
$$\mu(g)=\sum_1^k \sum_{m\in \mathbb Z}\frac{p_i c_{\alpha_i}}{(1+|m|)^{1+\alpha_i}}\1_{\{s_i^m\}}(g),$$
where $\alpha_i\in (0,2)$, $p_i > 0$, $\sum _1^kp_i=1$, the finite set $\{s^{\pm 1}_i:1\le i\le k\}$ is a generating set of $\Gamma$, and
$c_\alpha^{-1}= \sum_{ m \in \mathbb Z}\frac{1}{(1+|m|)^{1+\alpha}}$. There are more measures in $\mathcal{SM}_1(\Gamma)$ because the individual component of the convex combination above  do not have to be exactly
$ \sum_{m\in \mathbb Z}\frac{c_{\alpha_i}}{(1+|m|)^{1+\alpha_i}}\1_{\{s_i^m\}}(g)$, they only have to be $\asymp$-comparable to such a measure.
\begin{exa} \label{exa-typStabH}
On the Heisenberg group $\mathbb H_3(\mathbb Z)$ viewed as the group of matrix
\eqref{eq:3dimHB},
consider the measures  $$\mu_4((x_1,x_2,x_3))=  \frac{c_\alpha}{\left(1+\sqrt{x_1^2+x_2^2+|x_3-x_1x_2/2|}\right)^{4+\alpha_4}},$$
(this is $\mu$ from Example \ref{ex3-4}) and
$$\mu_i ((x_1,x_2,x_3))= \frac{c_\alpha \1_{H_i}((x_1,x_2,x_3))}{(1+ |x_i|)^{1+\alpha_i }}, \;i=1,2,3,$$ where $ H_i=\{(x_1,x_2,x_3):  x_j=0 \mbox{ if } j\neq i\}$,
with $\alpha_1,\alpha_2,\alpha_3,\alpha_4\in (0,2)$. The measure $\mu =\frac{1}{4} \sum_{ i= 1}^4\mu_i$ is a good example of a measure in $\mathcal{SM}(\mathbb H_3(\mathbb Z))$.  This is because the expression $\sqrt{x_1^2+x_2^2+|x_3-x_1x_2/2|}$ is constant under taking inverse and is comparable to the
word-length on $\mathbb H_3(\mathbb Z)$ (e.g., on the natural minimal symmetric generating set).
\end{exa}

\begin{exa} \label{exa-typStabH-2}
On the Heisenberg group $\mathbb H_3(\mathbb Z)$ viewed as the group of matrix
\eqref{eq:3dimHB},  for $i=1,2$, $\alpha_i\in (0,2)$, and $H_1=\{(x_1,0,x_3): x_1,x_3\in \mathbb Z\}$,
 $H_2=\{( 0,x_2,x_3): x_2,x_3\in \mathbb Z\}$,  consider the measures
 $$
 \mu_i((x_1,x_2,x_3))=  \frac{c_{\alpha_i}}{\left(1+\sqrt{x_i^2+x_3^2}\right)^{\alpha_i+ 2}}\1_{H_i}(x_1,x_2,x_3).
 $$
  The measure $\mu =\frac{1}{2}(\mu_1+\mu_2)$ is another good example of a measure in
  $\mathcal{SM}(\mathbb H_3(\mathbb Z))$.
 The measure in Example \ref{ex3-4} is also in $\mathcal{SM}(\mathbb H_3(\mathbb Z))$.
 \end{exa}

\subsection{Weight systems on $\Gamma$ associated to measures in $\mathcal{SM}(\Gamma)$} \label{sec-muw}

Let $\mu\in \mathcal{SM}(\Gamma)$ be given by (\ref{mu}). From the data defining $\mu$, we extract a long generating tuple
$$\Sigma=(\sigma_1,\dots \sigma_\ell)$$
by listing one representative of   $\{s,s^{-1}\}$ for each  $s\in S_i$, $1\le i\le k$, with repetition when the same $s,s^{-1}$ belongs to more than one set $S_i$. Thus, we can think of each $\sigma_j$ as carrying a label that tells us from which $S_i$ it comes.
Using this label we
give each $\sigma_j\in \Sigma$ the positive weight $w(\sigma_j)=1/\alpha_i$ if $\sigma_j$ comes from $S_i$. Now, consider $\Sigma$ as a finite alphabet and consider the set of all finite length formal commutators over $\Sigma\cup \Sigma^{-1}$ where $\Sigma^{-1}$ is the set of formal inverse letters. We can proceed inductively.
Elements of $\Sigma \cup \Sigma^{-1}$ are length $1$ commutators. After formal commutators of length at most $n$ have been defined, the formal commutators of length at most $n+1$ are all the formal expressions of the form $[\tau,\theta]$ where $\tau$ and $\theta$ are commutators of length $s$ and $t$ with $s+t\le n+1$.  Recall that each formal commutator $\sigma^{\pm 1}$ of length $1$ has a weight
$w(\sigma^{\pm 1})=1/\alpha_i$ if $\sigma$ comes originally from $S_i$. Extend the weight function $w$ to all formal commutators by setting
$w([\tau,\theta])=w(\tau)+w(\theta)$.

A priori, there are countably many formal commutators but because $\Gamma $ is nilpotent and we will ultimately consider only the formal commutators that are not trivial when evaluated in $\Gamma$, we only have to deal with finitely many formal commutators, whose lengths are  at most the nilpotent class of $\Gamma$.  We now use weighted formal commutators to define a   non-increasing equence of subgroups of $\Gamma$. Recall that by convention and abuse of notation, each letter $\sigma$ in $\Sigma$ is also a group element in $\Gamma$. The following definition is essentially from \cite{SCZ-nil} where further details can be found.
See Definition 1.4 and Proposition 1.5 in \cite{SCZ-nil}.
\begin{defin} For any $s\ge 0$, let $\Gamma^{\Sigma,w}_s$  be the subgroup of $\Gamma$ generated by the values in $\Gamma$
of all the formal commutators over the alphabet $\Sigma$ with weight at least $s$. By construction
$\Gamma^{\Sigma,w}_t\subseteq \Gamma^{\Sigma,w}_s$ if $s\le t$. Also, $[\Gamma^{\Sigma,w}_s,\Gamma^{\Sigma,w}_t]\subseteq \Gamma^{\Sigma,w}_{s+t}.$
\end{defin}
\begin{defin} There is a greatest $t$ such that $\Gamma^{\Sigma,w}_t=\Gamma$, call it  $w_1$.
By induction, having defined $w_j$, define $w_{j+1}$ to be the largest $t\in (w_j,\infty]$ such that $\Gamma^{\Sigma,w}_s=\Gamma^{\Sigma,w}_{t}$ for all $w_j<s\le t$. This defines a finite strictly increasing sequence
$$w_1<w_2<\dots< w_j<w_{j+1}<\dots <w_{j_*+1} =\infty$$
such that
$$
 \Gamma^{\Sigma,w}_{w_{j+1}} \subsetneq \Gamma^{\Sigma,w}_{w_{j}}, \;\Gamma^{\Sigma,w}_{s}=\Gamma^{\Sigma,w}_{w_j} \mbox{ for } s\in (w_{j-1},w_j], \quad  \Gamma^{\Sigma,w}_{s}=\{e\} \mbox{ for } s>w_{j_*}.
$$
By construction $\big[ \Gamma,\Gamma^{\Sigma,w}_{w_j} \big] \subset \Gamma^{\Sigma,w}_{w_{j+1}}.$
Call $A_{w_j}$  the abelian group
$$
A_{w_j}=\Gamma^{\Sigma,w}_{w_j}/\Gamma^{\Sigma,w}_{w_{j+1}}, \quad
1\le j\le j_*.
$$
\end{defin}

\begin{defin} Set     $\gamma_0 (\Sigma,w)=\sum_1^{j_*} w_j \mbox{Rank}(A_{w_j})$, 
where $\mbox{Rank}(A)$ denotes the torsion free rank of the finitely generated abelian group $A$.
\end{defin}

That the construction described above and  the definition of the   positive real $\gamma_0 (\Sigma,w)$
 is relevant to the study of random walks driven by measures in $\mathcal{SM}(\Gamma)$ is apparent from the following theorem from \cite{SCZ-nil,CKSWZ1}.

\begin{theo}[\cite{SCZ-nil,CKSWZ1}] \label{T:6.7}
Let $\Gamma$ be a finitely generated nilpotent group.
For any probability measure $\mu$ in $\mathcal{SM}(\Gamma)$ with associated data $(\Sigma,w)$ as above, there are constants $c(\mu)$ and $C(\mu)$ such that, for all $n$,
$$ c(\mu) n^{ - \gamma_0 (\Sigma,w)}\le \mu^{(n)}(e)\le C(\mu)n^{  -\gamma_0 (\Sigma,w)}.$$
\end{theo}

\subsection{Quasi-norms on $\Gamma$ associated with elements of $\mathcal{SM}(\Gamma)$}
The previous section associates  to any measure $\mu\in \mathcal{SM}(\Gamma)$ a weight system built on the $\ell$-tuple of group  elements
$\Sigma=(\sigma_1,\dots,\sigma_\ell)$ obtained by listing consecutively with possible repetitions all the elements of the sets $S_i$, $1\le i\le k$,
and the attached weight  $w(\sigma)=1/\alpha_i$ if $\sigma$ comes from $S_i$.  Recall that in this construction, we view $\Sigma$ as an abstract  alphabet.   This data allows us to construct a quasi-norm on the countable group $\Gamma$ based on the writing of any element $g$ of $
\Gamma$ as a word over the alphabet $\Sigma\cup \Sigma^{-1}$.  For any finite word $\word\in  \cup_0^\infty (\Sigma\cup \Sigma^{-1})^m$,
set
$$\mbox{deg}_\sigma(\word)=\mbox{number of times the letters $\sigma,\sigma^{-1}$ are used in $\word$}.
$$

The following definition is from \cite{SCZ-nil,CKSWZ1}.

\begin{defin} \label{def-quasinorm}
Given $\Gamma$, $\Sigma=(\sigma_1,\dots,\sigma_\ell)$ and weight $w$ as above, for each element $g\in G$, set
$$\|g\|_{\Sigma,w}=
\inf\left\{ \max_{\sigma\in \Sigma}\{ (\mbox{deg}_\sigma(\word))^{1/w(\sigma)}\}:\word\in \cup_0^\infty (\Sigma\cup \Sigma^{-1})^m,\quad  g=\word \mbox{ in } G\right\}.$$
By convention, $\|e\|_{\Sigma,w}=0$.
\end{defin}

\begin{rem} When $w(\sigma)=w_0$ for all $\sigma\in \Sigma$, the quasi-norm $\|\cdot\|_{\Sigma,w}$ satisfies
$$
\frac{1}{\ell} |g|_{\Sigma} \le \|g\|^{w_0}_{\Sigma,w}\le |g|_{\Sigma}
 \quad \hbox{for every }  g\in \Gamma,
 $$
where $|\cdot|_{\Sigma}$ denotes the usual word length of the finite symmetric generating set $\Sigma\cup \Sigma^{-1} \subset \Gamma$.
\end{rem}
\begin{rem} It may be worth noting that, in general, it is hard to compute or estimate $\|g\|_{\Sigma,w}$ for a given $g\in \Gamma$. The reference \cite{SCZ-nil} gives many results in this direction and these results will be useful in the sequel. This is related to the use of coordinate systems in so far as the question of estimating $g\in \Gamma$ becomes a precise question only when $g$ is given in terms of some parameter set, i.e., some sort of (possibly partial) coordinate system, see \cite[Theorem 2.10]{SCZ-nil}. To help the reader understand this comment, we suggest the following question: given a fixed $g\in \Gamma$, what is the behavior of $\|g^m\|_{\Sigma,w}$ as a function of $m$?  See \cite[Proposition 2.17 ]{SCZ-nil}.\end{rem}

\section{Adapted approximate group dilations}\label{S:9}}

\subsection{Searching for adapted dilations}\label{S:9.1}

The goal of this section is to associate to each probability measure  $\mu$
in $\mathcal{SM}(\Gamma)$ an adapted approximate dilation structure. This includes making the choice of an appropriate polynomial coordinate system for  the simply connected nilpotent Lie group $G=(\mathbb R^d,\cdot)$ in which $\Gamma$ embeds as a co-compact discrete subgroup.  The given measure $\mu$
determines
 uniquely certain features of the appropriate coordinate systems and associated approximate group dilations but not all. Among the feature that are determined uniquely (in this case, up to an arbitrary multiplicative positive constant) is a vector of non-decreasing weight values $b_j$, $1\le j\le d$, so that, in the chosen coordinate system $u=(u_i)_1^d\in\mathbb R^d$ for $G$, the appropriate approximate dilation structure is of the form $\delta_t(u)=(t^{b_i}u_i)_1^d$.
  Among the exponential coordinates of the first and second kind,
 the group structure $G_\bullet=(\mathbb R^d,\bullet)$  defined by
$$u\bullet u'=\lim_{t\to \infty} \delta_{1/t}(\delta_t(u)  \cdot  \delta_t(u'))$$
(understood up to isomorphisms) depends only on $\Gamma$ and $\mu$ and not on
 the particular choice of a coordinate system.
An interesting question is if this remains true beyond these exponential coordinate systems.

In the next subsections, we describe two key constructions: the construction of adapted exponential coordinates of the first kind and that of adapted coordinate of the second kinds. The essential difference between the two constructions is that, in the discussion of exponential coordinates of the first kind, we assume that the simply nilpotent Lie group $G$ in which $\Gamma$ sits as a
 co-compact subgroup is already given to us together with its Lie algebra and canonical exponential map. All we need to construct is an adapted linear basis of this Lie algebra based on the nature of the measure $\mu$. In the case of exponential coordinates of the second kind, we start from scratch with only the finitely generated torsion free nilpotent group $\Gamma$ carrying the measure $\mu$
and, following Malcev and Hall, we construct a ``discrete coordinate system'' for $\Gamma$ which is adapted to $\mu$ and, in turn, ``generates'' for us the simply connected Lie group $G$ and its adapted exponential coordinates of the second type. It is only a posteriori (and with some work) that one can check that certain features of these two constructions are identical.

\subsection{Exponential  coordinates of the first kind } \label{S:9.2}

This  subsection focuses on the situation when the torsion free finitely generated group $\Gamma$ is given to us as a co-compact discrete subgroup of a simply connected nilpotent Lie group $G$ with Lie algebra $\mathfrak g=(\mathbb R^d,[\cdot,\cdot])$ and the group $G$ is given in the (canonical) exponential coordinates of the first kind. This  identifies
the group $G$ with $(\mathbb R^d,\cdot)$ where the product $\cdot$ is given by the famous Campbell-Hausdorff formula
\eqref{e:2.3}.

In the next definition, we are given a probability measure $\mu\in \mathcal{SM}(\Gamma)$ and the associated data $\Sigma , w$ as in Section \ref{sec-muw} and we transfer the weight system to the Lie algebra $\mathfrak g$.
Observe that the tuple (use the same ordering as for $\Sigma$)
$$\Sigma_{\mathfrak g}=(\varsigma_i\in \mathfrak g: \exp(\varsigma_i)=\sigma_i\in \Sigma)=(\varsigma_1,\dots,\varsigma_\ell)$$
must be an algebraically generating set for $\mathfrak g$ in that this set together with all iterated brackets of its elements generates $\mathfrak g$ linearly. Indeed, because the exponential map is a global diffeomorphism between $\mathfrak g$ and $G$, if $\Sigma_{\mathfrak g}$ did not generate $\mathfrak g$, $\Gamma$ would be contained in a proper closed connected Lie subgroup
of $G$. This would contradict the fact that $\Gamma$ is co-compact in $G$.

We now trivially transfer the weight function $w: \Sigma\ra (0,\infty)$ to a function defined on $\Sigma_{\mathfrak g}$ by setting $w^\mathfrak g(\varsigma)=w(\sigma)$ if $\sigma=\exp(\varsigma)$. This   
 leads  to the definition of a weight system $w^\mathfrak g$ on the formal (Lie) commutators of the $\varsigma$'s in a way that is formally analogous to what we did on $\Gamma$.
\begin{defin}  Let $\mathfrak g^{\Sigma,w}_s$ be the Lie sub-algebra of $\mathfrak g$ generated by the evaluation in $\mathfrak g$ of all formal commutators of the $\varsigma\in \Sigma_{\mathfrak g}$ whose weight is at least $s$. By construction,
$$\mathfrak g^{\Sigma,w}_t\subseteq \mathfrak g^{\Sigma,w}_s \mbox{ if } s\le t$$
and
$$[\mathfrak g^{\Sigma,w}_s , \mathfrak g^{\Sigma,w}_t]\subseteq  \mathfrak g^{\Sigma,w}_{s+t}.$$
\end{defin}

\begin{defin} \label{filtration} There is a greatest $t$ such that $\mathfrak g^{\Sigma,w}_t=\mathfrak g$, call it $w^\mathfrak g_1$. By induction, having defined $w^\mathfrak g_j$,
there is a greatest $t\in (w^\mathfrak g_j, \infty]$ such that $\mathfrak g^{\Sigma,w}_s= \mathfrak g^{\Sigma,w}_t$ for all $s\in (w^\mathfrak g_j,t]$.  Call it $w^\mathfrak g_{j+1}$. Let $j^\mathfrak g_\star$ be the largest integer $j$ such that $w^\mathfrak g_j<\infty$ so that $w^\mathfrak g_{j^\mathfrak g_\star+1}=\infty$ and $ \mathfrak g^{\Sigma,w}_{w^\mathfrak g_{j^\mathfrak g_\star+1}}=\{0\}$.
This defines a finite strictly decreasing sequence of sub-Lie algebras
$$ \mathfrak g=\mathfrak g^{\Sigma,w}_{w^\mathfrak g_1}\supset \dots \supset \mathfrak g^{\Sigma,w}_{w^\mathfrak g_{j^\mathfrak g_\star}}\supset \{0\}$$
with the property that
$$
[ \mathfrak g, \mathfrak g^{\Sigma,w}_{w^\mathfrak g_j} ]\subseteq \mathfrak g^{\Sigma,w}_{w^\mathfrak g_{j+1}},
\quad j=1,\dots,j^\mathfrak g_\star.
$$
 \end{defin}

\begin{defin}[Adapted direct sum decomposition]  \label{def-directsum}We say that a direct sum decomposition of $\mathfrak g$,
$$\mathfrak g=\oplus_1^{j^\mathfrak g_\star} \mathfrak n_j,$$
is adapted to $(\Sigma,w)$ if, for all $j\in \{1,\dots,j^\mathfrak g_\star\}$,
$$ \mathfrak g^{\Sigma,w}_{w^\mathfrak g_j}= \oplus_{\ell=j}^{j^\mathfrak g_\star} \mathfrak n_\ell.$$
\end{defin}

\begin{rem} To construct an adapted   direct sum
decomposition,  start from the top and set $\mathfrak n_{j^\mathfrak g_\star}=\mathfrak g^{\Sigma,w}_{w^\mathfrak g_{j^\mathfrak g_\star}}$.
By descending induction, having constructed
$\mathfrak n_j,\dots,\mathfrak n_{j^\mathfrak g_{\star}}$ so that $\mathfrak g^{\Sigma,w}_{w^\mathfrak g_j}= \oplus_{\ell=j}^{j^\mathfrak g_\star} \mathfrak n_\ell$,
pick a linear complement of  $\mathfrak g^{\Sigma,w}_{w^\mathfrak g_{j}}$ inside $\mathfrak g^{\Sigma,w}_{w^\mathfrak g_{j-1}}$ and call it $\mathfrak n_{j-1}$.
\end{rem}

\begin{defin}[Approximate Lie dilation structure (first kind)]  \label{def-dilw}
Given a direct sum decomposition that is adapted to $(\Sigma,w)$, consider the group of invertible linear maps
$$
\delta_t: \mathfrak g\ra \mathfrak g,\quad \;t>0,
$$
 define by
 $$\delta_t (v)=t^{\mathfrak w^\mathfrak g_j} v  \quad \mbox{ for all } v\in \mathfrak n_j, 1\le j\le j_\star^\mathfrak g.$$
 Let $\boldsymbol \eps=(\eps_i)_1^d$ be a linear basis of $\mathbb R^d$ adapted to the direct sum
 $\mathfrak g=\oplus_1^{j^\mathfrak g_\star} \mathfrak n_j$, let $u=(u_i)_1^d$ be the corresponding coordinate system, and let
 $$
 b_i= \mathfrak w^\mathfrak g_j  \quad \mbox{ if } \eps_i\in   \mathfrak n_j
 $$
 so that $$\delta_t(u)=(t^{b_i}u_i)_1^d.$$
 \end{defin}

We shall see below in Proposition \ref{pro-twoweights} and Corollary \ref{cor-twoweights} that the important quantity
   $\gamma_0 (\Sigma, w)$ 
is given in terms of the sequences $(b_i)_1^d$ and $(\mathfrak w^{\mathfrak g}_j)_1^{j^\mathfrak g_\star }$ by
$$
 \gamma_0 (\Sigma,w)  =\sum_1^d b_i=\sum_1^{j^\mathfrak g_\star}\mathfrak w^{\mathfrak g}_j \dim(\mathfrak n_j).
$$

\begin{pro} The maps $(\delta_t)_{t>0}$ defined above form an approximate Lie dilation structure on $\mathfrak g$.
In any exponential coordinate system of the first kind adapted to the direct sum decomposition $\mathfrak g=\oplus_1^{j^\mathfrak g_\star} \mathfrak n_j,$  $(\delta_t)_{t>0}$  is  a straight approximate  group dilation structure on $G$.\end{pro}
\begin{proof} By linearity, it suffices to prove that for any $v_i\in \mathfrak n_{j_i}$, $i=1,2$,
$$\delta^{-1}_t ([\delta_t (v_1),\delta_t(v_2)])$$ has a limit when $t$ tends to infinity. By construction, $\delta_t(v_i)= t^{w^\mathfrak g_{j_i}} v_i$
and $[v_1,v_2]\in \mathfrak g^{\Sigma,w}_{w^\mathfrak g_{N}}= \oplus_{\ell\ge N}\mathfrak n_\ell$ where $w^\mathfrak g_N\ge w^\mathfrak g_{j_1}+w^\mathfrak g_{j_2}$,
 namely, $$[v_1,v_2] =\sum_{\ell=N}^{j^\mathfrak g_{\star} } f_\ell, \quad
 f_\ell\in \mathfrak n_\ell.$$
It follows that
$$\delta^{-1}_t ([\delta_t (v_1),\delta_t(v_2)])= \sum_{\ell=N}^{j^\mathfrak g_{\star}} t^{w^\mathfrak g_{i_1}+w^\mathfrak g_{i_2}-w^\mathfrak g_{\ell}} f_\ell .
$$
The limit of this expression when $t$ tends to infinity exists because  $w^\mathfrak g_{i_1}+w^\mathfrak g_{i_2}\le w^\mathfrak g_N\le w^\mathfrak g_\ell$ for all $\ell\ge N$.
If $w^\mathfrak g_N>w^\mathfrak g_{i_1}+w^\mathfrak g_{i_2}$, the limit is $0$. If $w^\mathfrak g_N=w^\mathfrak g_{i_1}+w^\mathfrak g_{i_2}$, the limit is $f_N$.
\end{proof}

\subsection{Building adapted exponential coordinates of the second kinds from $\Gamma$ and $\mu$}\label{S:9.3}}

In this  subsection, we start with the  given discrete torsion free nilpotent group $\Gamma$ (described, perhaps, by generators  and relations, or as a subgroup of a bigger group) and we explain how to construct the Lie group $G$ using well-known ideas related
 to coordinate of the second  kind. This is done in \cite{Mal,PHall} and we refer the reader to the treatment in \cite[Theorem 4.9, Section 4.3]{CMZ}.

\subsubsection*{Hall-Malcev coordinates}

Theorem 4.9 of \cite{CMZ} asserts that, for any finitely generated torsion free nilpotent group $\Gamma$, any descending central series (this means that $\Gamma_i/\Gamma_{i+1}$ is central in $\Gamma/\Gamma_{i+1}$ for each $1\le i\le n$)
$$\Gamma=\Gamma_1\rhd \Gamma_2\rhd \dots \rhd \Gamma_n\rhd \Gamma_{n+1}=\{e\}$$
with $\Gamma_i/\Gamma_{i+1}$ infinite cyclic, and any sequence of elements $\tau_i\in G$ such that $\Gamma_i= \langle \Gamma_{i+1},\tau_i\rangle$,
each element $\gamma\in \Gamma$ can be written uniquely
$$
\gamma= \tau_1^{u_1}\cdot \tau_2^{u_2}\cdot \dots \cdot \tau_n^{u_n},
\quad u_1,u_2,\dots,u_n\in \mathbb Z.
$$
 Moreover, for any $k\in \mathbb Z$ and any $\gamma'= \tau_1^{u'_1}\cdot \tau_2^{u'_2}\cdot \dots \cdot \tau_n^{u'_n}$,
 $$\gamma^k=  \tau_1^{g_1(u,k)}\cdot \tau_2^{g_2(u,k)}\cdot \dots \cdot \tau_n^{g_n(u,k)}$$
 and
 $$
 \gamma \cdot \gamma'=   \tau_1^{f_1(u,u')}\cdot \tau_2^{f_2(u,u')}\cdot \dots \cdot \tau_n^{f_n(u,u')} ,
 $$
where $u=(u_1,\dots,u_n)$, $u'=(u'_1,\dots,u'_n)$, and $f_i,g_i$, $1\le i\le n$ are polynomials with rational coefficients in there respective variables.

 Furthermore (\cite[Theorems 4.11-4.12 ]{CMZ}), by interpreting these coordinates in $\mathbb R^n$ instead of $\mathbb Z^n$, one obtains a simply
 connected nilpotent Lie group  of which $\Gamma$ is a discrete co-compact subgroup.

 For our present purpose, the task is to produce a descending central series
$$\Gamma=\Gamma_1\rhd \Gamma_2\rhd \dots \rhd \Gamma_n\rhd \Gamma_{n+1}=\{e\}$$
with $\Gamma_i/\Gamma_{i+1}$ infinite cyclic, which is adapted to the measure $\mu$.   Using the sequence $\Gamma^{\Sigma,w}_{w_j}$, $1\le j\le j_*$ is a good first guess. If each quotient $A_{w_j}=\Gamma^{\Sigma,w}_{w_j}/\Gamma^{\Sigma,w}_{w_{j+1}}$ is free abelian (i.e., has no torsion), then we  can produce a descending central series
$$\Gamma=\Gamma_1\rhd \Gamma_2\rhd \dots \rhd \Gamma_n\rhd \Gamma_{n+1}=\{e\}$$
which refines the sequence $\Gamma^{\Sigma,w}_{w_j}$, $1\le j\le j_*$, and has  $\Gamma_i/\Gamma_{i+1}$ infinite cyclic. In addition, we can find a sequence of elements $\tau_i\in \Gamma$, each of which is a commutator of the elements in $\Sigma$, such that $\Gamma_i= \langle G_{i+1},\tau_i\rangle$ and such that $$w(\tau_i)=w_j \mbox{ if and only if } \Gamma^{\Sigma,w}_{w_{j+1}}\supset \Gamma_i \supseteq \Gamma^{\Sigma,w}_{w_j}.$$

The problem we face is that it is NOT always the case that the groups $A_{w_j}$ are torsion free (even in the simplest of all cases when $\Gamma=\mathbb Z$!).

 \subsubsection*{Modified weight system on $\Gamma$}  Given a measure $\mu\in \mathcal{SM}(\Gamma)$ as in (\ref{mu}), we defined in Section \ref{sec-muw} a generating set $\Sigma =(\sigma_1,\dots,\sigma_\ell)$ and a weight system $w$ on formal commutators which generates the
 descending central sequence of subgroups $\Gamma^{\Sigma,w}_{w_j}$, $1\le j\le j_*$.

 Consider the finite set of all formal commutators over the alphabet $\Sigma$ which are not trivial in $\Gamma$ and organize that finite set as a long tuple
 $\Sigma_{\mbox{\tiny com}}=(c_1,\dots,c_L)$.  Let $\underline{\Sigma}_{\mbox{\tiny com}}=(\underline{c}_1,\dots,\underline{c}_L)$ be the evaluation of $\Sigma_{\mbox{\tiny com}}$ in $\Gamma$.

 Let us introduce a (modified) weight function, $\underline{w}$, on $\underline{\Sigma}_{\mbox{\tiny com}}$ by setting, for each $\underline{c}$ appearing in the tuple $\underline{\Sigma}_{\mbox{\tiny com}}$,
 $$
 \underline{w}(\underline{c})=\max \left\{ w_j: \exists \, m\in \mathbb{N},\;\underline{c}^m\in \Gamma^{\Sigma,w}_{w_j}, 1\le j\le j_*
 \right\}.
 $$
 For commutators whose evaluation in $\Gamma$ is trivial, we can set $\underline{w}(c)=\infty$.
 Following \cite[Section 2.2]{SCZ-nil}, we set
 $$\mbox{core}(\Sigma,w)=\{ \sigma_i:  \underline{w}(\sigma_i)=w(\sigma_i), 1\le i\le \ell\}.$$

 The function $\underline{w}$ is no less than $w$ and has the property that, if $c=[c_1,c_2]$ is nontrivial in $\Gamma$ then
 $$\underline{w}(\underline{c})\ge \underline{w}(\underline{c}_1)+ \underline{w}(\underline{c}_2) .$$
 It follows that the induced weight of a formal commutator  $\mathbf c $ over the alphabet $\underline{\Sigma}_{\mbox{\tiny com}}$ whose evaluation in $\Gamma$ is not  trivial is actually equal to  the $\underline{w}$ weight of the same commutator view as an element of $\underline{\Sigma}_{\mbox{\tiny com}}$. Moreover, $\mbox{core}(\Sigma_{\mbox{\tiny com}},\underline{w})=\Sigma_{\mbox{\tiny com}}.$

 \begin{defin} For any $s\ge 0$, let $\Gamma^{\mbox{\tiny com}}_s$  be the subgroup of $\Gamma$ generated by the values in $\Gamma$
of all the formal commutators over the alphabet $\Sigma$ with $ \underline{w}$-weight at least $s$. By construction
$\Gamma^{\mbox{\tiny com}}_t\subseteq \Gamma^{\mbox{\tiny com}}_s$ if $s\le t$. Also, $[\Gamma^{\mbox{\tiny com}}_s,\Gamma^{\mbox{\tiny com}}_t]\subseteq \Gamma^{\mbox{\tiny com}}_{s+t}.$
\end{defin}

\begin{defin} There is a greatest $t$ such that $\Gamma^{\mbox{\tiny com}}_t=\Gamma$, call it  $\underline{w}_1$.
By induction, having defined $\underline{w}_j$, define $\underline{w}_{j+1}$ to be the largest $t\in (\underline{w}_j,\infty]$ such that $\Gamma^{\mbox{\tiny com}}_t= \Gamma^{\mbox{\tiny com}}_s$ for all $\underline{w}_j<s\le t$. This defines a finite strictly increasing sequence
$$0<\underline{w}_1<\underline{w}_2<\dots< \underline{w}_j<\underline{w}_{j+1}<\dots <\underline{w}_{j^{\mbox{\tiny com}}_{*}+1} =\infty$$
such that
$$
\Gamma^{\mbox{\tiny com}} _{\underline{w}_{j+1}} \subset \Gamma^{\mbox{\tiny com}}_{\underline{w}_{j}},  \ \; \Gamma^{\mbox{\tiny com}}_{s}=\Gamma^{\mbox{\tiny com}}_{\underline{w}_j}
\ \mbox{ for } s\in (\underline{w}_{j-1},\underline{w}_j], \quad
\Gamma^{\mbox{\tiny com}}_{s}=\{e\}
\ \mbox{ for } s>\underline{w}_{j^{\mbox{\tiny com}}_*}.
$$
By construction $[\Gamma,\Gamma^{\mbox{\tiny com}}_{\underline{w}_j}]\subset \Gamma^{\mbox{\tiny com}}_{\underline{w}_{j+1}}.$
Call $A^{\mbox{\tiny com}}_{\underline{w}_j}$  the abelian group
$$
A^{\mbox{\tiny com}}_{\underline{w}_j}=\Gamma^{\mbox{\tiny com}}_{\underline{w}_j}/\Gamma^{\mbox{\tiny com}}_{\underline{w}_{j+1}},
\quad  1\le j\le j^{\mbox{\tiny com}}_*.
$$
\end{defin}
 The following lemma follows immediately from the construction outlined above.

 \begin{lem} \label{lem-tau}The groups $A^{\mbox{\tiny \rm com}}_{\underline{w}_j},  \;1\le j\le j^{\mbox{\tiny \rm com}}_*$ are free abelian and each is generated by a finite subset of the commutators $\underline{c} \in \Sigma_{\mbox{\tiny \rm com}}$. Consequently, there exists a   descending central series
 $$\Gamma=\Gamma_1\rhd \Gamma_2\rhd \dots \rhd \Gamma_d\rhd \Gamma_{d+1}=\{e\}$$
 refining  the descending central series
 $$ \Gamma=\Gamma^{\mbox{\tiny \rm  com}} _{\underline{w}_{1}}  \rhd \Gamma^{\mbox{\tiny \rm com}} _{\underline{w}_{2}}  \rhd \dots \rhd \Gamma^{\mbox{\tiny\rm  com}} _{\underline{w}_{j^{\mbox{\tiny com}}_*}}  \rhd \Gamma=\Gamma^{\mbox{\tiny \rm  com}} _{\underline{w}_{j_*^{\mbox{\tiny \rm com}}+1}}=\{e\}$$
 and a sequence $\tau_i=\underline{c}_{\ell_i}$, $1\le j\le d$, in $\Sigma_{\mbox{\tiny \rm com}}$ such that
 $\Gamma  _{i}/  \Gamma_{i+1} $ is an infinite cyclic,
 $ \Gamma _i=\langle \tau_i ,\Gamma_{i+1}\rangle$, $1\le i\le d$, and
 $$\Gamma^{\mbox{\tiny \rm  com}} _{\underline{w}_{j}}= \langle \tau_i: \underline{w}(\tau_i)\ge \underline{w}_j\rangle.$$
 \end{lem}
 Because of this lemma, it is clear that \cite[Theorems 4.9, 4.11,4.12]{CMZ} apply and provide a set of coordinate of the second kind
 $$
 \Gamma=\left\{\gamma=\tau_1^{u_1}\cdot \tau_2^{u_2}\cdot \dots \cdot \tau_d^{u_d}, \;u_1,u_2,\dots,u_d\in \mathbb Z
 \right\}$$
 for $\Gamma$, as well as an embedding of $\Gamma$ as a co-compact discrete subgroup a simply connected Lie group $G$
 \begin{align*} G =&\{g=\tau_1^{x_1}\cdot \tau_2^{x_2}\cdot \dots \cdot \tau_d^{x_d}, \;x_1,x_2,\dots,x_d\in \mathbb R\}\\
  =& \{g=\exp({x_1}\zeta_1)\cdot \exp({x_2}\zeta_2)\cdot \dots \cdot \exp({x_d}\zeta_d), \;x_1,x_2,\dots,x_d\in \mathbb R\},\end{align*}
 where $\zeta_i=\log \tau_i \in \mathfrak g$.

 \begin{defin}[Approximate group
 	dilation structure (second kind)]  \label{def-dilw2}
 In the exponential coordinate system of the second kind  $(x_i)_1^d$  introduced above, consider the group of straight dilations
$$\delta_t:\mathbb R^d\to \mathbb R^d,\quad t>0, \quad  x\mapsto \delta_t (x)=(t^{\underline{w}_i}x_i)_1^d.$$  \end{defin}

\begin{pro} The maps $(\delta_t)_{t>0}$ defined above form an approximate group dilation structure on $G=(\mathbb R^d,\cdot)$.
\end{pro}

  By the same token, we obtain an associated coordinate system of the first kind
$$(y_1,\dots,y_d)\mapsto \exp\left( \sum _1^d y_i \zeta_i\right),$$ which is compatible with the weight $w^\mathfrak g$ introduced earlier and such that $b_i=w^\mathfrak g(\zeta_i)=\underline{w}(\tau_i)$, $1\le i\le d$.  The straight dilation groups we introduced in these two  distinct coordinate systems  have the same exponents $b_i$ in their respective coordinate systems. Viewed as maps from $G$ to $G$, they are clearly different in general even so we use the same notation $\delta_t$ in both cases;
see Example \ref{E:3.4} for such an example where the matrix coordinate system is an exponential coordinate system of second kind.
 This is because there are really no good reasons to consider both coordinate systems at the same time, except to understand that this parallel constructions yield compatible results at the end.

 \subsection{Relations between the filtrations associated with $w$, $w^\mathfrak g$ and  \underline{$w$}}\label{S:9.4}

Although there are great similitudes in the construction of the (discrete group) filtrations $\Gamma^{\Sigma,w}_{w_j}$, $1\le j\le j_*$,
and $\Gamma^{\mbox{\tiny com}}_{\underline{w}_j}$, $1\le j\le j^{\mbox{\tiny com}}_*$, of the group $\Gamma$, and
the (Lie algebra) filtration $\mathfrak g^{\Sigma,w}_{\mathfrak w^\mathfrak g_j}$, $1\le j\le j^\mathfrak g_\star$ of $\mathfrak g$, there are also differences.

We start with a comparison of the coordinates of the first and second kind in this context. It is not hard to see that the definitions of the sequences
$w^\mathfrak g_j$, $1\le j\le j^\mathfrak g_\star$, and $\underline{w}_i$, $1\le i\le j^{\mbox{\tiny com}}_*$, and the above remark concerning the relations between group and Lie algebra commutators, imply that these sequences of weights are actually equal, that is,
\begin{equation}\label{equalw}j^\mathfrak g_\star=j^{\mbox{\tiny com}}_*\;\mbox{ and }\quad w^\mathfrak g_j=\underline{w}_j, \quad 1\le j\le j^\mathfrak g_\star.\end{equation}

More generally, each (discrete group) formal commutator $\boldsymbol \tau$ on the alphabet $\Sigma=(\sigma_1,\dots,\sigma_\ell)$ corresponds in an obvious formal way to a formal Lie commutator  $\boldsymbol \theta$ on the alphabet  $\Sigma^\mathfrak g=(\zeta_1,\dots,\zeta_\ell)$
in such a way that the Campbell-Hausdorff formula provides
a formal equality
\begin{equation}\label{formalCH}
 \boldsymbol \tau =\exp(\boldsymbol \zeta)=\exp ( \boldsymbol \theta +\mathbf R_{\boldsymbol \tau}),
\end{equation}
where $\mathbf R_{\boldsymbol \tau}$ is a formal series of Lie commutators with $w^\mathfrak g$-weights strictly larger than $\underline{w}(\boldsymbol \tau)=w^\mathfrak g(\boldsymbol \theta)$.
The concrete meaning of this formal identity in the present context is that it is an equality when evaluated  over any pair $\Gamma\subset G$
where $G$ is a simply connected nilpotent Lie group with algebra $\mathfrak g$, with the formal series $\mathbf R_{\boldsymbol \sigma}$ reducing to a finite sum. Obviously, the evaluation $\theta$ of $\boldsymbol \theta$ in $\mathfrak g$ belongs to $\mathfrak g^{\Sigma,w}_{w^\mathfrak g(\boldsymbol\theta)}$. It follows that the evaluation $\zeta=\theta+R_\sigma$ of $\boldsymbol \zeta$ in $\mathfrak g$ also belongs to $\mathfrak g^{\Sigma,w}_{w^\mathfrak g(\boldsymbol\theta)}$.

\subsubsection*{Two choices of exponential coordinate of the first kind} From the discussion above, it becomes clear that there are at least two very natural
exponential coordinate systems of the first kind associated with the sequence  $(\tau_i)_1^d$  of elements of $\Gamma$ given by Lemma \ref{lem-tau}.

\subsubsection*{Choice 1: Lie commutators} Each $\tau_i$ is a commutator built on $\Sigma=(\sigma_i)_1^\ell$. Let $\theta_i$ be the Lie commutator over $\Sigma^\mathfrak g=(\varsigma_i)_1^\ell$ that corresponds formally to $\tau_i$. Here $\sigma_i=\exp(\varsigma_i)$ as before and the last sentence means that   $\theta=[\varsigma,\varsigma']$ if $\tau=[\sigma,\sigma']$ with $\sigma=\exp(\varsigma),\sigma'=\exp(\varsigma')$. By (\ref{formalCH}) $\underline{w}(\tau_i)=w^\mathfrak g(\theta_i)$, and the final subsequence of $(\theta_i)_1^n$ corresponding to those $i$ such that
$\underline{w}(\tau_i)\ge w^\mathfrak g_j=\underline{w}_j$ is a linear basis of $ \mathfrak g^{\Sigma,w}_{w^\mathfrak g_j}$. In particular,
$$\underline{\mathfrak n}_j=\left\{ \zeta\in \mathfrak g: \zeta=\sum_{i: \underline{w}(\tau_i)=w^\mathfrak g_j} z_i \theta_i , \;z_i\in \mathbb R \right\},\quad  1\le j\le j^\mathfrak g_\star$$
provides an adapted direct sum decomposition of $\mathfrak g$ in the sense of Definition \ref{def-directsum}.

\subsubsection*{Choice 2:  Logarithms of group commutators} Each $\tau_i$ can be written uniquely as $\tau_i=\exp(\zeta_i)$,  where $\zeta_i$ and $\theta_i$ are related via (\ref{formalCH}). It follows that the final subsequence of $(\zeta_i)_1^n$ corresponding to those $i$ such that
$\underline{w}(\tau_i)\ge w^\mathfrak g_j=\underline{w}_j$ is (also) a linear basis of $ \mathfrak g^{\Sigma,w}_{w^\mathfrak g_j}$. In particular,
$$\underline{\mathfrak n}'_j=\left\{ \zeta\in \mathfrak g: \zeta=\sum_{i: \underline{w}(\tau_i)=w^\mathfrak g_j} z_i \zeta_i , \;z_i\in \mathbb R \right\},\quad  1\le j\le j^\mathfrak g_\star$$
provides an adapted direct sum decomposition of $\mathfrak g$ in the sense of Definition \ref{def-directsum}.

From the description of these two related coordinate systems, it follows that, for each $j\in \{1,\dots, j^{\mbox{\tiny com}}_*=j^\mathfrak g_\star\}$, the group $\Gamma^{\mbox{\tiny com}}_{\underline{w}_j}$ is a co-compact discrete subgroup of the Lie group
$\exp(\mathfrak g^{\Sigma,w}_{w^\mathfrak g_j}) $ (recall that ${\underline{w}_j}=w^\mathfrak g_j$). Note that, by definition, any two exponential coordinate systems are always related by a linear change of basis in $\mathfrak g$. In the present case, these linear changes of coordinates have an obvious triangular form with unit diagonal and they respect the increasing filtration $ \mathfrak g^{\Sigma,w}_{w^\mathfrak g_j}$, $1\le j\le j^\mathfrak g_\star$.

The following proposition records the relations between the objects related  to the original weight system $w$ on $\Gamma$ and those related to the Lie algebra weight $w^\mathfrak g$. The proof follows classical arguments developed in \cite{Mal}, see also \cite[Appendix]{GrTa} and \cite{SCZ-nil,CMZ}. It is omitted.

\begin{pro} \label{pro-twoweights}
The finite sequence of weight-values $w^\mathfrak g_j$, $1\le j\le j^\mathfrak g_{j_\star}$, is a subsequence of the
increasing finite sequence of weight-values $w_j$, $1\le j\le j_*$, and $w^{\mathfrak g}_{j_\star}=w_{j_*}$.  If $w^{\mathfrak g}_{i-1} <w_j\le  w^{\mathfrak g}_i$  for some
$1\le i \le j \le j_*$, then
$$\Gamma^{\Sigma, w}_j\subset \exp(\mathfrak g^{\Sigma,w}_{\mathfrak w^\mathfrak g_i})$$
and the quotient
$$  \exp(\mathfrak g^{\Sigma,w}_{\mathfrak w^\mathfrak g_i})/ \Gamma^{\Sigma, w}_{w_j}$$
is compact.
If $j\in \{1,\dots, j_*\}$ is such that the value $w_j$ does not appear in $(w^\mathfrak g_i)_1^{j^\mathfrak g_{\star}}$, then
$$\Gamma^{\Sigma,w}_{w_{j+1}}/\Gamma^{\Sigma,w}_{w_j} \mbox{ is a finite abelian group}.$$
\end{pro}
The following is an immediate corollary.

\begin{cor} \label{cor-twoweights}
  $\gamma_0 (\Sigma,w) 
=\sum_1^{j_*}w_j \mbox{\em Rank}(\Gamma^{\Sigma,w}_{w_j}/ \Gamma^{\Sigma,w}_{w_{j+1}})=\sum_1^{j^\mathfrak g_\star}\mathfrak w^{\mathfrak g}_j \dim(\mathfrak n_j).$
\end{cor}

\subsubsection*{An associated exponential coordinate system of the second kind}  By the Hall-Malcev construction reviewed in Section \ref{S:9.3}, the sequence $(\tau_i)_1^d$  of elements of $\Gamma$ given by Lemma \ref{lem-tau} and the sequence of their logarithm $(\zeta_i)_1^d$ in $\mathfrak g$ give us an exponential
coordinate system of the second kind in which an element $g$ of the group $G$ is written
$$
g=\prod_1^d \exp( y_i \zeta_i),\quad y=(y_i)_1^d\in \mathbb R^d.
$$
Recall that we also have exponential coordinates  $(x)_1^d\in \mathbb R^d $ of the first kind such that
$$g=\exp\left(\sum x_i\zeta_i\right).$$
By using the Campbell-Hausdorf formula, we obtain a polynomial map
$$x= M(y) =(M_i(y))_1^d  \ \mbox{ such that } \  g=\prod_1^d \exp( y_i \zeta_i)=\exp\left(\sum x_i\zeta_i\right)$$
and this map has a specific  triangular structure which can be described as follows.  For a multi-index of length $q$,  $I=(i_1,\dots,i_q)\in\{1,\dots,d\}^q$,
set $\underline{w}_I=\sum_1^q w_{i_j}$.  We say that a polynomial $p$ in the coordinate $(y_i)_1^d$ has weight at most $w$ if it can be written as a linear combination of  $y^I=y_{i_1}\dots y_{i_q}$   with $\underline{w}_I\le w.$
Then the map $M$ has the form
$$ M_i(y)=y_i +m_{i}(y), $$
where $m_i$ is a polynomial of  weight at most $\underline{w}_i$ with no linear terms.

Let us use the notation
$$\delta_t: \mathbb R^d\to \mathbb R^d,  \  u=(u_i)_1^d\mapsto \delta_t(u)=(t^{b_i}u_i)_1^d, \quad b_i=\underline{w}_i, \; t>0,
$$
and note that we can use these dilations in the $x$ coordinate system  as a well as in the $y$ coordinate system discussed above. We find that
$$\delta_{1/t}\circ M \circ \delta_t (y)= y_i + t^{-b_i} m_i(\delta_t (y)).$$
Because $m_i$ has weight at most $b_i$, this expression has a limit when $t$ tends to infinity which is of the form
$$y_i+m_i^\infty(y),$$
where $m^\infty_i$ is a linear combination of terms of weight exactly $b_i$.  This defines a polynomial map
$$M^\infty:\mathbb R^d\to \mathbb R^d,$$
which is a group isomorphism between the limit groups $G^1_\bullet $ (obtained by using the approximate group dilations $\delta_t$ in the exponential coordinates of the first kind) and the group $G^2_\bullet$ (obtained by using the approximate group dilations $\delta_t$ in the exponential coordinates of the second kind).

\subsection{More choices of coordinate systems} \label{S:9.5}

There are many more possible choices of exponential coordinates of first and second kind that suit our needs.  The key structure that must be preserved for our purpose is the filtration $\mathfrak g^{\Sigma,w}_j$, $1\le j\le j^{\mathfrak g}_\star$, of the Lie algebra $\mathfrak g$ which is canonically associate to $\Sigma,w$.  After that, a number of choices have to be made, the first of which is the choice of  the direct sum $\mathfrak g=\oplus_{j=1}^{j^{\mathfrak g}_\star} \mathfrak n_j$ so that
$$\mathfrak g^{\Sigma,w}_j= \oplus_{i\ge j} \mathfrak n_j.$$
One then need to pick  an adapted linear basis $\boldsymbol \eps =(\eps_i)_1^d$. Any such   
choice gives both an exponential coordinate system of the first kind
$$
g=\exp\left(\sum_1^d x_i \eps_i\right), \quad x=(x_i)_1^d\in \mathbb R^d,
$$
and
an exponential coordinate system of the second kind
$$
g=\prod_{i=1}^d \exp\left(y_i \eps_i\right), \quad y=(y_i)_1^d\in \mathbb R^d.
$$
Each of these choices of coordinates, call it $(u_1,\dots,u_d)\in \mathbb R^d$, comes with its own straight approximate group dilations
\begin{equation}\label{def-delta}
\delta_t: \mathbb R^d\to \mathbb R^d,
 \  u=(u_i)_1^d\mapsto \delta_t(u)=(t^{b_i}u_i)_1^d, \quad b_i=\underline{w}_i,\;t>0.
 \end{equation}

Everything that has been said above for the special case $\eps_i=\zeta_i$ applies as well to these other choices (including the properties of the maps $M$ and $M^\infty$). The choice $\boldsymbol \eps=\boldsymbol \zeta$ is justified mostly by the fact that, in that coordinate system, the discrete group $\Gamma$ is represented as a set as $\mathbb Z^d\subset  \mathbb R^d$. This is not the case in most other coordinate systems. If one remains in the class of exponential coordinate systems of the first type, moving from one such system to another is captured by a linear change of coordinate in $\mathbb R^d=\mathfrak g$. If one  move from  a system of the first kind to one of the second kind or between   systems of the second kind, the maps capturing the changes of coordinates are polynomial maps with a special structure reflecting the preservation of the filtration $\mathfrak g^{\Sigma,w}_j$, $1\le j\le j^{\mathfrak g}_\star$, of the Lie algebra $\mathfrak g$ (the best way to think of a change of coordinates involving at least one system of the second kind is to pass through the associated system of the first kind: this step is described by the map $M$ above).

\subsection{Comparison of the quasi-norms on $\Gamma, G$ and $G_\bullet$}\label{S:9.6}

Consider the exponential coordinate systems of first and second type associated with a basis $\boldsymbol \eps=(\eps_i)_1^d$ adapted to the filtration $\mathfrak g^{\Sigma,w}_j$ of the Lie algebra $\mathfrak g$ as considered in the preceding subsection. It comes with a family of approximate dilations given by \eqref{def-delta}.
On $\mathbb R^d$, consider the usual Euclidean norm $\|\cdot\|_2$ and the quasi-norm
$$
N_w(z)=\max_{1\le i\le d}\{|z_i|^{1/b_i}\},
\quad  b_i=\underline{w}_i=w^\mathfrak g_i,
$$
and note that, for all $z\in \mathbb R^d$,  $N_w(\delta_t(z))= tN_w(z)$.
The structure of the change of coordinate map $M$ between exponential coordinates of the first ($x=(x_i)_1^d$) and second ($y=(y_i)_1^d)$ type  shows that there are constants $0<c\le C <\infty$ such that, if $g=\exp(\sum_1^d x_i \zeta_i)=\prod_1^d\exp(y_i \zeta_i)$ then
$$cN_w(x)\le N_w(y)\le CN_w(x).$$

\begin{lem}  Referring to the above setup and notation,
there is a constant $C_*$
such that for any  $R\ge 1$ and  any $\zeta_i\in \mathfrak g$  with
$N_w(\zeta_i)\le R,\;i=1,2,$ we have
$\exp(\zeta_1)\exp(\zeta_2)=\exp(\zeta) $
with  $$N_w(\zeta)\le  C_* R.$$
\end{lem}
\begin{proof} This follows  from the Campbell-Hausdorff formula because of the properties of the direct sum decomposition along the subspaces $\mathfrak n_j$ and its relation to the weight system $w.$ Note that this is not correct in general for small $R$. This reflects the fact that the coordinate system and the quasi-norm $N_w$ have been chosen to capture the large scale geometry of the situation.
\end{proof}

The following proposition is one of the important keys to the results presented in this article. It relates the geometry of the discrete group $\Gamma$ equipped with the quasi-norm $\|\cdot\|_{\Sigma,w}$ (Definition \ref{def-quasinorm}) to the geometry of $N_w$ in the above coordinate systems.

\begin{pro} \label{pro-keycompnorm}
There are constants $c,C\in (0,\infty)$ such that, for any
$$\gamma=\exp\left( \sum_1^d x_i\eps_i\right)  \ \hbox{ with }  \  x=(x_i)_1^d\in \mathbb R^d,
$$
$$cN_w(x)\le \|\gamma\|_{\Sigma,w}\le C N_w(x).$$
Similarly, there are constants $c,C\in (0,\infty)$ such that, for any
$$
\gamma=\prod_1^d\exp\left(y_i\eps_i\right)
 \ \hbox{ with }  \  y=(y_i)_1^d\in \mathbb R^d,
$$
$$cN_w(y)\le \|\gamma\|_{\Sigma,w}\le C N_w(y).$$\end{pro}
Thanks to earlier considerations, it suffices to prove the first set of inequalities which refers to exponential coordinates of the first kind.
\begin{proof}[Proof of $cN_w(\zeta)\le \|\gamma\|_{\Sigma,w}$]  To simplify notation, set $N_w=N$. In \cite{SCZ-nil}, it is proved that there exists a finite tuple $(i_1,\dots,i_q)$,
$i_j\in\{1,\dots,\ell\}$, $1\le j\le q$, such that any $\gamma\in \Gamma$ with $\|\gamma\|_{\Sigma,w}=R$ can
be written as
\begin{equation}\label{finiteproduct}
\gamma= \prod_1^q \sigma_{i_j}^{z_j},\quad
 |z_j|\le CR^{ w(\sigma_{i_j})}.
 \end{equation}
Since $\sigma_i=\exp(\varsigma_i)$, $\sigma_i^x=\exp(x\varsigma_i)$ for any $i\in \{1,\dots,\ell\}$. Because $\varsigma_i$ has weight $w^\mathfrak g(\varsigma_i)=w(\sigma_{i})$, by construction, there is a $k_i $ with $w^\mathfrak g_{w_{k_i}}\ge w(\sigma_i)$ such that
$\varsigma_i\in \mathfrak  g^{w^\mathfrak g}_{k_i}$. In particular,
$$\varsigma_{i_j}= \sum_{k=k_{i_j}}^{j^\mathfrak g_\star }\xi_k, \quad \xi_k\in \mathfrak n_k,$$
and
$$\exp (z_j\varsigma_{i_j}) =\exp\left(\sum_{k=k_{i_j}}^{j^\mathfrak g_\star }z_j\xi_k \right)
$$
with $$\|z_{j}\xi_k\|_2\le  \max_{1\le i\le \ell}\{\|\varsigma_{i}\|_2\} \times |z_{j}|\le C' R^{w(\sigma_{i_j})}\le C' R^{w^\mathfrak g_{k}},$$
because $R\ge 1$ and $w^\mathfrak g_k\ge w(\sigma_{i_j})$ for all $k\ge k_{i_j}$.  That is,
$$N(z_j\xi_{k})\le C''R.$$

Because formula \eqref{finiteproduct} gives any $\gamma$ as a product of at most $q$  elements $\exp(z_j \varsigma_{i_j})$ with
$N( z_j \varsigma_{i_j})\le C'R$, it follows that any $\gamma=\exp(\sum_1^nx_i\eps_i)\in \Gamma$  satisfies
$$
N(x)\le C''C_*^{q-1}R= C''C_*^{q-1} \|\gamma\|_{\Sigma,w}.$$
\end{proof}
\begin{proof}[Proof of $\|\gamma\|_{\Sigma,w}\le CN(\zeta)$]
 The proof is by induction on the dimension $n$ of $\mathfrak g$.  If the dimension is $0$, there is nothing to prove.  Assume that  for all cases when the dimension of $\mathfrak g$ is less then $m$,   there exists a constant $\widetilde{C}$ such that  $\|\tilde{\gamma}\|_{\widetilde{\Sigma},\tilde{w}}\le \widetilde{C}N(\tilde{\zeta})$ for all $\tilde{\gamma}\in \widetilde{\Gamma} \subset (\mathbb R^m,\cdot)=\widetilde{G}$.    Consider $\Gamma,\Sigma,w,G=(\mathbb R^{m+1},\cdot)$.   Let $g\in \Gamma$ be a non-trivial element of the highest weight $w_{j_*}$ which is a commutator of the elements  $\sigma_i$ forming the tuple $\Sigma$ (this includes the elements of $\Sigma$ which are considered commutators of length $1$). Let $a\ge 1$ be the length of this commutator and $\sigma_{i_1},\dots, \sigma_{i_a}$,
 be the list of $\sigma_i$ used  to write $g$ as a commutator of length $a$ with the property that $w_{j_*}=\sum _1^a w(\sigma_{i_a})$.
 The element $g$ must commute with all elements in $\Gamma$ and it is of the form $g=\exp (\theta)$  where  $\theta$ is the Lie commutator over $\Sigma^\mathfrak g$ associated with the writing of $g$ as a commutator  over $\Sigma$.
Formally, let us use the notation $\mathbf c_G(x_1,\dots,x_a)$ to express the formal group commutator in question evaluated at the group elements $x_1,\dots,x_a $ so that $g=\mathbf c_G(\sigma_{i_1},\dots,\sigma_{i_a})$. Let $\mathbf c_\mathfrak g$ be the corresponding formal Lie commutator so that
$\theta=\mathbf c_\mathfrak g(\zeta_{i_1},\dots,\zeta_{1_a})$.  For any $a$-tuple of reals $t_1,\dots, t_a$, we also have
$$\mathbf c_G\left(e^{t_1 \zeta_{i_1}},\dots,\dots, e^{t_a \zeta_{i_a}}\right)=\exp( t_1\dots t_a \mathbf c_\mathfrak g(\zeta_{i_1},\dots, \zeta_{i_a})).$$

 Let $\Theta=\{ \exp(s\theta): s\in \mathbb R\}$ be the central  one parameter subgroup of $G$ associated with $\theta$ and
 consider the simply connected nilpotent group $\widetilde{G}=G/\Theta$ and its discrete subgroup $\widetilde{\Gamma}$ which is the image of $\Gamma$ by the projection map  $\pi: G\to \widetilde{G}$. The subgroup $\widetilde{\Gamma}$ is generated by the tuple $\widetilde{\Sigma}=(\pi(\sigma_1),\dots,\pi(\sigma_\ell))$.
 The dimension of $\tilde{\mathfrak g}$ is $m-1$. We can choose it to be the orthogonal complement of $\theta$ in $\mathfrak g$ so that $d\pi$ is the
 orthogonal projection  onto $\widetilde{\mathfrak g}$.

For any $\gamma=\exp(\zeta) \in \Gamma$ with $N(\zeta)=R$, we have
 $$N(\tilde{\zeta})\le N(\zeta)=R.$$ Moreover, applying the induction hypothesis, we can write  $\tilde{\gamma}=\exp(\tilde{\zeta})=\pi(\gamma)\in \widetilde{\Gamma}$ as a word over the alphabet $\widetilde{\Sigma}\cup\widetilde{\Sigma}^{-1}$
with
$$\|\widetilde{\gamma}\|_{\widetilde{\Sigma},\tilde{w}}\le \widetilde{C} N(\tilde{\zeta}).$$
Using this word representing of $\pi(\gamma)$, replacing  each $\tilde{\sigma}_i$ by $\sigma_i$ to obtain a word over the alphabet $\Sigma\cup \Sigma^{-1}$, and evaluating in $G$ give us an element
$\bar{\gamma}\in \Gamma$  and an element $\bar{\zeta}\in \mathfrak g$ such that
$$\left\{\begin{array}{l} \bar{\gamma} = \exp( \bar{\zeta}),\\
\pi(\bar{\gamma})=\tilde{\gamma},\quad d\pi (\bar{\zeta})=\tilde{\zeta},\\
\gamma =\bar{\gamma} \exp(t\theta)  \quad \mbox{for some  real } t \le C R^{w_{j_*}}. \end{array}\right.
  $$
  The estimate on $t$ is from \cite[Theorem 2.10]{SCZ-nil} (together with an application of the Campbell-Hausdorff formula in our special system of coordinates).  By construction, $\exp(t\theta)\in \Gamma$, and  \cite[Theorem 2.10]{SCZ-nil} implies that $$\|\exp(t\theta)\|_{\Sigma,w}\le CR=CN(\zeta).$$
It follows that
$$\|\gamma\|_{\Sigma,w} \le C'( |\bar{\gamma}_{\Sigma,w}|+ CN(\zeta))\le C'(\widetilde{C}+C) N(\zeta).$$
\end{proof}

The following proposition captures the fact that $N_w$ is almost a quasi-norm (a quasi-norm at large scale) on $G=(\mathbb R^d,\cdot)$ and is a quasi-norm on $G_\bullet=(\mathbb R^d,\bullet)$. The first fact follows from the
adapted triangular nature of multiplication in the the type of coordinate system considered here. The second fact then
follows from the homogeneity of $N_w$ together with the fact that $(\delta_t)_{t>0}$ is a group dilstion structure on $G_\bullet$.
\begin{pro} For any exponential coordinate system of the first or second kind adapted to the filtration $\mathfrak g^{\Sigma,w}_j$, $1\le j\le j^\mathfrak g _\star$, we have, for any $z,z'\in \mathbb R^d$,
$$N_w(z\cdot z')\le C(N_w(z)+N_w(z')+1),$$
where the group law $\cdot$ refers to the multiplication in $G=(\mathbb R^d,\cdot)$. Moreover, in the same linear basis for $\mathbb R^d$, we have
$$N_w(z\bullet z')\le C(N_w(z)+N_w(z')),$$
where $\bullet$ is the group law on $G_\bullet=(\mathbb R^d,\bullet)$ associated with the approximate group dilation $\delta_t(z)=(t^{b_i}z_i)_1^d)$.
\end{pro}

\section{The main results for random walks driven by measures in $\mathcal{SM}(\Gamma)$}
\label{S:10}

\subsection{The limit theorems for $\mathcal {SM}  (\Gamma)$}\label{S:10.1}

In this section we state our main results concerning measures in $\mathcal{SM}(\Gamma)$. They are direct applications of Theorems \ref{WT1} and \ref{localCLT}.  We state these results in adapted coordinate systems. Namely, given $\mu\in \mathcal{SM}(\Gamma)$ and the simply connected Lie group $G$ containing $\Gamma$ as a co-compact discrete subgroup, we choose to write $G=(\mathbb R^d,\cdot)$ using {\em one} of the polynomial coordinate systems described in Section \ref{S:9.5} above.  This coordinate system is adapted to the filtration $(\mathfrak g^{\Sigma,w}_j)_j$ of the Lie algebra $\mathfrak g$, itself built from the data describing the measure $\mu$ as an element of $\mathcal{SM}(\Gamma)$.  In particular, in this coordinate system, we have an approximate group dilation structure given by (\ref{def-delta}) which defines a limit group structure $G_\bullet=(\mathbb R^d,\bullet)$. The law $\bullet=\bullet_\mu$ defining this limit structure depends on $\mu$.

 Below, we  show that for {\em any} measure $\mu\in \mathcal{SM}(\Gamma)$, there are  a suitable approximate group dilation structure 
$(\delta_t)_{t >0}$ given by (\ref{def-delta}) and  a    norm $\| \cdot \|$ 
on $\Gamma$ so that  assumptions \eqref{e:4.2}, (R1)-(R2)-(E1)-(E2)
 and (T$\Gamma$) are all satisfied with the common constant $\beta>0$.  
    This is in contrast to condition (A) which may or may not be satisfied. Recall that condition (A) is the requirement that
the measure $\mu_t= t\delta_{1/t}(\mu)$, $t\ge 1$, defined  
by  \eqref{def-mut}  converges vaguely on $\mathbb R^d\setminus\{0\}$ to a measure $\mu_\bullet$
as $t$ tends to infinity.   Because the dilations $(\delta_t)_{t>0}$ have been carefully constructed from $\mu$, the family $(\mu_t)_{t>0}$ is always tight and, if (A) is satisfied then   (T$\bullet$)   is satisfied and the support of the limit $\mu_\bullet$ generates $G$; see the subsections below for the proofs.

 \medskip

 Recall that
 $\{ \mathbb P^x_\bullet;\,  x\in   G_\bullet \}$ is the family of probability measures
 induced by the limit symmetric L\'evy process $X^\bullet$
on $\D([0,M_0],\R^d)$.

Fix an arbitrary increasing sequence of reals $T_k$ that tends to infinity, e.g., $T_k=k$, and recall the notation
$\hat{X}^k_t,\; t>0$, $\hat{P}^k_t,\; t>0$, and  $\mathbb P^{[x]_k}_k,\; x\in G$ associated with the space-time rescaled discrete random walk, see  (\ref{def-hat}). In this notation, $[x]_k$ is  the closest point of $x$ (any one of, if there are more than one such points)  on $\Gamma_{T_k}$ in the norm $\|\cdot\|$, and
$$\mathbb P^{x}_k(\hat{X}^k_t=y)=\mu^{([tT_k])}((\delta_{T_k}(x))^{-1}\cdot \delta_{T_k}(y)),\quad
x,y\in \Gamma_{T_k}.$$

Hence, applying  Theorems \ref{WT1} and \ref{localCLT}, we obtain the following theorem.

\begin{theo}\label{WTStab}
Let $\mu\in \mathcal{SM}(\Gamma)$. Referring to the above set-up and notation,
assume that condition {\rm (A)} holds true, that is, the measure $\mu_t= t\delta_{1/t}(\mu)$, $t\ge 1$, defined at {\rm (\ref{def-mut})} converges vaguely on $\mathbb R^d\setminus\{0\}$ to a Radon measure $\mu_\bullet$ on $\mathbb R^d\setminus\{0\}$
 as $t$ tends to infinity.
\begin{itemize}
\item For any bounded continuous function $f$ on $\mathbb R^d$,
$\hat P_s^k f$ converges uniformly on compacts to
$P_{\bullet,s}f$. Furthermore, for each $M_0>0$ and  for
   every  $x\in \mathbb R^d$, $\hat \P^{[x]_k}_k$ converges weakly to $\P_\bullet^x$ on the space $\D([0,M_0],\R^d)$
   equipped with ${\cal J}_1$-topology.

\item  For any $U_2>U_1>0$ and $r>1$,
$$\lim_{k\to\infty}\sup_{{x\in   \R^d}: \|x\|\le r} \,\sup_{t\in [U_1,U_2]}
\left|\det (\delta_{T_k}) \mu^{([tT_k])}_k(\delta_{T_k}([x]_k))-  p_\bullet (t,x)\right|=0.
$$
\end{itemize}
\end{theo}

\subsection{The hypotheses  (R1)-(R2), and (E1)-(E2) when  $\mu\in \mathcal{SM}(\Gamma)$}\label{S:10.2}

Using the constructions described in the previous two sections and the results from \cite{CKSWZ1}, we can now show that any probability measure in $\mathcal{SM}(\Gamma)$ satisfies the hypotheses (R1)-(R2)-(E1)-(E2) in the context of properly chosen exponential coordinates of the first or second kinds.   Let us assume that we are given $\mu\in \mathcal{SM}(\Gamma)$ and the associated data $\Sigma,w$  as in Section \ref{sec-muw} and quasi-norm $\|\cdot\|_{\Sigma,w}$ as in Definition \ref{def-quasinorm}.  We assume that $\Gamma$ is given as a co-compact  subgroup of a simply connected Lie group $G$ and that an adapted global exponential coordinate system of the first or second kind has been chosen as explained in Section \ref{S:9.6} so that $\Gamma\subset G=   (\mathbb R^d,\cdot)$.
Moreover, $(\mathbb R^d,\cdot)$ is equipped with a straight approximate dilation structure
$$(\delta_t)_{t>0}: \quad \delta_t(z)=(t^{b_i}z_i)_1^d$$ with   the group limit
$(\mathbb R^d,\bullet)$.  Here the basis for $\mathbb R^d$ can also be identified as in Section \ref{S:9.6} with a linear basis of $\mathfrak g$ which is compatible with the direct sum decomposition
$$\mathfrak g=\oplus_1^{j^\mathfrak g_\star} \mathfrak n_j$$
in Definition \ref{def-directsum}. Each subspace $\mathfrak n_j$ is associated with a weight value $\underline{w}_j= w^\mathfrak g_j 
 >2$ and
for any index $i$ such that the corresponding basis element is in $\mathfrak n_j$, $b_i=\mathfrak w^\mathfrak g_j$.

We pick 
$$
 0<b< \min\{b_i: 1\le i\le d\}=(\max\{\beta_i: 1\le i\le d\})^{-1}, \quad \beta_i :=1/b_i
 $$
and set
$$|\gamma|_\Gamma=\|\gamma\|_{\Sigma,w}^b.$$
By construction, this is a norm on $\Gamma$, that is, $|\gamma\cdot\gamma'|_\Gamma\le |\gamma|_\Gamma+|\gamma'|_\Gamma$ for all $\gamma,\gamma'\in \Gamma$.
We also let $\|\cdot\|$ be a norm on $(\mathbb R^d,\bullet)$ (i.e., satisfying the triangle inequality $\|g\bullet g'\|\le \|g\|+\|g'\|$ for all $g,g'\in G_\bullet$), which is also equivalent to
$$
N_w (z)^b= \max\{|z_i|^{b/b_i}: 1\le i\le d\},   \quad z\in \mathbb R^d.
$$
By \cite{HebSik}, such a norm always exists.
 This norm $\| \cdot \|$ has the dilation property \eqref{e:4.2} with $\beta = \max_{1\leq i\leq d} \beta_i = 1/b$.

\smallskip

By Proposition \ref{pro-keycompnorm}, we have a tight comparison between the discrete object $|\cdot |_\Gamma$ and   the  continuous homogeneous norm
$\|\cdot\|$ on $G_\bullet=(\mathbb R^d,\bullet)$, namely, there are constant $0<c,C<\infty$ such that, for any $\gamma\in \Gamma\subset \mathbb R^d$,
\begin{equation}\label{CondN} c\|\gamma\|\le |\gamma|_{\Gamma} \le C\|\gamma\|.
\end{equation}

\subsection*{Conditions (R1)-(R2) and (E1)-(E2)}

Recall that condition (R1) reads
\begin{itemize}
\item[(R1)] There are constants $C_1$ and $\kappa$  such that, for any bounded  function $u$ defined on $\Gamma$ and $\mu$-harmonic in   $ B(r):=\{x\in \R^d: \|x \| < r\}$, we have
\[
  |u(y)-u(x)|\le C_1 \|u\|_{\infty}\left(\frac{\|x^{-1}\cdot y\|}{r}\right)^{\kappa}
\quad    \hbox{for } x, y \in B(r/2).  
\]
\end{itemize}

For any $\mu\in \mathcal{SM}(\Gamma)$, \cite[Corollary 6.10]{CKSWZ1} gives the following $\Gamma$-version of this property
\begin{itemize}
\item[(R$\Gamma$1)] There are constants $C_1$ and $\kappa$  such that, for any bounded  function $u$ defined on $\Gamma$ and $\mu$-harmonic in  $B_\Gamma (r):=\{x\in \Gamma:  |x|_\Gamma <  r\}$, 
 we have
\[
  |u(y)-u(x)|\le C_1 \|u\|_{\infty}\left(\frac{|x^{-1}\cdot y|_{\Gamma}}{r}\right)^{\kappa}
\quad  \hbox{for } x, y \in B_\Gamma (r/2). 
\]
\end{itemize}
To pass from this $\Gamma$ version, (R$\Gamma$1), to the desired (R1) we use the key norm comparison (\ref{CondN}) and a simple covering argument to adjust the permitted range of $x,y$ from one  
  statement  to the other.

Similarly, recall that condition (R2) reads
\begin{itemize}
\item[(R2)] There  are positive  constants $C_2 >0$ and $\beta>0$  such that, for all $n,m\in \mathbb N$  and  $x,y\in \Gamma$,
\begin{equation}\label{eq:3-29}
|\mu^{(n+m)}(xy)-\mu^{(n)}(x)|\le \frac{C_1}{V( n^{1/\beta})}  \left(\frac{m}{n+1}+\sqrt{\frac{\|y\|^{\beta}}{n+1}} \right),
\end{equation}
where $V(r):=\sharp \{g\in \Gamma: \|g\|\le r\}$.
\end{itemize}

The fact that (R2) holds true for any probability $\mu$ in $\mathcal{SM}(\Gamma)$ follows straightforwardly from \cite[Theorem 5.5 
(3)-(4)]{CKSWZ1} (see also  \cite[Proposition A.3]{CKSWZ1}), together with  (\ref{CondN}).
Regarding related recent results concerning the regularity of stable-like transition kernels in the abelian case, see \cite{ChKa}.

Regarding
the exit times conditions  (E1)-(E2), which are expressed using the norm $\|\cdot\|$ on $\mathbb R^d$,  for any measure $\mu\in \mathcal{SM}(\Gamma)$, they follow from (\ref{CondN}) together with \cite[Theorem 5.5(5)]{CKSWZ1} (for (E1)) and \cite[Lemma 6.6]{CKSWZ1} (for (E2)),  with the exponent $\beta>0$
being in the same as those in (R2). 
  (In these results of \cite{CKSWZ1}, $\beta = 1/w_*$ there, where $0<w_*:= \min\{w(s); s\in \Sigma\}$, which is our $b$; 
see Example 2.9 and Proposition 2.11(c) there.)

\subsection{Condition (T$\Gamma$)} \label{S:10.3}

  Verifying condition (T$\Gamma$) for any measure $\mu$ in   $\mathcal{SM}(\Gamma)$ requires some work. Any  $\mu\in \mathcal{SM}(\Gamma)$ is  a finite convex combination of   probability
  measures of a certain type and it suffices to prove  (T$\Gamma$) for any such building block, $\nu$. By definition, any such probability
  measure $\nu$  has
   the following property:  there is a subgroup $H$ of $\Gamma$ with finite, symmetric generating set $S$, word length $|\cdot |_{S}$ and volume growth exponent $d_H$, and an exponent $\alpha\in (0,2)$, such that
\begin{equation} \label{mu-hyp1}
\nu(x)\asymp \left\{\begin{array}{cl} (1+|x|_S)^{-\alpha-d_H} &\mbox{ if } x\in H,\\0& \mbox{otherwise.}\end{array}\right.\end{equation}
Note that the discrete subgroup $H$ is contained as a co-compact discrete subgroup in a unique closed  connected Lie subgroup $L=L_H$ of $G$.  As in Subsection \ref{S:10.2}, we assume we have made the choice of an adapted coordinate system for $G$ and of  an appropriate approximate dilation structure $(\delta_t)_{t>0}$.

Recall that $G$ is described by a polynomial global coordinate chart $G=(\mathbb R^d,\cdot)$ in which   the ebesgue measure is a Haar measure for $G$. Let $m\le d$ be the dimension of $L$. This closed Lie subgroup can  be described parametrically as an embedded  sub-manifold of $\mathbb R^d$ given by a polynomial map   $i_H=i$ from $\mathbb R^m$ into  $ \mathbb R^d$:
\begin{equation} \label{iH}
     v=(v_1,\dots,v_m) \in \R^m\mapsto  i_H(v)=(i_1(v),\dots,i_d(v)) \in \R^d
\end{equation}
 with a polynomial inverse on its image. Assume further that this map $i$ is also a group isomorphism on its image, that is,
$$L=(\mathbb R^m,\cdot_H)  \quad \mbox{and} \quad \;i(v)\cdot i(w)=i(v\cdot_H w).$$
In fact, we can use this formula to define $\cdot_H$ on $\mathbb R^m$. However, it will be convenient to assume that $L=(\mathbb R^m,\cdot)$ is an exponential coordinate system of the first type for $L$.

The Lie group $L$ is, of course, nilpotent and simply connected, and we assume that the global coordinate system $(\mathbb R^m,\cdot)$ is an exponential  coordinate system of the first type compatible with the
lower central series of $L$:
$$
L_1=L \supset L_2=[L,L]\supset \dots\supset L_j=[L,L_{j-1}]\supset \dots\supset L_t\supset L_{r_H+1}=\{0\},
$$
where $r_H$ is the smallest $j$ such that $L_{j+1}=\{0\}$. Namely, there is a strictly increasing $r_H$-tuple  of  integers $k_j$,  $1\le j\le r_H$, $k_1=1, k_{r_H}=m$ such that
$$L_{j}=\{(0,\dots,0,v_{k_{j}},\cdots, v_m): v_{k_{j}},\dots,v_m\in \mathbb R\}.$$
In this coordinate system for $L$, the straight dilation
\begin{equation}\label{gammaH}\gamma^H_t(v)= (t^{p_i} v_i)_1^m,\quad  p_i= j \;\mbox{ if } \quad  k_j\le i\le k_{j+1}-1\end{equation}
form an approximate group dilation structure with limit $L_*=(\mathbb R^m,*)$, a stratified nilpotent Lie group of homogeneous dimension $d_H$ with
$$d_H=\sum_{j=1}^{r_H} j (k_{j+1}-k_{j}).$$
According to Pansu's theorem, see \cite{Pansu1983} and \cite{Breuillard}, the word length $|\cdot |_S$ has the property that there is a norm $|\cdot|_*$ on $L_* $ such that
$$|\gamma^H_t (v)|_*=t|v|_*  \quad \mbox{for all }\;t>0,\, v\in \mathbb R^m,$$
and
\begin{equation}\label{Pansu}
\lim_{v\in H,\;v\ra \infty}|v|_S/|v|_*=1.\end{equation}
Moreover,
$$ |v|_*\asymp  \max_i\{|v_i|^{1/p_i}:v=(v_1,\dots,v_m)\}.$$
This implies  that for any $v\in \mathbb R^m$ with $\|v\|_2\le 1$ and $r\in (0,1]$  such that $\|\gamma^H_{1/r} v \|_2=1$, we have
\begin{equation} \|v\|_2\le r.  \label{2gamma1}
\end{equation}
 Moreover,  for any $v\in H$, if we define $t_v$ by $\|\gamma^H_{1/t_v} v\|_2=1$, then we have
 \begin{equation} |v|_S\asymp t_v. \label{2gamma2}
\end{equation}

  By construction, because  the probability measure 
$\nu$ is one of the building blocks of $\mu$, the  approximate dilation structure $(\delta_t)_{t>0}$ has the property that
\begin{equation}\label{mudelta}
\lim_{t\to  \infty} \delta_t^{-1}\circ  i_H\circ \gamma^H_{t^{1/\alpha}} (v)=:p(v)\end{equation}
exists for all $v\in \mathbb R^m=i^{-1}_H(L)$. This limit is uniform on compact sets and  the map $p$ is a continuous map (in fact a smooth map) from $\mathbb R^m$ to $\mathbb R^d$.  In the following Lemma, we consider any approximate dilation structure $(\delta_t)_{t>0}$ such that the limit in (\ref{mudelta}) exists.

 \begin{lem}[The map $p$ is a group homomorphism] \label{lem-phom}  Assume that $(\delta_t)_{t>0}$ is an approximate dilation structure on $G=(\mathbb R^d,\dot)$, that $\nu$   is a probability measure on $H$ satisfying 
  {\rm (\ref{mu-hyp1})}, and  that the limit $p$ at {\rm (\ref{mudelta})} exists   for all $v\in \mathbb R^m=i^{-1}_H(L)$.
If we equip $\mathbb R^m$  with the limit group structure $L_*=(\mathbb R^m,*)$ associated with the dilations $(\gamma^H_t)_{t>0}$, the map $p$
is a continuous group homomorphism from $L_*$ to $G_\bullet$ $($it is typically neither injective nor onto$)$.  Let $L^*_\bullet=p(L_*)\subseteq G_\bullet$.  For any $x\in L^*_\bullet$ and $u\in L$,
 $$
 \delta_{t^\alpha} (x)=p(\gamma^H_tu).
$$
Define $\gamma^{L^*_\bullet}_t$ on $L^*_\bullet$ by $\gamma^{L^*_\bullet}_t(x)=p(\gamma^H_tu). $
This is a group of group diffeomorphisms on $L^*_\bullet$. Namely,
 $$\gamma^{L^*_\bullet}_s\circ \gamma^{L^*_\bullet}_t=\gamma^{L^*_\bullet}_{st}
 \quad \mbox{and} \quad
 \gamma^{L^*_\bullet}_t(x\bullet y)=\gamma^{L^*_\bullet}_t(x)\bullet \gamma^{L^*_\bullet}_t(y), \quad
 s,t>0,\, x,y\in L^*_\bullet.
 $$
 Moreover, for any $s>0$,
 $$\delta_{1/s}\circ \gamma^{L^*_\bullet}_{s^{1/\alpha}}=\mbox{\rm Id} \quad  \mbox{on }  L^*_\bullet.$$
\end{lem}
\begin{proof} We can approximate
$p(u)\bullet p(v)$ by $\delta_{1/t}\left(\delta_t(p(u))\cdot \delta_t(p(v))\right) $ with $t$ large enough.  
Note that
$$\delta_{1/t}\left(\delta_t(p(u))\cdot \delta_t(p(v))\right)= \delta_{1/t}\circ i_H \circ \gamma^H_{t^{1/\alpha}}\left(\gamma^H_{1/t^{1/\alpha}}(\gamma^H_{t^{1/\alpha}}(u)\cdot\gamma^H_{t^{1/\alpha}}(v))\right).$$
Since, for large $s$, we can approximate $\gamma^H_{1/s}(\gamma^H_s(u)\cdot\gamma^H_s(v))$ by $u*v$  and the convergence of $ \delta_{1/t}\circ i_H \circ \gamma^H_{t^{1/\alpha}}$ to $p$ is uniform on compact sets,
  it follows that $p$ is a continuous group homomorphism from $L_*$ to $G_\bullet$.  The remaining statements are straightforward.
\end{proof}

\begin{lem}\label{L:8.3}
 Let    $\nu$    be a probability measure on $H$ as in   \eqref{mu-hyp1} and let $(\delta_t)_{t>0}$
   be an  approximate dilation structure on $G=(\mathbb R^d,\dot)$ satisfying   \eqref{mudelta}. 
   Let  $ \nu_t=t \delta_{1/t}(\nu)$ and let $J_t$ the associated  jump kernel from {\rm Proposition \ref{weaklimmeas}}   with $\nu$ in place of $\mu$ there. For any compact subset $K\subset \mathbb R^d$,
 \begin{eqnarray} \label{WCA2}
 \lim_{\eta\to 0}\limsup_{t\to \infty}
 \iint_{\{(x, y)\in K\times K: \|x_\bullet^{- 1}\bullet y\|_2 \leq  \eta\}} \|x_\bullet^{-1}\bullet y\|_2^2 J_t(dx,dy)&=0,\\ \label{WCA1}
   \lim_{R\to \infty}\limsup_{t\to \infty}
\int_K \int_{B_\bullet(x,R)^c} J_t(dx,dy) &=0.
\end{eqnarray}
\end{lem}

To prove this lemma, note  that, for any $f\ge 0$,
 $$\int f(x,y)J_t(dxdy)$$ is dominated by  a constant times
 $$t\det(\delta_{1/t}) \sum_{x,y\in \delta_{1/t}(\Gamma),x\neq y \atop{\delta_t(x)^{-1}\cdot\delta_t(y)\in H}}f(x,y) \frac{1}{(1+|\delta_t(x)^{-1}\cdot\delta_t(y)|_S)^{\alpha+d}}.$$

 The two functions $f$ of interest here are
 $$
  f(x,y)= \1_K(x)\1_K(y)\1_{\{\|x_\bullet^{- 1}\bullet y\|_2\le \eta\}}(y) \|x_\bullet^{- 1}\bullet y\|_2
 $$
and
$$
f(x,y)=\1_K(x)\1_{B_\bullet(x,R)^c}(y).
$$
The results follow from the  
 facts that   $\mu$ has the form (\ref{mu-hyp1}) and that $\delta_t$ is compatible with $\gamma^H_{t^{1/\alpha}}$
in the sense that (\ref{mudelta}) holds true.

\begin{proof}[Proof of (\ref{WCA1})]We need to bound
\begin{align*}  \lefteqn{I(R,t)=  t\det(\delta_{1/t}) \sum_{x,y\in \delta_{1/t}(\Gamma)\atop{\delta_t(x)^{-1}\cdot\delta_t(y)\in H}
}  \frac{ \1_K(x)\1_{B_\bullet(x,R)^c}(y)}{(1+|\delta_t(x)^{-1}\cdot\delta_t(y)|_S)^{\alpha+d_H}}}&&\\
&=   t\det(\delta_{1/t}) \sum_{x,y\in \delta_{1/t}(\Gamma)\atop{\delta_t(x)^{-1}\cdot\delta_t(y)\in H}}  \frac{ \1_K(x)  \1_{B_\bullet(x,R)^c}(y)  \1_{\{|\cdot|_S\le \eps (R^\beta t)^{1/\alpha}\}}(\delta_t(x)^{-1}\cdot \delta_t(y))}{(1+|\delta_t(x)^{-1}\cdot\delta_t(y)|_S)^{\alpha+d_H}}\\
&\quad  + t\det(\delta_{1/t}) \sum_{x,y\in \delta_{1/t}(\Gamma) \atop{\delta_t(x)^{-1}\cdot\delta_t(y)\in H} }  \frac{ \1_K(x)\1_{B_\bullet(x,R)^c}(y) \1_{\{|\cdot|_S>\eps(tR^\beta)^{1/\alpha}\}}(\delta_t(x)^{-1}\cdot \delta_t(y))
}{(1+|\delta_t(x)^{-1}\cdot\delta_t(y)|_S)^{\alpha+d_H}}\\
&=  I_1(R,t)+I_2(R,t).\end{align*}
The second sum, $I_2(R,k)$ is the main term and we treat it first by going back to $H$.
\begin{align*} I_2(R,k)&\le   t\det(\delta_{1/t}) \sum_{x,y\in \delta_{1/t}(\Gamma)\atop{\delta_t(x)^{-1}\cdot\delta_t(y)\in H}}  \frac{ \1_K(x) \1_{\{|\cdot|_S>\eps(R^\beta t)^{1/\alpha}\}}(\delta_t(x)^{-1}\cdot\delta_t(y))
}{(1+|\delta_t(x)^{-1}\cdot\delta_t(y)|_S)^{\alpha+d_H}}\\ &= 
 \left(\det(\delta_{1/t}) \sum_{x\in \Gamma}  \1_{\delta_t(K)}(x) \right)\left( t\sum_{z\in H\atop{|z|_S>\eps(R^\beta t)^{1/\alpha}}}\frac{1
}{(1+|z|_S)^{\alpha+d_H}}\right).
\end{align*}
The first factor is clearly bounded by a constant depending only on $K$ because $\Gamma$ is a co-compact lattice in $G$ so that
\begin{equation}\label{eq:10-19oe}
\sum_{x\in \Gamma}  \1_{\delta_t(K)}(x) \le C(K) \det(\delta_t).
\end{equation}

Using a decomposition by the
 Dyadic annulus in $H$, $\{x\in H: 2^{k}\le |\cdot|_S< 2^{k+1}\}$, it is elementary to verify that, for all $R,t>1$, the second factor satisfies
 $$ t\sum_{z\in H\atop{|z|_S>\eps(R^\beta t)^{1/\alpha}}}\frac{1
}{(1+|z|_S)^{\alpha+d_H}} \asymp  t (\eps (R^\beta t)^{1/\alpha})^{-\alpha}  \asymp \eps^{-\alpha}R^{-\beta}.$$
This proves that $$ \lim_{R\to \infty} \sup _{t\ge 1} \{I_2(R,t)\}=0.$$

To finish the proof of (\ref{WCA1}), we show that we can chose $\eps>0$ such that, for any $R\ge 1$ and $t$ large enough, $I_1(R,t)=0$. To see that, we will use the
description of $H$ as a discrete subgroup of $\mathbb R^m$ embedded into $\mathbb R^d$ via the $i_H:\mathbb R^m\to \mathbb R^d$, see (\ref{iH}),
and the approximate dilation structure $(\gamma^H_t)_{t>0}$.  A basic fact about this structure is that,  there is a constant $C_1$ such that,
for any $x\in H$ and $t\ge |x|_S$,
$$\|\gamma^H_{1/t^{1/\alpha} }(x)\|_2\le C_1,$$
where  $\|\cdot\|_2$ is the Euclidean norm on $\mathbb R^m$.
Now, for any $x\in H$, and any $t>0$, we have
$$N(i_H(x)) =tN(\delta_{1/t}\circ i_H (x))= tN\left(\delta_{1/t}\circ i_H\circ \gamma^H_{t^{1/\alpha}}(\gamma^H_{1/t^{1/\alpha}}(x))\right).$$
For large enough $t$ and $\|z\|_2\le C_1$,
$$ N\left(\delta_{1/t}\circ i_H\circ \gamma^H_{t^{1/\alpha}}(z)\right) \asymp N(p(z)),$$
because $\lim_{t\to \infty}\delta_{1/t}\circ i_H\circ \gamma^H_{t^{1/\alpha}} =p$ uniformly on compact sets.
So, if $|x|_S$ is large enough and we choose $t^{1/\alpha}=|x|_S$, we obtain
$$N(i_H(x))\le C_2 |x|_S^\alpha.$$

Now, consider $x,y\in \delta_{1/t}(\Gamma)$ such that $x\in K$, $z=\delta_t(x)^{-1}\cdot \delta_t(y)\in H$  and
$|z|_S\le \eps (R^\beta t)^{1/\alpha}$.   Because $x\in K$, we have $N(\delta_t(x))\le C_3t$ and thus
$$N(y)\le t^{-1}N(\delta_t(y))\le C_4t^{-1}( N( \delta_t(x))+N(\delta_t(x)^{-1}\cdot  \delta_t(y)))\le C_5 R^\beta,$$
that is,
$$\|y\|\le C_5^{1/\beta}R.$$
Because $x$ and $y$ are confined in a compact set (that depends on $R$), for $t$ large enough,  $x_\bullet^{-1}\bullet y $ is close to $\delta_{1/t}(  \delta_t(x)^{-1} \cdot \delta_t(y))$ so that
$$\|x_\bullet^{-1}\bullet y \|^\beta= N(x_\bullet^{-1}\bullet y) \le C_6 t^{-1}
N( \delta_t(x)^{-1}\cdot \delta_t(y))\le \eps^\alpha C_7 R^\beta,$$
which implies $\1_{B_\bullet(x,R)^c}(y)=0$ by taking $\eps$ small enough.

This  proves that  we can find $\eps>0$ small enough such that, for any $R$ fixed and $t$ large enough,
$$\1_K(x)\1_{B_\bullet(x,R)^c}(y) \1_{\{|\cdot|_S\le \eps (R^\beta t)^{1/\alpha}\}}(\delta_t(x)^{-1}\cdot\delta_t(y)) =0.$$
This ends the proof of (\ref{WCA1}).\end{proof}

\begin{proof}[Proof of (\ref{WCA2})]We need to bound
$$  J(K,\eta,t)=  t\det(\delta_{1/t}) \sum_{x,y\in \delta_{1/t}(\Gamma)\cap K\atop{\delta_t(x)^{-1}\cdot\delta_t(y)\in H}
}  \frac{\1_{\{\|\cdot\|_2\le \eta\}}(x_\bullet^{-1}\bullet y)\|x_\bullet^{- 1}\bullet y\|_2^2}{(1+|\delta_t(x)^{-1}\cdot\delta_t(y)|_S)^{\alpha+d_H}}.$$
On $K\times K$ and for $t$ large enough, $x_\bullet^{-1}\bullet y$ is close to $\delta_{1/t}(\delta_t(x)^{-1}\cdot \delta_t(y))$ so that
\begin{align*}
  J(K,\eta,t) &\le  t\det(\delta_{1/t}) \sum_{x,y\in=\Gamma \cap\delta_t( K)\atop{x^{-1}\cdot y\in H}
}  \frac{\1_{\{\|\cdot\|_2\le \eta\}}( \delta_{1/t}(x^{-1}\cdot y))\|\delta_{1/t}(x^{-1}\cdot  y) \|_2^2}{(1+|x^{-1}\cdot y|_S)^{\alpha+d_H}}\\
&\le   \left(\det(\delta_{1/t}) \sum_{x\in \Gamma}\1_{\delta_t(K)}(x)\right)\left(t  \sum_{z \in H}
  \frac{\1_{\{\|\cdot\|_2\le \eta\}}( \delta_{1/t}(z))\|\delta_{1/t}(z) \|_2^2}{(1+|z|_S)^{\alpha+d_H}}\right).\end{align*}
As in \eqref{eq:10-19oe}, the first factor is bounded for any fixed compact $K$. So we are left with inspecting
$$J'(K,\eta,t)=  t\sum_{z \in H}
  \frac{\1_{\{\|\cdot\|_2\le \eta\}}( \delta_{1/t}(z))\|\delta_{1/t}(z) \|_2^2}{(1+|z|_S)^{\alpha+d_H}}.$$

   In order to use  the dilation structure $\gamma^H_t$, we represent $H$ as a discrete set in $\mathbb R^m$ which injects into $\mathbb R^d$ via the map $i_H:\mathbb R^m\to \mathbb R^d$.
   Recall that   $\delta_{1/t} \circ i_H(\gamma^H_{t^{1/\alpha}}(z))\to p(z)$ uniformly on compact sets so that
 \begin{align*}\lefteqn{\1_{\{\|\cdot\|_2\le \eta\}}( \delta_{1/t}(i_H(z)))\|\delta_{1/t}(i_H(z)) \|_2^2}&&\\
  &\le  
  \1_{\{\|\cdot\|_2\le \eta\}}( p(\gamma^H_{1/t^{1/\alpha}}(z)))\|  p(\gamma^H_{1/t^{1/\alpha}}(z))\|_2^2 \\
  & \le  \min\{\eta^2, \|  p(\gamma^H_{1/t^{1/\alpha}}(z))\|_2^2 \} \\
&\le     C_1
  \min\{ \eta^ 2 , \|\gamma^H_{1/t^{/\alpha}}(z)\|_2^2 \} \\
  &\le    C_1 C_2
  \min\{ \eta^ 2 , t^{-2/\alpha} |z|^2_S\}  ,\end{align*}
  because $ \|  p(v)\|_2\le C_1 \|v\|_2$ for any $v\in \mathbb R^m$ with $\|v\|_2^2\le 1$ and, by (\ref{2gamma1})-(\ref{2gamma2}),
  $$\|\gamma^H_{1/t^{1/\alpha}} z \|_2 \le C_2 t^{-1/\alpha}|z|_S.$$
  It follows that
  \begin{align*}J'(K,\eta,t)&\le  C_3 t\sum_{z \in H}
  \frac{\min\{\eta^2,t^{-2/\alpha}|z|^2_S\}}{(1+|z|_S)^{\alpha+d_H}}\\
  &\le  C_3 t\sum_{z \in H, |z|_S> \eta t^{1/\alpha}}
  \frac{\eta^2}{(1+|z|_S)^{\alpha+d_H}}\\
 &\quad+ C_3 t^{1-2/\alpha}\sum_{z \in H, |z|_S\le  \eta t^{1/\alpha}}
  \frac{|z|_S^2}{(1+|z|_S)^{\alpha+d}}\\
  &\le  C_4\left(t\eta^2 (\eta t^{1/\alpha})^{-\alpha} + t^{1-2/\alpha} (\eta t^{1/\alpha})^{2-\alpha} \right)=2C_4 \eta^{2-\alpha}.
  \end{align*}
  This proves that $ \lim_{\eta\to 0}\limsup_{t\to \infty}J(K,\eta,t)=0$ as desired.
\end{proof}

\subsection{Condition (T$\bullet$) holds automatically for measures in $\mathcal{SM}(\Gamma)$}\label{S:10.4}

The careful reader will have notice that the title of this subsection needs additional context because, for a measure $\mu$ in $\mathcal{SM}(\Gamma)$ (and a coordinate system as discussed above), we do have an associated  approximate  group dilation structure
$(\delta_t)_{t>0}$ and a limit group structure $G_\bullet=(\mathbb R^d,\bullet)$ but, in general,  the family of  
 measures $\mu_t=t\delta_{1/t}(\mu)$ does not converge vaguely on $\mathbb R^d\setminus\{0\}$ and thus, (T$\bullet$) does not make immediate sense.   There is, however, a simple way to correctly interpret the title of this section.
\begin{lem} For any $\mu\in \mathcal{SM}(\Gamma)$ and associated approximate group dilation structure  $(\delta_t)_{t>0}$  in a coordinate system as above, and any vague sub-limit $\mu_\bullet$ of the family $\{\mu_t\}_{t>0}$ as $t$ tends to infinity, condition {\rm (T$\bullet$)} holds true.
\end{lem}

To prove this lemma, it suffices to prove the similar statement for each  component $\nu$ of the measure $\mu$.
 So, we assume (\ref{mu-hyp1}).  Lemma \ref{lem-phom}  shows that any vague sub-limit $\nu_\bullet$ of $\nu_t$ is supported on $L^*_\bullet$ and is
 bounded from above and below by multiples of the measure
 $$
 \nu_\psi(f)=\int_{\mathbb R^m} f(p(u))\psi(u)du,
 \quad   f\in \mathcal C_c(\mathbb R^d\setminus \{0\}) ,
 $$
where $p:\mathbb R^m\to \mathbb R^d$ is defined by  (\ref{mudelta}) and $\psi:\mathbb R^m\setminus \{0\}\to [0,\infty)$
is given by $\psi(u)=|u|_*^{-\alpha - d_H}$.
Let $J_\psi$  be the associated jump kernel measure. Then, for any compact subset $K\subset V$, we claim that
 \begin{eqnarray} \label{WCA2nu}
 \lim_{\eta\to 0}
 \iint_{\{(x, y)\in K\times K: \|x_\bullet^{-1}\bullet y\|_2 \leq  \eta\}} \|x_\bullet^{-1}\bullet y\|_2^2 J_\psi(dx,dy)&=0,\\ \label{WCA1nu}
   \lim_{R\to \infty}
\int_K \int_{B_\bullet(x,R)^c} J_\psi(dx,dy) &=0.
\end{eqnarray}

 \begin{proof}[Proof of \eqref{WCA2nu} and \eqref{WCA1nu}]  By Lemma \ref{lem-phom},
$$\delta_t(p(y))=p(\gamma^H_{t^{1/\alpha}}
y).$$
 It follows that
 \begin{align*}
\int_K \int_{B_\bullet(x,R)^c} J_\psi(dx,dy)&= \int _K \int _{B(R)^c}d\nu_\psi(dy)dx\\
&=   |K| \int _{\mathbb R^m} \1_{B(R)^c}(p(y))\psi(y)dy\\
&=  |K| \int_{\mathbb R^m}  \1_{B(1)^c}(\delta_{1/R^{\beta}}(p(y)))\psi(y)dy\\
&=  |K| \int_{\mathbb R^m}  \1_{B(1)^c}(p(\gamma ^H_{1/R^{\beta/\alpha}}(y)))\psi(y)dy\\
&=  |K| R^{d_H\beta/ \alpha}\int_{\mathbb R^m}  \1_{B(1)^c}(p(y))\psi(\gamma ^H_{R^{\beta/\alpha}}(y))dy\\
&=  |K| R^{-\beta} \int_{\mathbb R^m}  \1_{B(1)^c}(p(y))\psi(y)dy\quad  \stackrel{R\to \infty}\longrightarrow \;0 . \end{align*}
For the last step, note that $p$ is continuous so that  $\1_{B(1)^c}(p(y))$ is equal to $0$ in a neighborhood of $0$ in $\mathbb R^m$ and thus,  $ \int_{\mathbb R^m}  \1_{B(1)^c}(p(y))\psi(y)dy<\infty$.

Similarly, consider
$$I(K,\eta)= \iint_{\{(x, y)\in K\times K: \|x_\bullet^{-1}\bullet y\|_2 \leq  \eta\}} \|x_\bullet^{-1}\bullet y\|_2^2 J_\psi(dx,dy).$$

 Write
 \begin{align*}
I(K,\eta)&=  \int_{x\in K} \int_{ \{xz\in K: \|z\|_2 \leq  \eta\}} \|z\|_2^2 \nu_\psi (dz) dx\\
&\le   |K| \int_{ \{\|z\|_2 \leq  \eta\}} \|z\|_2^2 \mu_\bullet (dz)\\
&=  |K| \int_{\mathbb R^m} \1_{B (C_1\eta^\theta)}(p(u)) \|p(u)\|_2^2 \psi(u)du\\
&=  |K|  \int_{\mathbb R^m}  \1_{B(C_1)}(\delta_{\eta^{-\theta}}(p(u))) \|p(u)\|_2^2\psi(u)du\\
&=  |K|  \int_{\mathbb R^m}  \1_{B(C_1)}(p(\gamma^H_{\eta^{-\theta/\alpha}}(u))) \|p(u)\|_2^2\psi(u)du\\
&= |K|  \eta^{d_H\theta/\alpha}\int_{\mathbb R^m}  \1_{B(C_1)}(p(u))  \|p(\gamma^H_{\eta^{\theta/\alpha}}(u))\|_2^2 \psi(\gamma^H_{\eta^{\theta/\alpha}}(u))du \\
&=  |K|  \eta^{-\theta}  \int_{\mathbb R^m}  \1_{B(C_1)}(p(u))  \|\delta _{\eta^{\theta}}(p(u))\|_2^2 \psi(u)du\\
&\le  |K|\eta^{\theta\left(\frac{2}{\beta_+}-1\right)} \int_{\mathbb R^m}  \1_{B(C_1)}(p(u))  \|p(u)\|_2^2 \psi(u)du,\end{align*}
where $\beta_+=\max\{\beta_i\}<2$.
This integral is finite because $$ \1_{B(C_1)}(p(u)) \|p(u)\|_2^2\le C_2 \frac{\|u\|_2^2}{1+\|u\|_2^2}$$ and, using ``polar coordinates'' adapted to the dilation structure $(\gamma^H_t)_{t>0}$ on $H_*=(\mathbb R^m,*)$,
$$\int_{\mathbb R^m}\frac{\|u\|_2^2}{1+\|u\|_2^2}\psi(u)du \le 1+ \int_0^1 r^{-1+2-\alpha}dr,$$
because $\|u\|_2\le r$ if $\|\gamma^H_{1/r} u \|_2=1$ and $r\le 1$ (see (\ref{2gamma1})-(\ref{2gamma2})).
 \end{proof}

\subsection{Sufficient condition for (A) when $\mu\in \mathcal{SM}(\Gamma)$} \label{S:10.5} 

In this subsection, we explain why Theorem \ref{WTStab} applies to  a large class of examples in $\mathcal{SM}(\Gamma)$ that includes the two main examples described in Section \ref{sec-IntroExa}.  To give sufficient conditions for a measure in $\mathcal{SM}(\Gamma)$ to satisfy condition (A), we proceed component by component and follow the basic setup of Section \ref{S:10.3}.  Namely, we give sufficient conditions on a measure $\nu$ satisfying (\ref{mu-hyp1}) for the family $\nu_t=t\delta_{1/t}(\nu)$ to have a vague limit $\nu_\bullet$ on $\mathbb R^d\setminus \{0\}$.   Recall that $\nu$ is supported on a discrete subgroup $H$ contained as a co-compact closed subgroup in  a closed Lie subgroup $L$ of $G$.  
 The groups $G$ and $H$  
both have a global coordinate system $G=\mathbb R^d$ and $L=\mathbb R^m$. See Section \ref{S:10.3}.  We consider the following additional conditions:
\begin{enumerate}
\item[(SA1)] There exists an everywhere defined measurable non-negative function $\phi$ on $\mathbb R^m$ such that
  $\nu( i_H(v))=\phi(v)  $, where $i_H$ is the polynomial map from $\R^m$ to $\R^d$ defined by \eqref{iH}.  
  For $v\in \mathbb R^m, v\neq 0$, we set
\begin{equation}\label{phit} \phi^H_t (v)= t^{1+d_H/\alpha}\phi( \gamma^H_{t^{1/\alpha }}v).
\end{equation}
\item[(SA2)] There exists a continuous function $\psi:\mathbb R^m\setminus \{0\}\to [0,\infty)$ such that
\begin{equation} \label{barpsi1}
\hbox{for any } v\in \mathbb R^m\setminus \{0\},\quad |\phi^H_t(v)-\psi(v)|\le \eta(t) \bar{\psi}(v),\end{equation}
where, $ \lim_{t\to \infty}\eta(t)=0$  and $\bar{\psi}$ is locally bounded on $\mathbb R^m\setminus \{0\}$. \end{enumerate}

\begin{rem}
By the construction,  the function $\psi$ must satisfy
$$
\psi(v) \asymp |v|_{*}^{-(\alpha+d_H)} \quad  \hbox{for }  v\in \R^m\setminus \{0\},
$$
and
$$
\psi(\gamma^H_{t^{1/\alpha}}(v))=t^{-1-d_H/\alpha}\psi(v)  \quad \hbox{for }  v\in \R^m\setminus \{0\}  \hbox{ and }  t>0.
$$
\end{rem}

\medskip

\begin{rem}\label{R:10.6}
 Regarding hypothesis (SA1), two typical examples are:
\begin{enumerate}
\item[(i)] The function $\phi$ is a  continuous function on $\mathbb R^m$ and $\mu $ is defined in terms of $\phi$. For instance, this covers Example \ref{ex3-4}.

\item[(ii)]  The function $\phi$ may not be continuous but satisfies $\phi(xy)=\phi(x)$ for all $x\in H$ and $y\in \Omega_H$,
  where $\Omega_H $  is a relatively compact connected fundamental domain
for the action of $H$ on $L_H=\mathbb R^m$  (that is, $\Omega_H $  is a relatively compact connected  subset of
$L_H$ so that $L_H=\cup_{h\in H} h\Omega_H$).

\end{enumerate}
In this second case, we can define the function $\phi$ in terms of $\mu$ using the formula $$\phi(xy)=\mu(x), \quad x\in H, \, y\in \Omega_H.
$$
\end{rem}

\begin{exa}Consider the case when $\mu(h)=c (1+|h|_S)^{-\alpha-d_H}\1_H(h)$. Following  Remark \ref{R:10.6}(ii)  above, we can extend this function defined on $H$ to a function $\phi$ defined on $L_H=\mathbb R^m$ by setting $\phi$ to be  constant on the translates of a precompact fundamental domain.  For $x\in L_H$, let $\tilde{x}\in H$ be the representative of $x$ so that $ \tilde{x}^{-1}x\in \Omega_H, \;\tilde{x}\in H$. Then,
$$\phi^H_t(v)=c (t^{-1/\alpha}+|\widetilde{i_H(v)}|_S)^{-\alpha-d_H}, $$ and,  setting
$$\psi(v)=\frac{c}{|v|_{*}^{\alpha+d_H}},$$
Pansu's theorem (see \cite{Breuillard,Pansu1983}) gives
$$\lim_{t\to \infty}\phi^H_t(v)= \frac{c}{|v|_{*}^{\alpha+d_H}}=\psi(v).$$
Furthermore,
$$| \phi^H_t(v)-\psi(v)|\le C\frac{t^{-1/\alpha}}{|v|_{*}^{\alpha+d_H+1}}.$$
\end{exa}

\begin{pro} \label{pro-deslim}
Under assumptions
  \rm (SA1)-(SA2),
the measure
$$
\nu_t=t\delta_t^{-1}(\nu):
\quad \nu_t(f) := t\sum_{x\in H} f(\delta_t^{-1}(x))\nu(x) \quad \hbox{for } f\in
\mathcal C_c(\mathbb R^d\setminus \{0\}),
$$
 converges  vaguely  on $\mathbb R^d \setminus \{0\}$ to a 
  symmetric Radon  measure  $\nu_\bullet$ on  $\mathbb R^d \setminus \{0\}$ 
  given by
$$\nu_\bullet(f)=\int_{\mathbb R^m} f(p(u))\psi(u)du,\quad   f\in \mathcal C_c(\mathbb R^d\setminus \{0\}),$$
where $p:\mathbb R^m\to \mathbb R^d$ is defined by {\rm (\ref{mudelta})} and $\psi:\mathbb R^m\setminus \{0\}\to [0,\infty)$ by  {\rm (\ref{barpsi1})}.
\end{pro}

\begin{proof} This follows by a sequence of algebraic manipulation and approximations as follows.
  We use the notations introduced above and drop the superscript $H$ (if there is one), in particular,
  $d=d_H$, $\alpha$, $\gamma_t=\gamma^H_t$, $\phi$,
$\phi_t=\phi^H_t$, $\psi,$   $i=i_H:\mathbb R^m\to \mathbb R^d$, and the norm  $|\cdot|_*$ on $L_* $.
For  any $f\in \mathcal C_c(\mathbb R^d\setminus \{0\})$,   we have
 the scaled down copy of $H$ in $\mathbb R^m$
 \begin{align*}
\nu_t(f)&=   t\sum_{x\in H} f(\delta_t^{-1}(x))\nu(x)=   t\sum_{u\in i^{-1}(H) } f(\delta_t^{-1}(i(u)))\phi(u)\\
&=  t ^{-d/\alpha} \sum_{  u \in  \gamma_{t^{-1/\alpha}}(i^{-1}(H))} f(\delta_t^{-1}\circ i\circ \gamma_{t^{1/\alpha}}(u))\phi_t(u) \\
&= t ^{-d/\alpha} \sum_{  u\in   \gamma_{t^{-1/\alpha}}(i^{-1}(H))} f(\delta_t^{-1}\circ i\circ \gamma_{t^{1/\alpha}}(u))\psi (u) \\
&\quad+   t ^{-d/\alpha} \sum_{  u\in   \gamma_{t^{-1/\alpha}}(i^{-1}(H))} f(\delta_t^{-1}\circ i\circ \gamma_{t^{1/\alpha}}(u))(\phi_t(u)-\psi(u))\\
&=   t ^{-d/\alpha} \sum_{   u\in   \gamma_{t^{-1/\alpha}}(i^{-1}(H))} f(p(u))\psi(u) \\
&\quad +  t ^{-d/\alpha} \sum_{   u\in   \gamma_{t^{-1/\alpha}}(i^{-1}(H))} (f(\delta_t^{-1}\circ i\circ \gamma_{t^{1/\alpha}}(u))-f(p(u)))\psi(u) \\
&\quad+   t ^{-d/\alpha} \sum_{  u\in    \gamma_{t^{-1/\alpha}}(i^{-1}(H))} f(p(u))(\phi_t(u)-\psi(u))\\
&\quad+   t ^{-d/\alpha} \sum_{  u\in    \gamma_{t^{-1/\alpha}}(i^{-1}(H))} (f(\delta_t^{-1}\circ i\circ \gamma_{t^{1/\alpha}}(u))-f(p(u)))(\phi_t(u)-\psi(u))\\
&=  \Sigma_1(f,t)+\Sigma_2(f,t)+\Sigma_3(f,t)+\Sigma_4(f,t) .
\end{align*}
Now, the multivariate Riemann sum $\Sigma_1(f,t)$ of the continuous function $f\circ p \times \psi$ satisfies
$$\lim_{t\to \infty} \Sigma_1(f,t)=\int_{\mathbb R^m} f(p(u)) \psi(u) du$$ because, for any large real $R$,
\begin{align*}
\Sigma_1(f,t) &=   t ^{-d/\alpha} \sum_{u\in  \gamma_{t^{-1/\alpha}}(i^{-1}(H)); |u|_{*}\le R} f(p(u))\psi(u)
\\ &\quad+ t ^{-d_H/\alpha} \sum_{u\in  \gamma_{t^{-1/\alpha}}(i^{-1}(H)); |u|_{*}> R} f(p(u))\psi(u) .\end{align*}
The first term tends to $\int_{|u|_{*}\le R} f(p(u)) \psi(u) $, whereas the second term is bounded by
$CR^{-\alpha}$ because $\psi(u)\asymp |u|_{*}^{-\alpha-d_H}$. Similarly, $\int_{|u|_{*}>R}\psi(u)du \le CR^{-\alpha}$ and this proves the stated limit for  $\Sigma_1(f,t)$. Using our various hypotheses regarding $\mu$, the limits for
 $\Sigma_2(f,t), \Sigma_3(f,t)$ and $\Sigma_4(f,t)$  are easily seen to be equal to $0$.
\end{proof}

\subsection{The illustrative case of measures in  $\mathcal{SM}_1(\Gamma)$} \label{S:10.6} 

The simplest case illustrating the previous subsection is related to the treatment of measures in $\mathcal {SM}_1(\Gamma)$ when the building blocks have the form  $$\nu(g)=c_\alpha \sum_{k\in\mathbb Z} (1+|k|)^{-\alpha-1}\1_{\sigma^k}(g)$$ for some $\sigma\in \Gamma\subset G$, that is, $H=\langle \sigma\rangle \subset \Gamma\subset G$. Here, of course, $d_H=1$.  We use exponential coordinates of the first kind.
Recall that the element $\sigma\in H$ is of the form $\sigma=\exp(\zeta)=\zeta$ for some $$\zeta=(\zeta_1,\dots,\zeta_d) \in \mathfrak g=\mathbb R^d.$$ This is because the exponential map is the identity in our setup.
Define the function $\phi: V\to [0,\infty)$ by $\phi(x)=c_\alpha (1+|s|)^{-\alpha-1}$ if $x =s\zeta$ and $\phi(x)=0$ otherwise so that
$$\nu(f) = \sum_{y\in \{\sigma^k: k\in \mathbb Z\}} f(y)\phi(y).$$
Set $\phi_t(x)= c_\alpha(t^{-1/\alpha}+|s|)^{-\alpha-1}$ if $x=s\zeta$ and $0$ otherwise so that
$$\phi(t^{1/\alpha} s\zeta )=t^{-1-1/\alpha}\phi_t(s\zeta).$$
We also set  $\psi(s\zeta)=c_\alpha |s|^{-\alpha-1}$ for $s\neq 0$  and $\psi (y)=0$ if $y\notin \{s\zeta: s\in \mathbb R\}$ so that, for each $s\neq 0$,
$$\phi_t(z\zeta)-\psi(s\zeta)=c_\alpha\frac{|s|^{1+\alpha}-(t^{-1/\alpha}+|s|)^{1+\alpha}}{(t^{-1/\alpha}+|s|)^{1+\alpha} |s|^{1+\alpha}}
\to 0  \quad \mbox{as } t\to \infty.$$

Assume, in addition, that we are given an approximate dilation structure $\delta_t$ which can be expressed as
$$\delta_t(x)=(t^{w_1}x_1,\dots, t^{w_d}x_d)$$ in the basis of $\mathbb R^d$.  We want to understand the limit of $t\delta_t^{-1}(\nu)$ which is given on a continuous  function $f$ with compact support in $V\setminus \{0\}$ by
\begin{align*}t\delta_t^{-1}(\nu)(f)&=  t \sum_{y=\delta_t^{-1}( k\zeta) :k\in \mathbb Z }\phi (\delta_t (y)) f(y) \\
&=  t \sum_{y=\delta_t^{-1}( k\zeta) :k\in \mathbb Z }\phi (t^{1/\alpha} t^{-1/\alpha}\delta_t( y)) f(y)\\
&=  t^{-1/\alpha}  \sum_{y=t^{-1/\alpha } k\zeta :k\in \mathbb Z }\phi_t (y) f( \delta_t^{-1} (t^{1/\alpha} y)).
\end{align*}

Now we need to consider different cases depending on how $\delta^{-1}_t$ acts on $\zeta$.  Indeed,
$\delta_t^{-1} (t^{1/\alpha} s\zeta)= (t^{1/\alpha-w_i}\zeta_i)_1^d$.   In order to have a vague limit, we need to assume that, for every $i\in \{1,\dots,d\}$ such that $\zeta_i\neq 0$, $w_i\ge 1/\alpha$. If that is the case, then
$$\lim_{t\to \infty} \delta_t^{-1}(t^{1/\alpha} (s\zeta)) = s(\zeta^\infty_i)_1^d \quad \mbox{ with }
\zeta^\infty_i =p(\zeta) =\left\{\begin{array}{cl} \zeta_i & \mbox{ if } w_i=1/\alpha,\\
0 & \mbox{ otherwise.}\end{array}\right.
$$
Under this assumption (i.e., the approximate dilation structure $(\delta_t)_{t>0}$ is admissible for $\mu$), we write
\begin{align*}\lefteqn{t\delta_t^{-1}(\nu)(f)= t^{-1/\alpha}  \sum_{y=t^{-1/\alpha } k\zeta :k\in \mathbb Z }\phi_t (y) f( \delta_t^{-1} (t^{1/\alpha} y))}&&\\
&=  c_\alpha t^{-1/\alpha}  \sum_{y=t^{-1/\alpha } k\zeta :k\in \mathbb Z } \frac{ f(t^{-1/\alpha} kp(\zeta))}{(t^{-1/\alpha} |k|)^{1+\alpha}}\\
& \quad +
 c_\alpha t^{-1/\alpha}  \sum_{y=t^{-1/\alpha } k\zeta :k\in \mathbb Z } \frac{ [f( t^{-1/\alpha} k \delta_t^{-1} (t^{1/\alpha} (\zeta)))-
 f(t^{-1/\alpha} kp(\zeta))]}{(t^{-1/\alpha} |k|)^{1+\alpha}}\\
 & \quad +   t^{-1/\alpha}  \sum_{y=t^{-1/\alpha } k\zeta :k\in \mathbb Z } [f( t^{-1/\alpha} k \delta_t^{-1} (t^{1/\alpha} (\zeta))) (\phi_t(t^{-1/\alpha}k\zeta)-\psi(t^{-1/\alpha}k\zeta))].
  \end{align*}
 Because $f$  is   a compactly supported continuous  function
   in $V\setminus \{0\}$, the second and third sums tend to $0$ while the first sum
 tends to
 $$\int _{\mathbb R} c_\alpha |s|^{-\alpha-1} f(sp(\zeta))ds=\int _{\mathbb R}
 \psi(s\zeta) f(sp(\zeta))ds.$$

\section{Appendix: Nilpotent groups}\label{S:11}

\subsection{Definition of nilpotent groups}\label{S:11.1} 

In Section \ref{S:1}, we gave the classical definition of a nilpotent group and we recall it here. 

\begin{defin} \label{D:11.1} 
A nilpotent group is a group $G$ with identity element $e$ which has a central series of finite length, that is, there is a finite sequence of normal subgroups
 so that 
 $$
 \{e\} =K_0 \lhd K_1 \lhd  \cdots \lhd K_n=G
 $$
 with $K_{i+1}/K_i$ contained in the center of $G/K_i$ for $0\leq i \leq n -1$.
See, for example,  \cite[Definition  2.3]{CMZ}.
\end{defin}

An alternative definition of nilpotent group uses commutators. For two elements $x,y$ of a group $G$, the commutator of $x$ and $y$ is
$[x,y] :=x^{-1}y^{-1}xy$. For two subsets $A,B$ of $G$, $[A,B]$ denotes the group generated by all commutators $[a,b]$ for  $a\in A$ and $
b\in B$. See \cite[Lemma 1.4]{CMZ} for a collection of commutator identities. The lower central series of a group $G$ is defined inductively by setting $G_1=G$ and  $ G_{i+1}=[G,G_{i }] $ for $i\geq 1$.
It is a non-increasing sequence of subgroups of $G$. A group is nilpotent if and only if its lower central series terminates, that is,  there is an
integer  $r\geq 1 $ such that $G_i=\{e\}$, for all $i\ge   r+1 $. The smallest such $r$ is called the nilpotent class of the group $G$. 

\begin{exa} In the Heisenberg group of $3$ by $3$ upper-triangular matrices with diagonal entries equal to $1$, any commutator of length $3$, $[M_1,[M_2,M_3]]$, is the identity and there are elements that do not commute. Hence the Heisenberg group is nilpotent of class $2$. This applies to either the discrete Heisenberg group $\mathbb H_3(\mathbb Z)$ with integers matrix entries
or the real Heisenberg group $\mathbb H_3(\mathbb R)$ with real matrix entries.
\end{exa}  

\subsection{Definition of nilpotent Lie groups and Lie algebras} 
We refer the reader to \cite[Sections 1.1 and 1.2]{CorGr} for a short introduction to nilpotent Lie algebra and connected nilpotent Lie groups.  In the case of  a Lie algebra $(\mathfrak g,[\cdot,\cdot])$, the bracket $[\cdot,\cdot]$ is the key structural operation and the descending lower central series is defined inductively by $\mathfrak g_1=\mathfrak g$,
and $\mathfrak g_{i+1}=[\mathfrak g,\mathfrak g_i]$ for $i\ge 1$. 
  The Lie algebra is said to be nilpotent if there is an integer   $r\geq 0$  so that    $\mathfrak g_{r+1}=\{0\}$. The smallest such $r$ is the nilpotent class of $\mathfrak g$. 
   A connected Lie group is nilpotent 
 {according to Definition \ref{D:11.1} 
if and only if its Lie algebra is nilpotent.  Any simply connected nilpotent Lie group of topological dimension $d$ can be identified via the exponential map with $\mathbb R^d$ equipped with a group law given in coordinate by polynomial function. See, e.g.,  \cite[Theorem 1.2.1]{CorGr}.  The Campbell-Baker-Hausdorff formula 
(e.g., \cite[Page 11]{CorGr})   expresses the group product in this coordinate system.

\subsection{Embeddings into Lie groups}  \label{S:11.4}

Consider the following two natural questions. When can one embed a finitely generated torsion free 
nilpotent group $\Gamma$ as a co-compact subgroup into a nilpotent Lie group $G$?
Which connected simply connected nilpotent Lie group contains  a co-compact finitely generated subgroup? 

The first question is answered by constructions due to Malcev and P. Hall which provide such embeddings for any finitely generated torsion free nilpotent group. This is the subject of \cite[Chapter 4]{CMZ}. This result is used in this monograph, both as a black box, to embed $\Gamma$ as a co-compact subgroup  into a nilpotent Lie group $G$, and, more concretely, to construct coordinate systems. 

The answer to the second question is negative (there are connected simply connected nilpotent Lie groups that do not admit co-compact discrete subgroups). See, e.g.,  \cite[Theorem 5.1.8 and Example 5.1.13]{CorGr}. This result is not needed for the purpose of this monograph.  

\subsection{Volume growth}

A finitely generated group $\Gamma$ is  naturally equipped with the family of all word distances. A word distance is associated with a finite symmetric generating set $S$ (symmetric means that $g^{-1}\in S$ if $g\in S$). The length $|g|_S$ of an element $g$ is the least number $m$ of elements in $S$ that allow to write $g$ as a product $g=\sigma_1\dots\sigma_m$
  using   elements  $\sigma_i$ from $S$. 
By convention, $|e|_S=0$. The associated 
left-invariant  distance is $d_S(g,h)=|g^{-1}h|_S$.  Given two finite  symmetric generating sets $S$ and $T$, there are positive constants $a=a(S,T)$ and $A=A(S,T)$ such that
$$ \;a |g|_s\le |g|_T\le A|g|_S
\quad \hbox{for all } g\in   \Gamma.
$$

The volume growth of $\Gamma$ with respect to $S$ is 
$$V_S(t)=\#\{g\in \Gamma: |g|_S\le t\},$$
the number of points in any closed balls of radius $m$ in $(\Gamma, d_S(\cdot,\cdot))$.
If $S,T$ are two generating sets as above then there are positive constants $b=b(S,T)$ and $B=B(S,T)$ such that
$$ 
  bV_S( bt)\le V_T(m)\le BV_S(Bt)
  \quad \hbox{for all } t>0 . 
  $$
In the case of a finitely generated nilpotent group $\Gamma$ of nilpotent class $r$, the behavior of the volume growth function $V_S$ can be understood in terms of the lower central series $\Gamma_1=\Gamma$, $\Gamma_{i+1}=[\Gamma,\Gamma_i]$ for $i\geq 1$ as follows. The quotient groups $\Gamma_i/\Gamma_{i+1}$ are finitely generated abelian groups.   As any such group, the quotient  $\Gamma_i/\Gamma_{i+1}$ is the product of a finite abelian group and $\mathbb Z^{\ell_i}$ for some integer $\ell_i=\mbox{rank}(\Gamma_i/\Gamma_{i+1})$ which is called  the torsion-free rank of this abelian group. Set  
\begin{equation}\label{e:11.1}
D=D(\Gamma)=\sum_{j=1}^r j  \, \mbox{rank}(\Gamma_i/\Gamma_{i+1}).
\end{equation}
Then  there are constant $c=C(S)$ and $C=C(S)$ such that,
\begin{equation}\label{e:11.2}  
 c(1+t)^D \le V_S(t)\le C(1+t)^D
\quad \hbox{for all }  t \geq 0  . 
\end{equation}
See, e.g., \cite[Theorem VII.C.26]{delaHarpe} for references and comments on this result.

\medskip

If the nilpotent $\Gamma$ above is a discrete co-compact subgroup of a connected Lie group $G$ then, for any fixed left-invariant Riemannian metric on $G$, the Haar measure $|B(r)|$ of  the  ball of radius $r$ around the identity element $e$ satisfies
$$  c_1 r^D\le |B(r)|\le C_1r^D
\quad \hbox{for all } r\geq 1, 
$$
where $D=D(\Gamma)$ is as above. The positive constants $c_1,C_1$ depend on the choice of the Riemannian metric.

\begin{exa} The Heisenberg group $\mathbb H_3(\mathbb Z)$ is a co-compact subgroup of $\mathbb H_3(\mathbb R)$.
The elements of $\mathbb H_3(\mathbb Z)$ with at most one non-zero non-diagonal entry in the the top-right corner
is the center of $\mathbb H_3(\mathbb Z)$ as well as the commutator subgroup $[\mathbb H_3(\mathbb Z),\mathbb H_3(\mathbb Z)]$. It follows that the parameter $D=D(\mathbb H_3(\mathbb Z))$ is equal to $2+1\times 2 =4$.  Any left-invariant Riemannian metric on $\mathbb H_3(\mathbb R)$ has large-scale volume growth of type $r^4$.\end{exa}

\bigskip

 {\bf Acknowledgement.} The authors thank Mark Meerschaert for  references
\cite[Corollary 8.2.12]{MS1} and \cite[Theorem 4.1]{MS2}.
 The research of ZC is partially supported by Simons Foundation Grant 520542,
 TK by the Grant-in-Aid for Scientific Research (A) 17H01093 and 22H00099, 
Japan, LSC by NSF grants DMS-1707589 and DMS-2054593, and JW by NNSFC grant 11831014 and 12071076.

 \vskip 0.3truein
 
 \providecommand{\bysame}{\leavevmode\hbox to3em{\hrulefill}\thinspace}
\providecommand{\MR}{\relax\ifhmode\unskip\space\fi MR }
\providecommand{\MRhref}[2]{%
  \href{http://www.ams.org/mathscinet-getitem?mr=#1}{#2}
}
\providecommand{\href}[2]{#2}

\vskip 0.3truein

\footnotesize{
 {\bf Zhen-Qing Chen}

Department of Mathematics, University of Washington, Seattle,
WA 98195, USA.

E-mail: {\tt zqchen@uw.edu}

\bigskip

{\bf Takashi Kumagai}

Department of Mathematics, Faculty of Science and Engineering, 
Waseda University, 

3-4-1 Okubo, Shinjuku-ku, Tokyo 169-8555, Japan.

E-mail: {\tt t-kumagai@waseda.jp}

\bigskip

{\bf Laurent Saloff-Coste}

Department of Mathematics, Cornell University, Ithaca, NY 14853, USA.

E-mail: {\tt lsc@uno.math.cornell.edu}

\bigskip

{\bf Jian Wang}

College of Mathematics and Informatics, \\
\indent Fujian Key Laboratory of Mathematical Analysis and Applications (FJKLMAA),\\  \indent Center for Applied Mathematics of Fujian Province (FJNU),\\
\indent Fujian Normal University, Fuzhou 350007,
P.R. China. E-mail: {\tt jianwang@fjnu.edu.cn}

\bigskip

{\bf Tianyi Zheng}

Department of Mathematics, UC San Diego, San Diego, CA 92093-0112, USA.

E-mail: {\tt tzheng2@math.ucsd.edu}
 
\end{document}